\documentclass[a4paper]{article}

\usepackage[a4paper,  left=2.5cm,  right=2.5cm, top=3cm, bottom=2.5cm]{geometry}

\usepackage[utf8]{inputenc}
\usepackage[T1]{fontenc}

\usepackage{bbm}

\usepackage{amsmath}
\usepackage{amssymb}
\usepackage{amssymb}

\usepackage{tikz}
\usepackage{pgfplots}
\usepackage{subcaption} %%%For subfigures...
\usetikzlibrary{decorations.pathmorphing}
\usepackage{nomencl}
\makenomenclature
\usepackage{mathrsfs}

\usepackage{ifthen}

\usepackage{xr} %%%% zref-xr= xr+Hyperref 

\usepackage{amsthm}
\newtheorem{theorem}{Theorem}[section] 
\newtheorem{lemma}[theorem]{Lemma}

\newtheorem{assumption}[theorem]{Assumption}
\newtheorem{example}[theorem]{Example}
\newtheorem{remark}[theorem]{Remark}
\newtheorem{definition}[theorem]{Definition}
\newtheorem{notation}[theorem]{Notation}

\newtheorem{sol}{Solution}

\let\PROOF=\proof
\renewcommand\proof{\PROOF[\bfseries\proofname]}

\newcommand{\intL}{\int\displaylimits}

\newcommand{\RarrBold}{\textbf{``\boldmath$\Rightarrow$''}}
\newcommand{\LarrBold}{\textbf{``\boldmath$\Leftarrow$''}}

\newcommand{\bbA}{\mathbb{A}}

\newcommand{\C}{\mathbb C}
\newcommand{\D}{\mathbb D}
\newcommand{\E}{\mathbb E}
\newcommand{\Gen}{\mathbb L}
\newcommand{\M}{\mathbb M}
\newcommand{\N}{\mathbb N}
\renewcommand{\P}{\mathbb P}

\newcommand{\R}{\mathbb R}
\newcommand{\T}{\mathbb{T}}
\newcommand{\W}{\mathbb W}
\newcommand{\Z}{\mathbb Z}

\newcommand{\Cem}{\mathscr{C}}

\newcommand{\calA}{\mathcal{A}}

\newcommand{\calH}{\mathcal{H}}

\newcommand{\calK}{\mathcal{K}}
\newcommand{\calL}{\mathcal{L}}
\newcommand{\calM}{\mathcal{M}}
\newcommand{\calN}{\mathcal{N}}

\newcommand{\calW}{\mathcal{W}}

\newcommand{\sfC}{{\sf C}}

\newcommand{\relE}[2]{\calH \lb \left. #1 \middle| #2 \right. \rb } %%% H( #1 | #2 ) Relative entropy

\newcommand{\Zd}{\Z^{d}}

\newcommand{\Td}{\T^{d}}
\newcommand{\TN}{\Td_{N}}
\newcommand{\TR}[1][]{\Td_{#1} \times \R}

\newcommand{\Ti}{\lbs 0,T \rbs}
\newcommand{\Tihop}{\lb 0,T \rbs} %%%% Ti half open at 0
\newcommand{\TTR}[1][\Ti]{#1 \times \TR}

\newcommand{\ThopTR}{\TTR[\Tihop]}

\newcommand{\TT}[1][\Ti]{#1 \times \Td}

\newcommand{\TiR}[1][\Ti]{#1 \times \R}

\newcommand{\Wsp}{\calW}
\newcommand{\w}{w}
\newcommand{\TW}{\Td \times \Wsp}
\newcommand{\TWR}[1][]{\Td_{#1} \times \Wsp \times \R}
\newcommand{\TWC}{\TW \times \Csp{\Ti}}

\newcommand{\WC}{\Wsp \times \Csp{\Ti}}
\newcommand{\TTWR}[1][\Ti]{#1 \times \TWR}

\newcommand{\CspSymbol}{\sfC}%%%% space of continuous functions
\newcommand{\CspUp}[1]{\CspSymbol^{#1}} %%% C^{#1}
\newcommand{\CspLoUp}[2]{\CspUp{#2}_{#1}} %%% C^{#1}
\newcommand{\CspHelper}[3][]{\CspLoUp{#1}{#2} \!\lb #3 \rb } %%% C^{#2}_{#1} (#3)
\newcommand{\Csp}[2][]{\CspHelper{#1}{#2}  } %%% C^{#1}(#2)
\newcommand{\CspL}[3][]{\CspHelper[#2]{#1}{#3} } %%% C^{#1}_{#2}(#3)

\newcommand{\MspUp}[3][]{\M_{#1}^{#2} \! \lb #3 \rb} %%%%% M_{#1}^{#2} (#3)
\newcommand{\MOneUp}[2]{\MspUp[1]{#1}{#2}} %%%%% M_{1}^{#1} (#2)
\newcommand{\MOne}[1]{\MOneUp{}{#1}} %%%%% M_{1} (#2)
\newcommand{\MOneL}[1]{\MOneUp{L}{#1}} %%%%% M_{1}^{L} (#1)

\newcommand{\Nd}{N^{d}}
\newcommand{\RN}{\R^{N^{d}}}

\def\1{\ifmmode {1\hskip -3pt \rm{I}}
\else {\hbox {$1\hskip -3pt \rm{I}$}}\fi} %indicator

\newcommand{\dd}{{\rm d}} %%%Fuer integrale

\DeclareMathOperator*{\supp}{supp} %%%% Esssup

\newcommand{\sProOne}[2][]{\ensuremath \left\langle  #2 \right\rangle_{
		 \ifthenelse{\equal{#1}{}}{}{#1}
		}} %%%% < #2 >_{#1}
\newcommand{\sPro}[3][]{\sProOne[#1]{#2 , #3}}        %%%% < #2,#3 >_{#1}
\newcommand*{\defeq}{\mathrel{\vcenter{\baselineskip0.5ex \lineskiplimit0pt
                     \hbox{\scriptsize.}\hbox{\scriptsize.}}}%
                     =}
\newcommand*{\eqdef}{=\mathrel{\vcenter{\baselineskip0.5ex \lineskiplimit0pt
                     \hbox{\scriptsize.}\hbox{\scriptsize.}}}%
                     }             
\newcommand{\ul}[1]{\protect\underline{#1}} %%Underline
\newcommand{\ol}[1]{\overline{#1}} %%Overline
\newcommand{\what}[1]{\widehat{#1}}

\def\rightmarkWennLeerDannLinks{
\ifthenelse{\equal{\rightmark}{}}{\leftmark}{\rightmark}
}

\newcommand{\linInt}[1][]{\pi^{N}_{\ul{\theta}^{N}_{\HelperISn{#1}} \ifthenelse{\equal{#1}{}}{}{(#1)}  }}

\newcommand{\empM}[1][]{ \empMHelper[#1]{N}}

\newcommand{\HelperISn}[1]{\ifthenelse{\equal{#1}{n}}{n}{}}
\newcommand{\HelperISNOTn}[1]{\ifthenelse{\equal{#1}{n}}{}{(#1)}}

\newcommand{\empMHelper}[2][]{ \ensuremath \mu^{#2}_{\underline{\theta}^N_{\HelperISn{#1}}
\ifthenelse{\equal{#1}{}}{}{\HelperISNOTn{#1}}
}}

\newcommand{\empP}{\ensuremath \mu^{N}_{\underline{\theta}^N_{\Ti} }}
\newcommand{\empPt}[1]{\ensuremath \mu^{N}_{\underline{\theta}^N_{#1} }}

\newcommand{\empPSectionTitleHelper}[1][]{\empP}

\newcommand{\abs}[1]{\left| #1\right|}

\newcommand{\absabs}[1]{  \left\Vert #1 \right\Vert} % || #1 ||

\newcommand{\iNorm}[2][]{\ensuremath \abs{#2}_{ \infty #1 }} %%% || #2 ||_{\infty #1}
\newcommand{\Lp}[1]{L^{#1}} 
\newcommand{\Ltwo}{\Lp{2}}

\newcommand{\LtwoN}[2][]{\LpN[#1]{2}{#2}} % || #2 || _{L^{2}(#1)}

\newcommand{\LpN}[3][]{\NormHelper[#1]{\Lp{#2}}{#3}}  % || #3 || _{L^{#2}(#1)}

\newcommand{\NormHelper}[3][]{\ensuremath \absabs{#3}_{
		#2
		\ifthenelse{\equal{#1}{}}{}{(#1)}
		}}							% || #3 || _{#2 (#1)}
		
\DeclareMathOperator{\Proj}{Proj} %%%% Projection
\DeclareMathOperator{\distOp}{dist} %%%% Distance

\newcommand{\dist}[2]{\distOp \lbr #1 , #2 \rbr }

\newcommand{\floor}[1]{\lfloor #1 \rfloor}

\newcommand{\emptyline}{\vspace{0.5\baselineskip}\par}

\newcommand{\kreis}[1]{\unitlength1ex\begin{picture}(2.5,2.5)%
\put(1.25,0.75){\circle{2.5}}\put(1.25,0.75){\makebox(0,0){#1}}\end{picture}}

\newcommand\encircle[1]{%
  \tikz[baseline=(X.base)] 
    \node (X) [draw, shape=circle, inner sep=-0.25mm] {\strut #1};}

\newcommand{\oh}{\frac{1}{2}}

\newcommand{\lb}{\left(}
\newcommand{\rb}{\right)}
\newcommand{\lbr}{\left\{}
\newcommand{\rbr}{\right\}}
\newcommand{\lbs}{\left[}
\newcommand{\rbs}{\right]}

\let\OLDinf\inf

\renewcommand{\inf}[1][]{
\ifthenelse{\equal{#1}{}}{\OLDinf}{\OLDinf \lbr #1 \rbr}
}

\newcommand{\Holder}{H\"older }

\newcommand{\Ito}{It\^{o}}

\usepackage{enumitem}

\newlist{itemNoIntend}{itemize}{3}
\newlist{itemNoLeftIntend}{itemize}{1}
\newlist{itemNoItemsep}{itemize}{1}
\setlist[itemNoIntend]{
	topsep=0em,
	partopsep=0em,
	parsep=0em,
	itemsep=0em,
	leftmargin=0cm,
	itemindent=0cm,
	labelwidth=0cm,
	labelsep=0cm,
	align=left
}
\setlist[itemNoLeftIntend,1]{
	leftmargin=1em,
	itemindent=0cm,
	align=left
}
\setlist[itemNoItemsep,1]{
	itemsep=0em,
	align=left
}

\setlist*[itemNoIntend,itemNoLeftIntend,itemNoItemsep]{
	label=\textbullet
}

\newlist{enuAlph}{enumerate}{5}
\newlist{enuAlphNoIntendBf}{enumerate}{5}

\setlist[enuAlphNoIntendBf,1]{%
	label=\emph{\alph*.)},
	ref=\alph*.)
}

\setlist[enuAlph,1]{%
	label=\emph{\alph{enuAlphi}.)},
	ref=\alph{enuAlphi}.)
}
\setlist[enuAlph,2]{%
	label=\emph{\alph{enuAlphi}.\alph{enuAlphii})},
	ref=\alph{enuAlphi}.\alph{enuAlphii})
}

\newlist{enuRom}{enumerate}{5}
\newlist{enuRomNoIntend}{enumerate}{5}
\newlist{enuRomSimple}{enumerate}{5}
\newlist{enuRomSimpleBf}{enumerate}{5}
\newlist{enuRomNoIntendBf}{enumerate}{5}

\setlist*[enuRom,enuRomSimple,enuRomSimpleBf,enuRomNoIntendBf,enuRomNoIntend]{%
	label=\emph{(\roman*)},
	ref=(\roman*) 
	}

\setlist*[enuRomSimpleBf,enuRomNoIntendBf]{%
	label=\textbf{\emph{(\roman*)}}
	}

\setlist*[enuRomNoIntendBf,enuRomNoIntend]{%
	topsep=0em,
	leftmargin=0cm,
	itemindent=0.25em,
	labelwidth=0cm,
	labelsep=0.25em,
	align=left
}

\newlist{enuDbBr}{enumerate}{5}

\setlist[enuDbBr,1]{%
	label=(\arabic*.),
	ref=(\arabic*.)
}

\newlist{enuBr}{enumerate}{5}

\setlist[enuBr,1]{%
	label=\arabic*.),
	ref=\arabic*.)
}

\setlist[enuBr,2]{%
	label=\arabic{enuBri}.\arabic{enuBrii}.),
	ref=\arabic{enuBri}.\arabic{enuBrii}.)
}

\newlist{enuSimple}{enumerate}{5}
\newlist{enuSimpleBf}{enumerate}{5}

\newlist{steps}{enumerate}{5}
\newlist{stepsInt}{enumerate}{5}

\setlist*[enuRomSimple,enuSimple,enuSimpleBf,steps,enuRomSimpleBf]{%
	topsep=0em,
	partopsep=0em,
	parsep=0em,
	itemsep=0em,
	leftmargin=0cm,
	itemindent=0cm,
	labelwidth=0cm,
	labelsep=0cm,
	align=left,
	listparindent=\parindent
}
\setlist*[enuSimpleBf]{%
	label=\textbf{\arabic*. }
	}
\setlist*[stepsInt]{
	leftmargin=1.6cm
}

\setlist*[enuSimple,steps,stepsInt]{%
	label=\arabic*.
} %%%Nexte ebene label*=... Hinzuf�gen.

\setlist*[stepsInt,1]{label=\textbf{Step \arabic{stepsInti}: },ref=Step~\arabic{stepsInti}}

\setlist*[steps,1]{label=\textbf{Step \arabic{stepsi}: },ref=Step~\arabic{stepsi}}
\setlist*[steps,2]{label=\textbf{Step \arabic{stepsi}.\arabic{stepsii}: },ref=Step~\arabic{stepsi}.\arabic{stepsii}}
\setlist*[steps,3]{label=\textbf{Step \arabic{stepsi}.\arabic{stepsii}.\arabic{stepsiii}: },ref=Step~\arabic{stepsi}.\arabic{stepsii}.\arabic{stepsiii}}
\setlist*[steps,4]{label=\textbf{Step \arabic{stepsi}.\arabic{stepsii}.\arabic{stepsiii}.\arabic{stepsiv}: },ref=Step~\arabic{stepsi}.\arabic{stepsii}.\arabic{stepsiii}.\arabic{stepsiv}}

\newcommand{\step}[1][]{\item  \ifthenelse{\equal{#1}{}}{}{ \textbf{#1:}}}

\newcommand{\itemAdd}[1]{\item \textbf{#1:}}
\usepackage[notref,notcite]{showkeys} %%% Options: notref, notcite
\renewcommand{\showkeyslabelformat}[1]{%
\fbox{
\begin{minipage}[t]{1.8cm}
\normalfont\fontsize{3}{3}\ttfamily#1
\end{minipage}}
}

\usepackage[colorlinks]{hyperref}

\usepackage[small]{titlesec}

\numberwithin{equation}{section}

\renewcommand\thesubsubsection{\thesubsection.\arabic{subsubsection}}
\makeatletter
\renewcommand\paragraph{\@startsection{paragraph}{4}{\z@}%
            {-2.5ex\@plus -1ex \@minus -.25ex}%
            {1.25ex \@plus .25ex}%
            {\normalfont\normalsize\bfseries}}
\makeatother

\usepackage{fancyhdr}
\title{Hydrodynamic Limit and Propagation of Chaos for a continuous spin Langevin dynamics with long range interaction}
\author{current state of work on this topic of Bovier, Ioffe, Mueller}

\makeindex

\renewcommand{\empM}{\mu^{N}}
\newcommand{\empMThe}[1][]{\mu^{N}_{\ul{\theta}^{N}_{#1}}}
\renewcommand{\empP}{\mu^{N}_{\Ti}}
\renewcommand{\empPt}[1]{\ensuremath \mu^{N}_{#1} }

\renewcommand{\showkeyslabelformat}[1]{}
\makeatletter
\newcommand{\dateFT}[1]{%
  \let\@oldtitle\@title%
  \gdef\@title{\@oldtitle\footnotetext{\emph{Date:} #1}}%
}
\newcommand{\subjclass}[2][1991]{%
  \let\@@oldtitle\@title%
  \gdef\@title{\@@oldtitle\footnotetext{#1 \emph{Mathematics subject classification.} #2}}%
}
\newcommand{\keywords}[1]{%
  \let\@@@oldtitle\@title%
  \gdef\@title{\@@@oldtitle\footnotetext{\emph{Key words and phrases.} #1.}}%
}
\renewcommand{\thanks}[1]{%
  \let\@@@@oldtitle\@title%
  \gdef\@title{\@@@@oldtitle\footnotetext{#1}}%
}

\renewcommand*{\title}[2][]{\gdef\shorttitle{#1}\gdef\@title{#2}}
\makeatother

\renewenvironment{abstract}
               {\list{}{\rightmargin\leftmargin}%
                \item[\hspace{\leftmargin}\emph{Abstract.}]\relax}
               {\endlist}

\begin{document} 

\title[Path large deviations]{Path large deviations for interacting diffusions with local mean-field interactions in random environment}

\author{Patrick E.  M\"uller}

%%%	 \address{P. E. M\"uller\\Institut f\"ur Angewandte Mathematik\\
%%%	Rheinische Friedrich-Wilhelms-Universit\"at\\ Endenicher Allee 60\\ 53115 Bonn, Germany }
%%%	\email{patrick.mueller@iam.uni-bonn.de}

%%%\AtEndDocument --Prob: Seitenumbruch
	\newcommand{\AdresseAtEnd}{\bigskip\bigskip{\footnotesize%
	  \textsc{P. E. M\"uller\\Institut f\"ur Angewandte Mathematik\\
		 	Rheinische Friedrich-Wilhelms-Universit\"at\\ Endenicher Allee 60\\ 53115 Bonn, Germany } \par  
	  \textit{E-mail address}: \texttt{patrick.mueller@iam.uni-bonn.de} 
	}}

	%\date{\today}
	\date{\vspace{-5ex}}
	\dateFT{\today}


\thanks{P.E.M. is partially supported through the German Research Foundation in 
the Collaborative Research Center 1060 "The Mathematics of Emergent Effects", 
and  the Bonn International Graduate School in Mathematics (BIGS) in 
the Hausdorff Center for Mathematics (HCM)
}

\maketitle

\begin{abstract}
We consider a system of $\Nd$ spins  in random environment with a random local mean field type interaction.
Each spin has a fixed spatial position on the torus $\Td$, an attached random environment and a spin value in $\R$ that evolves according to a space and environment dependent Langevin dynamic. 
The interaction between two spins depends on the spin values, on the spatial distance and the random environment of both spins.
We prove the path large deviation principle from the hydrodynamic (or local mean field McKean-Vlasov) limit and derive different expressions of the rate function for the empirical process and for the empirical measure of the paths.
To this end, we generalize an approach of Dawson and G\"artner.
By the space and random environment dependency, this requires new ingredients compared to mean field type interactions.
Moreover, we prove the large deviation principle by using a second approach.
This requires a generalisation of Varadhan's lemma to nowhere continuous functions. 
\end{abstract}

\tableofcontents

\section{Introduction}
\label{sec::Introduction}

We consider a system of $\Nd$ interacting spins with spin values $ \ul{\theta}^{N}_{t} \in \RN$ at time $t>0$, that evolve according to the following Langevin dynamics
\begin{align}
\label{eq::SDE}
\begin{split}
	\dd \theta^{k,N}_{t} 
	&=
	b \left( \frac{k}{N}, \w^{k,N} , \theta^{k,N}_{t} , \empPt{t}\right) \dd t
	+
	\sigma
	\dd B^{k,N}_{t}
	,
	\\
	\theta^{k,N}_{0}
	&\sim
	\nu_{k} \in \MOne{\R}
	,
\end{split}
\end{align}
for $k \in \TN=\Zd / N \Zd$, the periodic $d$-dimensional lattice of length $N$.
Besides the random initial spin values, there are two sources of randomness in this system. 
On the one hand, a random environment $\w^{k,N} \in \Wsp \subset \R^{m}$ is attached to each spin. It expresses differences in the environment and in the nature of the spins. The random environment is distributed according to $ \zeta_{\frac{k}{N}} \in \MOne{\Wsp}$ and it is frozen over time.
On the other hand each spin value is subject to independent stochastic fluctuations, given by independent Brownian motions $B^{k,N}$.

We consider very general drift coefficients $b: \TWR \times \MOne{\TWR} \rightarrow \R$. This coefficient has to be continuous at least on a subset of the probability measures,  but it might be unbounded for example.
The drift coefficient depends on the fixed normalised spatial position $\frac{k}{N}$ of the spin, 
on the random environment $\w^{k,N}$  attached to this spin
and on the current spin value.
Moreover $b$ depends through the the empirical measure $\empPt{t}$, defined as
\begin{align}
\label{def::EmpMt}
	\empPt{t} 
	\defeq
	\frac{1}{\Nd } \sum_{k \in \TN} \delta_{ \left( \frac{k}{N} , \w^{k,N} , \theta^{k,N}_{t}  \right)}	
	\in 
	\MOne{\TWR}
	,
\end{align}
also on the spatial positions, the random environments and the spin values of all other spins.
This dependency of the drift coefficient on the empirical measure $\empPt{t}$ models the interaction between the spins.
Note that the geometric structure of the system, i.e. the spatial position of the spins, is highly relevant. On the one hand the drift coefficient depends directly on $\frac{k}{N}$ and on the random environment $\w^{k,N}$.
On the other hand the spatial positions of all spins and the full random environment take effect in the interaction through the empirical measure.
Moreover the initial distribution and the distribution of the random environment depend on the spatial position.

\emptyline
The following model is a concrete example of a spin system with local mean field interaction, covered by the more general model \eqref{eq::SDE}. 
Let the diffusion coefficient $\sigma$ be equal to $1$, take $\Wsp \subset \R^{m}$ compact  and choose
the drift coefficient as
\begin{align}
\label{eq::DriftCoef::LocalMF}
		b \left( x , \w , \theta  , \mu \right)
		=
		-\partial_{\theta} \Psi \left( \w,\theta \right) 
		+
		\int_{\TWR}
			J \left( x-x' , \w, \w' \right) \theta'
		\mu \left( \dd x', \dd \w' , \dd \theta' \right)
		,
\end{align}
for $\left( x, \w, \theta, \mu \right) \in \TWR \times \MOne{\TWR}$, with $\Psi$ a single spin potential and $J$ a weight function of the spatial distance between the spins.
For example $\Psi$ can be chosen as $\Psi \left( \w, \theta \right) = \theta^{4}+\w_{1} \theta$ or $\theta^{2}+\w_{1} \theta$ or $\theta^{4}-\theta^{2}+\w_{1} \theta$.
Then the first coordinate of the random environment $\w_{1}$ represents a random chemical potential.
With these coefficients, the SDE \eqref{eq::SDE} is given by
\begin{equation}
\label{SDE::LocalMF}
	\dd \theta^{k,N}_{t}  
	=  
		-  \partial_{\theta} \Psi \left( \w^{k,N}, \theta^{k,N}_{t}  \right)  \dd t
		+
		\frac{1}{\Nd } \sum_{ j \in \TN}  J\left(\frac{k-j}{N} , \w^{k,N}, \w^{j,N}  \right)  \theta^{j,N}_{t}  \dd t
			+ \dd B^{k,N}_{t}
	.
\end{equation}
We are mainly interested in this model. However we show most of the results in a much more general setting.

\emptyline

Given a realisation $\ul{\theta}^{N}_{\Ti} = \left\{ t \mapsto \ul{\theta}^{N}_{t} \right\}$ of the solution of \eqref{eq::SDE} and a realisation of the random environment $\ul{\w}^{N}$, let us denote by $\empP $ the empirical process, that is the path of the empirical measures $\empPt{t}$ defined in \eqref{def::EmpMt}, i.e.
\begin{align}
\label{def::empM}
	\empP 
		\defeq
		\left\{ 
			t 
			\mapsto 
			\empPt{t}
			\defeq
			\frac{1}{\Nd } \sum_{k \in \TN} \delta_{ \left( \frac{k}{N}, \w^{k,N} , \theta^{k,N}_{t}  \right)}
		\right\}
 		\in
 		\Csp{\Ti  , \MOne{\TWR}}
 		,
\end{align}
and by $L^{N}$ the empirical measure on $\TWC$
\begin{align}
\label{def::LN}
	L^{N} 
	=
	L^{N} \left( \ul{\w}^{N},\ul{\theta}_{\Ti} \right)
	\defeq 
		\frac{1}{\Nd } \sum_{k \in \TN} \delta_{ \left( \frac{k}{N} , \w^{k,N}, \theta^{k,N}_{\Ti}  \right)}
	\in
	\MOne{ \TWC } 
	.
\end{align}

One can show, in particular for the local mean field system \eqref{SDE::LocalMF}, that the empirical process $\empP$, converges to a deterministic continuous trajectory on $\MOne{\TWR}$ when the number of spins $N$ tends to infinity (see~\cite{BoIoMuHydro} when removing the random environment or see \cite{LucStaMFL} for a slightly different model with bounded interaction).
Without the random environment, each measure on this trajectory has a density $\xi_{t}$ with respect to the Lebesgue measure.
Moreover $\xi \in \Csp[1,0,2]{\ThopTR}$ and the time evolution of $\xi$ is the classical solution of the following PDE (that we call local mean field McKean-Vlasov equation),
\begin{align}
\label{eq::PDEHydro}
\begin{split}
\partial_{t} \xi_{t} \left( x,\theta \right)
		=
		\partial_{\theta}
		\left(
		\left(  \Psi' \left( \theta \right)
			- \int_{\TR} J \left( x'- x \right) \theta'  \xi_{t} \left( x',\theta' \right) \dd\theta' \dd x'  
		\right) 
		\xi_{t} \left( x,\theta \right)
		\right)
		+
		 	\frac{1}{2} \partial^{2}_{\theta} \xi_{t} \left( x,\theta \right)
	.
\end{split}
\end{align}

The aim of the current paper is to investigate the large deviations from the hydrodynamic limit for the general 
system \eqref{eq::SDE} and in particular for the local mean field system \eqref{SDE::LocalMF}.
This is motivated by the overall desire for understanding metastability and other long time phenomena 
in models like \eqref{eq::SDE}.
For infinite dimensional system an approach to answer these questions is a generalisation of 
the Freidlin-Wentzell theory (see the introduction of \cite{DawGarLDP}, \cite{FreiWentz} 
Chapter~10.5, \cite{FariJonaLa}).
By the Freidlin-Wentzell theory, these questions are related to large deviation principles (see 
\cite{FreiWentz}).

The heuristic underlying idea is, that for large $N$, the system usually follows the deterministic flow 
described by the hydrodynamic limit (e.g. \eqref{eq::PDEHydro}) towards a globally attracting, stable 
point in $\MOne{\TR}$. 
However for finite $N$ the randomness allows the system to deviate from the deterministic flow. 
Even a transition from one stable point to another might occur.
These deviations from the deterministic flow are exponentially unlikely for large $N$.
Therefore large deviation is an appropriate theory to determine the probability of such an event.
\emptyline

We prove in the following  that the  families of random elements $\left\{ \empP \right\}$ and $\left\{ L^{N} \right\}$, 
satisfy large deviation principles.
Moreover, we derive different representation of the rate functions and show relations between the two principles, the rate functions and the minimizer of the rate function. 
In particular we show that the rate function $S_{\nu,\zeta}$ corresponding to the family $\left\{ \empP \right\}$ has the following expression
\begin{align}
\label{eq::Sheuristisch}
		S_{\nu,\zeta} \left( \mu_{\Ti} \right)
		=
		\int_{0}^{T} \int_{\Wsp} \abs{ \partial_{t}\mu_{t}   - \left( \Gen_{\mu_{t},.,.} \right)^{*}  \mu_{t}  }^{2}_{\mu_{t} }  \dd t
		+
		\relE{\mu_{0}}{\dd x \otimes  \zeta_{x} \otimes \nu_{x}}
		,
\end{align}
for suitable $\mu_{\Ti} \in \Csp{\Ti, \MOne{\TWR}}$. 
We define the norm $\abs{.}_{\mu_{t}}$ and the relative entropy $\relE{.}{.}$ later. The operator $\left( \Gen_{\mu,x,\w} \right)^{*}$ is for each $\mu \in \M_{\varphi,\infty}$ and $\left( x,\w \right) \in \TW$ the formal adjoint of the following operator, acting on $f \in \CspL[2]{b}{\R}$,
\begin{align}
	\label{def::InteractingGenerator}
	\Gen_{\mu, x,\w } f \left( \theta \right)  
		\defeq
		 \frac{\sigma^{2}}{2} \partial^{2}_{\theta^{2}} f \left( \theta \right)  
		+
		b \left( x, \w, \theta, \mu \right)  \partial_{\theta} f \left( \theta \right) 
		.
\end{align}
It is at this stage important to see, that the rate function $S$ measures somehow the deviation from the hydrodynamic equation.
This is what we expect from the mentioned heuristics.

\emptyline
Dynamic large deviation principles, for models similar to \eqref{eq::SDE}, but with the huge difference that the spatial structure is not relevant, are considered by many authors (e.g. \cite{DawGarLDP}, \cite{Tan}, \cite{BruFiniteKullb}, \cite{FenKur}, \cite{PraHolMKV}, \cite{Cabana:2016qy}).
In these models, the empirical measure in the drift coefficient is replaced by $\mu^{N,MF}_{\ul{\theta}^{N}_{t} } \defeq \frac{1}{\Nd} \sum_{k \in \TN} \delta_{\theta^{k,N}_{t}}$.
In the concrete model \eqref{SDE::LocalMF}, this is for example the case when $J \equiv 1$, i.e when a mean field interaction (Curie-Weiss model) is considered.

For these kind of models, with weak interaction and irrelevance of the spatial structure and without random environment, a dynamical large deviation principle for $\left\{ \mu^{N,MF}_{\ul{\theta}^N_{\Ti} } \right\}$ is derived in \cite{DawGarLDP}. This principle is used in \cite{DawGarLdpFrEn} to connect the quasi potential with the free energy function.
The idea of the approach in \cite{DawGarLDP} is to freeze an empirical process in the drift coefficient to get a system of $N^{d}$ independent time inhomogeneous spins. For this system a large deviation principle is derived. Finally this large deviation principle is converted to the large deviation principle for the interacting system. The main difficulty thereby is to show that the rate function has the particular form similar to \eqref{eq::Sheuristisch}.
In this paper we generalize the approach of \cite{DawGarLDP}, to derive a large deviation principle for the space and random environment dependent empirical processes $\left\{ \empP \right\}$ and also for the empirical measures $\left\{ L^{N} \right\}$.
Changes are in particular required due to the space and random environment dependency of the drift coefficient, of the empirical process and of the initial data.
Moreover we consider the space of continuous functions on the usual space of probability measures $\Csp{\Ti , \MOne{\TWR}}$, equipped with the usual topologies (the uniform topology and the weak convergence) and not, as in \cite{DawGarLDP}, a subset of this space with a stronger topology.
Note that the diffusion coefficient in \cite{DawGarLDP} depends on the spin value and we consider only the case of a constant diffusion coefficient.
However all the results of this paper also hold for non constant diffusion coefficient. 
We only fixed it to be constant, to simplify the notation.

\emptyline
In \cite{Tan} a large deviation principles for the empirical measure
$ L^{N,MF} 
	\defeq 
		\frac{1}{\Nd } \sum_{k \in \TN} \delta_{  \theta^{k,N}_{\Ti}}
	\in
	\MOne{ \C }$
is derived. 
In  \cite{PraHolMKV}, a mean field interaction with random environment is considered.
In both models, the authors assume that the drift coefficient $b$ is bounded and does not depend on the spatial positions of the spins.
By this boundedness, the authors can transfer a large deviation principle of $\Nd$ independent Wiener process to the large deviation principle for the interacting system, by an application of the Varadhan lemma.
From this large deviation principle on $\MOne{\C}$, the contraction principle easily shows a large deviation principle for the empirical process.
However the rate function does not have the representation \eqref{eq::Sheuristisch}.
In \cite{PraHolMKV}, the authors try to show that the derived rate function is equal to a similar expression as \eqref{eq::Sheuristisch}.
Unfortunately in the proof of this result (in Step~4 of the proof of the Theorem~3), the authors use a circular reasoning.
Therefore only for some trajectories of measures the equality of these two expressions is proven.
Also in \cite{BruFiniteKullb} only for this subset of the trajectories the equality of the rate function is proven. 

Nevertheless we generalise the approach of \cite{PraHolMKV} to be applicable also for the unbounded, local mean field model \eqref{SDE::LocalMF}.
This leads to one of the two proofs, that we give here, of the large deviation principle of the empirical measure $L^{N}$.
However we are not able to correct the mentioned circular reasoning in \cite{PraHolMKV}.
Therefore we infer a large deviation principle for the empirical process $\empP$ from this, but we can not show, in this way, that it equals \eqref{eq::Sheuristisch}.
We get this equality of the rate functions by the more general approach of Chapter~\ref{sec::LDPempP} and the uniqueness of the rate function.

Also for mean field interaction with random environment, but without the space dependency, a large deviation principle for the empirical measure is derived in   \cite{Cabana:2016qy}. Moreover the authors characterise the minima of the rate function.
The model considered in \cite{Cabana:2016qy} is similar to the concrete example \eqref{SDE::LocalMF}, however only bounded interactions are considered.

\emptyline

In \cite{FenKur} a third kind of approach is used to prove the LDP for non geometric structure dependent, interacting system like in \cite{DawGarLDP}, but with more restrictive assumptions.
The authors connect the LDP with a variational problem arising from control theory (see Example~1.14, Chapter~13.3 and Theorem~13.37 of \cite{FenKur}).
However we do not try to generalise this approach to the system considered here, because the assumptions are more restrictive than \cite{DawGarLDP} and, at least from the point of view of the probabilistic theory,  the approach gives  less inside in the probabilistic structure.

\emptyline

A direct approach to derive a large deviation principle with a rate function of the form \eqref{eq::Sheuristisch} is used in \cite{KipOlla} for independent Brownian motions.
This can be generalised to models with mean field interaction.
However for the space dependent model we consider, we can not apply this approach. It requires that the hydrodynamic limit has a unique weak solution (see also \cite{GuionLD} page~40).

\emptyline

A dynamical large deviation principle for a system with a space dependent interaction (similar to the interaction in \eqref{SDE::LocalMF} but without the random environment), is studied in  \cite{ComNucle}.
The main difference to the model we consider here, is that only the spin values  $\pm 1$ are considered and that the spin values evolve according to Glauber dynamics.
The proofs of their results are highly dependent on their chosen jump dynamic and does not cover models with Langevin dynamics.

\subsection{Results for the concrete example \texorpdfstring{\protect\eqref{SDE::LocalMF}}{(\ref{SDE::LocalMF})} of a local mean field model}
\label{sec::Intro::ResultLMF}

For the concrete example \eqref{SDE::LocalMF} of a local mean field model, we state in this chapter the main results of this paper.
These are the large deviation principles for the families $\left\{ \empP \right\}$ and $\left\{ L^{N} \right\}$.
Note that we derive these  principles also for the more general system of interacting SDEs \eqref{eq::SDE} in the next chapters.
However for the local mean field model  \eqref{SDE::LocalMF}, definitions and assumptions are more comprehensible.

Let the spin system characterised by the model \eqref{SDE::LocalMF} satisfy the following assumptions.
Here $\Wsp$  is a compact subset of $\R^{m}$ for a $m>0$.

\begin{assumption}
\label{ass::LMF::Init}
The family of  initial distributions $\left\{ \nu_{x} \right\}_{x \in \Td} \subset \MOne{ \R } $  is Feller continuous, i.e. $\nu_{x^{(n)} }$ converges to $ \nu_{x} $ when $x^{(n)} \rightarrow x$, or equivalently the map $x  \mapsto  \int_{\R} f \left( \theta \right)  \nu_{x} \left( \dd \theta \right) $ is continuous for all $f \in \CspL{b}{\R}$.
\end{assumption}

See Chapter~\ref{sec::Prel::SanovT::AssExamples} for examples of Feller continuous initial distributions.

\begin{assumption}
\label{ass::LMF::Init::Integral}
~\vspace{-1em}
\begin{align}
			\sup_{x \in \Td}
				\int_{\R}  e^{ 2\Psi \left( \theta \right) }   \nu_{x} \left( \dd \theta \right)
			<
			\infty
\end{align}
\end{assumption}

\begin{assumption}
	\label{ass::LMF::Medium}
	The family of distributions of the random environment $\left\{ \zeta_{x} \right\}_{x \in \Td} \subset \MOne{ \Wsp }$ is Feller continuous.
\end{assumption}

\begin{assumption}
\label{ass::LMF::J}
	The interaction weight $J : \TW \times \Wsp  \rightarrow \R$ is in 
	$\Ltwo \left( \Td,\Csp{ \Wsp \times \Wsp} \right)$ and satisfies the following conditions:
	\begin{itemNoIntend}
	\item
	There is a $\ol{J} \in \Ltwo \left( \Td \right)$, such that
	$ \sup_{ \left( \w,\w' \right)  \in  \Wsp \times \Wsp} \abs{ J \left( x , \w,\w'  \right) } <  \ol{J} \left( x \right)$ for all $x \in \Td$.
	\item 
		 $J$ is even on $\Td$, i.e. $J \left( x \right) = J \left( -x \right)$ for all $x \in \Td$.
	\item
	Moreover
		\begin{align}					
				\frac{1}{\Nd}  
				\sum_{i \in \TN } 
							\sup_{\w,\w' \in \Wsp}
							\abs{ J  \left( \frac{i}{N}, \w , \w' \right)  -  \Nd \int_{\Delta_{i,N}} J \left( x, \w , \w' \right) \dd x }^{2} 
				\rightarrow 
				0
				,
		\end{align}
		when $N \rightarrow \infty$, with $\Delta_{i,N} \defeq \left\{ x \in \Td : \abs{ x- \frac{i}{N} } < \frac{1}{2 N} \right\}$.

	\end{itemNoIntend}
\end{assumption}

\begin{example}
\label{exa::LMF::J}
This assumption is in particular satisfied in the following cases:
	\begin{itemNoIntend}
		\item	
			$J$ is continuous in all variables.
		\item
			$J \left( x, \w, \w \right) = J^{1} \left( x \right) J^{2} \left( \w, \w' \right)$
			or $J \left( x, \w, \w \right) = J^{1} \left( x \right) + J^{2} \left( \w, \w' \right)$.
			In both situations:
			\begin{itemNoIntend}[leftmargin=1cm,label=\tiny$\blacksquare$]
			\item
				$J^{2} \in \Csp{\Wsp \times \Wsp}$, for example  
							$J^{2} \left( \w, \w' \right) = \w \w'$ or
							$J^{2} \left( \w, \w' \right) = \w-\w'$.
			\item
				$J^{1} \in \Ltwo \left( \Td \right)$ is even and 
			\begin{itemNoIntend}[leftmargin=1cm,label=\textbf{- }]
				\item
					either continuous, or
				\item
					$J^{1} = 1_{A}$ for $A \subset \Td$ a rectangular, or 
				\item
					$J^{1}$ can even have a singularity like $J^{1} \left( x \right) = \abs{x}^{-\frac{1}{2} + \epsilon}$ with $J^{1} \left( 0 \right) =0$.
			 	\end{itemNoIntend}
			 \end{itemNoIntend}
	\end{itemNoIntend}
\end{example}

\begin{remark}
		We use the assumption that $J$ is even, only in the proof of the large deviation principle of $\left\{ L^{N} \right\}$ (Theorem~\ref{thm::LMF::LDPLN}),
		but not in the proof of the large deviation principle of $\left\{ \empP \right\}$ (Theorem~\ref{thm::LMF::LDPempM}).
\end{remark}

\begin{assumption}
\label{ass::LMF::Psi}
		$\Psi \left( \theta, \w \right) = \ol{\Psi} \left( \theta \right) + \w_{1} \theta$,
		for $\left( \w,\theta \right) \in  \Wsp \times \R$,
		where $\ol{\Psi}$ is a polynomial of even degree greater or equal two, with positive coefficient of that degree. 
		Define
		\begin{align} 
			\label{ass::LMF::Psi::def::cPsi}
			c_{\Psi}
			\defeq
			\liminf_{\abs{\theta}\to\infty}
					\frac{\ol{\Psi}  \left( \theta  \right) }{\abs{\theta}^2} 
			,
		\end{align}
		with $c_{\Psi}= \infty$ if the degree of the polynomial is greater than two.
		Assume 
		\begin{align} 
		 \label{ass::LMF::Psi::eq:psi-con1}
					c_{\Psi}
				> 
					\LpN{1}{\ol{J}}
			.
		\end{align}
\end{assumption}

We state some examples of $\Psi$ that satisfy these assumption after \eqref{eq::DriftCoef::LocalMF}.
Also more general $\Psi$ are covered by the approach we state. For example  the randomness could merge into the single particle potential in a more general way than just as a chemical potential.

\emptyline

We infer from these assumption that the corresponding martingale problem is well posed for each fixed $\ul{\w}^{N} \in \Wsp^{\Nd}$ and each fixed initial values $\ul{\theta}^{N} \in \RN$, i.e. that there is a unique weak solution to \eqref{SDE::LocalMF} (see Remark~\ref{rem::LDPempP::MProbPNwelldef}).
Hence there is a unique measure $P^{N}_{\ul{\w}^{N}, \ul{\theta}^{N}} \in \MOne{\Csp{\Ti}^{\Nd}}$, which is the law of the solution  $\ul\theta^{N}_{\Ti}$ of the $\Nd$-dimensional SDE \eqref{SDE::LocalMF} with initial values  $\ul{\theta}^{N}$ and  with fixed environment $\ul{\w}^{N}$.

\begin{notation}
	\label{nota::LMF::General}
	We use the following notation:
	\begin{itemNoLeftIntend}[itemsep=0em]
		\item
		For each $N \in \N$, we denote the initial distribution of the $\Nd$-dimensional spin system by
		$\nu^{N} \defeq \bigotimes_{k\in \TN} \nu_{\frac{k}{N}} \in \MOne{ \RN} $.
		\item
		Analogue we define $\zeta^{N} \defeq \bigotimes_{k\in \TN} \zeta_{\frac{k}{N}}  \in \MOne{\Wsp^{\Nd}}$.
		\item
		We denote by 
		$P^{N}_{\ul{\w}^{N}} \defeq \int_{\RN} P^{N}_{\ul{\w}^{N},\ul{\theta}^{N}} \nu^{N} \left( \dd \ul{\theta}^{N} \right)  \in \MOne{ \Csp{\Ti}^{\Nd} } $,
		the law of the paths of the $\Nd$-dimensional spin system  with a given environment $\ul{\w}^{N} \in \Wsp$ and with initial distribution $\nu^{N}$.
		\item
		We use the symbol $P^{N} =  \zeta^{N} \left( \dd \w \right)  \otimes P^{N}_{\ul{\w}^{N}} \in \MOne{ \Wsp^{\Nd} \times \Csp{\Ti}^{\Nd} }  $ for the joint distribution of the random environment and of the paths of the spin system.
	\end{itemNoLeftIntend}
\end{notation}

The $\empP$ and $L^{N}$ are both images of $\ul{\w}^{N}$ and $\ul{\theta}^{N}_{\Ti}$. Therefore we consider $\empP$ and  $L^{N}$ as random elements under $P^{N}$.

\emptyline

The following norm appears in the rate function $S_{\nu,\zeta}$ in \eqref{eq::Sheuristisch} (compare this norm to the $-1$ Sobolev norm).
\begin{definition}
\label{def::Minus1Norm}
	For a measure $\pi \in \MOne{ \TWR } $ and $\xi$ a distribution on the space of test functions $\CspL[\infty]{c}{\TWR}$, define
	\begin{align}
	\begin{split}
		\abs{ \xi }^{2}_{\pi}
		&\defeq 
		\frac{1}{2}
		 \sup_{f \in \D_{\pi}}
		\frac{\abs{\sPro{ \xi}{f} }^{2}}
		{   \sigma^{2} \int_{\TWR}  \left(\partial_{\theta} f \left(x, \w, \theta \right)\right) ^{2} \pi \left( \dd x, \dd \w, \dd \theta \right)   }
		\\
		&=
		\sup_{f \in \CspL[\infty]{c}{\TWR} }
			\left\{
				\sPro{ \xi}{f}
				-
				\frac{\sigma^{2}}{2} \int_{\TWR}  \left(\partial_{\theta} f \left(x, \w, \theta \right)\right) ^{2} \pi \left( \dd x, \dd \w, \dd \theta \right)
			\right\}
		,
	\end{split}
	\end{align}
	with $\D_{\pi} = \left\{ f \in \CspL[\infty]{c}{\TWR}  : \int_{\TWR}  \left(\partial_{\theta} f \left(x, \w, \theta \right)\right) ^{2} \pi \left( \dd x, \dd \w, \dd \theta \right) \neq 0 \right\} $.
	
	With abuse of notation we also use the symbol $\abs{\xi}_{\pi}$ for $\pi \in \MOne{\R}$ and $\xi$ a distribution on the space of test functions $\CspL[\infty]{c}{\R}$.
\end{definition}

\begin{theorem}[{For a more general version of this Theorem, see Theorem~\ref{thm::LDPLN}}]~
\label{thm::LMF::LDPLN}
		Let the Assumption~\ref{ass::LMF::Init}, Assumption~\ref{ass::LMF::Init::Integral}, Assumption~\ref{ass::LMF::Medium}, Assumption~\ref{ass::LMF::J} and Assumption~\ref{ass::LMF::Psi} hold.
		Then the family $\left\{L^{N}, P^{N} \right\}$ satisfies on  $\MOne{\TWC}$ a large deviation principle with a good rate function.
		We derive two different representation of the rate function (see Theorem~\ref{thm::LDPLN} and Theorem~\ref{thm::LDPLN::LMF::LDP}).
\end{theorem}

\begin{theorem}[{For a more general version of this Theorem, see Theorem~\ref{thm::LDPempP}}]~
\label{thm::LMF::LDPempM}
Under the same assumptions as Theorem~\ref{thm::LMF::LDPLN}, 
the family $\left\{ \empP, P^{N} \right\}$ satisfies on  $\Csp{ \Ti , \MOne{\TWR}}$ a large deviation principle with good rate function
\begin{align}
		S_{\nu,\zeta} \left( \mu_{\Ti} \right)
		\defeq
		\int_{0}^{T} \abs{ \partial_{t}\mu_{t}   - \left(\Gen^{\textnormal{LMF}}_{\mu_{t},.,.} \right)^{*}  \mu_{t}   }^{2}_{\mu_{t} } \dd t
		+
		\relE{\mu_{0}}{\dd x \otimes  \zeta_{x} \otimes \nu_{x}}
		,
\end{align}
when $\mu_{\Ti} \in \Csp{\Ti, \MOne{\TWR}}$ is weakly differentiable, $\sup_{t \in \Ti} \int_{\TWR} \theta^{2} \mu_{t} \left( \dd x, \dd \w, \dd \theta \right)$ is finite, and $\mu_{t} = \dd x \otimes \mu_{t,x}$ with $\mu_{t,x} \in \MOne{\Wsp \times \R}$.
Otherwise $S_{\nu,\zeta} \left( \mu_{\Ti}  \right) = \infty$.

Moreover the integral with respect to $\TW$  and the supremum in the norm in $S_{\nu,\zeta}$ can be interchanged, i.e.
$	S_{\nu,\zeta} \left( \mu_{\Ti}  \right)
=
S^{\TW}_{\nu,\zeta} \left(  \mu_{\Ti}  \right)$ defined as
	\begin{align}
				 \int_{0}^{T} \int_{\Td} \int_{\Wsp}
				 \abs{
				 	\partial_{t}  \mu_{t,x,\w} - \left(\Gen^{\textnormal{LMF}}_{\mu_{t},x,\w}\right)^{*}   \mu_{t,x,\w}
				 }_{\mu_{t,x,\w}}
				 \mu_{0,x,\Wsp} \left( \dd \w \right)
				 \dd x \dd t
						+
						\relE{\mu_{0}}{\dd x \otimes  \zeta_{x} \otimes \nu_{x}}
		,
	\end{align}
	with $\mu_{t,x,\Wsp} \in \MOne{\Wsp}$ and $\mu_{t,x,\w} \in \MOne{\R}$ such that $\mu_{t,x} = \mu_{t,x,\Wsp} \lb \dd \w \rb \otimes \mu_{t,x,\w}$.
	
	We state further representation of $S_{\nu,\zeta}$ in  Chapter~\ref{sec::OtherRepRFS}.
\end{theorem}

\subsection{Structure of the paper and idea of the proofs}

This paper is organised as follows.
We state in Chapter~\ref{sec::Prel} some preliminaries that are required in the subsequent chapters.
At first this comprises some definitions and notations (Chapter~\ref{sec::Prel::Defs}).
Then in Chapter~\ref{sec::Prel::SanovT}, we generalise the Sanov Theorem to vectors of space ($\Td$), random environment ($\Wsp$) and spin value ($\R$) dependent empirical measures.
This is also a generalisation of the Sanov type theorem in \cite{DawGarLDP}, because of the additional space and random environment dependency, and because we allow  the initial values to be random.
Then we state a generalisation of the Arzel\'a Ascoli theorem for sets and measures on $\TWC$ (Chapter~\ref{sec::Prel::AA}), and we generalise the definitions and results on distribution-valued functions of the Chapter~4.1 of \cite{DawGarLDP} to the space $\TWR$ (Chapter~\ref{sec::Prel::Distr}).
Finally (in Chapter~\ref{sec::Prel::RelBetwSpaces}) we discuss how the spaces, on which $L^{N}$ and $\empP$ are defined, are related.

\emptyline

In Chapter~\ref{sec::LDPempP} we state and prove the large deviation principle for the family of the empirical process $\left\{\empP \right\}$ in a more general setting than the concrete example \eqref{SDE::LocalMF} considered in Theorem~\ref{thm::LMF::LDPempM}.
The concrete example \eqref{SDE::LocalMF} (with the assumptions of Chapter~\ref{sec::Intro::ResultLMF}) is covered by the weaker assumptions in  Chapter~\ref{sec::LDPempP} (we show this in Chapter~\ref{sec::LocalMF}).
To prove the large deviation principle, we generalise the approach of \cite{DawGarLDP} to the space and random environment dependent model we consider here.
We explain the approach and changes compared to \cite{DawGarLDP}, in detail in Chapter~\ref{sec::LDPempP}.
We first (Chapter~\ref{sec::LDPempP::Inde}) derive a LDP of the empirical process for the independent system and finally transfer this LDP to an LDP for the interacting system (Chapter~\ref{sec::LDPempP::Inter}).
The existence of the LDP for the independent system is a direct consequence  of the Sanov type result of Chapter~\ref{sec::Prel::SanovT}.
The better part of the Chapter~\ref{sec::LDPempP::Inde} is dedicated to showing that the rate function actually has a form like \eqref{eq::Sheuristisch}.
In contrast to the model considered in  \cite{DawGarLDP}, the model \eqref{eq::SDE} treated here, is space and environment dependent (in the drift coefficient, the empirical process and the initial distribution), the initial values are random and we consider the whole space $\Csp{\Ti, \MOne{\TWR}}$ and not a subspace with a stronger topology.
This leads to many changes in the proofs.
In particular in the proof of a lower bound on the rate function of the independent system (Chapter~\ref{sec::LDPempP::2RepBoundSI::Lower}), we require a solution to a PDE that is continuous in the space and environment variable.
We prove the existence and uniqueness of such a solution in Chapter~\ref{sec::LDPempP::PDE}. 

\emptyline

In Chapter~\ref{sec::OtherRepRFS}, we state different representations of the rate function for the empirical process. These expressions might be useful when working on the mentioned long time behaviour (see also \cite{DawGarLdpFrEn} in the mean field case), in particular when the model is not reversible.

\emptyline

In Chapter~\ref{sec::LDPLN}, we show that the same approach as in Chapter~\ref{sec::LDPempP} can be used to derive a large deviation principle for the family $\left\{ L^{N} \right\}$, provided that this family is exponentially tight.
We prove the exponential tightness for the concrete example \eqref{SDE::LocalMF} of the local mean field model in Chapter~\ref{sec::LDPLN::LMF}.
Moreover we derive for this model a second representation of the rate function. In this second expression of this function, the influence of the entropy and of the interaction becomes obvious.

\emptyline
In Chapter~\ref{sec::Comp}, we show  at first  (Theorem~\ref{thm::LDPLN::RelMinimaRF}), that the minimizer of the rate functions of $\left\{ \empP\right\}$ and $\left\{ L^{N}\right\}$ are one to one related.
Then in Chapter~\ref{sec::Comp::LdpLNtoEmp}, we infer from the large deviation principle of $\left\{ L^{N} \right\}$, the large deviation principle of $\left\{ \empP \right\}$, by an application of the contraction principle.
However the rate function does not have the desired form $S_{\nu,\zeta}$ given in \eqref{eq::Sheuristisch}.
In Chapter~\ref{sec::Comp::UpperBoundSnu} we show that the rate function is at least an upper bound on $S_{\nu,\zeta}$.
However we are not able to prove that it is also an lower bound, without using the result of Chapter~\ref{sec::LDPempP}.

\emptyline
In Chapter~\ref{sec::LMF::LDPLN}, we derive the large deviation principle for the empirical measure $\left\{L^{N}\right\}$ for the concrete example \eqref{SDE::LocalMF} of the local mean field model by a different approach than in Chapter~\ref{sec::LDPLN}.
We consider at first the simpler model without the interaction part.
For this model we get easily a large deviation principle (by the Sanov type result of Chapter~\ref{sec::Prel::SanovT}).
Then we transfer the LDP to the model with interaction by the Girsanov transformation.
This would typically follow from the Varadhan's lemma (see e.g. \cite{PraHolMKV}).
However the exponent in the Girsanov transform is only continuous on special subsets of $\MOne{\TWR}$.
Therefore we need to use a generalisation of Varadhan's lemma to functions that are unbounded and nowhere continuous.

We state this generalisation in Appendix~\ref{sec::Vara} in a very general and abstract form, because it may be of independent interest.

There are two reason, why we state this second approach of the proof of the large deviation principle for $\left\{L^{N}\right\}$ in Chapter~\ref{sec::LMF::LDPLN}.
On the one hand the idea of investigating separately the entropy and adding then the interaction, is easier to follow and seems to be more comprehensible (form a physical point of view) compared to the approach of Chapter~\ref{sec::LDPLN}.
On the other hand this approach gives a nice example how the generalisation of the Varadhan's lemma can be applied. 

\emptyline
\noindent\textbf{Acknowledgements.}
This paper is part of the authors PhD thesis, created under the supervision of A. Bovier.
The paper uses many ideas from the collaboration with A. Bovier and D. Ioffe.
I am very thankful for their fruitful input.
The author thanks F. den Hollander for helpful discussions on the subject and G. Uraltsev for his useful hints concerning the arising PDEs.

\section{Preliminaries}
\label{sec::Prel}

\subsection{Definitions and notations}
\label{sec::Prel::Defs}

\begin{notation}
\label{nota::ML}
	Let $Y$ be a polish space.
	We denote by $\MOne{Y}$ the space of probability measures on $Y$ equipped with the topology of weak convergence.
	
	We write $\MOneL{ \Td \times Y }$ for the subset of $\MOne{\Td \times Y}$, that consists of those measures, that have the Lebesgue measure as projection to $\Td$.
\end{notation}
The measures in $\MOneL{ \Td \times Y }$ are also called Young measures (see \cite{AttButMichVar} Definition 4.3.1).

\begin{notation}
\label{nota::Cem}
	We denote the space of continuous functions from $\Ti$ into $\MOne{\TWR}$ by
	\begin{align}
				\Cem 
				\defeq
					\Csp{\Ti, \MOne{\TWR}}
	\end{align}
	and its subspace with values in $\MOneL{\TWR}$ by
	\begin{align}
			\Cem^{L}
				\defeq
					\Csp{\Ti, \MOneL{\TWR}}
			.
	\end{align}
\end{notation}

Let $\varphi \in \Csp[2]{\R} $ be a non-negative function, such that $\lim_{\abs{\theta} \rightarrow  \infty} \varphi \left( \theta \right) = \infty$.

\begin{definition}
\label{def::MR}
		We denote the subset of $\MOne{\TWR}$ of measures, whose integral with respect to a $\varphi \in \Csp{\R}$ is bounded by $R >0$ by
		\begin{align}
				\M_{\varphi,R} 
				\defeq
				\left\{
							\mu \in \MOne{\TWR}
							:
							\int_{\TWR}  \varphi \left( \theta \right) \mu \left( \dd x, \dd \w, \dd \theta  \right)
							\leq
							R
				\right\}
				.
		\end{align}
		Moreover we denote the subset of $\MOne{\TWR}$, with finite integral with respect to $\varphi$ by
		\begin{align}
						\M_{\varphi,\infty}
						\defeq
						\bigcup_{R>0}
						 \M_{\varphi,R}
						=
						\left\{
									\mu \in \MOne{\TWR}
									:
									\int_{\TWR}  \varphi \left( \theta \right) \mu \left( \dd x, \dd \w, \dd \theta  \right)
									<
									\infty
						\right\}
						.
				\end{align}
With abuse of notation we use rarely also the symbol $\M_{\varphi,R}$ for the appropriate subspace of $\MOne{\TR}$.
\end{definition}

\begin{definition}
\label{def::CemR}
	We denote the subset of $\Cem$,  that consists of the paths which are everywhere in $\M_{\varphi,R}$, for a $R>0$, by
	\begin{align}
				\Cem_{\varphi,R}
				\defeq 
						\left\{ \mu_{\Ti} \in \Cem : 
									\sup_{t \in \Ti } 
													\int_{\TWR}  \varphi \left( \theta \right) \mu_{t} \left( \dd x, \dd \w, \dd \theta  \right) 
									\leq R \right\} 
						\subset 
						\Cem
				.
	\end{align}
	For the union of these sets we use the symbol
	\begin{align}
		\Cem_{\varphi,\infty} 
		\defeq
			\bigcup_{R=1}^{\infty} \Cem_{\varphi,R} 
		=
			\left\{ \mu_{\Ti} \in \Cem : \sup_{t \in \Ti } 
													\int_{\TWR}  \varphi \left( \theta \right) \mu_{t} \left( \dd x, \dd \w, \dd \theta  \right) 
											< \infty \right\}
		.
	\end{align}
\end{definition}

We endow $\M_{\varphi,R},\M_{\varphi,\infty},\Cem_{\varphi,R} $ and $\Cem_{\varphi,\infty} $ with the subspace topology of $\MOne{\TWR}$ and $\Cem$ respectively.
By this property these spaces differ from the definition used in \cite{GarOnTheMKVLimit} and \cite{DawGarLDP}.
There the authors equip the spaces with a stronger topology.

\begin{definition}
\label{def::DecomMu}
	For a measure $\mu \in \MOneL{\TWR}$,  we denote by $\mu_{x} \in \MOne{\Wsp \times \R}$ the regular conditional probability measures such that 
	$\mu = \dd x \otimes  \mu_{x}$. 
	
	For the projection of $\mu_{x}$ on the environment coordinate $\Wsp$, we use the symbol $\mu_{x,\Wsp}$ and for the corresponding regular conditional probability measures $\mu_{x,\w} \in \MOne{\R}$.
	Then $\mu = \dd x \otimes \mu_{x,\Wsp} \left( \dd \w \right) \otimes \mu_{x,\w}$.
\end{definition}

\begin{definition}
	We define the relative entropy between two probability measures $\mu, \nu \in \MOne{Y}$ on a Polish space $Y$, by
	\begin{align}
				\relE{\mu}{\nu}
				\defeq
				\begin{cases}
							\int_{Y} \log \left(  \frac{\dd \mu}{\dd \nu}\right) \nu
							\quad &\textnormal{if } \mu << \nu
							\\
							\infty &\textnormal{otherwise.}
				\end{cases}
	\end{align}
\end{definition}

\pagebreak[2]
\begin{notation}
We use the following notation.
\begin{itemNoLeftIntend}
\item
With $x,y,z$ we usually denote macroscopic coordinates, i.e. positions on the torus $\Td$. Whereas by
$i, j, k$ we denote microscopic coordinates, i.e. positions on the discrete torus $\TN$ . These two coordinate
systems are related by $x = \frac{i}{N}$.
\item
As time variables we use the letters $s,t,u$.
\item
We use the letters $\theta,\eta$ for the spin values.
With $\theta_{\Ti}$ we denote the whole path of the spin value, i.e. an element of $\Csp{\Ti}$.
With $\theta_{t} \in \R$ we denote the spin value at time $t \in \Ti$.
\item
For a $\Nd$-dimensional vector of spin values, numbered by $k \in \TN$, we use the symbol $\ul\theta^{N}$ and analogue $\ul\theta^{N}_{\Ti}$, $\ul\theta^{N}_{t}$. We write $\theta^{k,N}$ for the element at position $k \in \TN$ in this vector.
\item
We use the letter $\w$ for a value of the random environment. Again $\ul{\w}^{N}$ is the $\Nd$-dimensional vector of the environment and $\w^{k,N}$ the specific value of the environment associate with the position $k \in \TN$.

\item
We use lower-case letters, mostly $\mu$, $\nu$, $\pi$ for measures on $\MOne{\TWR}$, $\MOne{\TR}$ or $\MOne{\R}$ ($\nu$ is usually the distribution of the initial values).
For the path on measures, i.e. for an element in $\Cem$, we write $\mu_{\Ti}$.
For the measure at time $t \in \Ti$ of the path $\mu_{\Ti}$ we write $\mu_{t}$.

\item
We use upper-case letters, in most cases $Q$ or $\Gamma$, for measures on $\MOne{\TWC}$.

\item
We denote the spaces of continuous functions from $X$ to $Y$ by $\Csp{X,Y}$. For its subset of bounded functions we use the notation $\CspL{b}{X,Y}$, of functions that vanish at the boundary $\CspL{0}{X,Y}$ and of functions with compact support $\CspL{c}{X,Y}$.
With a superscript like in $\Csp[k]{X,Y}$ we state the $k$-times continuous differentiability.
To shorten the notation we often skip $Y$ if $Y=\R$, i.e. $\Csp{X} = \Csp{X,\R}$.

\end{itemNoLeftIntend}
\end{notation}

\subsection{A Sanov type result}
\label{sec::Prel::SanovT}

Let $Y_{1},...,Y_{r}$ be polish spaces for $r\geq 1$
and let $\left\{  Q_{x, \w, \theta} : \left( x, \w ,\theta \right) \in \TWR \right\} $ be a family of probability measures on $Y= Y_{1} \times ... \times Y_{r}$.

We generalise in this chapter at first the Sanov type Theorem~3.5 of \cite{DawGarLDP} to the setting we consider here (Lemma~\ref{lem::SanovT}).
More precisely we add the space dependency and the random environment in the vector of the empirical measure, i.e.
	for $ \left( y^{i} \right) _{i \in \TN}  \in Y^{\Nd} $ and $\left( \w^{i,N} \right) \in \Wsp^{\Nd}$, we define the vector in $\MOne{ \TW \times Y_{1} }  \times ... \times \MOne{ \TW \times Y_{r} } $ by
\begin{align}
	L^{N}_{r}
	\defeq
		 \left(
			 N^{-d} \sum_{i \in \TN} \delta_{ \left( \frac{i}{N} , \w^{i,N} , y^{i}_{1} \right)} 
			 ,...,
			 N^{-d} \sum_{i \in \TN} \delta_{ \left( \frac{i}{N} ,\w^{i,N}  y^{i}_{r} \right)} 
		 \right)  
	.
\end{align}
Moreover we allow the parameter $\theta$ in the measures of the distribution of the $y_{1},...,y_{r}$ to be random and not fixed as in \cite{DawGarLDP}.
New in this chapter compared to \cite{DawGarLDP} is also that we prove (Lemma~\ref{lem::SanovT::RelEntropy}), that the rate function can be expressed as a relative entropy.

\emptyline

We need the following two assumptions, that imply in particular that the integrals in Lemma~\ref{lem::SanovT} are well defined and that we get a suitable convergence of the logarithmic moment generating function in the proof of this lemma.
\begin{assumption}
\label{ass::SanovT::MeasFeller}

		 $\left\{  Q_{x, \w, \theta} : \left( x, \w, \theta \right) \in \TWR \right\}  \subset \MOne{ Y } $  is Feller continuous.

\end{assumption}

For each $\w \in \Wsp$ and each $x \in \Td$, we define $Q_{x,\w} \defeq \int_{\R} Q_{x,\w,\theta} \nu_{x} \left( \dd \theta \right) \in \MOne{ Y } $, by averaging over the parameter $\theta$.
With this $Q_{x,\w}$, define the product measures $Q^{N}_{\ul{\w}^{N}} \in \MOne{Y^{\Nd}}$ and the joint measure $Q^{N} \in \MOne{\Wsp^{\Nd} \times Y^{\Nd}}$ similar as in Notation~\ref{nota::LMF::General}.

\begin{lemma}[compare to  \cite{DawGarLDP} Theorem~3.5 for mean field LDP]
\label{lem::SanovT}
	If the Assumption~\ref{ass::LMF::Init}, Assumption~\ref{ass::LMF::Medium}  and Assumption~\ref{ass::SanovT::MeasFeller} hold,
	then the family $\left\{ L^{N}_{r}, Q^{N} \right\}$ satisfies a large deviation principle on the space $\MOne{ \TW \times Y_{1} }  \times ... \times \MOne{ \TW \times Y_{r} }$ with good rate function
	\begin{align}
	\label{lem::SanovT::RF}
	\begin{split}
		L_{\nu,\zeta}  \left( \Gamma^{1},...,\Gamma^{r} \right)  
		= 
			&\sup_{
				\substack{ f_{1}\in \CspL{b}{\TW \times Y_{1} } 
								\\
								...
								\\ 
								f_{r} \in \CspL{b}{\TW \times Y_{r} }}
						}
		\left\{ 	
			\sum_{\ell=1}^{r}
				\int_{\TW \times Y_{\ell}} f_{\ell} \left(  x, \w, y_{\ell}  \right)  \Gamma^{\ell} \left( \dd x, \dd \w, \dd y_{\ell} \right) 
				\right.
				\\
			&\qquad\qquad
			-
				\left.
				\int_{\Td}
					\log
					\left(
					\int_{\Wsp}
					 \int_{Y}
						e^{\sum_{\ell=1}^{r} f_{\ell} \left( x,\w,y_{\ell} \right) }
						Q_{x,\w} \left( \dd y_{1}, ..., \dd y_{r} \right) 
						\zeta_{x} \left( \dd \w \right)
					\right) 
				\dd x
			\right\} 
	\end{split}
	\end{align}
for $\Gamma^{\ell} \in \MOne{ \TW \times Y_{\ell} } $.
\end{lemma}

In the case when $r=1$, i.e. $Y = Y_{1}$, we can represent the rate function as a relative entropy.

\begin{lemma}
\label{lem::SanovT::RelEntropy}
	If $r=1$ 
	then for $\Gamma = \dd x \otimes \Gamma_{x} \in \MOneL{\TW \times Y} $
	\begin{align}
	\label{eq::lem::SanovT::RelEntropy}
	\begin{split}
		L_{\nu,\zeta} \left( \Gamma \right)
		&=
				\relE{ \Gamma  }{  \dd x \otimes \zeta_{x} \left( \dd \w \right) \otimes Q_{x,\w} }
			=
				\int_{\Td} \relE{ \Gamma_{x}  }{   \zeta_{x} \left( \dd \w \right) \otimes Q_{x,\w} } \dd x
		\\
		&=
					\int_{\Td}
					\int_{\Wsp}
						\relE{ \Gamma_{x,\w}}{ Q_{x,\w} }
						\Gamma_{x,\Wsp} \left( \dd \w \right)
					\dd x
					+
					\int_{\Td} \relE{ \Gamma_{x,\Wsp}  }{ \zeta_{x} } \dd x
				.
	\end{split}
	\end{align} 
	Otherwise $	L_{\nu,\zeta} \left( \Gamma \right) = \infty$.
	Here $\Gamma_{x,\Wsp} \in \MOne{\Wsp}$ is defined as in Definition~\ref{def::DecomMu}.
	 
\end{lemma}

Before we prove these two lemmas in Chapter~\ref{sec::Prel::SanovT::Proofs}, we state in Chapter~\ref{sec::Prel::SanovT::Ass} some immediate consequences of the assumptions.
We need these consequences in the proofs.
Moreover they show that the integrals and measures used for example in the definition of $L_{\nu,\zeta}$, are well defined.

Then we show in Chapter~\ref{sec::Prel::SanovT::AssDiff} how the Assumption~\ref{ass::LMF::Init} could be weakened.
Finally in Chapter~\ref{sec::Prel::SanovT::AssExamples} we give some examples of families $\left\{ \nu_{x} \right\}$ that satisfy the assumptions.

\subsubsection{Preliminaries of the proof of the Sanov type result}
\label{sec::Prel::SanovT::Ass}

\paragraph{Implication of the Assumption~\ref{ass::SanovT::MeasFeller}}

We infer now from the Assumption~\ref{ass::SanovT::MeasFeller}, the following stronger continuity result.
\begin{lemma}
\label{lem::SanovT::ass::FellerWithAddx}
The Assumption~\ref{ass::SanovT::MeasFeller} causes that the map $x, \w, \theta \mapsto  \int f \left( x,\w, y \right)  Q_{x,\w,\theta} \left( \dd y \right) $ is continuous for each $f \in \CspL{b}{\TW \times Y}$ .
\end{lemma}

\begin{proof}
	Fix an arbitrary sequence $ \left( x^{(n)}, \w^{(n)},\theta^{(n)} \right)  \rightarrow  \left( x, \w, \theta \right) \in \TWR$. Then
	\begin{align}
	\begin{split}
		&\abs{  \int f \left( x^{(n)}, \w^{(n)}, y \right)  Q_{x^{(n)}, \w^{(n)}, \theta^{(n)}} \left( \dd y \right)  
					-  \int f \left( x,\w, y \right)  Q_{x,\w, \theta} \left( \dd y \right)  }
		\\
		&\leq
		\abs{  \int f \left( x^{(n)},\w^{(n)}, y \right)  - f \left( x,\w, y \right)  Q_{x^{(n)}, \w^{(n)}, \theta^{(n)}} \left( \dd y \right)  }
		\\
		&\quad+
		\abs{  \int f \left( x,\w, y \right) \left(  Q_{x^{(n)},\w^{(n)}, \theta^{(n)}} \left( \dd y \right)  - Q_{x,\w,\theta} \left( \dd y \right) \right)   }
		\eqdef
		\encircle{1} + \encircle{2}
		.
		\end{split}
		\end{align}
		Due to the Feller continuity of $Q_{x,\w,\theta}$ (Assumption~\ref{ass::SanovT::MeasFeller}), the sequence $Q_{x^{(n)},\w^{(n)}, \theta^{(n)}}$ is tight (Prokhorov's theorem).
		Hence there is for each $\epsilon>0$, a compact set $K^{\epsilon} \subset Y$, such that $Q_{x^{(n)},\theta^{(n)}} \left( Y \backslash K^{\epsilon} \right)  \leq \epsilon$ for all $n$. Therefore we get
		\begin{align}
		\label{eq::lem::SanovT::ass::FellerWithAddx::pf::1}
		\encircle{1}
		\leq
			\sup_{y \in K^{\epsilon}} \big| f \left( x^{(n)},\w^{(n)},y \right) -f \left( x,\w, y \right)  \big|
			+ 2 \iNorm{f} Q_{x^{(n)},\w^{(n)} \theta^{(n)}} \left( Y \backslash K^{\epsilon} \right) 
		\leq
			\epsilon
			,
		\end{align}
	by the continuity of $f$ and the compactness of $K^{\epsilon}$ for $n$ large enough.
	From the Feller continuity (Assumption~\ref{ass::SanovT::MeasFeller}), we infer moreover that $\encircle{2}$ is bounded by $\epsilon$ for $n $ large enough.

\end{proof}

\paragraph{Implications of the Assumption~\ref{ass::LMF::Init}}

In this section, we show that the Assumption~\ref{ass::LMF::Init} implies in particular a convergence, which we need to prove the large deviation result (in Lemma~\ref{lem::SanovT})
and that $\dd x \otimes \nu_{x} \in \MOne{\TR}$ is well defined.
\begin{lemma}
\label{lem::SanovT::ass::Init}

Let the Assumption~\ref{ass::LMF::Init} holds.
\begin{enuAlph}
	\item
	\label{lem::SanovT::ass::Init::Convergence}
	For all $f \in \CspL{b}{\TWR}$, for which there is a constant $c>0$ such that $f \geq c$,
		\begin{align}
			\mkern-60mu
			\frac{1}{\Nd} \sum_{ k \in \TN} \log \left(
								 \int_{\Wsp \times \R} \mkern-12mu
								 f \left( \frac{k}{N}, \w, \theta \right) 
										 \nu_{\frac{k}{N}} \left( \dd \theta \right)  \zeta_{\frac{k}{N}} \left( \dd \w \right)
								 \right)
			\rightarrow
			\int_{\Td} \log \left( \int_{\Wsp \times \R} \mkern-12mu
									 f \left( x,\w, \theta \right) \nu_{x} \left( \dd \theta \right) \zeta_{x} 
									 \left( \dd \w \right)  \right) \dd x
			.
		\end{align}

	\item
	\label{lem::SanovT::ass::Init::WellDefinedLimit}
	The probability measure $\nu \left( \dd x, \dd \theta \right) \defeq \dd x \otimes \nu_{x}  \left( \dd \theta \right) \in \MOneL{\TR}$, defined by
	\begin{align}
		\nu \left[ A \times B \right] = \int_{A} \int_{B} \nu_{x} \left( \dd \theta \right) \dd x
	\end{align}
	for $A \subset \Td$ and $B \subset \R$, both Borel measurable, is well defined.

\end{enuAlph}
\end{lemma}

\begin{proof}[of Lemma~\ref{lem::SanovT::ass::Init}~\ref{lem::SanovT::ass::Init::Convergence}]
Fix a $f \in \CspL{b}{\TWR}$ such that $f>c>0$.
By the Feller continuity of $\nu_{x}$ (Assumption~\ref{ass::LMF::Init}) and $\zeta_{x}$ (Assumption~\ref{ass::LMF::Medium}), the function
\begin{align}
\label{eq::pf::lem::SanovT::ass::Init::Convergence::Hf}
		 x  
		 \mapsto
		 H_{f} \left( x \right) 
		\defeq 
		\int_{\Wsp}\int_{\R} f \left( x, \w, \theta \right) \nu_{x} \left( \dd \theta \right) \zeta_{x} \left( \dd \w \right)
\end{align} 
is continuous. This can be shown by the same proof, that we used for Lemma~\ref{lem::SanovT::ass::FellerWithAddx}.
Note that the Feller continuity of $\nu_{x}$ and $\zeta_{x}$ implies the Feller continuity of $\nu_{x} \otimes \zeta_{x}$ (see for example \cite{BilConv} Theorem~2.8~(ii)).
Then $H_{f}$ is, as a continuous function, also Riemann integrable.

By the continuity of $\log$ on $\left[ c, \iNorm{f} \right] \subset \R$,  also $x \mapsto \log H_{f} \left( x \right)$ is Riemann integrable.
This Riemann integrability implies the convergence of the sums in Lemma~\ref{lem::SanovT::ass::Init}~\ref{lem::SanovT::ass::Init::Convergence}.
\end{proof}

\begin{proof}[of Lemma~\ref{lem::SanovT::ass::Init}~\ref{lem::SanovT::ass::Init::WellDefinedLimit}]
By the Feller continuity of $\nu_{x}$ (Assumption~\ref{ass::LMF::Init}), the maps $F_{f}: \Td \rightarrow \R$,
\begin{align}
\label{eq::pf::lem::SanovT::ass::Init::Convergence::Ff}
	F_{f} \left( x \right)  
	\defeq
		\int_{\R} f \left( \theta \right) \nu_{x} \left( \dd \theta \right)
	.
\end{align}
 are continuous and therefore also Borel-measurable, for all non negative $f \in 	\CspL{b}{\R}$.
This implies that $F_{f}$ is also Borel-measurable for all $f= \1_{B}$ with $B \subset \R$ a arbitrary rectangle, by a pointwise approximation of $\1_{B}$ with continuous function.
Then $F_{\1_{B}}$ is also Borel measurable for all Borel measurable $B \subset \R$ (as pointwise limits).
Therefore the function 
\begin{align}
		P \left( x,B \right) 
		= 
			\int_{B} \nu_{x} \left( \dd \theta \right)
		=
		F_{\1_{B}} \left( x \right)
\end{align} 
is a probability kernel (or regular conditional probability measure) for each Borel set $B \subset \R$.
Hence the $\nu$ (given in Lemma~\ref{lem::SanovT::ass::Init}~\ref{lem::SanovT::ass::Init::WellDefinedLimit}) is a well defined probability measure.
\end{proof}

\paragraph{Implications of Assumption~\ref{ass::SanovT::MeasFeller} and Assumption~\ref{ass::LMF::Init}}

\begin{lemma}
\label{lem::SanovT::ass::dxQxnu}
	By Assumption~\ref{ass::SanovT::MeasFeller}, Assumption~\ref{ass::LMF::Init} and Assumption~\ref{ass::LMF::Medium}, then
	 $\dd x \otimes \zeta_{x} \left( \dd \w \right) \otimes Q_{x,\w}$, characterised by
	 	\begin{align}
	 	\left( \dd x \otimes \zeta_{x} \left( \dd \w \right) \otimes Q_{x,\w}  \right)  \left[ A_{1} \times A_{2} \times A_{3} \right] 
	 	=
	 	 \int_{A_{1}} \int_{A_{2}} \int_{A_{3}} Q_{x,\w} \left( \dd y \right)  \zeta_{x} \left( \dd \w \right) \dd x
	 	\end{align}
	 for $A_{1} \subset \Td$, $a_{2} \subset \Wsp$ and $A_{3} \subset Y$,
	  is a well defined probability measure in  $\MOne{ \TW \times Y}$.
	 
	 Moreover for all $f \in \CspL{b}{Y} $,
		 $\left( x,\w \right) \mapsto \int_{Y} f \left( y \right) Q_{x,\w} \left( \dd y \right)$ is continuous.
\end{lemma}

\begin{proof}
		We show at first that $\left( \zeta_{x} \left( \dd \w \right) \otimes Q_{x,\w}  \right)$ is well defined for each $x \in \Td$, by constructing a probability kernel.
		For $f \in \CspL{b}{Y} $, the function $\TWR \ni \left( x, \w,\theta \right) \mapsto \int_{Y} f \left( y \right) Q_{x,\w,\theta} \left( \dd y \right) $ is continuous and bounded by Assumption~\ref{ass::SanovT::MeasFeller}.		
		Then also 
		\begin{align} 
		\label{eq::pf::lem::SanovT::ass::Init::Convergence::Hbarf}
			\TW \ni
				\left( x, \w \right)
			\mapsto 
				\ol{H}_{f} \left( x, \w \right) 
			\defeq 
				 \int_{\R} \int_{Y} f \left( y \right) Q_{x, \w, \theta} \left( \dd y \right)  \nu_{x} \left( \dd \theta \right)
		\end{align}
		 is continuous by Assumption~\ref{ass::LMF::Init} (this can be shown as the continuity of \eqref{eq::pf::lem::SanovT::ass::Init::Convergence::Hf}).
		 As in the proof of Lemma~\ref{lem::SanovT::ass::Init}~\ref{lem::SanovT::ass::Init::WellDefinedLimit}, we infer from this that 
		 $P \left( x, \w, A \right) = \int_{Y} \1_{A}\left( y \right) Q_{x,\w} \left( \dd y \right) $  is a probability kernel.
		 Hence $\left( \zeta_{x} \left( \dd \w \right) \otimes Q_{x,\w}  \right) \in \MOne{\Wsp \times Y}$ is well defined for all $x \in \Td$.
		 
		 By the same arguments, also $P \left( x, B \right) = \int_{\Wsp \times Y} \1_{B}\left( \w, y \right) Q_{x,\w} \left( \dd y \right) \zeta_{x} \left( \w \right)$ is a probability kernel. This requires the Assumption~\ref{ass::LMF::Medium}.
		 Therefore the $\left( \dd x \otimes \zeta_{x} \left( \dd \w \right) \otimes Q_{x,\w} \right)$ is well defined.
\end{proof}

\subsubsection{Proof of Lemma~\ref{lem::SanovT} and Lemma~\ref{lem::SanovT::RelEntropy}}
\label{sec::Prel::SanovT::Proofs}

\begin{proof}[of Lemma~\ref{lem::SanovT}]
The log moment generating function can be calculated for  $ f= \left( f_{1},...,f_{r} \right)  \in \CspL{b}{\TW \times Y_{1} } \times ... \times \CspL{b}{\TW \times Y_{r} } $ by
\begin{align}
\label{eq::pf::lem::SanovT::logMoment}
\begin{split}
	\Gamma_{\nu,\zeta} \left( f \right)  
	&=
	 \lim_{N \rightarrow \infty} N^{-d} \log 
			 \int_{\Wsp^{\Nd} \times Y^{\Nd}}
						 e^{\Nd  \sPro{  L^{N}_{r}  }{  f  } } 
						 \zeta^{N} \left( \dd \ul{\w}^{N} \right) \otimes Q^{N}_{\ul{\w}^{N}}   \left( \dd \ul{y} \right) 
	\\
	&=
		\lim_{N \rightarrow \infty} N^{-d} \log \prod_{k \in \TN}
			\int_{\Wsp} \int_{Y} e^{\sum_{\ell=1}^{r} f_{\ell}\left( \frac{k}{N}, \w, y_{\ell} \right) } 
							Q_{\frac{k}{N},\w} \left( \dd y_{1},...,\dd y_{r} \right) \zeta_{\frac{k}{N}} \left( \dd \w \right)
	\\
	&=
		\lim_{N \rightarrow \infty} N^{-d} \sum_{k \in \TN }   \log 
			\int_{\Wsp} \int_{Y} e^{\sum_{\ell=1}^{r} f_{\ell}\left( \frac{k}{N}, \w, y_{\ell} \right) } 
			Q_{\frac{k}{N},\w } \left( \dd y_{1},...,\dd y_{r} \right) \zeta_{\frac{k}{N}} \left( \dd \w \right)
	\\
	&=
			 \int_{\Td} \log \left(
						\int_{\Wsp} \int_{Y} e^{\sum_{\ell=1}^{r} f_{\ell}\left( x, \w, y_{\ell} \right) } 
						Q_{x,\w } \left( \dd y_{1},...,\dd y_{r} \right) \zeta_{x} \left( \dd \w \right)
				\right) 
			\dd x
	,
\end{split}
\end{align}
where we use in the last equality the Lemma~\ref{lem::SanovT::ass::Init}~\ref{lem::SanovT::ass::Init::Convergence} and Lemma~\ref{lem::SanovT::ass::FellerWithAddx}.
Note that by Lemma~\ref{lem::SanovT::ass::FellerWithAddx} and by $H_{f}$ (defined in \eqref{eq::pf::lem::SanovT::ass::Init::Convergence::Hf}) being continuous, all integrals in \eqref{eq::pf::lem::SanovT::logMoment} are well defined. 

The right hand side of \eqref{eq::pf::lem::SanovT::logMoment} is finite and Gateaux differentiable. 
Also as in  \cite{DawGarLDP} we can show if $L_{\nu,\zeta} \left( \Gamma^{1},...,\Gamma^{r} \right)  < \infty$, then $\Gamma^{i} \in \MOne{ \Td \times Y_{i} } $. Therefore all conditions of Theorem~3.4 in  \cite{DawGarLDP} are satisfied and the claims of Lemma~\ref{lem::SanovT} are proven.
\end{proof}

\begin{proof}[of Lemma~\ref{lem::SanovT::RelEntropy}]
By Lemma~\ref{lem::SanovT}, we know that $\left\{ L^{N}_{r} \right\}$ satisfies under $\left\{Q_{v_{N}}^{N}\right\}$ a LDP with rate function $L_{\nu,\zeta} \left( \Gamma \right)$.
Now we show that the rate function $L_{\nu,\zeta}$ has the claimed representation \eqref{eq::lem::SanovT::RelEntropy}.
The measure $\left( \dd x \otimes \zeta_{x} \left( \dd \w \right) \otimes Q_{x,\w} \right)$ in the relative entropy is well defined by Lemma~\ref{lem::SanovT::ass::dxQxnu}.

\emptyline
\begin{steps}
\step[{If $L_{\nu,\zeta} \left( \Gamma \right)  < \infty$ then $\Gamma \in \MOneL{ \TW \times Y }$}]
\label{pf::lem::SanovT::RelEntropy::step::ML}

Fix  $\Gamma \in \MOne{ \TW \times Y } $ with $L_{\nu,\zeta} \left( \Gamma \right)  < \infty$.
Then $\int_{\TWR} f \left( x \right) \Gamma \left( \dd x , \dd \w,  \dd \theta \right) = \int_{\Td} f \left( x \right)  \dd x$ for all $f \in \CspL{b}{ \Td}$.
Indeed, assume there were a $f \in \CspL{b}{\Td}$ for which this is not satisfied.
Then for all $\lambda \in \R$, 
\begin{align}
	L_{\nu,\zeta} \left( \Gamma \right) 
	\geq 
		\lambda \int_{\TWR} f \left( x \right) \Gamma \left( \dd x , \dd \w, \dd \theta \right) 
		 -
		  \lambda \int_{\Td} f \left( x \right)  \dd x \not = 0
	.
\end{align}
Because $\lambda$ is arbitrary, this is a contradiction to $	L_{\nu,\zeta} \left( \Gamma \right)   < \infty$.
\emptyline

For each open $A \subset \Td$, we can find a sequence of $f_{n}\in \CspL{b}{\Td}$, such that $f_{n}\geq 0$, $f_{n} \nearrow \1_{A}$  (see e.g. \cite{AshReal} A6).
Therefore we get by the dominant convergence theorem that the projection of $\Gamma$  on $\Td$ has to be the Lebesgue measure.
The disintegration theorem for measures on a product space  (see \cite{AttButMichVar} Theorem~4.2.4)  states that $\Gamma = \dd x \otimes \Gamma_{x} $ with $\Gamma_{x} \in \MOne{ \Wsp \times Y } $.

\emptyline
\step[{$L_{\nu,\zeta} \left( \Gamma \right)   \leq \relE{  \dd x \otimes  \Gamma_{x}     }{  \dd x \otimes \zeta_{x} \left( \dd \w \right) \otimes Q_{x,\w} }$ for $\Gamma \in \MOneL{\TW \times Y} $}]
\label{pf::lem::SanovT::RelEntropy::step::LleqRel}

Fix $\Gamma \in \MOneL{\TW \times Y} $, such that $\relE{ \dd x \otimes  \Gamma_{x}   }{  \dd x \otimes \zeta_{x} \left( \dd \w \right) \otimes Q_{x,\w} } < \infty$. 
Hence $\dd x \otimes  \Gamma_{x}$ is absolute continuous with respect to $\dd x \otimes \zeta_{x} \left( \dd \w \right) \otimes Q_{x,\w}$ with density $\rho$: 
\begin{align}
\dd x \otimes  \Gamma_{x} \left( \dd \w, \dd y \right) 
=
\rho \left( x,\w, y \right)  \dd x \otimes \zeta_{x} \left( \dd \w \right) \otimes Q_{x,\w} \left( \dd y \right)
.
\end{align}
Because $\Gamma \in \MOneL{\TW \times Y} $, $\int_{\Wsp} \int_{Y} \rho \left( x,\w, y \right)  Q_{x,\w} \left( \dd y \right) \zeta_{x} \left( \dd \w \right) =1$ for all $ x  \in \Td$.
The claimed upper bound on $L_{\nu,\zeta} \left( \Gamma \right)$, follows from finally by the same steps as in the second point of the proof of Theorem~3.1 in \cite{MicRobLD}.

\emptyline
\step[{$L_{\nu,\zeta}  \left( \Gamma \right)  \geq \relE{  \dd x \otimes  \Gamma_{x}   }{  \dd x \otimes \zeta_{x} \left( \dd \w \right) \otimes Q_{x,\w}  }$ for $\Gamma \in \MOneL{\TW \times Y} $}]
\label{pf::lem::SanovT::RelEntropy::step::LgeqRel}

This is just an application of Jensen's inequality to the convex function $-\log$ in $L_{\nu,\zeta}$ and the variation formula of the relative entropy.

\emptyline
\step[{Second representation of rate function}]
\label{pf::lem::SanovT::RelEntropy::step::2Rep}

The second representation of the rate function follows by \cite{DupEllAWeakCon} Theorem~C.3.1.
\end{steps}\vspace{-\baselineskip}
\end{proof}

\begin{remark}
	When $r>1$ in Lemma~\ref{lem::SanovT::RelEntropy} , also the \ref{pf::lem::SanovT::RelEntropy::step::ML},\ref{pf::lem::SanovT::RelEntropy::step::LleqRel} and \ref{pf::lem::SanovT::RelEntropy::step::2Rep} of the proof of Lemma~\ref{lem::SanovT::RelEntropy} are true.
	However the \ref{pf::lem::SanovT::RelEntropy::step::LleqRel} is in general not true any more due to the larger set $\Csp{\TW \times Y}$ used in the variation formula of $\relE{.}{ \dd x \otimes \zeta_{x} \left( \dd \w \right) \otimes Q_{x,\w}  }$, compared to the set of functions used in the supremum in $L_{\nu,\zeta}$.
\end{remark}

\begin{remark}
	We could choose the initial distribution of the $\Nd$ dimensional system 
	 more general than $\nu^{N}$ being the product measures over the $\nu_{\frac{k}{N}}$,
	 and still get the results of Lemma~\ref{lem::SanovT} and Lemma~\ref{lem::SanovT::RelEntropy}.
	
	For example take measures $ \left\{ \nu^N_{k} \right\}_{k \in \TN, N \in \N} \subset \MOne{\R}$  and define the product measures $\nu^{N}$ with these measures instead of $\nu_{\frac{k}{N}}$.
	If for each $\epsilon >0$ and each positive $f \in \CspL{b}{\TWR}$, there is a $N_{\epsilon,f} \in \N$ such that
	\begin{align}
		\sup_{ N>N_{\epsilon,f} }
		\sup_{\w \in \Wsp}
		\sup_{k \in \TN} \abs{ 
					\int_{\R}
					\int_{Y} f \left( \frac{k}{N}, \w, y \right) Q_{\frac{k}{N},\w,\theta} \left( \dd y \right)
					\left(
					\nu^{N}_{k} \left( \dd \theta \right)
					-
					\nu_{\frac{k}{N}} \left( \dd \theta \right)
					\right)
					}
		<
			\epsilon
		,
	\end{align}
	then \eqref{eq::pf::lem::SanovT::logMoment} would also hold for these measures.

\end{remark}

\begin{remark}
We could exchange the space $\Td$ by an arbitrary compact Polish spaces $X$.
If adjusted assumptions hold for $X$, then we would get the same large deviation result.
We need the Lemma~\ref{lem::SanovT} in the sequel only with the space $\Td$. 
To simplify the comprehensibility, we state it here not in its most general form.
\end{remark}

\subsubsection{Weaker assumptions on the initial distributions than Assumption~\ref{ass::LMF::Init}}
\label{sec::Prel::SanovT::AssDiff}

We require for the proof of Lemma~\ref{lem::SanovT} the result of Lemma~\ref{lem::SanovT::ass::Init} but not necessarily the Assumption~\ref{ass::LMF::Init}.
Also for Lemma~\ref{lem::SanovT::RelEntropy}, we only need the result of Lemma~\ref{lem::SanovT} and Lemma~\ref{lem::SanovT::ass::dxQxnu}.

The results of Lemma~\ref{lem::SanovT::ass::Init} and of Lemma~\ref{lem::SanovT::ass::dxQxnu} also hold for initial distributions, that are not Feller continuous, i.e. that do not satisfy Assumption~\ref{ass::LMF::Init}, if weaker conditions are satisfied. We show this in the following lemma.

In the proof of Lemma~\ref{lem::SanovT::ass::Init}~\ref{lem::SanovT::ass::Init::Convergence}, we essentially need only that $H_{f}$ (defined in \eqref{eq::pf::lem::SanovT::ass::Init::Convergence::Hf}) is Riemann integrable.
Moreover we use in the proof Lemma~\ref{lem::SanovT::ass::Init}~\ref{lem::SanovT::ass::Init::WellDefinedLimit}, only that $F_{f}$ (defined in \eqref{eq::pf::lem::SanovT::ass::Init::Convergence::Ff}) is Borel-measurable for all $f\geq 0$.
Last but not least we need in the proof of Lemma~\ref{lem::SanovT::ass::dxQxnu}, that $\ol{H}_{f}$ (defined in \eqref{eq::pf::lem::SanovT::ass::Init::Convergence::Hbarf}) is Borel-measurable.
In the proofs of these lemmas we showed these properties by applying the Assumption~\ref{ass::LMF::Init}.
However these properties also follow from different conditions, as we show in the following Lemma.

\begin{lemma}
\label{lem::SanovT::ass::InitSuff}
	If we assume Assumption~\ref{ass::SanovT::MeasFeller} and Assumption~\ref{ass::LMF::Medium} and that
	\begin{enuRom}
	\item
	\label{lem::SanovT::ass::InitSuff::Riemann}
	$F_{f}$ is Riemann integrable for $f \in \CspL{b}{\R} $, $f \geq 0$,
	\item
	\label{lem::SanovT::ass::InitSuff::tight}
	the set $\left\{ \nu_{x} \right\}_{x \in \Td}$ is tight and
	\item
	\label{lem::SanovT::ass::InitSuff::Borel}
	 $F_{f}$ is Borel-measurable for all non-negative  $f \in \CspL{b}{\R} $,
	\end{enuRom}
	
then the statements of Lemma~\ref{lem::SanovT::ass::Init} and of Lemma~\ref{lem::SanovT::ass::dxQxnu} also hold.
	
\end{lemma}

\begin{remark}
	The conditions \ref{lem::SanovT::ass::InitSuff::Riemann}, \ref{lem::SanovT::ass::InitSuff::tight} and \ref{lem::SanovT::ass::InitSuff::Borel} are all implied by the Assumption~\ref{ass::LMF::Init}.
\end{remark}

\begin{proof}
	As discussed at the beginning of this chapter, we only need to show that $H_{f}$ is Riemann integrable and that $\ol{H}_{f}$ is Borel-measurable.

\begin{steps}
\step[$H_{f}$ is Riemann integrable]
\label{pf::lem::SanovT::ass::InitSuff::Riemann}

We fix an $f \in \CspL{b}{\TWR}$, with $f>c>0$.
For each $\epsilon>0$, we construct now a Riemann integrable function $H_{f,\epsilon} : \Td \rightarrow \R$ which satisfies 
\begin{align}
\label{eq::pf::lem::RiemWithX::ToShow}
	\iNorm{ H_{f} \left( . \right) - H_{f,\epsilon}  \left( . \right) }
	<
	\epsilon
	.
\end{align}
This implies the uniform convergence of Riemann integrable functions to $H_{f}$ and therefore also that $H_{f}$ is Riemann integrable.

\emptyline

For all $\epsilon>0$, there is a $N_{\epsilon} \in \N$, such that
\begin{align}
	\abs{
	\int_{\Wsp \times \R} \left( f \left( x_{1}, \w, \theta \right) -  f \left( x_{2}, \w, \theta \right)  \right)  
	\nu_{x_{1}} \left( \dd \theta \right)
	\zeta_{x_{1}} \left( \dd \w \right)
	}
	\leq
		\epsilon
	,
\end{align}
for $x_{1},x_{2} \in \Td$ with $\iNorm{ x_{1}-x_{2} } \leq \frac{1}{N_{\epsilon}}$.
This follows by the same calculation used in \eqref{eq::lem::SanovT::ass::FellerWithAddx::pf::1}, by the tightness of $\nu_{x}$ (condition \ref{lem::SanovT::ass::InitSuff::tight}) and the tightness of $\zeta_{x}$ (Assumption~\ref{ass::LMF::Medium}).
Hence \eqref{eq::pf::lem::RiemWithX::ToShow} is satisfied with
\begin{align}
	H_{f,\epsilon} \left( x \right)  
	\defeq 
	\int_{\Wsp} \int_{\R} f \left( \frac{i}{N_{\epsilon}}, \w, \theta \right) \nu_{x} \left( \dd \theta \right) \zeta_{x}\left( \dd \w \right)
	, 
\end{align}
where $i \in \Td_{N_{\epsilon}}$ is chosen such that $\iNorm{ \frac{i}{N_{\epsilon}} -x } \leq \frac{1}{2N_{\epsilon}}$.

Moreover $H_{f,\epsilon}$ is Riemann integrable.
Indeed, for each $i \in \Td_{N_{\epsilon}}$ and each $x \in \Td$  with  $\iNorm{ \frac{i}{N_{\epsilon}} -x } \leq \frac{1}{2N_{\epsilon}}$, the function in the integrand is always the function $f \left( \frac{i}{N_{\epsilon}},\w, \theta \right)$.
Therefore \ref{lem::SanovT::ass::InitSuff::Riemann} implies that $H_{f,\epsilon}$ is Riemann integrable on this interval.
There are only finitely many such rectangles and therefore $H_{f,\epsilon}$ is Riemann integrable on $\Td$.

\emptyline
\step[ $\ol{H}_{f}$ is Borel-measurable]

Fix a $f \in \CspL{b}{Y}$.
As shown in the proof of Lemma~\ref{lem::SanovT::ass::dxQxnu}, the function $\TWR \ni \left( x, \w, \theta \right) \mapsto \int_{Y} f \left( y \right) Q_{x,\w, \theta} \left( \dd y \right) $ is continuous and bounded by Assumption~\ref{ass::SanovT::MeasFeller}.
Therefore it suffices to prove that for all $g \in \CspL{b}{\TWR}$, the function $x,\w \mapsto \int_{\R} g \left( x,\w, \theta \right) \nu_{x} \left( \dd \theta \right)$ is Borel measurable.
By the same argument as in \ref{pf::lem::SanovT::ass::InitSuff::Riemann}, we can approximate $\ol H_{g}$ uniformly by $\ol H_{g,\epsilon}$.
Then we only need to show that the $\ol H_{g,\epsilon}$ are Borel-measurable.
This follows as in \ref{pf::lem::SanovT::ass::InitSuff::Riemann}, but now by condition \ref{lem::SanovT::ass::InitSuff::Borel} instead of  \ref{lem::SanovT::ass::InitSuff::Riemann}.

\end{steps}\vspace{-\baselineskip}
\end{proof}

\subsubsection{Examples of initial distributions}
\label{sec::Prel::SanovT::AssExamples}

\begin{example}
We give now three easy examples of initial distributions that satisfy the Assumption~\ref{ass::LMF::Init}.
\begin{enuRom}
	\item
		All initial distributions equal each other, i.e. $\nu_{x} = \nu_{0} \in \MOne{ \R } $.
	\item
		There is a function $g \in \Csp{\Td}$ such that $\nu_{x} = \delta_{g \left( x \right) }$.
	\item
			There is a function $g \in \Csp{\Td}$ such that $\nu_{x}$ is normal distributed with mean $g \left( x \right)$ and variance one, i.e. $\nu_{x} \sim N \left( g \left( x \right), 1 \right)$.		
\end{enuRom}

\end{example}

\begin{example}
Let us now state some examples of function, that satisfy the conditions of Lemma~\ref{lem::SanovT::ass::InitSuff} and are therefore also usable.
\begin{itemize}	
\item
	Let $A_{i} \subset \Td$ be measurable  disjoint rectangles  such that $\Td = \bigcup_{i=1}^{n} A_{i}$.
	Let $\mu_{x} = \mu_{A_{i}}$ for $x \in A_{i}$.
	Then $F_{f}$ is a step function and therefore Borel measurable and Riemann integrable.
	
	Moreover the set $\left\{ \mu_{x} \right\}$ is a finite set of probability measures and therefore tight.
	
	However the stronger Assumption~\ref{ass::LMF::Init} is in general not satisfied.

\item
	Explicit example of such measures are for example $\nu_{x} = \nu^{\textnormal{UP/Down}}$ respectively on the upper and lower half of the torus.
	The measures $\nu^{\textnormal{UP/Down}}$ could be for example  $\delta_{\pm 1}$ or $N \left( \pm 1 , 1 \right)$.

\item
	$F_{f}$ is also Borel measurable and Riemann integrable, if it is a uniform limit of step functions (i.e. a $d$-dimensional regulated function).
	Therefore even more general measures $\nu_{x}$ are possible, as long as these measures satisfy the tightness assumption.

\end{itemize}
\end{example}

\subsection{Extended Arzel\'a-Ascoli theorem}
\label{sec::Prel::AA}

We give now a mild generalisation of the Arzel\'a-Ascoli theorem to subsets of $\TWC$.
By the compactness of $\Td$ we basically only have to take care of the projections of a set $A \subset \TWC$ to the $\Wsp$ and the $\Csp{\Ti}$ component. For the latter projection we can use the conditions of the original  Arzel\'a-Ascoli theorem.

\begin{lemma}[Extended Arzel\'a-Ascoli Theorem]
\label{lem::AA}
\hspace{1cm}

\begin{enuRom}
\item
\label{lem::AA::AA}
$A \subset \TWC$ is relatively compact 
if and only if 
\begin{align}
\Proj_{\CspSymbol} \left[ A \right] = \left\{ \theta_{\Ti}  \in \Csp{\Ti} : \exists \left( x,\w \right) \in \TW :  \left( x, \w, \theta_{\Ti}  \right)  \in A  \right\}
\end{align}
 is equibounded and equicontinuous and $ \Proj_{\Wsp} \left[ A \right] $ is relatively compact.

\item
\label{lem::AA::AppliedToMeasure}
A sequence $\{Q^{(n)}\} \subset \MOne{ \TWC } $ is tight 
if and only if
\begin{enumerate}
\item
	for each $\eta>0$ there exists an $a>0$  such that for all $n>0$ and  $t \in \Ti $
	\begin{align}
		Q^{(n)}\left[   \left( x,\w,\theta_{\Ti}  \right)  \in \TWC : \abs{ \theta_{0}  }  \geq a \right]  \leq \eta
	\end{align} 
	and

\item
	for each $\kappa,\eta>0$ there exists $\delta \in  \left( 0,1 \right) $ such that for all $n>0$
	\begin{align}
			Q^{(n)}\left[   \left( x,\w,\theta_{\Ti}  \right)  \in \TWC : \sup_{\abs{ t-s } \leq \delta} \abs{ \theta_{t}-\theta_{s}   } \geq \kappa   \right]  
			\leq 
				\eta
	\end{align} 
	and
\item
	for each $\eta>0$ there exists an $M>0$  such that for all $n>0$
		\begin{align}
		Q^{(n)}\left[   \left( x,\w,\theta_{\Ti}  \right)  \in \TWC : \abs{ \w }  \geq M \right]  
		\leq \eta
		.
		\end{align} 
\end{enumerate}

\end{enuRom}

\end{lemma}

\begin{proof}
\begin{enuRomNoIntendBf}
\item
We claim that the relative compactness of $A$ is equivalent to the relative compactness of $\Proj_{\CspSymbol} \left[ A \right]$ and the relative compactness of  $\Proj_{\Wsp} \left[ A \right] $.

Then~\ref{lem::AA::AA} follows from the Arzel\'a-Ascoli theorem  (see for example  \cite{BilConv} Theorem~7.2).

\RarrBold
If $A$ is relatively compact, then,  for each $\epsilon$, there are $n=n \left( \epsilon \right) \in \N$ tuples $  \left( x^{(\ell)},  \w^{(\ell)}, \theta^{(\ell)}_{\Ti}  \right)_{\ell=1}^n  \subset \TWC$, such that $A \subset \bigcup_{\ell=1}^{n} B_{\epsilon} \left(  \left( x^{(\ell)},  \w^{(\ell)}, \theta^{(\ell)}_{\Ti} \right)  \right) $.
Then 
\begin{align}
		\Proj_{\CspSymbol} \left[ B_{\epsilon} \left(  \left( x^{(\ell)}, \w^{(\ell)}, \theta^{(\ell)}_{\Ti}  \right)  \right) \right]
		=
		B_{\epsilon} \left(  \theta^{(\ell)}_{\Ti}  \right)
\end{align}
and therefore 
$\Proj_{\CspSymbol} \left[ A \right] \subset \bigcup_{i=1}^{n} B_{\epsilon} \left( \theta^{(\ell)}_{\Ti}  \right) $.
Hence we found a finite open cover of $\Proj_{\CspSymbol} \left[ A \right]$, i.e. $\Proj_{\CspSymbol} \left[ A \right]$ is totally bounded and therefore relatively compact.

By the same argument there is a finite open cover for $\Proj_{\Wsp} \left[ A \right]$.

\LarrBold
If $\Proj_{\CspSymbol} \left[ A \right]$ is relatively compact, then $\Proj_{\CspSymbol} \left[ A \right] \subset \bigcup_{\ell=1}^{n} B_{\epsilon} \left( \theta^{(\ell)}_{\Ti}  \right) $.
If $\Proj_{\Wsp} \left[ A \right]$ is relatively compact, then $\Proj_{\Wsp} \left[ A \right] \subset \bigcup_{i=1}^{n'} B_{\epsilon} \left( \w^{(i)} \right) $.
 This implies that $A$ is totally bounded with open cover
$A \subset \bigcup_{\ell=1}^{n} \bigcup_{i =1}^{n'} \bigcup_{k \in \Td_{\frac{1}{\epsilon}}} B_{4\epsilon} \left(   \left( k \epsilon,\w^{(i)}, \theta^{(\ell)}_{\Ti}  \right)   \right) $.

\item
This claim follows by applying part~\ref{lem::AA::AA} of this lemma, as in the proof in \cite{BilConv} Theorem~7.3.
\end{enuRomNoIntendBf}\vspace{-\baselineskip}
\end{proof}

\subsection{Distribution-valued functions}
\label{sec::Prel::Distr}

In this chapter we state the definitions and results of Chapter~4.1 of \cite{DawGarLDP} transferred to the space-dependent setting considered here.

\begin{definition}
\label{def::Distr::TestFunc}
\begin{itemNoLeftIntend}[itemsep=0em]
	\item
	We denote by  $\D = \CspL[\infty]{c}{\TWR} $ the space of test functions having compact support and continuous derivatives of all orders with the usual inductive topology.
	\item
	For a compact set $K\subset \TWR$, let $\D_{K}$ be the subset of $\D$ of functions with support in $K$.
	\item
	By $\D'$ and $\D_{K}'$, we denote the space of real distributions on $\D$ respectively on $\D_{K}$.
	\item
	Moreover we write $\sPro{\xi}{f}$ for the application of $\xi \in \D'$ to $f \in \D$.
\end{itemNoLeftIntend}	
\end{definition}

\begin{definition} [Variation of Definition~4.1 in \cite{DawGarLDP}]
\label{def::Distr::AbsCont}
	A map $\xi   : \Ti \rightarrow \D'$ is called absolutely continuous
	 if for each compact set $K \subset \TWR$, there exist a neighbourhood $U_{K}$ of $0$ in $\D_{K}$ and a absolutely continuous function $H_{K}: \Ti \rightarrow \R$ such that
	 \begin{align}
	 	\abs{
			\sPro{\xi \left( u \right) }{f}
			-
			\sPro{\xi \left( v \right) }{f}
		}
	\leq
		\abs{
			H_{K} \left( u \right)  - H_{K} \left( v \right) 
		}
	 \end{align}
for all $u,v \in I$ and $f \in U_{K}$.
	 
\end{definition}

\begin{lemma}[Lemma~4.2 in \cite{DawGarLDP}]
\label{def::Distr::Derivatives}
	If $\xi  : \Ti \rightarrow \D'$ is absolutely continuous,
	then
	$\sPro{\xi \left( . \right) }{f} : \Ti \rightarrow \R$ is also absolutely continuous for each $f \in \D$.
	
	Moreover the time derivative of $\xi$ in the distributions sense
	\begin{align}
		\partial_{t}{\xi} \left( t \right)  = \lim_{h\rightarrow 0} h^{-1} \left( \xi \left( t+h \right)  - \xi \left( t \right)  \right) 
	\end{align}
exists for almost all $t \in \Ti$.

\end{lemma}

\begin{lemma}[Lemma~4.3 in \cite{DawGarLDP},integration by parts]
\label{def::Distr::IntbyParts}
	For all absolutely continuous map $\xi : \Ti \rightarrow \D'$ and each $f\in \CspL[\infty]{c}{\TTWR} $,
	\begin{align}
		\sPro{\xi \left( t \right) }{f \left( t \right) }-\sPro{\xi \left( s \right) }{f \left( s \right) } =
		\int_{s}^{t}\sPro{\partial_{t}{\xi} \left( u \right) }{f \left( u \right) } \dd u 
		+
		\int_{s}^{t}\sPro{\xi \left( u \right) }{\partial_{t}{f} \left( u \right) } \dd u 
		\textnormal{ .}
	\end{align}
\end{lemma}

The proofs of these two lemmas are analogue to the one of Lemma~4.2 in \cite{DawGarLDP} respectively Lemma~4.3 in \cite{DawGarLDP}.
The crucial property of $\D$ and $\D_{K}$ for the proofs is their separability.
This is the case for the spaces considered here as well as in \cite{DawGarLDP}.

\begin{remark}
	We apply the results of this chapter later to probability measure valued functions in $\Cem$.
	This is possible because each measure in $\MOne{ \TWR } $ is a Radon measure and hence also an element of $\D'$.
\end{remark}

\subsection{Relation between the spaces of the empirical measures and empirical processes}
\label{sec::Prel::RelBetwSpaces}

We are looking at two different levels of large deviation principles.
The higher level are the empirical measures $L^{N}$ in $\MOne{\TWC}$.
The second level are the empirical processes $\empP$ in $\Cem$.
Both elements are defined (see \eqref{def::LN} and \eqref{def::empM}) as images of the paths of the spin values on the space $\Csp{\Ti}^{\Nd}$ and of the random environment $\ul{\w}^{N}\in \Wsp^{\Nd}$.

\emptyline
Let us now define a map $\Pi : \MOne{\TWC} \rightarrow \Cem$, which maps  $L^{N}$ to $\empP$ for each $N \in \N$.

\begin{definition}
\label{def::Pi}
For $Q \in \MOne{\TWC}$ we define $\Pi \left( Q \right)_{\Ti} \in \Cem$ for each $t \in \Ti$ by 
 \begin{align}
 \begin{split}
	\Pi \left( Q \right)_{t}  \left( \dd x,\dd \w,\dd \theta \right) 
	&=
		Q \left[ y= \left( y_{x}, y_{\w} , y_{\Ti} \right) \in \TWC :  \left( y_{x},y_{\w}, y_{t}  \right)  \in \dd x \dd \w \dd \theta \right]  
	\\
	&= 
		Q \circ  \left( id_{\Td},id_{\Wsp}, \theta_{t} \right) ^{-1}  \left( \dd x, \dd \w, \dd \theta \right) 
\end{split}
\end{align}
for $ \left( x,\w,\theta \right)  \in \TWR$.
\end{definition}
The measure $\Pi \left( Q \right)_{t}$ is the one-dimensional distribution at time $t \in \Ti$ of the measure $Q \in \MOne{\TWC}$.
Let us show that $\Pi \left( Q \right)_{\Ti}$ of Definition~\ref{def::Pi} is actually an element of the space $\Cem$.

\begin{lemma}
\label{lem::MapPi::WellDef}
	The function $\Pi$  is well defined.
\end{lemma}

\begin{proof}
Fix a $Q \in \MOne{ \TWC }$. We have to show that $\Pi \left( Q \right)_{\Ti} $ is in  $\Cem$.
By the definition of $\Pi$, we know already that $\Pi \left( Q \right)_{t}  \in \MOne{ \TWR } $ for all $t \in \Ti$.
Now we prove the continuity in time.
Take a bounded $L_{f}$-Lipschitz continuous function $f\in \CspL{b}{\TWR}$ and $s,t \in \Ti $ with $\abs{s-t} < \delta$ , then
\begin{align}
\label{eq::lem::MapPi::WellDef::BoundWeakCon}
\begin{split}
	&\abs{
		\int f \left( y,\w, \theta \right)  \left( \Pi \left( Q \right)_{t}  -  \Pi \left( Q \right)_{s} \right)
	}
=
	\abs{
		\int f \left( y,\w, \theta_{t} \right)  - f \left( y, \w, \theta_{s}  \right)  Q \left( \dd y,\dd \w,\dd \theta_{\Ti}  \right) 
	}
\\
&\leq 
	\int \abs{f \left( y, \w, \theta_{t} \right)  - f \left( y, \w, \theta_{s}  \right) } 
			\1_{ \abs{\theta_{t}-\theta_{s} } < \kappa} 
		Q \left( \dd y, \dd \w, \dd \theta_{\Ti}  \right)  
	+
	2 \iNorm{f} Q \left[ \hat{\theta}: \abs{\theta_{t}-\theta_{s} } \geq \kappa \right]
\\
&\leq
	L_{f} \kappa 
	+
	2 \iNorm{f}
		Q \left[ \sup_{\abs{u-v}<\delta}\abs{\theta_{u} -\theta_{v}} \geq \kappa \right]
	\leq
	\epsilon
	,
\end{split}
\end{align}
when $\kappa= \frac{\epsilon}{2L_{f}}$ and $\delta$ is small enough (by the extended  Arzel\'a-Ascoli Lemma~\ref{lem::AA}~\ref{lem::AA::AppliedToMeasure}).
Hence the Portmanteau theorem implies that $\Pi \left( Q \right)_{t_{n}}  \rightarrow \Pi \left( Q \right)_{t} $ weakly in $\MOne{ \TWR } $ if $t_{n} \rightarrow t$.
\end{proof}

Moreover we show now that $\Pi$ is a continuous function.

\begin{lemma}
\label{lem::MapPi::Cont}
	The function $\Pi$  is continuous.
\end{lemma}

\begin{proof}
The proof of this lemma follows the ideas in the proof of \cite{DawGarLDP} Lemma~4.6 for the mean field model.
	
	Take a sequence $Q^{(n)} \rightarrow Q$ in $\MOne{ \TWC } $. 
	This implies that for each $t\in \Ti $ and each $f\in \CspL{b}{\TWR}$, that is Lipschitz continuous,
	\begin{align}
	\label{eq::lem::MapPi::Cont::ConvProj}
			\abs{
				\int_{\TWR}  f \left( x,\w,  \theta\right) \left( \Pi \left( Q^{(n)} \right)_{t} - \Pi \left( Q \right)_{t} \right) \left(\dd x, \dd \w, \dd\theta\right)  
			}
			\rightarrow 
			0
			. 
	\end{align}
	The topology on $\Cem$ is the topology of uniform convergence.
	Therefore we have to show that the convergence \eqref{eq::lem::MapPi::Cont::ConvProj} is uniform in $t$.
	The weak convergence of $Q^{(n)}$ implies tightness  (Prokhorov's theorem), because $\TWC$ is a separable metric space.
	  Moreover we can split the absolute value in \eqref{eq::lem::MapPi::Cont::ConvProj} into the following summands
\begin{align}
\begin{split}
	 \textnormal{\eqref{eq::lem::MapPi::Cont::ConvProj}}
	  \leq&
		\abs{
			\int_{\TWR} f \left( x, \w, \theta\right) \left( \Pi \left( Q^{(n)} \right)_{s} - \Pi \left( Q \right)_{s} \right) \left(\dd x, \dd \w, \dd\theta\right)  
		}
	\\
	&
	+
		\abs{
			\int_{\TWR} f \left( x, \w, \theta\right) \left( \Pi \left( Q^{(n)} \right)_{t} - \Pi \left( Q^{(n)} \right)_{s} \right) \left(\dd x, \dd \w, \dd\theta\right)  
		}
	\\
	&+
		\abs{
			\int_{\TWR} f \left( x, \w, \theta\right) \left( \Pi \left( Q \right)_{t} - \Pi \left( Q \right)_{s} \right) \left(\dd x, \dd \w, \dd\theta\right)  
		}
	\eqdef
	\encircle{1}+\encircle{2}+\encircle{3}
	.
\end{split}
\end{align}

The \encircle{2} and \encircle{3} are bounded by $\epsilon$ for all $t,s \in \Ti$ with $\abs{t-s}<\delta$ for a $\delta$ small enough.
This can be shown as in \eqref{eq::lem::MapPi::WellDef::BoundWeakCon}.
Moreover the $\delta$ is the same for all $n \in \N$, because the analogue of \eqref{eq::lem::MapPi::WellDef::BoundWeakCon} is bounded uniformly in $n$ by  Lemma~\ref{lem::AA}~\ref{lem::AA::AppliedToMeasure}.

For each $k \in \left\{ 1,..., \frac{T}{\delta} \right\}$, there is a $N_{k} \in \N$, such that $	\encircle{1}$ is bounded by $\epsilon$ for  all $n>N_{k}$.

Therefore we conclude that for all $n>\max_{k=0}^{\frac{T}{\delta}} N_{k}$
\begin{align}
	  	\sup_{t}
		\abs{
			\int_{\TWR}  f \left( x, \w, \theta\right) \left( \Pi \left( Q^{(n)} \right)_{t} - \Pi \left( Q \right)_{t} \right) \left(\dd x, \dd \w, \dd\theta\right)  
		}
	\leq
	3 \epsilon
	,
 \end{align}
i.e. the uniform (in $t\in\Ti $) convergence of \eqref{eq::lem::MapPi::Cont::ConvProj}.
\end{proof}

\begin{notation}
	\label{nota::PIforC}
	With abuse of notation, we use the symbol $\Pi$ also for:
	\begin{itemNoLeftIntend}
		\item 
			The analogue defined function $\MOne{\Csp{\Ti}} \rightarrow \Csp{\Ti, \MOne{\R}}$.
			Then $\Pi \left( q \right)_{\Ti} \in \Csp{\Ti, \MOne{\R}}$ for $q \in \MOne{\Csp{\Ti}} $.
		\item
			The analogue defined function $\MOne{\Wsp \times \Csp{\Ti}} \rightarrow \Csp{\Ti, \MOne{\Wsp \times \R}}$.
	\end{itemNoLeftIntend}
		
\end{notation}

In the following lemma we state 
that the projection of $\Pi \left( Q \right)$ to $\Td$ is the Lebesgue measure, if this is the case for $Q$.
Moreover we show that the projection of $\Pi$ to the environment coordinate is frozen over time.

\begin{lemma}
\label{lem::MapPi::YoungMedium}
	 For $Q \in \MOneL{\TWC}$, $\Pi \left( Q \right)_{t} \in \MOneL{ \TWR }$ for all $t \in \Ti$.
	 Moreover $\Pi \left( Q \right)_{t,x,\Wsp} = \Pi \left( Q \right)_{0,x,\Wsp} = Q_{x,\Wsp}$ (see Definition~\ref{def::DecomMu}) for all $t \in \Ti$.
\end{lemma}

\begin{proof}
		Fix a $Q \in \MOneL{\TWC}$ and a $t \in \Ti$.
		Then $Q = \dd x \otimes Q_{x}$ and it is easy to see that
		$\Pi \left( Q \right)_{t} = \dd x \otimes \Pi \left(  Q_{x} \right)_{t}$.
		
		Moreover $Q = \dd x \otimes Q_{x,\Wsp} \left( \dd \w \right) \otimes Q_{x,\w}$. Then for all $ t \in \Ti$
		\begin{align}
		\begin{split}
			\intL_{\TW} \mkern-10mu  f \left( x, \w \right) Q_{x,\Wsp} \left( \dd \w \right) \dd x
			&=
			\mkern-20mu 
			\intL_{\TWC} \mkern-20mu  f \left( x, \w \right) Q 
			=
			\mkern-5mu 
			\intL_{\TW} \mkern-10mu  f \left( x, \w \right) Q_{x,\Wsp} \left( \dd \w \right)  \dd x 
			\\
			&=
			\mkern-10mu
			\intL_{\TWR} \mkern-10mu  f \left( x, \w \right) \Pi \left( Q \right)_{t} 
			=
			\mkern-5mu
			\intL_{\TW}  \mkern-10mu  f \left( x, \w \right)  \Pi \left( Q \right)_{t,x,\Wsp} \left( \w \right)  \dd x
			.
		\end{split}
		\end{align}
		
\end{proof}

\section{The LDP of the empirical process }
\label{sec::LDPempP}

In this chapter we state and prove the large deviation principle for the family of empirical processes $\left\{\empP\right\}$ define  in \eqref{def::empM}.
We investigate a more general setting than the model considered in Theorem~\ref{thm::LMF::LDPempM} (in Chapter~\ref{sec::Intro::ResultLMF}). 
Therefore we state at first some notation and assumptions.
We show in Chapter~\ref{sec::LocalMF}, that the concrete example of a local mean field model considered in Chapter~\ref{sec::Intro::ResultLMF} satisfies these assumptions.

We examine the $\Nd$ dimensional system of interacting spins defined by \eqref{eq::SDE},
with drift coefficient $b: \TWR \times \MOne{\TWR} \rightarrow \R$ and diffusion coefficient $\sigma>0$.
As explained in the introduction, the interaction between the spins is modelled as a dependency of the drift coefficient $b$ on the empirical measure.

We define the  $\Nd$ dimensional diffusion generator corresponding to \eqref{eq::SDE} for fixed environment $\ul{\w}^{N}$, acting on $f \in \CspL[2]{b}{\R^{\Nd}}$ by
\begin{align}
	\label{def::NParticleGenerator}
	\Gen^{N}_{\ul{\w}^{N}} f \left( \ul{\theta}^{N} \right)  
	\defeq	
		\sum_{k \in \TN} \Gen_{\empM,\frac{k}{N},\w^{k,N}}  \:\:  f \left( \ul{\theta}^{N} \right) 
	,
\end{align}
where $ \Gen_{\empM,\frac{k}{N},\w^{k,N}} $ is the operator defined in \eqref{def::InteractingGenerator} with derivatives in the $\theta^{k,N}$ direction and with drift coefficient $b \lb \frac{k}{N}, \w^{k,N}, . , \empM \rb : \R \rightarrow \R$.
The $\empM \in \MOne{\TWR}$ is the empirical measure defined as in \eqref{def::EmpMt} with $\ul{\theta}^{N}$ and $\ul{\w}^{N}$.

\emptyline
For the proof of the large deviation principle, we require that the drift coefficient $b$ is chosen in such a way that the following assumption is satisfied.

\begin{assumption}
\label{ass::LDPempP}

There is a non-negative function $\varphi \in \Csp[2]{\R} $ with $\lim_{\abs{\theta} \rightarrow  \infty} \varphi \left( \theta \right) = \infty$,
such that:
\begin{enuAlph}
	\item
		The function $b: \TWR \times \M_{\varphi,\infty}  \rightarrow \R$ satisfies:
		\begin{enuAlph}
		\item
		\label{ass::LDPempP::bCont}
			 The restriction of $b$ to $\TWR \times \left( \M_{\varphi,R} \cap \MOneL{\TWR} \right) \rightarrow \R$ is continuous for all $R>0$.
		\item
			\label{ass::LDPempP::blockbd}
				For all $N \in \N$ and all $\ul{\w}^{N} \in \Wsp^{\Nd}$,  $b^{N} : \RN \rightarrow \RN$, defined by 
				\begin{align}
						b^{N} \left( \ul{\theta}^{N} \right) \defeq \left( b \left( \frac{k}{N}, \w^{k,N}, \theta_{k}, \empM \right) \right)_{k \in \TN}
						,
				\end{align}
				is a locally bounded measurable function.
	\end{enuAlph}
	\item
	\label{ass::LDPempP::IntBound}
		There is a constant $\lambda>0$ and a $\ol{N} \in \N$, such that for all $N >\ol{N}$ and all empirical measures $\empM$ (defined by $\ul{\theta}^{N} \in \RN$ and $\ul{\w}^{N} \in \Wsp^{\Nd}$),
		\begin{align}
			\int_{\TWR} \Gen_{\empM,x, \w}   \varphi \left(  \theta \right) 
									+ \frac{\sigma^{2}}{2} \abs{\partial_{\theta} \varphi \left( \theta \right) }^{2} 
						\empM \left(\dd x, \dd \w, \dd\theta \right)
			\leq
				\lambda 	
				\int_{\TWR}  \varphi \left(  \theta \right) \empM \left(\dd x, \dd \w, \dd\theta\right)
			.
		\end{align}
	\item
	\label{ass::LDPempP::Lyapu}
		For each $\mu_{\Ti} \in \Cem_{\varphi,\infty} \cap \Cem^{L}$, there is a constant $\lambda \left( \mu_{\Ti}  \right) >0$ such that
		\begin{align}
			\Gen_{\mu_{t}, x, \w}  \varphi \left(\theta \right) + \frac{\sigma^{2}}{2} \abs{\partial_{\theta} \varphi \left(\theta \right) }^{2}
			\leq
			\lambda \left(  \mu_{\Ti}  \right)  \varphi \left(\theta \right)
			,
		\end{align}
		for all $\left( t, x, \w, \theta \right) \in \TTWR$.
	\item
	\label{ass::LDPempP::MuIntCont}
		For each $R>0$ and each $\ol\mu_{\Ti}  \in \Cem_{\varphi,R} \cap \Cem^{L}$, 
		\begin{align}
					\int_{0}^{T} 
					\int_{\TWR}
						\sigma^{2} \abs{b \left( x, \w, \theta,\mu^{(n)}_{t}  \right) -b \left( x, \w, \theta,\ol{\mu}_{t}  \right) }^{2}
						\mu^{(n)}_{t} \left(\dd x, \dd \w, \dd \theta \right) \dd t 
				\rightarrow
				0
				,
		\end{align}
		for $n \rightarrow \infty$,
		when $\mu^{(n)}_{\Ti} \rightarrow \ol\mu_{\Ti}$, for a sequence $\left\{ \mu^{(n)}_{\Ti} \right\} \subset \left( \Cem_{\varphi,R} \cap \Cem^{L} \right)$ 
		or a sequence
		\begin{align}
		\left\{ \mu^{(n)}_{\Ti} \right\} 
		\subset
		\left\{ \mu_{\Ti} \in \Cem_{\varphi,R} : \mu_{\Ti}=\empP \textnormal{ is a empirical process for a } N \in \N \right\}
		.
		\end{align}
\end{enuAlph}
\end{assumption} 

\begin{example}
We show in Chapter~\ref{sec::LocalMF}, that the concrete example of a local mean field model considered in Chapter~\ref{sec::Intro::ResultLMF} satisfies the Assumption~\ref{ass::LDPempP}.
\end{example}

\begin{remark}
\label{rem::LDPempP::MProbPNwelldef}
For each given environment $\ul{\w}^{N} \in \Wsp^{\Nd}$, the Martingale problem for the generator $\Gen^{N}_{\ul{\w}^{N}}$ is well posed by the Assumption~\ref{ass::LDPempP}~\ref{ass::LDPempP::blockbd} and~\ref{ass::LDPempP::IntBound}.
	Indeed, from Theorem~10.1.2 of \cite{StrVarMultidi} and Theorem~7.2.1 of \cite{StrVarMultidi}, we infer  the uniqueness of the solution to the Martingale problem, because the drift coefficient is locally bounded and measurable (Assumption~\ref{ass::LDPempP}~\ref{ass::LDPempP::blockbd}).
	For the existence of a solution of the Martingale Problem, we apply Theorem~10.2.1 of \cite{StrVarMultidi} with 
	$\varphi \left( \ul{\theta}^{N} \right) \defeq \frac{1}{\Nd} \sum \varphi \left( \theta^{k,N} \right)$.
	The conditions of this theorem are satisfied by Assumption~\ref{ass::LDPempP}~\ref{ass::LDPempP::IntBound}.
	We denote by $P^{N}_{\ul{\w}^{N},\ul{\theta}^{N}} \in \MOne{ \Csp{\Ti}^{N} } $ the unique solution of this martingale problem.
	For a short discussion of the other assumptions, see Remark~\ref{rem::LDPempP::DiscAss}. 
	
	With $P^{N}_{\ul{\w}^{N},\ul{\theta}^{N}}$, we define $P^{N}_{\ul{\w}^{N}}$ and $P^{N}$ as in Notation~\ref{nota::LMF::General}.
\end{remark}

Besides the Assumption~\ref{ass::LMF::Init} on the Feller continuity of the initial distribution $\left\{ \nu_{x} \right\}$, we require that these measures satisfy  the following uniform integration condition. 
	\begin{assumption}
	\label{ass::LDPempP::Init::Integ}
		There is a  $\ell>1$ such that
		\begin{align}
				\sup_{ x \in \Td }\int_{\R} e^{\ell \varphi \left( \theta \right)} \nu_{x} \left( \dd \theta \right)
				<
				C
				.
		\end{align}
	\end{assumption}

The following large deviation principle is the main result of this chapter.
\begin{theorem}
	\label{thm::LDPempP}
	Let the Assumption~\ref{ass::LMF::Init}, Assumption~\ref{ass::LMF::Medium}, Assumption~\ref{ass::LDPempP} and   Assumption~\ref{ass::LDPempP::Init::Integ} hold.
	Then the family $ \left\{ \empP, P^{N} \right\}$ satisfies on $\Csp{ \Ti , \MOne{\TWR}}$   a large deviation principle with good rate function
	\begin{align}
		\label{eq::thm::LDPempP::RF}
		S_{\nu,\zeta} \left( \mu_{\Ti}  \right)  
		\defeq
		\begin{cases}
			\int_{0}^{T} \abs{ \partial_{t}\mu_{t}  - \left( \Gen_{\mu_{t},.,.} \right)^{*}  \mu_{t}  }^{2}_{\mu_{t} } \dd t
			+
					\relE{\mu_{0}}{\dd x \otimes  \zeta_{x} \otimes \nu_{x}}
			 &\textnormal{if } \mu_{\Ti} \in \bbA \cap \Cem_{\varphi,\infty}
		\\
			\infty
		&\textnormal{otherwise,}
		\end{cases}
	\end{align}
where the norm $\abs{.}_{\mu_{t}}$ is defined in Definition~\ref{def::Minus1Norm} and where
\begin{align}
\label{eq::thm::LDPempP::SetA}
\begin{split}
	\bbA
	\defeq
		\Big\{
		\mu \in \Cem^{L} 
			:
			\mu_{\Ti} 
			\textnormal{ is absolutely continuous 
						in the sense of Definition~\ref{def::Distr::AbsCont}}
		\Big\}
	.
\end{split}
\end{align}

Moreover the integral with respect to $\TW$  and the supremum in the norm in $S_{\nu,\zeta}$ can be interchanged, i.e.
 $	S_{\nu,\zeta} \left( \mu_{\Ti}  \right)  
		=
		S^{\TW}_{\nu,\zeta} \left(  \mu_{\Ti}  \right)  $,
		defined by
	\begin{align}
		\label{eq::thm::LDPempP::RF::ST}
	\begin{split}
						\int_{0}^{T} \int_{\Td} \int_{\Wsp}
						\abs{
							\partial_{t}  \mu_{t,x,\w} - \left(\Gen_{\mu_{t},x,\w}\right)^{*}   \mu_{t,x,\w}
						}_{\mu_{t,x,\w}}
						\mu_{0,x,\Wsp} \left( \dd \w \right)
						 \dd x \dd t
			+
			\relE{\mu_{0}}{\dd x \otimes  \zeta_{x} \otimes \nu_{x}}
	\end{split}
	\end{align}
	if $\mu_{\Ti} \in \bbA \cap \Cem_{\varphi,\infty} \cap \Cem^{L} $
	and  $	S^{\TW}_{\nu,\zeta} \left(  \mu_{\Ti}  \right) = \infty$ otherwise.
\end{theorem}

To prove this theorem, we generalise the proof of the large deviation principle for the mean field model of \cite{DawGarLDP}, to the space  and random environment dependent setting we consider here.
Therefore the structure of the proof of Theorem~\ref{thm::LDPempP} is similar to the structure of the corresponding proof in \cite{DawGarLDP}. 
However there are three main differences to \cite{DawGarLDP} in the model we consider here.
The main differnce is that the drift coefficient $b$ and the empirical process $\empP$ depend on $x \in\Td$ and on the random environment $\w \in \Wsp$.
Moreover in \cite{DawGarLDP} the initial spin values are fixed, whereas in the model we consider, the initial spin values can be distributed randomly.
Last but not least, we show the large deviation principle on the space $\Cem$ (and not, as in \cite{DawGarLDP}, on $\Cem_{\varphi,\infty}$ with another topology than the subspace topology).

Due to these differences, changes are necessary in the proofs (compared to the approach in \cite{DawGarLDP}).
Many of these changes are of technical nature.
We point out at the beginning of each proof of the partial results, how the proof differs from the corresponding proof in \cite{DawGarLDP}.
Then we state the proofs with emphasis on these necessary modifications.
Of course we explain proofs and parts of proofs, that are new, completely.

The proof of Theorem~\ref{thm::LDPempP}  is organised as follows.
\begin{enuBr}
\item
At first (Chapter~\ref{sec::LDPempP::Inde}), we prove a large deviation principle for a system of independent spins (see Theorem~\ref{thm::LDPempP::Inde}) and show that the rate function has the representation $S^{I}_{\nu,\zeta}$ (defined in \eqref{eq::thm::LDPempP::Inde::RF}), that is similar to $S_{\nu,\zeta}$.
We infer this large deviation principle from the generalised Sanov-type large deviation result derived in Chapter~\ref{sec::Prel::SanovT}.
The rest of this Chapter~\ref{sec::LDPempP::Inde} is dedicated to showing that the rate function has the representation $S^{I}_{\nu,\zeta}$.
\begin{enuBr}
\item
To show the form of the rate function, we derive at first two different representations $S^{I,1}_{\nu,\zeta}$ and $S^{I,2}_{\nu,\zeta}$ of the rate function (Chapter~\ref{sec::LDPempP::2RepRate}).
For both representation we use the Sanov-type large deviation result derived in Chapter~\ref{sec::Prel::SanovT}.
These proofs are formally almost equal to the corresponding proofs in \cite{DawGarLDP}.
The space and random environment dependency only leads to formal changes in the notation.
However the applied results of Chapter~\ref{sec::Prel::SanovT} are different from the Sanov-type results used in  \cite{DawGarLDP}, due to these new dependencies.
Moreover to be able to apply the Sanov type result, we show that the measures corresponding to the independent SDEs are Feller continuous.

\item
Next we show that $S^{I,1}_{\nu,\zeta}$  ($S^{I,2}_{\nu,\zeta}$)  is an  upper (lower) bound on the claimed form $S^{I,\Td}_{\nu,\zeta}$ ($S^{I}_{\nu,\zeta}$) of the rate function (Chapter~\ref{sec::LDPempP::2RepRate}). 

In the proof of the upper bound (Chapter~\ref{sec::LDPempP::2RepBoundSI::Upper}), we generalise an approach used in \cite{PraHolMKV}, which is partially based on approaches of \cite{FolRandom} and \cite{BruFiniteKullb}.
In contrast to \cite{PraHolMKV}, we consider the space dependency $x \in \Td$ in addition to the random environment $\w \in \Wsp$.

Note that the proof of the lower bound given in \cite{PraHolMKV} unfortunately has a gap and 
cannot be used. We give a 
 proof of the lower bound in  (Chapter~\ref{sec::LDPempP::2RepBoundSI::Lower}) that generalises the 
 ideas used in \cite{DawGarLDP}.
Besides the usual formal changes (due to the space dependency, compared to \cite{DawGarLDP}), we 
have to handle a new problem here.
The proof requires the existence of a solution to a boundary value partial differential equation, which 
has to be continuous in the space variable $x \in \Td$ and the environment variable $\w \in \Wsp$.
This condition is obviously not needed in \cite{DawGarLDP}.
Therefore we show in the Chapter \ref{sec::LDPempP::PDE}, that there exist such a solution. This 
chapter and the proof are new.

\item
Finally we derive another formula for $S^{I}_{\nu,\zeta}$. This is (again modulo changes due to the space dependency) similar to the corresponding proof in \cite{DawGarLDP}. However in \cite{DawGarLDP} this formula is used to derive the large deviation upper bound.
We do not use it in the proof of the large deviation upper bound, because it only bounds $S^{I}_{\nu,\zeta}$ (see the beginning of Chapter~\ref{sec::LDPempP::2RepBoundSI::Upper} for more details).
However we need this result in Chapter~\ref{sec::LDPempP::Inter} to show that the rate function $S_{\nu,\zeta}$ is actually lower semi-continuous.

\end{enuBr}

\item
In Chapter~\ref{sec::LDPempP::Inter}, we infer from this large deviation principle for independent spins, a local large deviation principle for the interacting spin system (Theorem~\ref{thm::LDPempP::Local}).
To do this, we define the independent generator $\Gen^{I}_{t,x,\w} \defeq \Gen_{\ol\mu_{t},x,\w}$ for fixed $\ol\mu_{\Ti} \in \Cem_{\varphi,\infty} \cap \Cem^{L}$.
For the empirical process defined with the spin values that evolve according to the Langevin dynamics with this generator, we know by Chapter~\ref{sec::LDPempP::Inde} a large deviation principle.
From this principle, we infer the local large deviation principle under $\left\{ P^{N} \right\}$, with the help of exponential bounds (that we show in Chapter~\ref{sec::LDPempP::Inter::ExpBounds}). 
This is a again a generalisation of \cite{DawGarLDP}.
Moreover we give a new proof of the local large deviation principle around $\ol\mu_{\Ti}$ that are not in $\Cem_{\varphi,\infty} \cap \Cem^{L}$ (see Chapter~\ref{sec::LDPempP::Inter::PfLocalLDP}).
This is necessary because we assume the continuity of $b$ only on a subset of $\M_{\varphi,R}$ (see Assumption~\ref{ass::LDPempP}~\ref{ass::LDPempP::bCont}). 
Also with the mentioned exponential bounds, we prove the exponential tightness of $\left\{ \empP , P^{N} \right\}$  (Theorem~\ref{thm::LDPempP::ExpBound}).
Finally we infer from the exponential tightness and the local large deviation principle, the Theorem~\ref{thm::LDPempP}.

\end{enuBr}

We explain the steps and proofs in more details in the respective chapters.
We finish this chapter with a short discussion how the Assumption~\ref{ass::LDPempP} enter into this approach.

\begin{remark}
\label{rem::LDPempP::DiscAss}
As explained in Remark~\ref{rem::LDPempP::MProbPNwelldef}, we use Assumption~\ref{ass::LDPempP}~\ref{ass::LDPempP::blockbd} and~\ref{ass::LDPempP::IntBound}, to infer that the Martingale problem for the generator $\Gen^{N}_{\ul{\w}^{N}}$ is well defined.
Moreover the Assumption~\ref{ass::LDPempP}~\ref{ass::LDPempP::IntBound} implies the exponential bounds in Chapter~\ref{sec::LDPempP::Inter}.
We get analogue results for the independent system defined by the generator $L^{I}_{t,x,\w}$ due to Assumption~\ref{ass::LDPempP}~\ref{ass::LDPempP::bCont} and~\ref{ass::LDPempP::Lyapu}.
Finally, we require Assumption~\ref{ass::LDPempP}~\ref{ass::LDPempP::MuIntCont} to show that $S_{\nu,\zeta}$ is a good rate function (here we need the sequences in $\Cem^{L}$) and to connect the independent system with the interacting system when deriving the local large deviation principle in Chapter~\ref{sec::LDPempP::Inter} (here we need the sequences of empirical processes).	 
\end{remark}

\subsection{Independent spins}
\label{sec::LDPempP::Inde}

In this chapter we investigate the large deviation principle of the empirical process for systems of independent spins.
As explained, we derive such a system by fixing the interaction between the spins in the SDE \eqref{eq::SDE}.
Therefore we consider a drift coefficient  $b^{I}: \TTWR \rightarrow \R$ here that depends not any more on the empirical measure but on the time.

For each $x \in \Td$, $\w \in \Wsp$ and $t \in \Ti$, define the time-dependent diffusion generator
\begin{align}
\label{def::GeneratorIndependent}
	\Gen^{I}_{t,x,\w} 
	\defeq
	 \frac{1}{2} \sigma^{2} \frac{\partial^{2}}{\partial^{2} \theta} 
			+
			b^{I} \left( t,x, \w, . \right)  \frac{\partial}{\partial \theta} 
			,
\end{align}
 that corresponds to the SDE 
 \begin{align}
 \label{eq::SDEIndep}
	 \dd \theta^{x}_{t} 
	 =
	 b^{I} \left( t, x , \w, \theta^{x}_{t}  \right) 
	 \dd t
	 +
	 \sigma 
	 \dd B^{x}_{t}
	 .
 \end{align}

Let us assume that $b^{I}$ is chosen such that the following assumptions are satisfied.

\begin{assumption}
\label{ass::LDPempP::Inde}
	\begin{enuAlph}
	\item
	\label{ass::LDPempP::Inde::bCont}
		$b^{I}$ is continuous on $\TTWR$.
	\item
	\label{ass::LDPempP::Inde::MPro}
		For each $x \in \Td$ and each $\w \in \Wsp$, the Martingale problem for $\Gen^{I}_{t,x,\w}$ is well posed , with corresponding family of probability measures  
		$\left\{ P^{I}_{t,x,\w,\theta} \in \MOne{ \Csp{\left[ t,T \right]} } ,  \left( t,x, \w, \theta \right) \in \TTWR  \right\} $.
	\end{enuAlph}
\end{assumption}

We interpret $P^{I}_{t,x, \w, \theta}$ as the measure of the path of the spin value at the position $x \in \Td$ with initial value $\theta \in \R$ at time $s \in \Ti$ and fixed environment $\w \in \Wsp$, that evolves according to \eqref{def::GeneratorIndependent}.
We use the shorter notation $P^{I}_{x,\w,\theta}$, when $t=0$.
By \eqref{def::GeneratorIndependent}, the spin values at position $x,y \in \Td$ evolve mutually independent for $x \not = y$.

\begin{notation}
\label{nota::LDPempP::Inde::General}
	We write $P^{I}_{x,\w}$ for the distribution of the path of the spin value at the position $x \in \Td$ with fixed environment $\w \in \Wsp$ and with initial distribution $\nu_{x}$ at time $0$, i.e. $P^{I}_{x,\w} = \int_{\R} P^{I}_{x,\w,\theta} \nu_{x} \left( \dd \theta \right)$.
	
	Similar to Notation~\ref{nota::LMF::General}, we define 
	$P^{I,N}_{\ul{\w}^{N}} $ and $P^{I,N}$ (now with $P^{I}_{x, \w, \theta}$) .
\end{notation}

\emptyline

The following large deviation principle with the particular form of the rate function is the main result of this chapter.

\begin{theorem}
\label{thm::LDPempP::Inde}

Let the Assumption~\ref{ass::LMF::Init}, Assumption~\ref{ass::LMF::Medium} and Assumption~\ref{ass::LDPempP::Inde} hold.
Then the family $ \left\{ \empP, P^{I,N} \right\}$ satisfies on $\Csp{ \Ti , \MOne{\TWR}}$   a large deviation principle with good rate function 
\begin{align}
\label{eq::thm::LDPempP::Inde::RF}
	S^{I}_{\nu,\zeta} \left( \mu_{\Ti}  \right)  
	\defeq
	\begin{cases}
		\int_{0}^{T} \abs{ \partial_{t}\mu_{t}  - \left( \Gen^{I}_{t,.,.} \right)^{*}  \mu_{t}  }^{2}_{\mu_{t} } \dd t
		+
		\relE{\mu_{0}}{\dd x \otimes  \zeta_{x} \otimes \nu_{x}}
		\quad
		&\textnormal{if } \mu_{\Ti} \in \bbA
		\\
		\infty
		&\textnormal{otherwises,}
	\end{cases}
\end{align}
with $\bbA$ defined in \eqref{eq::thm::LDPempP::SetA}.

Moreover
 $	S^{I}_{\nu,\zeta} \left( \mu_{\Ti}  \right)  
		=
		S^{I,\Td}_{\nu,\zeta} \left(  \mu_{\Ti}  \right)  $,
		defined by
	\begin{align}
	\label{eq::thm::LDPempP::Inde::RF::SIT}
	\begin{split}
						 \int_{0}^{T} \int_{\Td} \int_{\Wsp}
							 \abs{
							 	\partial_{t}  \mu_{t,x,\w} - \left(\Gen^{I}_{t,x,\w}\right)^{*}   \mu_{t,x,\w}
								 }_{\mu_{t,x,\w}}
							\mu_{0,x,\Wsp} \left( \dd \w \right)
							\dd x \dd t
					+
						\relE{\mu_{0}}{\dd x \otimes  \zeta_{x} \otimes \nu_{x}}
	\end{split}
	\end{align}
if $\mu_{\Ti} \in \bbA$ and $	S^{I,\Td}_{\nu,\zeta} \left(  \mu_{\Ti}  \right) = \infty$ otherwise.
\end{theorem}

\begin{remark}
	The rate functions $S_{\nu,\zeta}$ (of Theorem~\ref{thm::LDPempP}) and $S^{I}_{\nu,\zeta}$ (of Theorem~\ref{thm::LDPempP::Inde}) are related to each other.
	Set $\Gen^{I}_{t,x,\w}=\Gen_{\ol{\mu}_{t},x,\w}$ for a $\ol\mu_{\Ti}  \in \Cem_{\varphi,\infty}$.
	And let $S^{I}_{\nu,\zeta}$ be the rate function defined by \eqref{eq::thm::LDPempP::Inde::RF} corresponding to this generator.
	Then $S_{\nu,\zeta} \left( \ol\mu_{\Ti}  \right)  = S^{I}_{\nu,\zeta} \left( \ol\mu_{\Ti}  \right) $.	
	We use this relation in Chapter~\ref{sec::LDPempP::Inter}.
\end{remark}

\begin{proof}[of Theorem~\ref{thm::LDPempP::Inde}]
It is easy to see that the family $\left\{\empP , P^{I,N} \right\}$ satisfies a large deviation principle 
by Lemma~\ref{lem::SanovT} and the contraction principle
(see the proof of Lemma~\ref{lem::LDPempP::indep::1Rep}).

The main difficulty of the proof of Theorem~\ref{thm::LDPempP::Inde} is to show that the rate function 
$S_{\nu,\zeta}$ has the form \eqref{eq::thm::LDPempP::Inde::RF}.
To prove this, we generalise the approach used to prove  Theorem~4.5 in \cite{DawGarLDP} to the setting we consider here.

As in \cite{DawGarLDP}, we derive two different representations, 
$S^{I,1}_{\nu,\zeta}$ and $S^{I,2}_{\nu,\zeta}$, of the rate function and  show that these provide  a lower bound  on $S^{I}_{\nu,\zeta}$ and  an 
upper bound  on $S^{I,\Td}_{\nu,\zeta}$,  respectively

To get the first representation, we use the contraction principle and transfer the LDP for $\left\{L^{N}, P^{I,N}\right\}$, that we get by Lemma~\ref{lem::SanovT}, to the LDP for $\left\{\empP , P^{I,N} \right\}$.

\begin{lemma}[compare to \cite{DawGarLDP} Lemma~4.6 for the mean field case]
	\label{lem::LDPempP::indep::1Rep}
	The family $ \left\{ \empP, P^{I,N} \right\}$ satisfies on $\Csp{ \Ti , \MOne{\TWR}}$  a large deviation principle with rate function
	\begin{align}
	\label{eq::lem::LDPempP::indep::1Rep::S1Inf}
		S^{I,1}_{\nu,\zeta} \left( \mu_{\Ti}  \right)  = \inf_{ Q \in \MOne{ \TWC }  : \Pi \left( Q \right)_{\Ti} =\mu_{\Ti}  } L_{\nu,\zeta}^{1} \left( Q \right) 
	\end{align}
	for $\mu_{\Ti}  \in \Cem$, with
		\begin{align}
		\begin{split}
		L_{\nu,\zeta}^{1}  \left( Q \right)  
		&= 
			\int_{\Td} \int_{\Wsp}
				\relE{Q_{x,\w}}{ P^{I}_{x,\w}}
				Q_{x,\Wsp} \left( \dd \w \right)
			\dd x
			+
			\int_{\Td}
			\relE{ Q_{x,\Wsp} } { \zeta_{x} }
			\dd x
		\\
		&=
		\sup_{f \in \CspL{b}{\TWC}} 
			\left\{ 
				\int_{\TWC} f \left(  x,\w,\theta_{\Ti}   \right)  Q \left( \dd x,\dd \w, \dd \theta_{\Ti}  \right) 
				\right.
		\\
		&\qquad\qquad
				-
				\left.
				\int_{\Td}
					\log
					\left(
					\int_{\Wsp}
					 \int_{\Csp{\Ti}}
						e^{f \left( x,\w,\theta_{\Ti}  \right) }
						P^{I}_{x,\w} \left( \dd \theta_{\Ti}  \right) 
						\zeta_{x} \left( \w \right)
					\right) 
				\dd x
			\right\} 
		\end{split}
		\end{align}
for $Q \in \MOneL{ \TWC } $ and $L^{1}_{\nu,\zeta} \left( Q \right) = \infty$ otherwise.
	
	In particular, $S^{I,1} \left( \mu_{\Ti} \right)$ is only finite if $\mu_{t} \in \MOneL{\TWR}$ for all $t \in \Ti$ 
	and if $\mu_{t,x,\Wsp} = \mu_{0,x,\Wsp}$ for all $t \in \Ti$ and almost all $x \in \Td$.
	
\end{lemma}

To derive the second representation, we define for $0 \leq s \leq t \leq T$ the operator acting on $f \in \CspL{b}{\TWR}$ by
\begin{align}
\label{def::LDPempP::OperatorU}
		U_{s,t}f \left( x,\w, \theta \right)  
		\defeq
		 \int_{\Csp{\left[s,T\right]}} f \left( x,\w, \theta_{t} \right)  P^{I}_{s,x,\w,\theta} \left( \dd \theta_{\left[s,T\right]}  \right) 
	.
\end{align}
With this operator we get the following representation of the rate function.

\begin{lemma}[compare to \cite{DawGarLDP} Lemma~4.7 for the mean field case]
	\label{lem::LDPempP::indep::2Rep}
	The family $ \left\{ \empP, P^{I,N} \right\}$ satisfies on $\Csp{ \Ti , \MOne{\TWR}}$   a large deviation principle with rate function
	\begin{align}
	\label{lem::LDPempP::indep::2Rep::S2}
		S^{I,2}_{\nu,\zeta} \left( \mu_{\Ti}  \right)  = \sup_{r \in \N, 0 \leq t_{1} < ... < t_{r} \leq T } L^{t_{1},...,t_{r}}_{\nu,\zeta} \left( \mu_{t_{1}},...,\mu_{t_{r}}  \right) 
			\textnormal{ for }
		 \mu_{\Ti}  \in \Cem
		 ,
	\end{align}
	where for $\mu_{i} \in \MOne{\TWR}$, $L^{t_{1},...,t_{r}}_{\nu,\zeta} \left( \mu_{1},...,\mu_{r} \right) $ is defined by
	\begin{align}
	\label{lem::LDPempP::indep::2Rep::L}
	\begin{split}
		&\sup_{f \in \CspL[\infty]{c}{\TWR} } \left\{  
					\int_{\TWR} \mkern-24mu f \left( x,\w,\theta\right) \mu_{1}
					 - 
					 \int_{\Td}    \log \left( 
															 \int_{\Wsp\times \R} \mkern-18mu
																	 U_{0,t_{1}} e^{f} \left( x, \w, \theta \right)  
																	\nu_{x} \left(\dd \theta\right) 
																	\zeta_{x} \left( \dd \w \right)
																\right) \dd x
				\right\} 
		\\
		&+
		\sum_{i=2}^{r}
			\sup_{f \in \CspL[\infty]{c}{\TWR} } \left\{  
					\int_{\TWR} \mkern-24mu   f \left( x,\w, \theta\right) \mu_{i}
					-
					\int_{\TWR}  \mkern-24mu  \log U_{t_{i-1},t_{i}} e^{f}  \left( x,\w,\theta\right) \mu_{i-1}
				\right\} 
		,
	\end{split}
	\end{align}
	where the $\mu_{i}$ integrate with respect to the variables $\dd x, \dd \w, \dd \theta$.
\end{lemma}

Finally we show that $ S^{I}_{\nu,\zeta}$, respectively 	$S^{I,\Td}_{\nu,\zeta}$, is bounded by these two rate functions.
\begin{lemma}
\label{lem::LDPempP::indep::Coincidence}
	For all $\mu_{\Ti} \in \MOne{\TWC}$
	\begin{align}
		S^{I,2}_{\nu,\zeta} \left( \mu_{\Ti} \right)
		\leq  
		S^{I}_{\nu,\zeta} \left( \mu_{\Ti} \right) 
		\leq
		S^{I,\Td}_{\nu,\zeta} \left( \mu_{\Ti} \right)
		\leq 
		S^{I,1}_{\nu,\zeta} \left( \mu_{\Ti} \right) 
		.
	\end{align}
Moreover $	S^{I,1}_{\nu,\zeta} \left( \mu_{\Ti} \right) < \infty$ implies that $\mu_{\Ti} $ is weakly differentiable.
\end{lemma}

From these three lemmas, we conclude the Theorem~\ref{thm::LDPempP::Inde} by the uniqueness of the rate function of a large deviation principle.

We prove the lemmas in the following chapters.
\end{proof}

\subsubsection{Two representation of the rate function (Proof of Lemma~\ref{lem::LDPempP::indep::1Rep} and  Lemma~\ref{lem::LDPempP::indep::2Rep})}
\label{sec::LDPempP::2RepRate}

\begin{proof}[of Lemma~\ref{lem::LDPempP::indep::1Rep}]
	We apply the Sanov type Lemma~\ref{lem::SanovT} with $r=1$, $Y=\Csp{\Ti}$ to conclude that the family
	$\left\{L^{N}, P^{I,N} \right\}$ satisfies on $\MOne{\TWC}$ a large deviation principle with rate function $L_{\nu,\zeta}$.
	The Lemma~\ref{lem::SanovT} requires Assumption~\ref{ass::LMF::Init}, Assumption~\ref{ass::LMF::Medium} and the following Feller continuity:
	\begin{lemma}
	\label{lem::LDPempP::indep::Feller}
			The Assumption~\ref{ass::LDPempP::Inde} implies that
			the family $\left\{ P^{I}_{t,x,\w,\theta} : \left( t,x,\w,\theta \right)  \in \TTWR \right\} $ is Feller continuous.
	\end{lemma}
	Before we prove this lemma, we finish the proof of Lemma~\ref{lem::LDPempP::indep::1Rep}.
	
	The map $\Pi$ (defined in Definition~\ref{def::Pi}) is  continuous (Lemma~\ref{lem::MapPi::Cont}). It maps each probability measure on $\TWC$ to a continuous measure valued trajectories in $\Cem$.
	Moreover for each fixed vector $\ul{\theta}^{N}_{\Ti}$ and each $\ul{\w}^{N}$, the image of the corresponding empirical path measure $L^{N}$ under $\Pi$ is  the corresponding empirical process $\empP$.
	Therefore, the contraction principle implies the large deviation principle for $\left\{\empP, P^{I,N} \right\}$ with the rate function $S^{I,1}_{\nu,\zeta}$.

	\emptyline
	The right hand side of \eqref{eq::lem::LDPempP::indep::1Rep::S1Inf} is only finite if there is a $Q \in \MOneL{\TWC}$ with $\Pi \left(  Q \right)_{\Ti} = \mu_{\Ti}$.
	This implies that $\mu_{t} \in \MOneL{\TWR}$ for all $t \in \Ti$, and 
	that $\mu_{t,x,\Wsp} = \mu_{0,x,\Wsp}$ for all $t \in \Ti$ and almost all $x \in \Td$,
	 by Lemma~\ref{lem::MapPi::YoungMedium}.
\end{proof}

\begin{proof}[of Lemma~\ref{lem::LDPempP::indep::Feller}]
Fix an  arbitrary convergent sequence $ \left( x^{(n)}, \w^{(n)}, \theta^{(n)} \right)  \rightarrow  \left( x, \w, \theta \right)  \in \TWR$.
We define $a_{n} \defeq a \equiv \sigma$ and  
$b^{I,(n)} \left( t, \eta \right) \defeq b^{I} \left( t,  \w^{(n)}, x^{(n)},\eta \right) $, 
$b^{I} \left( t, \eta \right) =b^{I} \left( t,x,\w,\eta \right) $ for $\left( t, \eta \right) \in \TiR$.
These functions are continuous by  Assumption~\ref{ass::LDPempP::Inde}~\ref{ass::LDPempP::Inde::bCont}.
Moreover we know, by Assumption~\ref{ass::LDPempP::Inde}~\ref{ass::LDPempP::Inde::MPro}, that $P^{I}_{x^{(n)}, \w^{(n)},\theta^{(n)}}$ is the solution to the Martingale problem corresponding to the drift coefficient $b^{I,(n)}$.

The Theorem~11.1.4 in \cite{StrVarMultidi}  implies that the solutions to the Martingale problem $P^{I}_{x^{(n)}, \w^{(n)},\theta^{(n)}}$ converge weakly to $P^{I}_{x,\w,\theta}$.
The conditions of Theorem~11.1.4 of \cite{StrVarMultidi}  are satisfied by Assumption~\ref{ass::LDPempP::Inde}.
Therefore $P^{I}_{x,\w,\theta}$ is Feller continuous.
\end{proof}

\begin{proof}[of Lemma~\ref{lem::LDPempP::indep::2Rep}]
This proof is a generalisation of the proof of \cite{DawGarLDP}~Lemma~4.6 and we use the ideas of this proof.
At first we prove a LDP for the finite dimensional distributions of $\left\{\empP\right\}$ (i.e. the distribution of $\empP$ at a finite number of times)
and in a second step we transfer this LDP to the LDP for $\left\{\empP\right\}$ by using the projective limit approach.

\emptyline
\begin{steps}
\step[{LDP for the finite dimensional distributions of $P^{I,N}$}]

Fix $N\geq 1$, $r \in \N$, $0=t_{0} \leq t_{1} < .... < t_{r} \leq T$.
We define the random elements
\begin{align}
	\empPt{t_{1} ,..., t_{r}}
	\defeq
		\left( 
			\empPt{t_{1}}
			, ... ,
			\empPt{t_{r}}
		\right)
	\in
			\left( \MOne{\TWR}\right)^{r}
	.
\end{align}
Then $\empPt{t_{1} ,..., t_{r}}$ depends only on the spin values at the times $t_{1},...,t_{r}$, i.e. on $\ul{\theta}^{N}_{t_{1}} , ... , \ul{\theta}^{N}_{t_{r}}$ and not any more on the whole path.

By Lemma~\ref{lem::SanovT} (with $Y_{1} =.... = Y_{r} = \R$), the family  $\left\{\empPt{t_{1} ,..., t_{r}} , P^{I,N}  \right\}$ satisfies
a large deviation principle on $\left( \MOne{ \TWR }  \right) ^{r}$ 
with rate function
\begin{align}
	L^{t_{1},...,t_{r}}_{\mu_{0}} \left( \mu_{1},...,\mu_{r} \right) 
	&=
	\sup_{f_{1},...,f_{r} \in \CspL{b}{\TWR}}
		\left[ 
			\sum_{\ell=1}^{r} 
						\int_{\TWR}  \mkern-18mu f_{\ell} \left( x, \w, \theta \right) \mu_{\ell} \left( \dd x, \dd \w, \dd \theta \right)
			-
			H \left( f_{1},...,f_{r} \right)  
		\right] 
\end{align}
for  $\mu_{\ell} \in \MOne{\TWR}$, where 
\begin{align}
	H \left( f_{1},...,f_{r} \right)   
		\defeq
			\int_{\Td}   
			\log \left( 
				\int_{\Wsp}
			 	\int_{\Csp{\Ti}} 
					e^{\sum_{\ell=1}^{r} f_{\ell} \left( x, \w, \theta_{t_{\ell}} \right) }
				P_{x,\w}^{I}  \left( \dd \theta_{\Ti} \right)  
				\zeta_{x} \left( \dd \w \right)
				\right)
			\dd x
	.
\end{align}
To show that this function coincides with \eqref{lem::LDPempP::indep::2Rep::L}, we first get by the Markov property of 
$\left\{  P_{t,x,\w,\theta} \right\} $ that
\begin{align}
\label{eq::pf::lem::LDPempP::indep::2Rep::HfirstTrans}
\begin{split}
	&H \left( f_{1},...,f_{r} \right)  
	\\
	&=
		\int_{\Td}  \log
		\left( 
				\int_{\Wsp}
			 	\int_{\Csp{\Ti}}
					\int_{\Csp{\Ti}}  \mkern-18mu
						e^{f_{r} \left( y,\w,\theta_{t_{r}}  \right) } 
					P^{I}_{t_{r-1},x,\w,\theta_{t_{r-1}}} \left( \dd \theta_{\Ti}  \right) 
		\right.
		\\
		&\mkern200mu		
					e^{\sum_{\ell=1}^{r-1} f_{\ell} \left( y, \w, \theta_{t_{\ell}} \right) }
				P_{x,\w}^{I}  \left( \dd \theta_{\Ti} \right)  
				\zeta_{x} \left( \dd \w \right)
		\Bigg)
		\dd x
	\\
	&= 
		\int_{\Td}  \log
				\left(
					\int_{\Wsp}
					\int_{\R} 
						U_{t_{0},t_{1}} \left( e^{f_{1}} ... U_{t_{r-1},t_{r}}e^{f_{r}} \right) \left( x, \w, \theta \right)
					\nu_{x} \left( \dd \theta \right)
					\zeta_{x} \left( \dd \w \right)
				\right)
		\dd x
	.
\end{split}
\end{align}

Now performing formally (by pushing through the space dependency) the same calculation as Dawson and G\"artner in \cite{DawGarLDP} page 275, we can transfer the right hand side of \eqref{eq::pf::lem::LDPempP::indep::2Rep::HfirstTrans} to the right hand side of \eqref{lem::LDPempP::indep::2Rep::L} with the supremum taken over all $f \in \CspL{b}{\TWR}$. But the operators  $U_{s,t}$ are continuous linear operators, hence the supremum over $\CspL[\infty]{c}{\TWR} $ equals the supremum over $\CspL{b}{\TWR}$.

\emptyline
\step[{Transfer of the LDP for $\left\{\empPt{t_{1} ,..., t_{r}}\right\}$ to the LDP for $\left\{\empP\right\}$}]

An LDP for $\left\{\empP\right\}$ follows from the LDP for the finite dimensional marginals of the first step, by the projective limit approach.
In \cite{DawGarLDP} on page 276 this is done for the mean field model.
This proof can be almost directly used in the setting we consider here. 
To have a complete picture, we state nevertheless the idea here.

To have a projective system corresponding to $\left( \MOne{\TWR} \right)^{r}$ with order relation $\subseteq$ for $\{t_{1},...,t_{r}\}$, we embed the space $\Cem$ into $\MOneUp{\Ti}{ \TWR} \defeq \{f: \Ti  \rightarrow \MOne{\TWR} \}$ furnished with the product topology. 

We know by Lemma~\ref{lem::LDPempP::indep::1Rep} already that $\left\{\empP, P^{I,N} \right\}$ satisfies a large deviation system on $\Cem$.
By the contraction principle,  $\left\{\empP, P^{I,N}\right\}$ also satisfies a large deviation system on $\MOneUp{\Ti}{ \TWR}$.
We denote its rate function by $\what{S}^2$. 
But this LDP can also be identified with the projective limit of the finite dimensional LDPs derived above.
Hence by the projective limit theorem  (\cite{DemZeiLarge} Theorem~4.6.1, \cite{DawGarLDP} Theorem~3.3)  we see that $\what{S}^2$ has the desired form \eqref{lem::LDPempP::indep::2Rep::S2} on $\MOneUp{\Ti}{ \TWR}$. 

Moreover  $\what{S}^2$ is infinite on $\MOneUp{\Ti}{ \TWR} \backslash \Cem$ and the random variables $\empP$ under $P^{I,N}$ are concentrated on $\Cem$.
Hence we can reduce the LDP to an LDP on $\Cem$ by Lemma~4.1.5~(b)  in \cite{DemZeiLarge}. 
This finishes the proof of Lemma~\ref{lem::LDPempP::indep::2Rep}.

\end{steps}\vspace{-\baselineskip}
\end{proof}

\subsubsection{Coincidence of the two representations with \texorpdfstring{$S^{I}_{\nu,\zeta}$}{SInu} (proof of Lemma~\ref{lem::LDPempP::indep::Coincidence})}
\label{sec::LDPempP::2RepBoundSI}

In this chapter we prove Lemma~\ref{lem::LDPempP::indep::Coincidence}. Therefore we show at first an upper bound on $S^{I,\Td}_{\nu,\zeta}$ and then a lower bound on $S^{I}_{\nu,\zeta}$.

For the upper bound (Chapter~\ref{sec::LDPempP::2RepBoundSI::Upper}) we generalise an approach used in \cite{PraHolMKV},  which is partially based on approaches of \cite{FolRandom} and \cite{BruFiniteKullb}.
In contrast to \cite{PraHolMKV}, we consider the space dependency $x \in \Td$ in addition to the random environment $\w \in \Wsp$.
Moreover we look at an independent system, whereas in \cite{PraHolMKV}  an interacting system is considered (see also Lemma~\ref{lem::LdpLNtoEmp::SuppBound}, where we use this approach also for an interacting systems).

The proof that we give for the lower bound (see Chapter~\ref{sec::LDPempP::2RepBoundSI::Lower}) is a generalisation of Chapter~4.4 in \cite{DawGarLDP} to the model we consider here.
We require in the proof the existence and uniqueness of a solution to a PDE.
In contrast to \cite{DawGarLDP}, this solution has to be continuous in the space variable $x \in \Td$ and the environment variable $\w \in \Wsp$.
We show the existence and this regularity of a solution in the completely new Chapter~\ref{sec::LDPempP::PDE}.
The rest of the proof in Chapter~\ref{sec::LDPempP::2RepBoundSI::Lower} generalises the proof in \cite{DawGarLDP}.
Moreover, we correct minor mistakes of the proof in \cite{DawGarLDP}.

\paragraph{Upper bound on \texorpdfstring{$S^{I,\Td}_{\nu,\zeta}$}{SITdnu}}
\label{sec::LDPempP::2RepBoundSI::Upper}

We show in this chapter that $S^{I,\Td}_{\nu,\zeta} \leq S^{I,1}_{\nu,\zeta}$.
As mentioned, the proof we state here, is based on an approach in \cite{PraHolMKV}.

\begin{lemma}
	\label{lem::LDPempP::indep::UpperB}
	If $S^{I,1}_{\nu,\zeta} \left( \mu_{\Ti} \right) < \infty$ for a $\mu_{\Ti} \in \Cem$,
	then
	\begin{align}
	S^{I,1}_{\nu,\zeta} \left( \mu_{\Ti} \right)
	=
	S^{I,\Td}_{\nu,\zeta} \left( \mu_{\Ti} \right) 
	,
	\end{align}
	and $t \mapsto \mu_{t,x,\w}$ is weakly differentiable for almost all $\left( x, \w \right) \in \TW$.
	
	In particular
	$S^{I,1}_{\nu,\zeta} \geq S^{I,\Td}_{\nu,\zeta} \geq S^{I}_{\nu,\zeta}$.
\end{lemma}

\begin{remark}
	\label{rem::LDPempP::indep::admissMu}
	Note that the lemma only states the equality of $S^{I,1}_{\nu,\zeta} $ and $S^{I,\Td}_{\nu,\zeta}$, when $S^{I,1}_{\nu,\zeta}  \left( \mu_{\Ti}  \right) < \infty$, i.e. when there is a $Q \in \MOne{\TWC}$, such that $L_{\nu,\zeta}^{1}  \left( Q \right) < \infty$ and $\Pi \left( Q \right)_{\Ti} = \mu_{\Ti}$.
	In \cite{FolRandom}, $\mu_{\Ti}$ that satisfy this condition are called admissible.
	
	Therefore this result is not enough to show the claimed equality in Theorem~\ref{thm::LDPempP::Inde} and we are bound to prove also a lower bound (in Chapter~\ref{sec::LDPempP::2RepBoundSI::Lower}).
\end{remark}

 \begin{proof}[of Lemma~\ref{lem::LDPempP::indep::UpperB}]
 Fix a $\mu_{\Ti}  \in \Cem$ with $S^{I,1}_{\nu,\zeta} \left( \mu_{\Ti}  \right) < \infty$.
 
 The idea of this proof is based on the steps~1-3 of the proof of Theorem~3  in \cite{PraHolMKV}, that are partly based on \cite{FolRandom} and \cite{BruFiniteKullb}. The proof is organised as follows.
We show in \ref{pf::lem::LDPempP::indep::UpperB::step::Minimiz}, that there is a $\ol Q \in \MOne{\TWC}$, which is a minimizer of the right hand side of \eqref{eq::lem::LDPempP::indep::1Rep::S1Inf} for $S^{I,1}_{\nu,\zeta} \left( \mu_{\Ti}  \right)$.
In \ref{pf::lem::LDPempP::indep::UpperB::step::OtherRepLI} we derive another representation of $S^{I,1}_{\mu} \left( \mu_{\Ti} \right)$, by applying a result of \cite{FolRandom}.
Finally (in \ref{pf::lem::LDPempP::indep::UpperB::step::NewRepEqual}) we show, that the new representation of $S^{I,1}_{\mu} \left( \mu_{\Ti} \right)$ equals $S^{I,\Td}_{\nu,\zeta}$. We use that $\mu_{t,x,\w}$ is the evolution of the time marginal of $\ol Q_{x,\w}$ and a weak solution of a Fokker-Planck equation.

\emptyline
\begin{steps}
\step[{There is a $\ol Q$ with  $L^{1}_{\nu,\zeta}  \left( Q \right) = S^{I,1}_{\mu} \left( \mu_{\Ti} \right)$ and nice properties}]
\label{pf::lem::LDPempP::indep::UpperB::step::Minimiz}
\nopagebreak[2]

We restrict the infimum in \eqref{eq::lem::LDPempP::indep::1Rep::S1Inf} to the set
\begin{align}
		A_{\mu,C}
		\defeq
			\Big\{  Q: \Pi \left( Q \right)_{\Ti} =\mu_{\Ti} \Big\}  
			\cap
			\left\{  Q: L^{1}_{\nu,\zeta}  \left( Q \right) \leq C \right\} 
		\subset 
			\MOne{\TWC}
			,
\end{align} 
for a $C>0$ large enough.
This set is non empty and compact  (the last set is compact because $L^{1}_{\nu,\zeta}$ is a good rate function and the first set is closed).
 Hence by the lower semi continuity of $L^{1}_{\nu,\zeta}$, there exists a $\ol Q \in A_{\mu,C}$ that is a minimiser of $L^{1}_{\nu,\zeta} $ in $A_{\mu,C}$.
This implies that
 $L^{1}_{\nu,\zeta}  \left( \ol Q \right)
			=
			S^{I,1}_{\mu} \left( \mu_{\Ti} \right)
$.

\emptyline
Therefore  $L^{1}_{\nu,\zeta}  \left( \ol Q \right) = \relE{\ol Q}{\dd x \otimes \zeta_{x} \left( \dd \w \right) \otimes P^{I}_{x,\w} } < \infty$ and
  $\ol Q \in \MOneUp{L}{\TWC}$.
 Let us write $\ol Q = \dd x \otimes \ol Q_{x} $ for $\ol Q_{x} \in \MOne{\WC}$ and
 $Q_{x} =  \ol Q_{x,\Wsp} \otimes \ol Q_{x,\w}$ for $\ol Q_{x,\w} \in \MOne{\Csp{\Ti}}$, $\ol Q_{x,\w} \in \MOne{\Wsp}$.
 Then for almost all $x \in \Td$ and $\ol Q_{x,\Wsp}$-almost all $\w \in \Wsp$,  $\relE{ \ol Q_{x,\w} }{ P^{I}_{x,\w}  }< \infty$,  $\relE{ \ol Q_{x,\Wsp}}{\zeta_{x}} < \infty$ and  $\Pi \left( \ol Q_{x,\w} \right)_{t} = \mu_{t,x,\w}$.
Moreover  $\Pi \left( \ol Q \right)_{t} = \dd x \otimes \ol Q_{x,\Wsp} \otimes \Pi \left( \ol Q_{x,\w} \right)_{t} = \mu_{t} \in \MOneUp{L}{\TWR}$.

 \emptyline
 \step[{Another representation of  $S^{I,1}_{\mu} \left( \mu_{\Ti} \right)$}]
 \label{pf::lem::LDPempP::indep::UpperB::step::OtherRepLI}
 
 By these properties, we get, for almost all $x \in \Td$, as in \cite{FolRandom} Theorem~II.1.31 and Remark~II.1.3  (see also \cite{LipShiStatis} Chapter~7 (in particular Theorem~7.11)), that there is a map $b^{x,\w} : \TiR \rightarrow \R$ such that $\ol Q_{x,\w}$ is the law of $\theta^{x,\w}_{\Ti}$ described by the following SDE
 \begin{align}
 \label{pf::lem::LDPempP::indep::UpperB::eq::SDEforQx}
 	\dd \theta^{x,\w}_{t}
 	=
 		\left( \sigma b^{x,\w} \left( t, \theta^{x,\w}_{t}\right) - b^{I} \left( t,x, \w, \theta^{x,\w}_{t}\right) \right) \dd t 
 		+ \sigma \dd B^{\ol Q_{x,\w}}_{t}
 	,
  \end{align}
  with $\theta^{x,\w}_{0} \sim \mu_{0,x,\w}$
 and
  \begin{align}
   \label{pf::lem::LDPempP::indep::UpperB::eq::GirQxPIx}
 	\frac{\dd \ol Q_{x,\w}}{\dd P^{I}_{x,\w} }
 	=
 	e^{\int_{0}^{T} b^{x,\w} \left( t , . \right)  \dd B^{\ol Q_{x,\w}}_{t} + \frac{1}{2} \int_{0}^{T} b^{x,\w} \left( t, . \right) ^{2} \dd t}
 	\frac{\dd \mu_{0,x,\w}}{\dd \nu_{x}}
	.
  \end{align}
 Here $B^{\ol Q_{x,\w}}_{t}$ is a Wiener process under $\ol Q_{x,\w}$.
Inserting this derivative in the relative entropy, we get
\begin{align}
\label{pf::lem::LDPempP::indep::UpperB::eq::relEexplForm}
\begin{split}
	\relE{ \ol Q_{x,\w} }{ P^{I}_{x,\w} } 
	-
	\relE {\mu_{0,x,\w}} { \nu_{x}}
	&=
	\frac{1}{2} \int_{\Csp{\Ti}} \int_{0}^{T}   \left( b^{x,\w} \left( t, \theta_{t}\right)  \right) ^{2} \dd t \; Q_{x,\w} \left( \dd \theta_{\Ti}  \right) 
	\\
	&=
	\frac{1}{2} \int_{0}^{T} \int_{\R}  \left( b^{x,\w} \left( t, \theta_{t}\right) \right) ^{2} \mu_{t,x,\w }\left( \dd \theta \right)  \: \dd t 
	.
\end{split}
 \end{align}
Integrating over $\mu_{0,x,\Wsp} = \ol Q_{x,\Wsp} \in \MOne{\Wsp}$ and then over $x \in \Td$ implies that
\begin{align}
\label{pf::lem::LDPempP::indep::UpperB::eq::S1IexplForm}
\begin{split}
S^{I,1}_{\nu,\zeta} \left( \mu_{\Ti}  \right)
=
L^{1}_{\nu,\zeta}  \left( \ol Q \right)
=
	&\frac{1}{2} \int_{\Td} \int_{\Wsp} \int_{0}^{T} \int_{\R}  
						\left( 
								b^{x,\w} \left( t,\theta \right) 
						\right) ^{2} 
						\mu_{t,x,\w}\left( \dd \theta \right)  \dd t 
						\: \mu_{0,x,\Wsp}  \left( \dd \w \right)
						 \dd x
	\\
	&+
	\int_{\Td} \int_{\Wsp} \relE {\mu_{0,x,\w}} { \nu_{x}} \mu_{0,x,\Wsp}  \left( \dd \w \right) \dd x
	+
	\int_{\Td} \relE{	\mu_{0,x,\Wsp}   }{ \zeta_{x} } \dd x
	.
\end{split}
 \end{align}

\emptyline
\step[{The new representation of $S^{I,1}_{\nu,\zeta} \left( \mu_{\Ti} \right)$ equals $S^{I,\Td}_{\nu,\zeta}$}]
\label{pf::lem::LDPempP::indep::UpperB::step::NewRepEqual}

We show now that for almost all $t \in \Ti$, almost all $x \in \Td$ and $Q_{x,\Wsp}$-almost all $\w \in \Wsp$
\begin{align}
\label{pf::lem::LDPempP::indep::UpperB::eq::ToShow2}
	\frac{1}{2} \int_{\R}  
							\left( 
								b^{x,\w} \left( t,\theta \right) 
							\right) ^{2} 
							\mu_{t,x,\w}\left( \dd \theta \right) 
	 =
		\abs{
			\partial_{t}  \mu_{t,x,\w} - \left(\Gen^{I}_{t,x,\w}\right)^{*}   \mu_{t,x,\w}
		}_{\mu_{t,x,\w}}
	,
\end{align}
with $\Gen^{I}_{t,x,\w} $ defined in \eqref{def::GeneratorIndependent}.

The equation \eqref{pf::lem::LDPempP::indep::UpperB::eq::ToShow2} can be shown as in the Steps~2 and~3 in the proof of Theorem~3 in \cite{PraHolMKV}.
Therefore we sketch the proof here only.
 
 The measure $\ol Q_{x,\w}$ is the law of \eqref{pf::lem::LDPempP::indep::UpperB::eq::SDEforQx} and by construction $\mu_{t,x,\w}$  is the evolution of the time marginal of this law. Hence $\mu_{t,x,\w}$ is a weak solution of the Fokker-Plank equation
 \begin{align}
 \label{pf::lem::LDPempP::indep::UpperB::eq::FPeq}
 	\partial_{t} \mu_{t,x,\w}
 	=
 	 - \partial_{\theta}  \left( \big[ \sigma b^{x,\w} \left( t, .\right) - b^{I} \left( t, x, \w,. \right) \big] \mu_{t,x,\w} \right)  + \frac{\sigma^{2}}{2} \partial^{2}_{\theta^{2}} \mu_{t,x,\w}
	.
 \end{align}
From this, we subtract now the generator $\left(\Gen^{I}_{t,x,\w}\right)^{*} $ 
 \begin{align}
 	\partial_{t} \mu_{t,x,\w} - \left(\Gen^{I}_{t,x,\w}\right)^{*} \mu_{t,x,\w} 
 	=
 	- \partial_{\theta}  \left(   \sigma b^{x,\w} \left( t, .\right)   \mu_{t,x,\w}  \right) 
 	,
 \end{align}
 what leads to 
 \begin{align}
 \label{pf::lem::LDPempP::indep::UpperB::eq::UpperBound}
 \begin{split}
	\abs{
		\partial_{t}  \mu_{t,x,\w} - \left(\Gen^{I}_{t,x,\w}\right)^{*} \mu_{t,x,\w}
	}_{\mu_{t,x,\w}}
	&=
		\frac{1}{2}
		\sup_{f \in  \D_{\mu_{t,x,\w}}}
			\frac{\abs{ \int_{\R} \sigma b^{x,\w} \left( t, \theta \right)  \partial_{\theta} f \left(\theta \right) \mu_{t,x,\w} \left( \dd \theta \right)}^{2} }
			{	\sigma^{2} \int_{\R}  \left(\partial_{\theta} f \left(\theta \right)\right) ^{2} \mu_{t,x,\w} \left( \dd \theta \right)}
	\\
	&\leq
		\frac{1}{2}
		\int_{\R}  \left( b^{x,\w} \left( t, \theta \right) \right)^{2}  \mu_{t,x,\w} \left( \dd \theta \right)
	,
\end{split}
 \end{align}
with
$\D_{\mu_{t,x,\w}} \defeq \left\{ f \in \CspL[\infty]{c}{\R} :  \int_{\R}  \left(\partial_{\theta} f \left(\theta \right)\right) ^{2} \mu_{t,x,\w} \left( \dd \theta \right) >0 \right\}$.

 To conclude \eqref{pf::lem::LDPempP::indep::UpperB::eq::ToShow2}, we have to show that the last inequality is actually an equality.
 This can be done as in Step~3 of the proof of Theorem~3 in \cite{PraHolMKV}, by showing that $\left\{  \partial_{\theta} f : f \in \D_{\mu_{t,x,\w}}\right\} $ is dense in $\Ltwo  \left( \R,\mu_{t,x,\w} \right) $. Then we take a approximating sequence $f_{n} \in \D_{\mu_{t,x,\w}}$, 
 $\partial_{\theta} f_{n} \rightarrow b^{x,\w}_{t} $ 
 and get the corresponding lower bound.

 \end{steps}\vspace{-\baselineskip}
 \end{proof}

\begin{remark}
\label{rem::LDPempP::indep::UpperB::Different}
	Instead of Lemma~\ref{lem::LDPempP::indep::UpperB}, we could also show similarly as in Lemma~4.9 in \cite{DawGarLDP}, that $S^{I,1}_{\nu,\zeta} \geq S^{I}_{\nu,\zeta}$, by using a representation of $S^{I}_{\nu,\zeta}$, that we derive in Lemma~\ref{lem::LDPempP::indep::3Rep}.
	This would require some changes  (compared to  \cite{DawGarLDP}), due to the space dependency and the initial distribution of the spin values that we consider here.
	However, the advantage of Lemma~\ref{lem::LDPempP::indep::UpperB} is that it bounds also $ S^{I,\Td}_{\nu,\zeta} $.
	This could be archived also by a variation of Lemma~4.9 in \cite{DawGarLDP} and a variation of Lemma~\ref{lem::LDPempP::indep::3Rep}, i.e. by moving the integral with respect to $x \in \Td$ out of the supremum in \eqref{eq::lem::LDPempP::indep::3Rep::SI}.
	However using this approach, one has to be careful whether functions are integrable with respect to $x \in \Td$ and $\w \in \Wsp$.
	
\end{remark}

\paragraph{Lower bound on \texorpdfstring{$S^{I}_{\nu,\zeta}$}{SInu}}
\label{sec::LDPempP::2RepBoundSI::Lower}

We prove in this chapter the following lower bound on $S^{I}_{\nu,\zeta}$.
The proof is a generalisation of the corresponding proof in \cite{DawGarLDP}. 
The most important difference to the original proof is that we derive for solutions of the arising PDE (see the proof of Lemma~\ref{lem::LDPempP::indep::S2leqSKilled}) also regularity in the space variable and the random environment variable.

\begin{lemma}[compare to Lemma~4.10 in \cite{DawGarLDP} for the mean field case]
\label{lem::LDPempP::indep::S2leqS}
	$S^{I,2}_{\nu,\zeta} \leq S^{I}_{\nu,\zeta}$.
\end{lemma}

\begin{proof}
It suffices to show, by \eqref{lem::LDPempP::indep::2Rep::L}, \eqref{eq::thm::LDPempP::Inde::RF} and the second formula of the norm in Definition~\ref{def::Minus1Norm}, that
\begin{align} 
\label{lem::LDPempP::indep::S2leqS::eq::ToShow}
\begin{split}
	&\int_{\TWR} f \left( x, \w, \theta \right) \mu_{t}  \left( \dd x, \dd \w, \dd \theta \right)
		-
		\int_{\TWR} \log U_{s,t} e^{f} \left( x, \w, \theta \right) \mu_{s}  \left( \dd x, \dd \w, \dd \theta \right)
	\\
	&\leq 
		\int_{s}^{t}
		\sup_{h \in \CspL[\infty]{c}{\TWR}} 
			\left(
				\sPro{\partial_{u}\mu_{u}}{h}
				-
				\int_{\TWR} \mkern-27mu
								\Gen^{I}_{u,x} h \left( x, \w, \theta \right) + \frac{\sigma^{2}}{2} \left(\partial_{\theta} h \left( x, \w, \theta \right) \right) ^{2} 
					\mu_{u}  \left( \dd x, \dd \w, \dd \theta \right)
			\right) 
		\dd u
\end{split}
\end{align}
for all $f \in \CspL[\infty]{c}{\TWR} $, $0 \leq s < t \leq T$, $\nu \in \MOneL{ \TWR } $ and $\mu_{\Ti}  \in \Cem$ with $S^{I}_{\nu,\zeta} \left( \mu_{\Ti}  \right)  < \infty$.
Indeed  by \eqref{lem::LDPempP::indep::S2leqS::eq::ToShow}, we bound separately each summand of the sum on the right hand side of \eqref{lem::LDPempP::indep::2Rep::L}.
For the first summand on the right hand side of \eqref{lem::LDPempP::indep::2Rep::L}, we have to differentiate between the cases $t_{1} =0$ and $t_{1}>0$ in the supremum in \eqref{lem::LDPempP::indep::2Rep::S2}. 
If $t_{1}>0$, then apply first the Jensen inequality to the this summand of the right hand side of \eqref{lem::LDPempP::indep::2Rep::L} before using \eqref{lem::LDPempP::indep::S2leqS::eq::ToShow}.
In the case $t_{1}=0$, the first summand on the right hand side of \eqref{lem::LDPempP::indep::2Rep::L} equals to $\relE{\mu_{0}}{\dd x \otimes \zeta_{x} \otimes \nu_{x}}$, which appears in formula \eqref{eq::thm::LDPempP::Inde::RF} of $S^{I}_{\nu,\zeta}$ (by a similar estimate as used in the proof of Lemma~\ref{lem::SanovT::RelEntropy}).

\emptyline

An easy heuristic proof for the mean-field counterpart to \eqref{lem::LDPempP::indep::S2leqS::eq::ToShow} is given in \cite{DawGarLDP} on page~282.
We refer to this heuristic to get an idea of the following proof. However in particular due to the  unbounded domain of the spin values, problems arise such that the heuristic does not make sense.
However we can prove \eqref{lem::LDPempP::indep::S2leqS::eq::ToShow} when restricting the analysis to compact sets (see Lemma~\ref{lem::LDPempP::indep::S2leqSKilled}) and infer from this \eqref{lem::LDPempP::indep::S2leqS::eq::ToShow}.
Therefore we define a new semi group corresponding to the diffusion processes which is killed when leaving the ball $B_{R}= \left\{   \left( x, \w, \theta \right)  \in \TWR : \abs{\theta} < R \right\} $ by
\begin{align}
	U^{R}_{s,t} f \left( x, \w, \theta \right) 
	=
	\int_{\Csp{\left[s,T \right]}} f \left( x, \w, \theta_{t}  \right)  \1_{\tau_{R}^{s}  >t } P_{s,x, \w, \theta} \left( \dd \theta_{\left[s,T \right]}  \right) 
\end{align}
for $f \in \CspL{b}{\TWC}$,
with $\tau_{R}^{s}  \left( \theta_{\left[s,T \right]}  \right) = \min \left\{    t \in [s,T] : \abs{\theta_{t}} \geq R   \right\} $.

\begin{lemma}[compare to Lemma~4.11 in \cite{DawGarLDP}]
\label{lem::LDPempP::indep::S2leqSKilled}
	Given a $\mu_{\Ti}  \in \Cem$ with $S^{I}_{\nu,\zeta} \left( \mu_{\Ti}  \right)  < \infty$,
	then for all $R>0$, $0 \leq s < t \leq T$ and $f \in \CspL[\infty]{c}{\TWR} $ with $f \leq 0$ and $\textnormal{supp} \left( f \right)  \subset B_{R}$.
	\begin{align}
	\label{lem::LDPempP::indep::S2leqSKilled::eqToShow}
	\begin{split}
			&
			\int f \left( x, \w, \theta \right) \mu_{t}  \left( \dd x, \dd \w, \dd \theta \right)
					-
					\int \log \left[ 1+ U^{R}_{s,t}  \left( e^{f}-1 \right)  \right] \left( x, \w, \theta \right) \mu_{s}  \left( \dd x, \dd \w, \dd \theta \right)
	\\
	&\leq
	  \int_{s}^{t}
		\sup_{h \in \CspL[\infty]{c}{\TWR} } 
			\left(
				\sPro{\partial_{u}\mu_{u}}{h}
				-
						\int
										\Gen^{I}_{u,x,\w} h \left( x, \w, \theta \right) + \frac{\sigma^{2}}{2} \left(\partial_{\theta} h \left( x, \w, \theta \right) \right) ^{2} 
							\mu_{u}  \left( \dd x, \dd \w, \dd \theta \right)
			\right) 
		\dd u
		,
	\end{split}
	\end{align}
where the integrals without bounds integrate over the space $\TWR$.

\end{lemma}

This lemma implies \eqref{lem::LDPempP::indep::S2leqS::eq::ToShow} by the same approximation approach given after Lemma~4.11 in \cite{DawGarLDP}.
Hence once we prove Lemma~\ref{lem::LDPempP::indep::S2leqSKilled}, the proof of Lemma~\ref{lem::LDPempP::indep::S2leqS} is finished.
\end{proof}

\begin{proof}[of Lemma~\ref{lem::LDPempP::indep::S2leqSKilled}]
In this proof we generalise the proof of Lemma~4.11 in~\cite{DawGarLDP} to the model considered here.
In contrast to \cite{DawGarLDP} we do not assume in Theorem~\ref{thm::LDPempP::Inde}, that the drift coefficient $b$ is locally \Holder continuous. However we need this assumption to get the existence of a solution to a PDE (see \ref{pf::lem::LDPempP::indep::S2leqSKilled::step::Hoelder::Solution}).
Therefore we assume at first (\ref{pf::lem::LDPempP::indep::S2leqSKilled::step::Hoelder}), that $b^{I}$ is \Holder continuous in time and spin value.
Finally, in \ref{pf::lem::LDPempP::indep::S2leqSKilled::step::Generalb}, we show how to generalise this to general drift coefficients.

Fix an $R>0$, an arbitrary $f \in \CspL[\infty]{c}{\TWR} $ with $f \leq 0$ and $\textnormal{supp} \left( f \right)  \subset B_{R}$ and arbitrary $0 \leq s < t \leq T$.
Let $\Wsp_{f} \subset \Wsp$ be a compact subset such that the projection on $\Wsp$ of the support of $f$  is contained in $\Wsp_{f}$

\emptyline
\begin{steps}
\step[{The drift coefficient is \Holder continuous}]
\label{pf::lem::LDPempP::indep::S2leqSKilled::step::Hoelder}

Let us assume that $b^{I}$ is  $\frac{1}{4}$-\Holder continuous in time and $\oh$-\Holder continuous in $\theta \in B_{R}$ on the subset $\TT \times \Wsp_{f} \times B_{R}$. Moreover let $b^{I}$ be continuous on $\TTWR$.
To generalise the ideas of \cite{DawGarLDP} to the space and random environment dependent model, we need in particular the existence of a unique solution to an initial boundary value problem. This solution has to be moreover continuous in the space variable $x \in \Td$ and in the random environment variable.
We prove the existence and uniqueness of such a solution in Theorem~\ref{thm::LDPempP::UniqClasPDE}. 
We follow the lines of the proof in \cite{DawGarLDP} with focus on the extensions needed to treat the space and random environment dependency.

\emptyline
\begin{steps}
\step[{Construction of a (non smooth) function that solves a PDE}]
\label{pf::lem::LDPempP::indep::S2leqSKilled::step::Hoelder::Solution}

By Theorem~\ref{thm::LDPempP::UniqClasPDE}, there is a unique classical solution  $g^{*}$ to the terminal boundary value problem
\begin{equation}
\label{eq::PDETerminal}
\begin{aligned}
	&\partial_{s} g \left( s,x, \w, \theta \right)  
				\!\!&&=&& 
				- \Gen^{I}_{s,x,\w} g  \left( s,x, \w, \theta \right)  
				\quad &&\left( s,x, \w, \theta \right)  \in \left[ 0,t \right) \times \Td \times \Wsp_{f} \times B_{R}
	\\
		&g \left( t,x, \w, \theta \right)  &&=&& e^{f \left( x, \w, \theta \right) }-1   
				&&\left( x, \w, \theta \right)  \in  \Td \times \Wsp_{f}  \times B_{R}
	\\
		&g \left( s,x, \w, \theta \right)  &&=&& 0     &&\left( s,x, \w, \theta \right)  \in \left[ 0,t \right)  \times \Td \times \Wsp_{f}  \times \partial B_{R}
	.
\end{aligned}
\end{equation}

This implies that $g^{*} \left( s,x, \w, \theta \right)  =0$ for $\left( s,x, \w, \theta \right)  \in \left[ 0,t \right] \times \Td \times \partial\Wsp_{f} \times B_{R}$. We define $g^{*}$ to be zero for $\w \not \in \Wsp_{f}$ or $\theta \not \in B_{R}$.

The $g^{*}$ satisfies for $\left( s,x, \w, \theta \right)  \in \TTWR$
\begin{align}
\label{eq::lem::LDPempP::indep::S2leqSKilled::gForms}
\begin{split}
	g^{*} \left( s,x, \w, \theta \right) 
	&=
	\int_{\Csp{\left[s,T \right]}} 
						g^{*} \left(  t \wedge \tau_{R} , x, \w, \theta_{t \wedge \tau_{R}}  \right) 
				P_{s,x, \w, \theta} \left( \dd \theta_{\left[s,T \right]}  \right) 
	\\
	&=
	\int_{\Csp{\left[s,T \right]}} 
				\left( e^{f \left( x, \w, \theta_{t}\right) } -1 \right) \1_{\tau_{R} >t} 
	 P_{s,x, \w, \theta} \left( \dd \theta_{\left[s,T \right]}  \right) 
	=
	U^{R}_{s,t}\left( e^{f }-1 \right) \left( x, \w, \theta \right)
	.
\end{split}
\end{align}
The first equality is true because $g^{*} \left( t \wedge \tau_{R} , x, \w, \theta \left( t \wedge \tau_{R} \right)  \right)$  is a $P_{s,x, \w, \theta}$ martingale for all $\left( s,x, \w, \theta \right)  \in \left[ 0,t \right] \times \TWR$ by Assumption~\ref{ass::LDPempP::Inde}~\ref{ass::LDPempP::Inde::MPro}.
The next equality is due to the boundary and the initial condition in \eqref{eq::PDETerminal}, respectively the chosen continuation of $g^{*}$.
Note that the equality of $g^{*}$ and the third representation is the corresponding Feynman-Kac formula (for fixed $\left( x,\w \right) \in \Td \times \Wsp_{f}$).
\emptyline

Define the function $h^{*} \defeq \log \left( g^{*}+1 \right)$.
This function solves 
\begin{align}
\label{eq::lem::LDPempP::indep::S2leqSKilled::PDE}
\begin{split}
	&\partial_{t} h = - \Gen^{I}_{t,x,\w} h - \frac{\sigma^{2}}{2} \left( \partial_{\theta} h \right)^{2} 
	\qquad \textnormal{ on } \TT \times \Wsp_{f} \times B_{R}
	\qquad\textnormal{ and } 
	\\
	&h \left( t,.,.,.\right) \big|_{\TW \times B_{R}}  = f \left( .,.,. \right)
	\quad
	\textnormal{ and } 
	\quad
	h \big\vert_{\partial B_{R}} =0
	.
\end{split}
\end{align}
If we could use the function $h$ on the right hand side of \eqref{lem::LDPempP::indep::S2leqSKilled::eqToShow}, 
then the integration by parts Lemma~\ref{def::Distr::IntbyParts} would prove Lemma~\ref{lem::LDPempP::indep::S2leqSKilled}.
Unfortunately $h^{*}$ is not in $\CspL[\infty]{c}{\TTWR}$.
By its construction and the compactness of $f$, the support of $g^{*}$ and thus of $h^{*}$ is compact, but $h^{*}$ is not smooth.

\emptyline
\step[{Smoothing of $g^{*}$}]

The last part of the proof consists of approaching $g^{*}$ with smooth functions $g_{\epsilon}$, defined by
\begin{align}
	g_{\epsilon}
	\defeq
	k_{\epsilon} *_{x, \w, \theta} g^{*}
\end{align}
with $k_{\epsilon} \left( x, \w, \theta \right) =k^{1}_{\epsilon} \left( x \right) k^{2}_{\epsilon} \left( \w \right)  k^{3}_{\epsilon} \left( \theta \right) $.
Here $k^{1}_{\epsilon}$ is a Dirac sequence (approximation to the identity) in $\Td$ such that $k^{1}_{\epsilon} \left( x \right) =\epsilon^{-d} k^{1} \left( \epsilon^{-1} x \right) $ and $k^{1} \in \CspL[\infty]{c}{ \Td }$, $k^{1}\geq 0$ and $\int_{\Td} k^{1} \left( x \right)  \dd x =1$.
Analogue we define $k^{2}_{\epsilon}$ and $k^{3}_{\epsilon}$ as a Dirac sequence on $\Wsp$ and  $\R$ respectively.

Then $g_{\epsilon} \in \CspL[\infty]{c}{\TWR} $ but it does not satisfy any more \eqref{eq::PDETerminal} and  
\begin{align}
	h_{\epsilon} \defeq \log \left( 1+ g_{\epsilon} \right)
\end{align}
does not satisfy any more \eqref{eq::lem::LDPempP::indep::S2leqSKilled::PDE}.
Therefore we can not use directly the  integration by parts Lemma~\ref{def::Distr::IntbyParts} to show
\eqref{lem::LDPempP::indep::S2leqSKilled::eqToShow}.

\emptyline
\step[{Smoothed function almost satisfies \eqref{lem::LDPempP::indep::S2leqSKilled::eqToShow}}]

Nevertheless we prove in the following that $h_{\epsilon}$ used on the right hand side of  \eqref{lem::LDPempP::indep::S2leqSKilled::eqToShow} (instead of the supremum) almost satisfies \eqref{lem::LDPempP::indep::S2leqSKilled::eqToShow}, with an error that vanishes as $\epsilon \rightarrow 0$.

Indeed by the integration by parts Lemma~\ref{def::Distr::IntbyParts}
\begin{align}
\label{lem::LDPempP::indep::S2leqSKilled::eqToShowForHepsilon}
\begin{split}
	\encircle{L}
	\defeq
		&\int_{\TWR} h_{\epsilon} \left( t,x, \w, \theta \right) \mu_{t}  \left( \dd x, \dd \w, \dd \theta \right)
		-
		\int_{\TWR} h_{\epsilon} \left( s,x, \w, \theta \right) \mu_{s}  \left( \dd x, \dd \w, \dd \theta \right)
	\\
	&=
	\int_{s}^{t} \sPro{\partial_{u}{\mu}_{u} }{h_{\epsilon} \left( u \right) } 
						+ 
					\int_{\TWR} \partial_{u} h_{\epsilon} \left( u,x, \w, \theta \right) \mu_{u}  \left( \dd x, \dd \w, \dd \theta \right)
		\quad
		\dd u 
	\\
	&=
	\int_{s}^{t}
		\sPro{\partial_{u}{\mu}_{u} }{h_{\epsilon} \left( u \right) } 
		- 
			\int_{\TWR} 
						\Gen^{I}_{u,x,\w} h_{\epsilon} \left( u,x, \theta \right)  + \frac{\sigma^{2}}{2} \abs{\partial_{\theta} h_{\epsilon} \left( u,x, \w, \theta \right) }^{2} 
			\mu_{u}  \left( \dd x, \dd \w, \dd \theta \right)
		\\
		&\quad +
		\int_{\TWR} 
			\frac{\left( \partial_{u}+\Gen^{I}_{u,x,\w} \right) g_{\epsilon} \left( u,x, \w, \theta \right)  }{1+g_{\epsilon} \left( u,x, \w, \theta \right) }
		\mu_{u}  \left( \dd x, \dd \w, \dd \theta \right)
	\quad \dd u
	\eqdef
		\encircle{R1} - \encircle{R2} + \encircle{R3}
	,
\end{split}
\end{align}
because $\partial_{u} h_{\epsilon}= \frac{\partial_{u}g_{\epsilon}}{1+g_{\epsilon}}$ and 
$\Gen^{I}_{u,x,\w} h_{\epsilon} =  \frac{\Gen^{I}_{u,x,\w} g_{\epsilon}}{1+g_{\epsilon}} - \frac{\sigma^{2}}{2} \abs{\partial_{\theta} h_{\epsilon}}^{2}$.

The \encircle{L} converges to the left hand side of \eqref{lem::LDPempP::indep::S2leqSKilled::eqToShow}, because $g_{\epsilon} \left( s \right)  \rightarrow g^{*} \left( s \right) $ uniformly on $\TWR$.
Indeed
\begin{align}
	\abs{g_{\epsilon} \left( s,x, \w, \theta \right) -g^{*} \left( s,x, \w, \theta \right) } 
	\leq
	\sup_{ \left( y,\eta \right)  \in  \supp \left\{ k_{\epsilon} \right\}} \abs{g^{*} \left( s, x+y, \w, \theta+\eta \right) -g^{*} \left( s,x, \w, \theta \right) }
\end{align}
 and $g^{*} \left( s \right) $ is uniformly continuous (as a continuous function with compact support).
Therefore $h_{\epsilon} \left( t \right)  \rightarrow f$ and $h_{\epsilon} \left( s \right)  \rightarrow \log \left( 1+g^{*} \left( s \right)  \right) $ uniformly.

The integrals \encircle{R1} and \encircle{R2}  are smaller or equal to the right hand side of \eqref{lem::LDPempP::indep::S2leqSKilled::eqToShow}.
We interpret \encircle{R3} as an error and show in the next step that it can be bounded from a above by a vanishing function.

\emptyline
\step[{A vanishing upper bound on \protect\encircle{R3}}]

By the following lemma we get a vanishing upper bound on the last integral \encircle{R3}  of \eqref{lem::LDPempP::indep::S2leqSKilled::eqToShowForHepsilon}.

\begin{lemma}[compare to Lemma~4.12 in \cite{DawGarLDP}]
\label{lem::LDPempP::indep::S2leqSBoundGeps}
	For $\epsilon>0$ small enough, there exists a continuous function $r_{\epsilon}$ on $\TTWR $, such that
	\begin{align}
		\left( \partial_{u}+\Gen^{I}_{u,x,\w} \right) g_{\epsilon} \left( u,x, \w, \theta \right)
		\leq
		r_{\epsilon} \left( u,x, \w, \theta \right)
		\qquad
		\textnormal{ for } 
		\left( u,x, \w, \theta \right) \in \TTWR 
	\end{align}
	and
	$r_{\epsilon} \rightarrow 0$ uniformly on $\TTWR$ for $\epsilon \rightarrow 0$.
\end{lemma}

We state the  proof of this lemma after we have finished the proof of Lemma~\ref{lem::LDPempP::indep::S2leqSKilled}.
By Lemma~\ref{lem::LDPempP::indep::S2leqSBoundGeps}
\begin{align}
	\encircle{R3}
	\leq
		\int_{s}^{t} \int_{\TWR} 
				\frac{ r_{\epsilon} \left( u,x, \w, \theta  \right)  }{1+g_{\epsilon} \left( u,x, \w, \theta  \right) }
			\mu_{u}  \left( \dd x, \dd \w, \dd \theta \right) \: \dd u
	. 
\end{align}
The right hand side vanishes for $\epsilon \rightarrow 0$,
because $r_{\epsilon} \rightarrow 0$ uniformly and $e^{-\iNorm{f} } \leq 1+ g_{\epsilon} \leq 1$ (by \eqref{eq::lem::LDPempP::indep::S2leqSKilled::gForms}).

Hence we conclude that \eqref{lem::LDPempP::indep::S2leqSKilled::eqToShow} holds for \Holder continuous drift coefficients.
 
\end{steps}
\emptyline
\step[{General drift coefficient $b^{I}$}]
\label{pf::lem::LDPempP::indep::S2leqSKilled::step::Generalb}

Last but not least we show now that Lemma~\ref{lem::LDPempP::indep::S2leqSKilled} also holds for general (non-\Holder continuous) drift coefficients provided that the Assumption~\ref{ass::LDPempP::Inde} is satisfied.
Therefore we approximate at first (\ref{pf::lem::LDPempP::indep::S2leqSKilled::step::Generalb::Approxb}) the drift coefficient $b^{I}$ by a sequence of \Holder-continuous functions $b^{I,(n)}$, that converge to $b^{I}$ on $\Csp{\TTWR}$.
Then we show that \ref{pf::lem::LDPempP::indep::S2leqSKilled::step::Hoelder} can be applied for all $b^{I,(n)}$ (\ref{pf::lem::LDPempP::indep::S2leqSKilled::step::Generalb::eqForbn}), i.e. that  \eqref{lem::LDPempP::indep::S2leqSKilled::eqToShow} holds for each $b^{I,(n)}$.
Finally we justify that we can take the limit on both sides of \eqref{lem::LDPempP::indep::S2leqSKilled::eqToShow} such that this inequality also holds for $b^{I}$.
To this end we only need to show that the left hand side of  \eqref{lem::LDPempP::indep::S2leqSKilled::eqToShow} for $b^{I,(n)}$ is in the limit greater than the corresponding one for $b^{I}$ and an analogue result for the right hand side (\ref{pf::lem::LDPempP::indep::S2leqSKilled::step::Generalb::LHS} and \ref{pf::lem::LDPempP::indep::S2leqSKilled::step::Generalb::RHS}).
For that matter we follow the ideas of Dawson and G\"artner  in Chapter~4.5 of \cite{DawGarLDP} and generalise their proof to the setting we consider here.
Along the way, we also fix a small issue of Dawson and G\"artner in their treatment of the left hand side (compare our \ref{pf::lem::LDPempP::indep::S2leqSKilled::step::Generalb::LHS} with their calculation on page 288).

\emptyline
\begin{steps}
\step[{Approximation of $b^{I}$}]
\label{pf::lem::LDPempP::indep::S2leqSKilled::step::Generalb::Approxb}

Denote by $\Wsp_{f,2}$ the open set of all points in $\Wsp$ with distance at most $1$ from $\Wsp_{f}$.

We approximate the continuous drift coefficient $b^{I}$  by functions $b^{I,(n)} \in \Csp{\TTWR}$.
These functions are chosen such that $b^{I,(n)}$  is on $\TT \times \Wsp_{f} \times B_{R}$  also $\frac{1}{4}$-\Holder continuous in time and $\frac{1}{2}$-\Holder continuous in  $B_{R} $.
Moreover $b^{I,(n)}=b^{I}$ outside of $\TT \times \Wsp_{f,2} \times B_{2R}  $ and $b^{I,(n)} \rightarrow b^{I}$ uniformly.  
Finding such a sequence is for example possible by the Stone-Weierstrass Theorem (on the compact set $\ol\Wsp_{f,2}$) and the Urysohn’s Lemma (with $\Wsp_{f}$ and $\Wsp_{f,2}$).

\emptyline
\step[{\eqref{lem::LDPempP::indep::S2leqSKilled::eqToShow}  holds for each $b^{I,(n)}$}] 
\label{pf::lem::LDPempP::indep::S2leqSKilled::step::Generalb::eqForbn}

One has to prove, that the Martingale problem for the generator $\Gen^{I,(n)}_{s,x,\w}$ with drift coefficient $b^{I,(n)}$ is well posed. But this we get from the (Cameron-Martin-) Girsanov theorem  (\cite{StrVarMultidi} Theorem~6.4.2)  because the difference between $b^{I,(n)}$ and $b^{I}$ is at most $\epsilon$ for $n$ large enough by the uniform convergence.
We call the corresponding solution $P^{I,(n)}_{s,x, \w, \theta}$ and its semi-group $U^{R,(n)}_{s,t}$.
Hence by \ref{pf::lem::LDPempP::indep::S2leqSKilled::step::Hoelder}, \eqref{lem::LDPempP::indep::S2leqSKilled::eqToShow} holds with $U^{R,(n)}_{s,t}$ and $\Gen^{I,(n)}$.

\emptyline
\step[{The LHS of  \eqref{lem::LDPempP::indep::S2leqSKilled::eqToShow} for $b^{I,(n)}$ is in the limit greater than the LHS  for $b^{I}$}]
\label{pf::lem::LDPempP::indep::S2leqSKilled::step::Generalb::LHS}

Fix $\left( s,t,x, \w, \theta \right) \in \Ti \times \TTWR$.
By \cite{StrVarMultidi} Theorem~11.1.4, $P^{I,(n)}_{s,x, \w, \theta} \rightarrow P^{I}_{s,x, \w, \theta}$, and by \cite{StrVarMultidi} Theorem~11.1.2, $\theta_{\left[s,T\right]} \mapsto \tau_{R}^{s} \left( \theta_{\left[s,T\right]} \right) $ is lower semi-continuous.
 Hence $\left\{ \tau_{R}^{s}  >t \right\}$ is an open set and $\1_{\tau_{R}^{s} >t}$ is lower semi-continuous.
The function $\left( e^{f}-1 \right)$ is non positive and continuous, what implies that $\left( e^{f \left( x, \w, \theta_{\Ti}  \right) }-1 \right) \1_{\tau_{R}^{s}  >t }$ is upper semi continuous.
 By the Portmanteau theorem
\begin{align}
\begin{split}
	\limsup_{n\rightarrow \infty} U^{R,(n)}_{s,t} \left( e^{f}-1 \right)  \left( x, \w, \theta \right)  
	=
	\limsup_{n\rightarrow \infty}  \int_{\Csp{\Ti}} \left( e^{f \left( x, \w, \theta_{\Ti}  \right) }-1 \right) \1_{\tau_{R}^{s}  >t } 
								P^{I,(n)}_{s,x, \w, \theta} \left( \dd \theta_{\Ti}  \right) 
	\\
	\leq
	\int_{\Csp{\Ti}} \left( e^{f \left( x, \w, \theta_{\Ti}  \right) }-1 \right) \1_{\tau_{R}^{s}  >t } P^{I}_{s,x, \w, \theta} \left( \dd \theta_{\Ti}  \right) 
	=
		U^{R}_{s,t} \left( e^{f}-1 \right)  \left( x, \w, \theta \right)  
	.
\end{split}
\end{align}
With the Fatou-Lebesgue theorem  (possible because $-1 < U^{R,(n)}_{s,t} \left( e^{f}-1 \right)  \left( x, \w, \theta \right) $) we conclude
\begin{align}
\begin{split}
	&\limsup_{n\rightarrow \infty} 
					\int_{\TWR} 
							\log \left[ 1+ U^{R,(n)}_{s,t}  \left( e^{f}-1 \right)  \right] \left(x, \w, \theta\right)
						\mu_{s}  \left( \dd x, \dd \w, \dd \theta \right)
	\\
	&\leq
							\int_{\TWR} 
									\limsup_{n\rightarrow \infty} \log \left[ 1+ U^{R,(n)}_{s,t}  \left( e^{f}-1 \right)  \right] \left(x, \w, \theta\right)
								\mu_{s}  \left( \dd x, \dd \w, \dd \theta \right)
	\\
	&\leq
		\int_{\TWR} 
				\log \left[ 1+ U^{R}_{s,t}  \left( e^{f}-1 \right)  \right] \left(x, \w, \theta\right)
			\mu_{s}  \left( \dd x, \dd \w, \dd \theta \right)
	.
\end{split}
\end{align}
 
 \emptyline
 \step[{The RHS of  \eqref{lem::LDPempP::indep::S2leqSKilled::eqToShow} for $b^{I,(n)}$ is in the limit smaller than the RHS for $b^{I}$}]
\label{pf::lem::LDPempP::indep::S2leqSKilled::step::Generalb::RHS}

By the triangle inequality we get
\begin{align}
	 \abs{ \partial_{u}\mu_{u}  - \left(\Gen^{I,(n)}_{u,.,.}\right)^{*}  \mu_{u} }^{2}_{\mu_{u}} 
	 \leq
	  \abs{ \partial_{u}\mu_{u}  - \left(\Gen^{I}_{u,.,.}\right)^{*}  \mu_{u} }^{2}_{\mu_{u}} 
	   +
		\abs{  \left(\Gen^{I}_{u,.,.} - \Gen^{I,(n)}_{u,.,.}\right)^{*}  \mu_{u}  }^{2}_{\mu_{u}} 
	.
\end{align}
The last summand is smaller or equal to $\frac{\sigma^{2}}{2} \int_{\TWR} \abs{b^{I,(n)} \left(  x, \w, \theta \right)-b^{I} \left(  x, \w, \theta \right)}^{2} \mu_{u}  \left( \dd x, \dd \w, \dd \theta \right)$, what vanishes when $n \rightarrow \infty$ by the uniform convergence.

\emptyline
\step[{Conclusion}]
\label{pf::lem::LDPempP::indep::S2leqSKilled::step::Generalb::Concl}
Hence we conclude
\begin{align}
\begin{split}
	&\int_{\TWR} f \left( x, \w, \theta \right) \mu_{t}  \left( \dd x, \dd \w , \dd \theta \right)
			-
			\int_{\TWR} \log \left[ 1+ U^{R}_{s,t}  \left( e^{f}-1 \right)  \right] \left( x, \w, \theta \right) \mu_{s}  \left( \dd x, \dd \w , \dd \theta \right)
	\\	
	&\leq
		\liminf_{n\rightarrow \infty}\left\{ 
						\int f \left( x, \w, \theta \right) \mu_{t}  \left( \dd x, \dd \w , \dd \theta \right)
						-
						\int  \log \left[ 1+ U^{R,(n)}_{s,t}  \left( e^{f}-1 \right)  \right] \left( x, \w, \theta \right) \mu_{s}  \left( \dd x, \dd \w , \dd \theta \right)
		 \right\}
	\\
	&\leq
	 \liminf_{n\rightarrow \infty}  \int_{s}^{t}
		\abs{ \partial_{u}\mu_{u}  - \left(\Gen^{I,(n)}_{u,.}\right)^{*}  \mu_{u} }^{2}_{\mu_{u}} \dd u
	\leq
		\int_{s}^{t} \abs{ \partial_{u}\mu_{u}  - \left(\Gen^{I}_{u,.}\right)^{*}  \mu_{u} }^{2}_{\mu_{u}} \dd u
	.
\end{split}
\end{align}

\end{steps}
\end{steps}\vspace{-\baselineskip}
\end{proof}

\begin{proof}[of Lemma~\ref{lem::LDPempP::indep::S2leqSBoundGeps}]
Fix $\left( s, x, \w, \theta \right) \in \TTWR$.
We get by the integration by parts formula (and the same argument as in \cite{DawGarLDP} in the proof of Lemma~4.12  to bound the derivatives at the boundary  $\partial B_{R}$), 
\begin{align}
\label{eq::pf::lem::LDPempP::indep::S2leqSBoundGeps::Defr}
\begin{split}
	&\left( \partial_{s}+\Gen^{I}_{s,x,\w} \right) g_{\epsilon} \left( s, x, \w, \theta \right) 
	\\
	&\leq
	\int
		k_{\epsilon} \left( x-x' , \w-\w', \theta-\theta' \right) 
		\left(
		\partial_{s} g \left( s,x', \w',\theta' \right) 
		 +
		 \frac{\sigma^{2}}{2}  \partial^{2}_{\theta'^{2}}g \left( s,x', \w',\theta' \right) 
		 \right.
	\\	
		 &\mkern330mu
		 +
		 b^{I} \left( s,x, \w, \theta \right)  \partial_{\theta'} g \left( s,x', \w',\theta' \right) 
		\bigg) 
		 \dd\theta' \dd \w' \dd x'
	\\
	&=
	\int
			k_{\epsilon} \left( x-x' , \w-\w', \theta-\theta' \right) 
		\left(  b^{I} \left( s,x, \w, \theta \right)  - b^{I} \left( s,x', \w', \theta' \right)  \right)
		\partial_{\theta'} g \left( s,x', \w', \theta' \right) 
	 	\dd\theta' \dd \w' \dd x'
	,
\end{split}
\end{align}
where the two integrals are over the spac $\Td \times \Wsp_{f} \times B_{R}$.
In the last equality we use that $g$ is a solution to \eqref{eq::PDETerminal}.
We denote the right hand side of \eqref{eq::pf::lem::LDPempP::indep::S2leqSBoundGeps::Defr}
by	$r_{\epsilon} \left( s,x, \w, \theta \right) $.

For each $\epsilon$, the integrand in $r_{\epsilon}$ is continuous and uniformly bounded, because $b^{I}$ and $\partial_{\theta'} g $ are continuous and we consider a compact set.
This implies that $r_{\epsilon}$ is continuous.

For all $\left( s, x, \w, \theta \right) \in \TTWR$
\begin{align}
	\abs{ r_{\epsilon} \left( s, x, \w, \theta \right)  }
	\leq
		\sup_{\substack{	x',x'' \in \Td; \w' \in \Wsp; \w'' \in \Wsp_{f}; \theta' ,\theta'' \in B_{2R}	\\
										\abs{x'-x''}< \epsilon, \abs{\w' - \w''} < \epsilon, \abs{ \theta'-\theta''}< \epsilon
					}}
			\abs{ b^{I} \left( s,x',\w',\theta' \right)  - b^{I} \left( s,x'',\w'',\theta'' \right) }
		\iNorm{\partial_{\theta'} g }
	,
\end{align}
for $\epsilon$ small enough.
The derivative $\partial_{\theta'} g $ is bounded and $b^{I}$ is uniform continuous on the compact set $\TT \times \ol\Wsp_{f,2} \times B_{2R}$.
Hence $r_{\epsilon}$ converges uniformly to $0$.
\end{proof}

\paragraph{PDE preliminaries}
\label{sec::LDPempP::PDE}

In this chapter we prove (see Theorem~\ref{thm::LDPempP::UniqClasPDE}) the uniqueness and the existence of a \Holder continuous (in time and spin value) solution of the terminal boundary value problem \eqref{eq::PDETerminal}, that is moreover continuous on $\Td$ and on a connected subset $\what{\Wsp} \subset \Wsp$. 
We did not find such a result in the literature due to the non-ellipticity in the $\TW$-directions.

In the proof of this result, we look at first at the PDE \eqref{eq::PDETerminal} with fixed $\left( x, \w \right) \in \Td \times \what{\Wsp}$.
For each of these PDEs, we get by a result of \cite{Lady} (that we repeat in Theorem~\ref{thm::LDPempP::Lady})  the existence and uniqueness of a solution $g_{x,\w}$ on $\Ti \times B_{R}$.
The main part of the proof then consists of showing that these solutions are continuous in $x \in \Td$ and $\w \in \what{\Wsp}$.

\emptyline

Now we define the \Holder space, on which we derive the solution.
We refer to the page~7 in \cite{Lady} for this definition (without the dependency on $\Td$).

\begin{definition}
	We denote by $H^{\ell / 2, 0, 0, \ell} \left( \TT \times \what{\Wsp} \times B_{R} \right) $ the Banach space of continuous functions on $\TT \times \what{\Wsp} \times B_{R}$, which have continuous derivatives $\partial^{r}_{t} \partial^{s}_{\theta}$, with $2 r + s \leq \ell$, 
	and with finite norm
	\begin{align}
			\abs{u}_{H^{\ell / 2, 0, \ell}} 
			=
				 \sum_{2r+s \leq \floor{\ell}} \iNorm{  \partial^{r}_{t} \partial^{s}_{\theta} u }
				 +
				 \sum_{2r+s = \floor{\ell}} \abs{  \partial^{r}_{t} \partial^{s}_{\theta} u }_{\ell-\floor{\ell}, \theta}
				 +
				 \sum_{2r+s \in \left\{ \floor{\ell}-1 , \floor{\ell} \right\}} \abs{  \partial^{r}_{t} \partial^{s}_{\theta} u }_{\frac{2r+s}{2}, t}
				 ,
	\end{align}
	where $\abs{ . }_{\ell-\floor{\ell}, \theta}$ and $\abs{ .  }_{\ell-\floor{\ell}, t}$ are the usual \Holder norms in $\theta\in \R$ and $t \in \Ti$ respectively.
	
	The space $H^{\ell / 2,  \ell} \left( \Ti \times B_{R} \right) $ is defined analogue, just without the dependency on $\Td \times \what{\Wsp}$.
	
\end{definition}

\begin{remark}
	For $\ell \in \left( 0, 1 \right)$, the norm $\abs{u}_{H^{\ell / 2, 0, \ell}} $ is simply $\iNorm{u} + 	\abs{   u }_{\ell, \theta} +\abs{   u }_{\frac{\ell}{2}, t}$.
\end{remark}

\begin{theorem}
\label{thm::LDPempP::UniqClasPDE}
Let $\ell > 0$ be a non integer number.
Assume that $b^{I} \in H^{\ell/2,0,0,\ell} \left( \TT \times \what{\Wsp} \times B_{R} \right) $ (that we use in the definition \eqref{def::GeneratorIndependent} of $\Gen^{I}$)  and that $i \in H^{0,0,\ell+2} \left( \Td \times \what{\Wsp} \times B_{R} \right) $. 
Then for each $R \in \R$, there is a unique solution $g^{*} \in H^{\ell /2 +1,0,0,\ell+2} \left( \left[ 0,t \right] \times \Td \times \what{\Wsp} \times \overline{B_{R}}   \right) $ of the following terminal boundary value problem 
\begin{equation}
\label{eq::thm::LDPempP::UniqClasPDE::PDE}
\begin{aligned}
	&\partial_{s} g \left( s,x, \w, \theta \right)  \!\!&&=&& - \Gen^{I}_{t,x,\w} g  \left( s,x, \w, \theta \right)  
		\quad &&\left( s,x, \w, \theta \right)  \in \left[ 0,t \right) \times \Td \times \what{\Wsp} \times B_{R}
	\\
		&g \left( t,x, \w, \theta \right)  &&=&& i \left( x, \w, \theta \right)    
				&&\left( x, \w, \theta \right)  \in  \Td \times \what{\Wsp} \times B_{R}
	\\
		&g \left( s,x, \w, \theta \right)  &&=&& 0    
				 &&\left( s,x, \w, \theta \right)  \in \left[ 0,t \right)  \times \Td \times \what{\Wsp} \times \partial B_{R} \left( 0 \right) 
	.
\end{aligned}
\end{equation}

\end{theorem}

In the proof of this theorem, we use the following version of the Theorem~5.2 in  Chapter~IV of \cite{Lady}.
Because we need it only for a specific class of PDEs, it is not as general as the original version of the theorem.

\begin{theorem}[\cite{Lady} Chapter~IV Theorem~5.2  ] 
\label{thm::LDPempP::Lady}
		Let $\ell > 0$ be a non integer number and  $\ol{i} \in H^{\ell+2} \left( B_{R} \right) $ and $\ol{b}^{I},  w \in H^{\ell/2,\ell} \left( \left[ 0, t \right] \times B_{R} \right) $.
		Then for each $R>0$, there is a unique classical solution $g^{*}\in H^{\ell /2 +1,\ell+2} \left(  \left[ 0,t \right] \times \overline{B_{R}} \right) $ of the following terminal boundary value problem
\begin{equation}
\label{eq::thm::LDPempP::Lady::PDE}
\begin{aligned}
	&\partial_{s} g \left( s,\theta \right)  
			\!\!&&=&&  
			- \left( \frac{\sigma^{2}}{2} \partial^{2}_{\theta^{2}} + \ol{b}^{I} \left( s,\theta \right)  \partial_{\theta} \right) g  \left( s, \theta \right)  + w \left( \theta,s \right)
								 \quad &&\left( s,\theta \right)  \in \left[ 0,t \right) \times B_{R}
	\\
		&g \left( t,\theta \right)  &&=&& \ol{i} \left( \theta \right)    &&\theta   \in   B_{R}
	\\
		&g \left( s, \theta \right)  &&=&& 0     &&\left( s,\theta \right)  \in \left[ 0,t \right) \times \partial B_{R} \left( 0 \right) 
	.
\end{aligned}
\end{equation}
Moreover the solution $g^{*}$ satisfies
\begin{align}
\label{eq::thm::LDPempP::Lady::NormInequality}
	\abs{    g^{*}   }_{   H^{\ell /2 +1,\ell+2} \left(  \left[ 0,t \right] \times \overline{B_{R}} \right)   }
	\leq
	C
	\left(
		\abs{   w  }_{   H^{\ell /2 ,\ell} \left(  \left[ 0,t \right] \times \overline{B_{R}} \right)    }
		+
		\abs{   i  }_{H^{\ell+2} \left( B_{R} \right)}
	\right)
\end{align}
for a constant $C>0$ independent of $w$ and $i$.

\end{theorem}

For a proof of this Theorem~\ref{thm::LDPempP::Lady} we refer to \cite{Lady}. Now we prove the Theorem~\ref{thm::LDPempP::UniqClasPDE}.

\begin{proof}[of Theorem~\ref{thm::LDPempP::UniqClasPDE}]
\begin{steps}
\step[The existence and regularity]

The PDE \eqref{eq::thm::LDPempP::UniqClasPDE::PDE} corresponds for a fixed tuple $\left( x,\w \right) \in \Td \times \what{\Wsp}$ to the PDE \eqref{eq::thm::LDPempP::Lady::PDE} with $w \equiv 0$, $\ol{i}\left( \theta \right) = i \left( x, \w, \theta \right)$, $\ol{b}^{I} \left( s, \theta \right) = b^{I} \left( s, x, \w, \theta \right)$, due to the independence in $x \in \Td$ and $\w \in \what{\Wsp}$ of the operator $\Gen^{I}_{t,x,\w}$.
Therefore we know by Theorem~\ref{thm::LDPempP::Lady}, that there is unique solution $g_{x,\w} \in  H^{\ell /2 +1,\ell+2} \left(  \left[ 0,t \right] \times \overline{B_{R}} \right)  $ of the corresponding PDE \eqref{eq::thm::LDPempP::Lady::PDE},  for each $\left( x,\w \right) \in \Td \times \what{\Wsp}$.
Set $g \left( . , x , \w , . \right)  \defeq g_{x,\w} $.
The function $g$ is a solution of \eqref{eq::thm::LDPempP::UniqClasPDE::PDE}.
To show the claimed regularity of this solution, 
we need to show that $\left( x,\w \right) \mapsto g_{x,\w}$ is a continuous map $\Td \times \what{\Wsp} \rightarrow  H^{\ell /2 +1,\ell+2}\left(  \left[ 0,t \right] \times \overline{B_{R}} \right)$.

Fix an arbitrary tuple $\left( x_{0}, \w_{0} \right) \in \Td \times \what{\Wsp}$. 
The proof of the continuity at $\left( x_{0}, \w_{0} \right)$ is organised as follows:
\begin{steps}
	\item
		First we define an operator $I_{x,\w} :H^{\ell /2 +1,\ell+2}\left(  \left[ 0,t \right] \times \overline{B_{R}} \right) \rightarrow H^{\ell /2 +1,\ell+2}\left(  \left[ 0,t \right] \times \overline{B_{R}} \right) $ for each $\left( x,\w \right) \in \Td \times \what{\Wsp}$.
	\item
		Then we show that $I_{x,\w}$ is a continuous contraction, when $\abs{x-x_{0}}$ and $\abs{\w-\w_{0}}$ are small enough.
	\item
		Next we show that the sequence $ \left( I_{x,\w} \right)^{n} \left( g^{*}_{x_{0},\w_{0}} \right)$ converges to $g^{*}_{x,\w}$ (also for $\abs{x-x_{0}}$ and $\abs{\w-\w_{0}}$ small enough).
	\item
		Finally we conclude from the previous steps the continuity of $g^{*}_{x,\w}$ at $\left( x_{0},\w_{0}\right) \in \Td \times \what{\Wsp}$.
\end{steps}

\emptyline
Let us carry out this program.
\emptyline 

\begin{steps}
\item

Define the operator
\begin{align}
	T_{s,x, \w} \defeq 
	\Gen^{I}_{s,x_{0},\w_{0}}-\Gen^{I}_{s,x,\w}
	=
	  \left( b^{I} \left( s,x_{0}, \w_{0}, . \right)  - b^{I} \left( s,x, \w, . \right)  \right) \partial_{\theta} 
	.
\end{align}
With this operator, $\Gen^{I}_{s,x,\w}$ can be seen as a perturbation of $\Gen^{I}_{s,x_{0},\w_{0}}$, by  $\Gen^{I}_{s,x,\w} =\Gen^{I}_{s,x_{0},\w_{0}}-T_{s,x,\w}$.
Moreover we define the operator
\begin{align}
	I_{x,\w} : H^{\ell /2 +1,\ell+2} \left(  \left[ 0,t \right] \times \overline{B_{R}} \right)   \rightarrow H^{\ell /2 +1,\ell+2} \left(  \left[ 0,t \right] \times \overline{B_{R}} \right)   
\end{align}
as the map that sends a function $v \in H^{\ell /2 +1,\ell+2} \left(  \left[ 0,t \right] \times \overline{B_{R}} \right)  $ to the (unique) solution of 
\begin{equation}
\begin{aligned}
	&\partial_{s} g \left( s,\theta \right)  \!\!&&=&&  - \Gen^{I}_{s,x_{0},\w_{0}}  g  \left( s, \theta \right)  +  T_{s,x,\w} v
								 \quad &&\left( s,\theta \right)  \in \left[ 0,t \right) \times B_{R}
	\\
		&g \left( t,\theta \right)  &&=&& i \left( x,\w,  \theta \right)    &&\theta   \in   B_{R}
	\\
		&g \left( s, \theta \right)  &&=&& 0     &&\left( s,\theta \right)  \in \left[ 0,t \right) \times \partial B_{R}
	.
\end{aligned}
\end{equation}

We get the existence and the uniqueness of a solution to this PDE from Theorem~\ref{thm::LDPempP::Lady}.

\emptyline
\item

We show now that $I_{x,\w}$ is a continuous contraction.

Fix arbitrary $u_{1},u_{2} \in H^{\ell /2 +1,\ell+2} \left(  \left[ 0,t \right] \times \overline{B_{R}} \right)  $.
By the definition, $I_{x,\w} \left( u_{1} \right) -I_{x,\w} \left( u_{2} \right) $ is the unique classical solution to
\begin{align}
\begin{split}
	\left( \partial_{s} + \Gen^{I}_{s,x_{0},\w_{0}} \right) \left( I_{x,\w} \left( u_{1} \right) -I_{x,\w} \left( u_{2} \right)  \right) = T_{s,x,\w}  \left( u_{1}-u_{2} \right) 
	\\
	\textnormal{ with $0$ terminal and $0$ boundary condition}
	.
\end{split}
\end{align}
Then by \eqref{eq::thm::LDPempP::Lady::NormInequality},  for $\abs{x_{0}-x}$ and $\abs{\w_{0}-\w}$ small enough,
\begin{align}
\label{eq::pf::thm::LDPempP::UniqClasPDE::IContract}
\begin{split}
		\abs{ I_{x,\w} \left( u_{1} \right)  - I_{x,\w} \left( u_{2} \right)  }_{H^{\ell/2+1,\ell+2}}
	&\leq
		C 
		\abs{ T_{.,x,\w}  \left( u_{1}-u_{2} \right)  }_{H^{\ell/2,\ell}}
	\\
	&\leq
		C 
		\abs{ b^{I} \left( .,x_{0}, \w_{0}, . \right) - b^{I} \left( .,x, \w, . \right)  }_{H^{\ell/2,\ell}}
		\abs{ \partial_{\theta}  \left( u_{1}-u_{2} \right)  }_{H^{\ell/2,\ell}}
	\\
	&\leq
		\epsilon
		\abs{ u_{1}-u_{2} }_{H^{\ell/2+1,\ell+2}}
\end{split}
\end{align}
In the last inequality we use that $b^{I} \in H^{\ell/2,0,0,\ell} \left( \TT \times \what{\Wsp} \times B_{R} \right) $.
This implies that $I_{x,\w}$ is a continuous contraction.
Note that the $\epsilon$ is independent of $\left( x,\w \right) \in \Td \times \what{\Wsp}$, as long as $\abs{x_{0}-x}$ and $\abs{\w_{0}-\w}$ are small enough, because the constant $C$ depends only on $\Gen^{I}_{.,x_{0},\w_{0}}$.

\emptyline
\item

For the solution $g^{*}_{x_{0},\w_{0}}$ of \eqref{eq::thm::LDPempP::UniqClasPDE::PDE} at $\left( x_{0}, \w_{0} \right)$, we define the sequence $\left\{ \left( I_{x,\w} \right)^{n} \left( g^{*}_{x_{0},\w_{0}}\right) \right\}_{n}$.
Then by \eqref{eq::pf::thm::LDPempP::UniqClasPDE::IContract}
\begin{align}
\label{eq::pf::thm::LDPempP::UniqClasPDE::Cauchy}
\begin{split}
	&\abs{  \left( I_{x,\w} \right)^{n+1} \left( g^{*}_{x_{0},\w_{0}}\right)  - \left( I_{x,\w} \right)^{n} \left( g^{*}_{x_{0},\w_{0}}\right)  }_{H^{\ell/2+1,\ell+2}}
	\\
	&\leq
		\epsilon \abs{   \left( I_{x,\w} \right)^{n} \left( g^{*}_{x_{0},\w_{0}}\right)  - \left( I_{x,\w} \right)^{n-1} \left( g^{*}_{x_{0},\w_{0}}\right)   }_{H^{\ell/2+1,\ell+2}}
	\leq
		\epsilon^{n} \abs{  I_{x,\w}  \left( g^{*}_{x_{0},\w_{0}} \right)   -  g^{*}_{x_{0},\w_{0}}   }_{H^{\ell/2+1,\ell+2}}
	\: 
	.
\end{split}
\end{align}
 Therefore $\left\{ \left( I_{x,\w} \right)^{n} \left( g^{*}_{x_{0},\w_{0}}\right) \right\}_{n}$ is a Cauchy sequence.
 The \Holder spaces are complete, hence there is a $u^{*}_{x,\w} \in H^{\ell/2+1,\ell+2}$ such that $ \left( I_{x,\w} \right)^{n} \left( g^{*}_{x_{0},\w_{0}}\right)  \rightarrow u^{*}_{x,\w}$.
The continuity of $I_{x,\w}$ implies that
 also $I_{x,\w} \left( \left( I_{x,\w} \right)^{n} \left( g^{*}_{x_{0},\w_{0}}\right) \right)  \rightarrow I_{x,\w} \left( u^{*}_{x,\w} \right) $. 
Therefore $u^{*}_{x,\w}=I_{x,\w} \left( u^{*}_{x,\w} \right) $.
By the definition of $I_{x,\w}$ and the uniqueness of Theorem~\ref{thm::LDPempP::Lady}, we conclude $u^{*}_{x,\w}=g^{*}_{x,\w}$. 

\emptyline
\item

Then by \eqref{eq::pf::thm::LDPempP::UniqClasPDE::Cauchy}
\begin{align}
\label{eq::pf::thm::LDPempP::UniqClasPDE::Continuity0}
\begin{split}
		\abs{g^{*}_{x,\w}- g^{*}_{x_{0},\w_{0}} }_{H^{\ell/2+1,\ell+2}}
	&\leq
		\sum_{n=0}^{\infty} \abs{ \left( I_{x,\w} \right)^{n+1} \left( g^{*}_{x_{0},\w_{0}}\right)  - \left( I_{x,\w} \right)^{n} \left( g^{*}_{x_{0},\w_{0}}\right) }_{H^{\ell/2+1,\ell+2}}
	\\
	&\leq
		\abs{   I_{x,\w}  \left( g^{*}_{x_{0},\w_{0}} \right)   -  g^{*}_{x_{0},\w_{0}}   }_{H^{\ell/2+1,\ell+2}} \frac{1}{1-\epsilon}
	.
\end{split}
\end{align}
We show now that the right hand side is bounded by a $\epsilon_{1}>0$ for  $\left( x,\w \right) \in \Td \times \what{\Wsp}$ with $\abs{x_{0}-x}$ and $\abs{\w_{0}-\w}$ small enough.
By construction $  I_{x,\w}  \left( g^{*}_{x_{0},\w_{0}} \right)   -  g^{*}_{x_{0},\w_{0}} $ is the solution to the PDE $\partial_{s} g=-\Gen^{I}_{s,x_{0},\w_{0}} g+T_{s,x,\w} g^{*}_{x_{0},\w_{0}}$ with $i \left( x, \w, . \right) -i \left( x_{0}, \w_{0}, . \right) $ boundary condition. Hence by \eqref{eq::thm::LDPempP::Lady::NormInequality}
\begin{align}
\label{eq::pf::thm::LDPempP::UniqClasPDE::Continuity}
\begin{split}
	\abs{  I_{x,\w}  \left( g^{*}_{x_{0},\w_{0}} \right)   -  g^{*}_{x_{0},\w_{0}}  }_{H^{\ell/2+1,\ell+2}}
	&\leq 
		C
		\left(
			\abs{T_{t,x,\w} g^{*}_{x_{0},\w_{0}} }_{H^{\ell/2+1,\ell+2}}
			+
			\abs{i \left( x, \w, . \right) -i \left( x_{0}, \w_{0}, . \right) }_{H^{\ell/2+1,\ell+2}}
		\right)
	.
\end{split}
\end{align}
Then as in \eqref{eq::pf::thm::LDPempP::UniqClasPDE::IContract} and finally by applying again \eqref{eq::thm::LDPempP::Lady::NormInequality} for $g^{*}_{x_{0},\w_{0}}$, we get that the right hand side of \eqref{eq::pf::thm::LDPempP::UniqClasPDE::Continuity} is smaller or equal to
\begin{align}
\label{eq::pf::thm::LDPempP::UniqClasPDE::Continuity2}
\begin{split}
		C
		\left(
			\epsilon 
			\abs{g^{*}_{x_{0},\w_{0}}}_{H^{\ell/2+1,\ell+2}}
		+
		\epsilon 
		\right)
	\leq
		\epsilon 
		C 
		\left( \abs{i \left( x_{0}, \w_{0}, . \right) }_{H^{\ell+2}} +1 \right)
	\leq 
		\epsilon_{1}
\end{split}
\end{align}
because  $i \left( x_{0}, \w_{0}, . \right)  \in H^{\ell+2}$.
Therefore $\abs{g^{*}_{x,\w}- g^{*}_{x_{0},\w_{0}} }_{H^{\ell/2+1,\ell+2}}< \epsilon_{1}$ for $\abs{x_{0}-x}$ and $\abs{\w_{0}-\w}$ small enough, by  \eqref{eq::pf::thm::LDPempP::UniqClasPDE::Continuity0}.

This is the claimed regularity of the solution $g^{*}_{x,\w}$ at $\left( x, \w \right) \in \Td \times \what{\Wsp}$.

\end{steps}

\emptyline
\step[The Uniqueness]

Let $g^{*}$ be a solution of \eqref{eq::thm::LDPempP::UniqClasPDE::PDE}.
Then, for each tuple $\left( x,\w \right) \in \Td \times \what{\Wsp}$, $g^{*}_{x,\w}$ has to be the unique solution of \eqref{eq::thm::LDPempP::Lady::PDE} with $w \equiv 0$, $\ol{i}\left( \theta \right) = i \left( x, \w, \theta \right)$, $\ol{b}^{I} \left( s, \theta \right) = b^{I} \left( s, x, \w,\theta \right)$. 
Therefore there is at most one solution of \eqref{eq::thm::LDPempP::UniqClasPDE::PDE} in $H^{\ell /2 +1,0,0,\ell+2} \left( \left[ 0,t \right] \times\Td \times \what{\Wsp} \times \overline{B_{R}}   \right) $.
\end{steps}\vspace{-\baselineskip}
\end{proof}

\begin{remark}
Using the calculation in \eqref{eq::pf::thm::LDPempP::UniqClasPDE::Continuity} and in \eqref{eq::pf::thm::LDPempP::UniqClasPDE::Continuity2},
 we could show even higher regularity, than continuity, of the solution in $\Td \times \what{\Wsp}$, if we assume higher regularity of $b$ and $i$ in $\Td \times \what{\Wsp}$.
\end{remark}

\subsubsection{Another representation of the rate function \texorpdfstring{$S^{I}_{\nu,\zeta}$}{SI}}

We state in the next lemma another representation of the rate function $S^{I}_{\nu,\zeta}$.
This representation is not used in the proof of Theorem~\ref{thm::LDPempP::Inde}.
As explained in Remark~\ref{rem::LDPempP::indep::UpperB::Different} we could use it to show an upper bound on $S^{I}$. 
Nevertheless, we prove this lemma here, because we need it in Chapter~\ref{sec::LDPempP::Inter} when showing that the rate function of the interacting system is actually lower semi-continuous

\begin{lemma}[Compare with  \cite{DawGarLDP}  Lemma~4.8 for the mean field case]
\label{lem::LDPempP::indep::3Rep}
	Take a $\nu \in \MOneL{\TWR}$ and a $\mu \in \Cem$.
	Then
	\begin{align}
	\label{eq::lem::LDPempP::indep::3Rep::SI}
		S^{I}_{\nu,\zeta} \left( \mu_{\Ti}  \right)  
		=
		\relE{\mu_{0}}{\dd x \otimes  \zeta_{x} \otimes \nu_{x}}
		+ 
		\sup_{f \in \CspL[1,0,2]{c}{ \TTWR  }} I \left( \mu_{\Ti} , f \right) 
	\end{align}
where
\begin{align}
\label{eq::lem::LDPempP::indep::3Rep::I}
\begin{split}
	I \left( \mu_{\Ti} ,f \right)  
		&= 
			\int_{\TWR} f \left( T ,x, \w, \theta \right) \mu_{T} \left( \dd x, \dd \w, \dd\theta \right)
			-
			\int_{\TWR} f \left( 0 ,x, \w, \theta \right) \mu_{0} \left( \dd x, \dd \w, \dd\theta \right)
			\\
			&-  
				\int_{0}^{T} 
				\int_{\TWR} 
					\left(\frac{\partial}{\partial t}+ \Gen^{I}_{t,x, \w}\right)  f \left( t,x,\theta \right)  
										- \frac{\sigma^{2}}{2} \left(\partial_{\theta} f \left( t,x,\theta \right) \right) ^{2}
					\mu_{t} \left( \dd x, \dd \w, \dd\theta \right)	
		\dd t
	.
\end{split}
\end{align}

\end{lemma}

\begin{proof}
Most parts of this proof are almost equal (modulo additional integrals with respect to $\Td$ and $\Wsp$) to the proof of Lemma~4.8 in  \cite{DawGarLDP}.
Therefore we only state the ideas and point out where things have to be changed due to the space and random environment dependency.

Fix a $\mu_{\Ti}  \in \Cem$ with 
$\relE{\mu_{0} }{\nu}  < \infty$.

\emptyline
\begin{steps}
\step

We define for $f \in \CspL[1,0,2]{c}{ \TTWR  }$
\begin{align}
\label{pf::lem::LDPempP::indep::3Rep::defEll}
\begin{split}
	\ell_{s,t} \left( f \right)
	&= 
				\int_{\TWR} f \left( t ,x, \w, \theta \right) \mu_{t} \left( \dd x, \dd \w, \dd\theta \right)
				-
				\int_{\TWR} f \left( s ,x, \w, \theta \right) \mu_{s} \left( \dd x, \dd \w, \dd\theta \right)
				\\
				&-  
					\int_{s}^{t} 
					\int_{\TWR} 
						\left(\partial_{u}+ \Gen^{I}_{u,x, \w}\right)  f \left( u,x, \w, \theta \right)  
					\mu_{u} \left( \dd x, \dd \w, \dd\theta \right)	
				\dd t
	.
\end{split}
\end{align}
Note that this is equal to $I\left( \mu,f \right)$ without the $\left(\partial_{\theta} f \left( t,.,.,. \right) \right) ^{2}$ part and with the restriction to the time interval $[s,t]$.
Analogue to (4.26) of \cite{DawGarLDP},   we can prove that
\begin{align}
	\abs{ \ell_{s,t}  \left( f \right) }^{2} 
		\leq 
		\int_{0}^{T} 
			\int_{\TWR} \mkern-18mu
					\sigma^{2} \left( \partial_{\theta} f \left( t,x, \w, \theta \right) \right) ^{2}
				\mu_{t} \left( \dd x, \dd \w, \dd \theta\right)
			\dd t
				 \sup_{g \in \CspL[1,0,2]{c}{ \TTWR  }} I \left( \mu_{\Ti} ,g \right)   
		.
\end{align}

\step
As in the second step in \cite{DawGarLDP} we can show that for each $g \in \CspL[1,0,2]{c}{ \TTWR  }$
\begin{align}
	I \left( \mu_{\Ti} ,g \right)  
	\leq 
		S^{I}_{\nu,\zeta} \left( \mu_{\Ti}  \right) 
		-
		\relE{\mu_{0}}{\dd x \otimes  \zeta_{x} \otimes \nu_{x}}
\end{align}
by applying the integration by parts Lemma~\ref{def::Distr::IntbyParts}.

\emptyline
\step
\label{pf::lem::LDPempP::indep::3Rep::Step::L2}
We may assume that $ \sup_{g \in \CspL[1,0,2]{c}{ \TTWR  }} I \left( \mu_{\Ti} ,g \right)  < \infty$.
Denote by $\what{L}^{2}_{\mu_{\Ti}} \left( s,t \right) $ the Hilbert space of all measurable maps $h: [s.t] \times \TWR \rightarrow \R$, with finite norm 
\begin{align}
	\abs{ h }_{\mu_{\Ti}} 
			\defeq 
				\int_{s}^{t} 
				\int_{\TWR} 
					\frac{\sigma^{2}}{2} \left( h \left( u,x, \w, \theta \right)\right) ^{2} 
				\mu_{u} \left( \dd x, \dd \w, \dd \theta\right)
				\dd u 
		.
\end{align}
Moreover let $L^{2}_{\mu_{\Ti}} \left( s,t \right) $ be the closure in $\what{L}^{2}_{\mu_{\Ti}} \left( s,t \right) $ of the subset consisting of the maps $\left( t, x, \theta \right) \mapsto \partial_\theta h \left( t, x, \theta \right)$ with $h \in \CspL[1,0,2]{c}{\left[ s, t \right] \times \TWR}$.

Similar as in the third step of the proof in \cite{DawGarLDP} (but now with the additional dependency on the space $\Td$), we can use this space to prove that there is a $h^{\mu_{\Ti} }  \in \what{L}^{2}_{\mu_{\Ti}} \left( s,t \right) $, such that
\begin{align}
\label{eq::lem::LDPempP::indep::3Rep::weakPDE}
	\ell_{0,t} \left( f \right)  
	= 
		\int_{0}^{t}
			\int_{\TWR} 
					\sigma^{2} h^{\mu_{\Ti} } \left( u,x, \w, \theta \right)  \partial_{\theta} f \left( u,x, \w, \theta \right)   
			\mu_{u} \left( \dd x, \dd \w, \dd \theta\right)
			\dd u
	.
\end{align}
The existence of such an $h^{\mu_{\Ti} }$, origins from applying the Riesz representation theorem for $\ell$.
Then the same arguments as in \cite{DawGarLDP} lead to 
\begin{align}
\label{eq::lem::LDPempP::indep::3Rep::supIeqH}
	\sup_{f \in \CspL[1,0,2]{c}{ \TTWR  }} I \left( \mu_{\Ti} ,f \right)  
	=
		\frac{1}{2}
		\int_{0}^{T} 
			\int_{\TWR} 
					\frac{\sigma^{2}}{2} \left( h^{\mu_{\Ti}} \left( t,x, \w, \theta \right) \right) ^{2} 
		\mu_{t} \left( \dd x, \dd \w, \dd \theta\right)
		\dd t
	.
\end{align}

\step
In this last part, one uses the right hand side of \eqref{eq::lem::LDPempP::indep::3Rep::supIeqH}  to show the equation \eqref{eq::lem::LDPempP::indep::3Rep::SI}.
This follows again from the same arguments as in \cite{DawGarLDP}, by showing that $\mu_{\Ti} $ is absolutely continuous as a map from $\Ti \rightarrow \D'$ and finally by applying the Lemma~\ref{def::Distr::Derivatives}.
\end{steps}\vspace{-\baselineskip}
\end{proof}

\subsection{From independent to interacting spins}
\label{sec::LDPempP::Inter}

In this chapter, we finish the proof of Theorem~\ref{thm::LDPempP} by generalising the proofs given in Chapter~5 of~\cite{DawGarLDP}.
As explained subsequent to the Theorem~\ref{thm::LDPempP}, we use the following local version of an LDP (Theorem~\ref{thm::LDPempP::Local}) and
exponential tightness result  (Theorem~\ref{thm::LDPempP::ExpBound}), to prove Theorem~\ref{thm::LDPempP}.

\begin{theorem}[compare to Theorem~5.2 in \cite{DawGarLDP} for the mean field version]
\label{thm::LDPempP::Local}
	Under the assumptions of  Theorem~\ref{thm::LDPempP}, the following statements are true
	for fixed $\ol\mu_{\Ti}  \in \Cem$.
	\begin{enuRom}
	\item
	\label{thm::LDPempP::Local::LowerBound}
	For all open neighbourhoods $V \subset \Cem$ of $\ol\mu_{\Ti}  $
		\begin{align}
		\label{thm::LDPempP::Local::eq::LowerBound}
			\liminf_{N \rightarrow \infty} N^{-d} 
												\log 
															P^{N}\left[ \empP \in V \right]
			\geq
			- S_{\nu,\zeta} \left( \ol\mu_{\Ti}  \right) 
			.
		\end{align}
	\item
	\label{thm::LDPempP::Local::UpperBound}
		For each $\gamma>0$, there is an open neighbourhood $V \subset \Cem$ of $\ol\mu_{\Ti}  $
		such that
		\begin{align}
		\label{thm::LDPempP::Local::eq::UpperBound}
			\limsup_{N \rightarrow \infty} N^{-d} \log 
								P^{N}
										\left[ \empP \in V \right]
			\leq
					\begin{cases}
							- S_{\nu,\zeta} \left( \ol\mu_{\Ti}  \right) + \gamma
							\quad&\textnormal{if } S_{\nu,\zeta} \left( \ol\mu_{\Ti}  \right) < \infty
							\\
							- \gamma
							&\textnormal{otherwise.}
					\end{cases}
		\end{align}
	\end{enuRom}
\end{theorem}

\begin{theorem}[compare to Theorem~5.3 in \cite{DawGarLDP} for the mean field version]
	\label{thm::LDPempP::ExpBound}
	Under the assumptions of  Theorem~\ref{thm::LDPempP}, 
	there is, for all $s>0$, a compact set $\calK_{s} \subset \Cem$, with $\calK_{s} \subset \Cem_{\varphi,R}$  for a $R$ large enough, such that
		\begin{align}
	\label{thm::LDPempP::ExpBound::eq::Bound+Initial}
		\limsup_{N \rightarrow \infty} N^{-d} \log 
					P^{N}\left[ \empP \in \Cem \backslash \calK_{s} \right]
		\leq 
				-s
		.
	\end{align}
\end{theorem}

We state the proofs of these two Theorems in  Chapter~\ref{sec::LDPempP::Inter::PfLocalLDP}. and Chapter~\ref{sec::LDPempP::Inter::PfExpTight}.

Before we infer from these results the Theorem~\ref{thm::LDPempP}, let us briefly state the idea of the proofs of these two theorems and explain how the rest of this chapter is organised.

\begin{enuBr}
\item
In Chapter~\ref{sec::LDPempP::Inter::Prel}, we show some preliminary lemmas.
At first (in Chapter~\ref{sec::LDPempP::Inter::AssInterToInde}) we show that the operator $\Gen^{I}_{t,x, \w} = \Gen_{\ol{\mu}_{t},x, \w}$ satisfies the assumptions of Chapter~\ref{sec::LDPempP::Inde}, for all $\ol\mu_{\Ti} \in \Cem_{\varphi} \cap \Cem^{L}$. This implies the validity of the results of Chapter~\ref{sec::LDPempP::Inde} for the independent system with fixed effective field $\ol{\mu}_{\Ti}$.

Then we show (in Chapter~\ref{sec::LDPempP::Inter::InCemVar}), that $\empP$ is in  $\Cem_{\varphi,\infty}$ almost surely under $P^{N}$, for all $N\in \N$.

Finally (Chapter~\ref{sec::LDPempP::Inter::ExpBounds}), we derive exponential small bounds for $P^{N}$. For example we show that the probability of being outside of $\Cem_{\varphi,R}$ is exponentially small.
The proofs of these result are for fixed initial data formally the same as the proofs in \cite{DawGarLDP}, at least after applying the result of Chapter~\ref{sec::LDPempP::Inter::InCemVar}.
However due to the different initial distribution, some new estimates are required.
Here we need the Assumption~\ref{ass::LDPempP::Init::Integ}.

\item
Next we prove in Chapter~\ref{sec::LDPempP::Inter::PfExpTight} the Theorem~\ref{thm::LDPempP::ExpBound}, by combining in a suitable way the exponential bounds.
The approach of this proof does not differ from the corresponding proof in \cite{DawGarLDP}.

\item
Finally we prove Theorem~\ref{thm::LDPempP::Local} in Chapter~\ref{sec::LDPempP::Inter::PfLocalLDP}.
Here we separate the proof in the cases when $\ol\mu_{\Ti}$ is in $\Cem_{\varphi,\infty}$, in $\Cem^{L}$ and when it is not in these sets.
The part of the proof when  $\ol\mu_{\Ti}$ is in  $\Cem_{\varphi,\infty} \cap \Cem^{L}$ is formally similar to the proof in \cite{DawGarLDP}.
We use, in this part, the exponential bounds derived in Chapter~\ref{sec::LDPempP::Inter::ExpBounds} as well as the large deviation principle for independent spins (derived in Chapter~\ref{sec::LDPempP::Inde}).
The other case, i.e. when $\ol\mu_{\Ti}$ not in $\Cem^{L}$ or not in $\Cem_{\varphi,\infty}$, are new here.
When $\ol\mu_{\Ti} \not \in \Cem^{L}$, we show that in a small neighbourhood around $\ol\mu_{\Ti}$, there is no empirical process for $N$ large enough. From this we conclude the local large deviation result.
For the case that $\ol\mu_{\Ti}$  is not in $\Cem_{\varphi,\infty}$, we infer the local large deviation result from the exponential bounds.

At the beginning of Chapter~\ref{sec::LDPempP::Inter::PfLocalLDP}, we explain the proof in more details.

\end{enuBr}

\begin{remark}
All the results of this section can be transferred to hold also on $\Cem_{\varphi, \infty}$ with the stronger topology  considered in  \cite{DawGarLDP}.
The proofs would formally be the same. 
\end{remark}

\begin{proof}[of  Theorem~\ref{thm::LDPempP}]
	This proof of Theorem~\ref{thm::LDPempP}  is similar to the proof of the corresponding mean field theorem in  \cite{DawGarLDP}. 
	Despite these similarities we state the proof here, because it illustrates how the Theorem~\ref{thm::LDPempP::Local} and Theorem~\ref{thm::LDPempP::ExpBound} are applied.
	Differences to \cite{DawGarLDP} arise only in the proof that $S_{\nu,\zeta}$ is a good rate function.
	This is mainly due to the space and random environment dependency and because the spin values do not start at fixed positions (as considered in \cite{DawGarLDP}), but are distributed according to $\nu$.

	\emptyline
	
	\begin{enuRomNoIntendBf}
	
	\itemAdd{The large deviation lower bound}
		
		Let $G \subset \Cem$ be a open set.
		The large deviation lower bound follows directly by applying  Theorem~\ref{thm::LDPempP::Local}~\ref{thm::LDPempP::Local::LowerBound} with $V=G$ for all $\mu_{\Ti}  \in G$.
		
	\itemAdd{The large deviation upper bound}
	
		Let $F \subset \Cem$ be a closed set.
		We assume that $\inf_{\mu \in F} S_{\nu,\zeta} \left( \mu \right) =\overline{s}<\infty$.
		The case when the infimum is not finite can be treated similarly.
		
		By Theorem~\ref{thm::LDPempP::ExpBound} we know that there is compact set $\calK \subset \Cem$ such that \eqref{thm::LDPempP::ExpBound::eq::Bound+Initial} is satisfied with $s=\overline{s}$.
		We further know by Theorem~\ref{thm::LDPempP::Local}~\ref{thm::LDPempP::Local::UpperBound}  that for a fixed $\gamma>0$ and for each $\mu_{\Ti}  \in F \cap \calK$, there is an open neighbourhood $V_{\mu_{\Ti} }$ of $\mu_{\Ti}$
	such that \eqref{thm::LDPempP::Local::eq::UpperBound} is satisfied for $\mu_{\Ti} $. 
	Because $F \cap \calK$ is compact, it is covered by a finite number of these neighbourhoods.
		Combining these results we get
	\begin{align}
	\begin{split}
		&\limsup_{N \rightarrow \infty} N^{-d} \log 
						P^{N}
								\left[ \empP \in F  \right]
		\\
		&\leq
			\max
			\left\{
				\limsup_{N \rightarrow \infty} N^{-d} \log 
							P^{N}\left[ \empP \in F \cap \calK \right]
			,
				\limsup_{N \rightarrow \infty} N^{-d} \log 
					 P^{N}\left[ \empP \not \in \calK \right]
			\right\}
		\\
		&\leq
			 -\overline{s} + \gamma
		.
		\end{split}
	\end{align}
	Because the parameter $\gamma$ is arbitrary, 
	this proves the large deviation upper bound.
		
		\emptyline
	\itemAdd{$S_{\nu,\zeta}$ is a good rate function}
	
		To show that $S_{\nu,\zeta}$ is a good rate function, we have to show that the level sets
		\begin{align}
					\calL^{\leq s} \left( S_{\nu,\zeta} \right)  \defeq \left\{ \mu_{\Ti}  \in \Cem : S_{\nu,\zeta} \left( \mu_{\Ti}  \right)  \leq s  \right\}
				\end{align}
			are compact in $\Cem$, for each $s \geq 0$.
		We show at first that the level set $\calL^{\leq s} \left( S_{\nu,\zeta} \right)  $ is relatively compact and then that it is closed.
		\begin{steps}
		\step[{$\calL^{\leq s} \left( S_{\nu,\zeta} \right)  $ is relatively compact}]
		
		By Theorem~\ref{thm::LDPempP::ExpBound} we know that there is a compact set $\calK_{s+\epsilon} \subset \Cem_{\varphi,R} \subset \Cem$, for $R>0$ large enough, such that \eqref{thm::LDPempP::ExpBound::eq::Bound+Initial} holds for $s+\epsilon$.
		We claim that $\calL^{\leq s} \left( S_{\nu,\zeta} \right)  \subset \calK_{s+\epsilon}$.
		Let us assume that there is a $\mu_{\Ti} \in \calL^{\leq s} \left( S_{\nu,\zeta} \right) $ that is not in $\calK_{s+\epsilon}$.
		Then we know by \eqref{thm::LDPempP::ExpBound::eq::Bound+Initial}  and
		Theorem~\ref{thm::LDPempP::Local}~\ref{thm::LDPempP::Local::LowerBound} (because $ \Cem  \backslash \calK_{s+\epsilon}$ is an open neighbourhood of $\mu_{\Ti}$), 
		that $s+\epsilon \leq S_{\nu,\zeta} \left( \mu_{\Ti} \right)$, a contradiction.
		
		\emptyline		

		\step[{$\calL^{\leq s} \left( S_{\nu,\zeta} \right)  $ is closed}]

		Let $I \left( \mu_{\Ti} ,f \right) $ be defined as in \eqref{eq::lem::LDPempP::indep::3Rep::I}.
		By Lemma~\ref{lem::LDPempP::indep::3Rep} we know that 
		\begin{align}
		S_{\nu,\zeta} \left( \mu_{\Ti}  \right)  = 
				\relE{\mu_{0} }{\dd x \otimes \zeta_{x} \otimes \nu_{x}} 
				+
				\sup_{f \in \CspL[1,0,2]{c}{\TTR}} I^{\Gen_{\mu_{\Ti},.,.}} \left( \mu_{\Ti} ,f \right)
		.
		\end{align}
		Moreover we know by the previous step and the definition of $S_{\nu,\zeta}$ that 
		$\calL^{\leq s} \left( S_{\nu,\zeta} \right)  
						\subset 
							\Cem_{\varphi,R}  \cap \Cem^{L}$,
		for a $R$ large enough.
		Therefore $\calL^{\leq s} \left( S_{\nu,\zeta} \right)   = \bigcap_{f \in \CspL[1,0,2]{c}{\TTR}} \calL^{\leq s}_{f,R} \left( S_{\nu,\zeta} \right)  $ with
		\begin{align}
			\calL^{\leq s}_{f,R} \left( S_{\nu,\zeta} \right)
			\defeq 
				\left\{ \mu_{\Ti} \in \Cem_{\varphi,R}  \cap \Cem^{L}
					: 
						I^{\Gen_{\mu_{\Ti} , . , . }} \left( \mu_{\Ti} ,f \right) 
						+
						\relE{\mu_{0}}{\dd x \otimes \zeta_{x} \otimes \nu_{x}}
						\leq 
						s
				\right\}
			.
		\end{align}
		It is hence enough to show that  $\calL^{\leq s}_{f,R} \left( S_{\nu,\zeta} \right)$ is closed for each $f \in \CspL[1,0,2]{c}{\TTWR}$.
		\emptyline
		
		The map  $\mu_{\Ti}  \mapsto I^{\Gen_{\mu_{\Ti} ,. ,.}} \left( \mu_{\Ti} ,f \right) $ is continuous as a function $\Cem_{\varphi,R}  \cap \Cem^{L} \rightarrow \R$ for all $R \in \R_{+}$ and for all $f \in \CspL[1,0,2]{c}{\TTWR}$.
		This follows from Assumption~\ref{ass::LDPempP}~\ref{ass::LDPempP::MuIntCont}.
		Moreover $\mu^{(n)} \rightarrow \mu$ implies that $\mu^{(n)}_{0} \rightarrow \mu_{0}$, and $\mu_{0} \mapsto \relE{\mu_{0}}{\dd x \otimes \zeta_{x} \otimes \nu_{x}} $ is lower semi continuous.
		From the continuity of $\mu_{\Ti}  \mapsto I^{\Gen_{\mu_{\Ti} ,. ,.}} \left( \mu_{\Ti} ,f \right) $ and the lower semi continuity of $\relE{.}{\dd x \otimes \zeta_{x} \otimes \nu_{x}} $ we infer, that the set $\calL^{\leq s}_{f,R} \left( S_{\nu,\zeta} \right) $ is closed in $\Cem_{\varphi,R}  \cap \Cem^{L}$.
		 Due to $\Cem_{\varphi,R}  \cap \Cem^{L}$ being closed in $\Cem$, this implies that $\calL^{\leq s}_{f,R} \left( S_{\nu,\zeta} \right) $ is also closed in  $\Cem$.
		
		\end{steps}
\end{enuRomNoIntendBf}\vspace{-\baselineskip}
\end{proof}

\subsubsection{Preliminaries}
\label{sec::LDPempP::Inter::Prel}

\paragraph{The assumptions of the corresponding independent systems are satisfied}
\label{sec::LDPempP::Inter::AssInterToInde}

Fix a $\ol\mu_{\Ti} \in \Cem_{\varphi, \infty} \cap \Cem^{L}$.
Define the function $b^{I} \left( t,x, \w, \theta \right)  \defeq b \left( x, \w, \theta,\ol{\mu}_{t} \right) $.
We show now that the Assumption~\ref{ass::LDPempP::Inde} is satisfied for the independent spin system (given by \eqref{eq::SDEIndep}) with this drift coefficient $b^{I}$, i.e. $\Gen^{I}_{t,x, \w} \defeq \Gen_{\ol{\mu}_{t},x, \w}$.

\begin{enuAlphNoIntendBf}
\item
\label{it::AssInterToInde::cont}
The Assumption~\ref{ass::LDPempP::Inde}~\ref{ass::LDPempP::Inde::bCont} is satisfied because of Assumption~\ref{ass::LDPempP}~\ref{ass::LDPempP::bCont} and $\overline{\mu}_{t} \in \M_{\varphi,R}$ for all $t \in \Ti$ and for a $R$ large enough.

\emptyline
\item
	We infer from Theorem~10.1.2 of \cite{StrVarMultidi}  the uniqueness of the Martingale problem for each tuple $\left( x, \w \right) \in \TW$, because  the drift coefficient is continuous (by \ref{it::AssInterToInde::cont}).
	To apply this theorem, let $G_{n}$ be a set with compact closure in $R^{N^{d}}$ and define a continuous and bounded function $b^{I,(n)}: \Ti \times \R$ to equal $b^{I} \left( ., x, . \right)$ on $G_{n}$.
	Then Theorem~7.2.1 of \cite{StrVarMultidi}  gives that for each $n$ the Martingale problem corresponding to $b^{I,(n)}$ is well defined. 
	To show the existence, we apply Theorem~10.2.1 of \cite{StrVarMultidi}.
	The conditions of this theorem are satisfied by Assumption~\ref{ass::LDPempP}~\ref{ass::LDPempP::Lyapu}, because $\Gen^{I}_{t,x, \w} = \Gen_{\ol{\mu}_{t},x, \w}$.

	Therefore the Martingale problem is well defined, i.e. Assumption~\ref{ass::LDPempP::Inde}~\ref{ass::LDPempP::Inde::MPro} is satisfied.

\end{enuAlphNoIntendBf}

\paragraph{The empirical process is with probability one in \texorpdfstring{$\Cem_{\varphi,\infty }$}{Cphi}}
\label{sec::LDPempP::Inter::InCemVar}

\begin{lemma}
\label{lem::LDPempP::Inter::PNOutCphi}
\begin{enuRom}
\item	
\label{lem::LDPempP::Inter::PNOutCphi::Nu}
	Let Assumption~\ref{ass::LDPempP::Init::Integ} hold.
	Then for all $N \in \N$, 
	\begin{align}
		\sup_{\ul{\w}^{N} \in \Wsp^{\Nd}}
				P^{N}_{\ul{\w}^{N}} \left[ \empP \in \Cem \backslash \Cem_{\varphi,\infty} \right] = 0
				.
	\end{align}
\item
\label{lem::LDPempP::Inter::PNOutCphi::MR}
	For any  $r>0$ and for all $N \in \N$, 
	\begin{align}
			\sup_{\ul{\w}^{N} \in \Wsp^{\Nd}}
		\sup_{\ul{\theta}^{N} \in \RN : \empMThe \in \M_{r,\varphi} } 
				P^{N}_{\ul{\w}^{N}, \ul{\theta}^{N}} \left[ \empP \in \Cem \backslash \Cem_{\varphi,\infty}  \right]
		=
		0
		,
	\end{align}
where $P^{N}_{\ul{\w}^{N}, \ul{\theta}^{N}} \in \MOne{ \RN}$ is defined as $P^{N}_{\ul{\w}^{N}}$ (see Notation~\ref{nota::LMF::General}) with fixed initial values $\ul{\theta}^{N}$.
\end{enuRom}
\end{lemma}

\begin{proof}
\begin{enuRomNoIntendBf}
\item
For all $R>0$ and $\ul{\w}^{N}  \in \Wsp^{\Nd}$
	\begin{align}
	\label{eq::pf::lem::LDPempP::Inter::PNOutCphi::TowardsR}
	\begin{split}
				P^{N}_{\ul{\w}^{N}} \left[ \empP \in \Cem \backslash \Cem_{\varphi,\infty} \right]
				&\leq
					P^{N}_{\ul{\w}^{N}} \left[ \ul{\theta}^{N}_{\Ti} :  \sup_{t \in \Ti} \frac{1}{\Nd} \sum_{k \in \TN} \varphi \left( \theta^{k,N}_{t} \right) > R  \right]
				\\
				&=
					P^{N}_{\ul{\w}^{N}} \left[ \ul{\theta}^{N}_{\Ti} : 
									\sup_{t \in \Ti} \log \left( 1+ \frac{1}{\Nd}  \sum_{k \in \TN} \varphi \left( \theta^{k,N}_{t} \right)  \right) > \log \left( R +1 \right)  \right]
				.
	\end{split}
	\end{align}
	
	We want to show that the right hand side converge to zero when $R$ tends to infinity.
	To do this, we use an approach that is for example used in the proof of Theorem~1.5 in \cite{GarOnTheMKVLimit} and apply it to the setting we consider here.
	
	Fix $\ul{\w}^{N} \in \Wsp^{\Nd}$.
	Applying \Ito's lemma to $h \left( \ul{\theta}^{N}_{ t } \right) \defeq \log \left( 1+ \frac{1}{\Nd}  \sum_{k \in \TN} \varphi \left( \theta^{k,N}_{t} \right)  \right)$, we get
	\begin{align}
	\begin{split}
			h \left( \ul{\theta}^{N}_{ t } \right)
			&\leq
				h \left( \ul{\theta}_{0} \right)
			+
			\int_{0}^{t}
				\left( 1+ \frac{1}{\Nd}  \sum_{k \in \TN} \varphi \left( \theta^{k,N}_{s} \right)  \right)^{-1}
					\mkern-9mu
					\int_{\TWR} \mkern-9mu \Gen_{\empPt{s},x, \w}\varphi \left(\theta\right) \empPt{s} \left( \dd x, \dd \w, \dd \theta \right)
			\dd s
			+
			M_{t}
			\\
			&\leq
			h \left( \ul{\theta}_{0} \right)
			+
			T
			+
			M_{t}
			,
	\end{split}
	\end{align}
	by Assumption~\ref{ass::LDPempP}~\ref{ass::LDPempP::IntBound}, where $\empPt{s}$ is the empirical measure defined by $\ul{\w}^{N}$ and $\ul{\theta}^{N}$.
	The $M_{t}$ is a continuous local $P^{N}_{\ul{\w}^{N}}$ martingale with $M_{0} =0$.
	Define the non negative $P^{N}_{\ul{\w}^{N}}$  supermartingale
	\begin{align}
			S^{R}_{t} 
			\defeq
				\min \left\{ 
						h \left( \ul{\theta}_{0}\right)
									+
									T
									+
									M_{t}
						,
						\log \left( R \right)
						\right\}
				.
	\end{align}
	By the Doob supermartingale inequality
	\begin{align}
	\label{eq::pf::lem::LDPempP::Inter::PNOutCphi::Doob}
	\begin{split}
			P^{N}_{\ul{\w}^{N}} \left[  	\sup_{t \in \Ti} h \left( \ul{\theta}^{N}_{ t } \right) > \log \left( R +1 \right)  \right]
			&\leq
				P^{N}_{\ul{\w}^{N}} \left[ \sup_{t \in \Ti} S^{R}_{t}> \log \left( R +1 \right)  \right]
			\leq
				\frac{1}{\log \left( R + 1 \right)} E_{P^{N}_{\ul{\w}^{N}}} \left[ S^{R}_{0} \right]
			\\
			&\leq
				\left( \log \left( R + 1 \right) \right)^{-\oh} 
				+
				\nu^{N} \left[   h \left( \ul\theta \right) > \left(  \log \left( R +1 \right)  \right)^{\oh}  - T \right]
			.
	\end{split}
	\end{align}
	To bound the probability, we apply the Chebychev inequality,
	\begin{align}
	\begin{split}
			\nu^{N} \left[   h \left( \ul\theta \right) > \left(  \log \left( R +1 \right)  \right)^{\oh} -T   \right]
			\leq
				e^{-\kappa \Nd \left( e^{\left( \log \left( R +1 \right)  \right)^{\oh} -T} - 1 \right) }
				\prod_{i \in \TN}
				\int_{\R} e^{\kappa \varphi \left( \theta \right) }  \nu_{\frac{i}{N}} \left( \dd \theta \right)
		.
	\end{split}
	\end{align}
	By Assumption~\ref{ass::LDPempP::Init::Integ}, the integral is bounded by a constant.
	Therefore the right hand side of \eqref{eq::pf::lem::LDPempP::Inter::PNOutCphi::Doob} converges to zero uniformly for all $\ul{\w}^{N}$,
	when $R$ tends to infinity.
	Combining this with \eqref{eq::pf::lem::LDPempP::Inter::PNOutCphi::TowardsR}, implies \ref{lem::LDPempP::Inter::PNOutCphi::Nu}.

\item
We get by the same arguments as in \ref{lem::LDPempP::Inter::PNOutCphi::Nu} (\eqref{eq::pf::lem::LDPempP::Inter::PNOutCphi::TowardsR} to \eqref{eq::pf::lem::LDPempP::Inter::PNOutCphi::Doob})
	\begin{align}
		\sup_{\ul{\theta}^{N} \in \RN : \empMThe \in \M_{r,\varphi} } 
				P^{N}_{\ul{\w}^{N},\ul{\theta}^{N}} \left[ \empP \in \Cem \backslash \Cem_{\varphi,\infty}  \right]
		\leq
				\left( \log \left( R+1 \right) \right)^{-\oh} 
		,
	\end{align}
for all $R>0$ large enough, when $r$ is fixed.
\end{enuRomNoIntendBf}\vspace{-\baselineskip}
\end{proof}

\paragraph{Exponential bounds}
\label{sec::LDPempP::Inter::ExpBounds}

In the next two lemmas we show that it is exponentially unlikely that an empirical process leaves the sets $\Cem_{\varphi,R}$.
At first we show it uniformly for fixed initial conditions in $\M_{r,\varphi}$ (Lemma~\ref{lem::LDPempP::Inter::OutCemR}), then for initial conditions distributed according to $\nu$ (Lemma~\ref{lem::LDPempP::Inter::OutCemR+Initial}).

\begin{lemma}[compare to Lemma~5.5 in \cite{DawGarLDP} for the mean field case]
\label{lem::LDPempP::Inter::OutCemR}
	For any  $r>0$, $R>0$ and for all $N \in \N$, 
	\begin{align}
		\sup_{\ul{\w}^{N} \in \Wsp^{\Nd}}
		\sup_{\ul{\theta}^{N} \in \RN : \empMThe \in \M_{r,\varphi} } 
				P^{N}_{\ul{\w}^{N}, \ul{\theta}^{N}} \left[ \empP \in \Cem \backslash \Cem_{\varphi,R} \right]
		\leq 
		e^{-\Nd R_{T}}
		,
	\end{align}
with $R_{T} = R e^{-\lambda T}-r$, where $\lambda$ is defined in Assumption~\ref{ass::LDPempP}~\ref{ass::LDPempP::IntBound}.
	
\end{lemma}

\begin{proof}
	First note that by Lemma~\ref{lem::LDPempP::Inter::PNOutCphi}~\ref{lem::LDPempP::Inter::PNOutCphi::MR}, it is enough to show for each $\ul{\w}^{N} \in \Wsp^{\Nd}$
	\begin{align}
			\sup_{\ul{\theta}^{N} \in \RN : \empM \in \M_{r,\varphi} } 
					P^{N}_{\ul{\w}^{N}, \ul{\theta}^{N}} \left[ \empMThe \in \Cem_{\varphi,\infty} \backslash \Cem_{\varphi,R} \right]
			\leq 
			e^{-\Nd R_{T}}
			.
		\end{align}
	This bound can be proven (at least formally) exactly as the proof of Lemma~5.5 in \cite{DawGarLDP}. Therefore we do not state it here. 
	Neither the different topology on $\Cem_{\varphi,\infty}$ considered in that paper nor the space dependency, is crucial in the proof.
	The proof requires the Assumption~\ref{ass::LDPempP}~\ref{ass::LDPempP::IntBound}.
\end{proof}

\begin{lemma}
\label{lem::LDPempP::Inter::OutCemR+Initial}
Let Assumption~\ref{ass::LDPempP::Init::Integ} hold.
	For all $s >0$, there is a $R=R_{s}>0$, such that for all $N \in \N$
	\begin{align}
		\sup_{\ul{\w}^{N} \in \Wsp^{\Nd}}
		P^{N}_{\ul{\w}^{N}} \left[  \empP \in  \Cem \backslash \Cem_{\varphi,R} \right]
		\leq 
		e^{-\Nd s}
		.
	\end{align}
\end{lemma}

\begin{proof}
	For all $R>0$, $\ul{\w}^{N} \in \Wsp^{\Nd}$
	\begin{align}
	\begin{split}
		P^{N}_{\ul{\w}^{N} } \left[  \empP \in  \Cem \backslash \Cem_{\varphi,R} \right]
		&=
			\int_{\RN} P^{N}_{\ul{\w}^{N}, \ul\theta^{N}} \left[ \empP \in \Cem \backslash \Cem_{\varphi,R} \right] \nu^{N} \left( \dd \ul\theta^{N} \right)
		\\
		&\leq
			\sum_{k=0}^{\infty} e^{-\Nd R e^{-\lambda T} + \Nd \left( k+1 \right)} \nu^{N} \left[ \M_{k+1,\varphi} \backslash \M_{k,\varphi} \right]
		,
	\end{split}
	\end{align}
where we use Lemma~\ref{lem::LDPempP::Inter::OutCemR} in the inequality.
For the probability of the right hand side we use the exponential Chebychev inequality with $\ell >1$
	\begin{align}
		\nu^{N} \left[ \M_{k+1,\varphi} \backslash \M_{k,\varphi} \right]
		\leq
			\nu^{N} \left[ \sum_{k \in \TN} \varphi \left( \theta^{k,N} \right) > \Nd k \right]
		\leq
			e^{- \ell \Nd k }\prod_{i \in \TN} \int_{\R} e^{\ell \varphi \left( \theta \right) } \nu_{\frac{i}{N}} \left( \dd \theta \right)
		\leq
			e^{- \ell \Nd k } C^{\Nd}
		,
	\end{align}
	by Assumption~\ref{ass::LDPempP::Init::Integ}.
	Then
	\begin{align}
	\begin{split}
		P^{N}_{\ul{\w}^{N} } \left[  \empP \in  \Cem \backslash \Cem_{\varphi,R} \right]
		&\leq
			C^{\Nd} e^{-\Nd R e^{-\lambda T} + \Nd}
			\sum_{k=0}^{\infty} e^{ \Nd k  \left( 1 - \ell \right) }
		\\
		&\leq
			C^{\Nd} e^{-\Nd R e^{-\lambda T} +\Nd} \frac{1}{1-e^{\Nd \left( 1- \ell \right)}}
		\leq
			e^{-\Nd R e^{-\lambda T} \oh}
		,
	\end{split}
	\end{align}
	for $R$ large enough.
\end{proof}

For the Theorem~\ref{thm::LDPempP::ExpBound}, we need compact subsets of $\Cem$.
These sets are characterised in the following lemma.

\begin{lemma}[Lemma~1.3 in \cite{GarOnTheMKVLimit}]
\label{lem::LDPempP::Inter::RelCompInCEM}
	Let $\left\{ f_{n} \right\}_{n}$ be a countable dense subset of $\CspL{c}{\TWR}$.
	A set $\calK$ is relatively compact in $\Cem$ 
	if and only if
	\begin{align}
	\label{lem::LDPempP::Inter::RelCompInCEM:eq:MainSet}
		\calK \subset \calK_{K} \cap \bigcap \calK_{n}
		,
	\end{align}
	with
	\begin{align}
		&\calK_{K} = \left\{ \mu_{\Ti}  \in \Cem : \mu_{t}  \in K \textnormal{ for all } t \in \Ti  \right\}
		\\
		&\calK_{n} = \left\{ \mu_{\Ti}  \in \Cem : 
						\left\{ t \mapsto \int_{\TWR} f_{n} \left(x, \w, \theta\right) \mu_{t} \left(\dd x, \dd \w, \dd \theta \right) \right\} 
						\in K_{n} \right\}
		,
	\end{align}
where $K \subset \MOne{\TWR}$ and $K_{n} \subset \Csp{\Ti}$ are compact.
\end{lemma}

For a proof of this lemma, see Lemma~1.3 in \cite{GarOnTheMKVLimit}. 

The next lemma states an exponential bound on the probability that the empirical process is outside of a subset of $\Cem$, that is defined via the projection to $\Csp{\Ti}$.
We use this set in Theorem~\ref{thm::LDPempP::ExpBound} as the set $\calK_{n}$, defined in Lemma~\ref{lem::LDPempP::Inter::RelCompInCEM} in the characterisation of relative compact subset of $\Cem$.

\begin{lemma}[compare to Lemma~5.6 in \cite{DawGarLDP} for the mean field case]
	\label{lem::LDPempP::Inter::CemRoutKf}
	For all $R>0, s>0$ and $f \in \CspL[\infty]{c}{\TWR} $, there exists a compact set $K \subset \Csp{\Ti}$, such that 
	for all $N \in \N$ 
	\begin{align}
	\sup_{\ul{\w}^{N} \in \Wsp^{\Nd}}
	\sup_{\ul{\theta}^{N} \in \RN}
	P^{N}_{\ul{\w}^{N}, \ul{\theta}^{N}} \left[  \empP \in  \Cem_{\varphi,R} \backslash \calK_{f} \right]
	\leq 
	e^{-\Nd s}
	,
	\end{align}
	with $\calK_{f} = \left\{ \mu_{\Ti}  \in \Cem :  \left\{ t \mapsto \int_{\TWR} f \left(x, \w, \theta\right) \mu_{t} \left(\dd x, \dd \w, \dd \theta \right) \right\} \in K \right\}$.
	
\end{lemma}

\begin{proof}
	Also this proof is formally exactly the proof of Lemma~5.6 in \cite{DawGarLDP} for each $\ul{\w}^{N} \in \Wsp^{\Nd}$. Indeed in the proof one only uses the  function $\left\{ t \mapsto \int_{\TWR} f \left(x, \w, \theta\right) \mu_{t} \left(\dd x, \dd \w, \dd \theta \right) \right\}$, which is (here as in \cite{DawGarLDP})  a function in $\Csp{\Ti ,\R}$ and one does not have to care about the structure within the integral.
	Moreover the topology of $\Cem$ is not relevant in the proof. 
	The proof requires the Assumption~\ref{ass::LDPempP}~\ref{ass::LDPempP::blockbd}.
\end{proof}

\subsubsection{Proof of Theorem~\ref{thm::LDPempP::ExpBound}}
\label{sec::LDPempP::Inter::PfExpTight}

\begin{proof}[of Theorem~\ref{thm::LDPempP::ExpBound}]
This proof equals the proof of Theorem~5.3 in \cite{DawGarLDP}, besides formal changes due to the space dependency.
The only generalisation is that we consider random initial data here. 

By Lemma~\ref{lem::LDPempP::Inter::RelCompInCEM} it is enough to define compact sets $K \subset \MOne{\TWR}$ and $K_{n} \subset \Csp{\Ti}$ to get a compact set in $\Cem$.
We set  $K=\M_{\varphi,R}$ and therefore $\calK_{K}=\Cem_{\varphi,R}$.
Moreover we choose by Lemma~\ref{lem::LDPempP::Inter::CemRoutKf} for each $n$ a $K_{n} \subset \Csp{\Ti}$, such that
\begin{align}
\label{eq::pf::thm::LDPempP::ExpBound::calKn}
		\sup_{\ul{\w}^{N} \in \Wsp^{\Nd}}
		\sup_{\ul\theta^{N} \in \RN} 
		P^{N}_{\ul{\w}^{N}, \ul\theta^{N}} \left[  \empP \in  \Cem_{\varphi,R} \backslash \calK_{n} \right]
		\leq 
		e^{- n \Nd s}
	.
\end{align}
Define the compact set $\calK \defeq \overline{ \Cem_{\varphi,R} \cap \bigcap \calK_{n}}$.
This is a subset of $\Cem_{\varphi,R}$, because $\Cem_{\varphi,R}$ is closed in $\Cem$.

\emptyline
By Lemma~\ref{lem::LDPempP::Inter::OutCemR+Initial} and \eqref{eq::pf::thm::LDPempP::ExpBound::calKn} we conclude for all $N\in \N$ and $R$ large enough
\begin{align}
\begin{split}
	 P^{N}\left[  \empP  \in \Cem \backslash \calK \right]
	&\leq
		 P^{N}\left[  \empP  \in \Cem \backslash  \Cem_{\varphi,R} \right]
		+
		\sum_{n=1}^{\infty}
		\:
		\sup_{\ul{\w}^{N} \in \Wsp^{\Nd}}
		\sup_{\ul\theta^{N} \in  \RN } 
		 P^{N}_{\ul{\w}^{N}, \ul\theta^{N}}  \left[  \empP  \in \Cem_{\varphi,R} \backslash  \calK_{n} \right]
	\\
	&\leq 
		e^{-\Nd s}
		+
		\sum_{n=1}^{\infty}
		e^{-n \Nd  s}
	.
\end{split}
\end{align}\vspace{-\baselineskip}
\end{proof}

 \subsubsection{Proof of Theorem~\ref{thm::LDPempP::Local}}
\label{sec::LDPempP::Inter::PfLocalLDP}

We prove in this chapter the Theorem~\ref{thm::LDPempP::Local}.
In the proof, we investigate separately the cases, when 
$\ol\mu_{\Ti} \in \Cem_{\varphi,\infty}$ (\ref{pf::thm::LDPempP::Local::step::InML} and \ref{pf::thm::LDPempP::Local::step::OutML}), and when it is not in this space (\ref{pf::thm::LDPempP::Local::step::NotCphi}).
Moreover we divide the first case in the subcases that $ \ol{\mu}_{t} \in  \Cem^{L}$ (\ref{pf::thm::LDPempP::Local::step::InML}), and when this is not true (\ref{pf::thm::LDPempP::Local::step::OutML}).
The ideas of the proofs of the three cases are as follows:
\begin{steps}[label=\textbf{Case \arabic{stepsi}: }]
\step

For $\ol\mu_{\Ti} \in \Cem_{\varphi,\infty} \cap  \Cem^{L}$, we reduce the claims of Theorem~\ref{thm::LDPempP::Local} to large deviation upper and lower bounds for a system of independent SDEs. 
For this independent system we know these large deviation bounds by Theorem~\ref{thm::LDPempP::Inde}.
To reduce the claims we choose at first (\ref{pf::thm::LDPempP::Local::step::DefIndeSys}), for each $N \in \N$, a system of spin values, that evolve mutually independent, with the constraint that their empirical process should be close to $\ol\mu_{\Ti}$ with high probability.
Therefore we choose the drift coefficient $\overline{b}^{I} \left( x, \w, \theta,t \right)  \defeq b \left( x, \w, \theta,\ol{\mu}_{t}  \right) $.
We regard the empirical process of interacting diffusions, in a small neighbourhood of $\ol\mu_{\Ti} $, as a small perturbation of the empirical process for the independent diffusions with drift coefficient $\overline{b}^{I}$. 
Then, in \ref{pf::thm::LDPempP::Local::step::Girsan}, we apply the  (Cameron-Martin-) Girsanov theorem and receive a density between the measures of the solution to the original SDE and the one of the SDE with drift coefficient $\overline{b}^{I}$.
Using this density, we reduce in \ref{pf::thm::LDPempP::Local::step::LowerBound} and \ref{pf::thm::LDPempP::Local::step::UpperBound} the claims of Theorem~\ref{thm::LDPempP::Local} to large deviation bounds for the independent system. We get these bounds by Theorem~\ref{thm::LDPempP::Inde}, which is applicable by Chapter~\ref{sec::LDPempP::Inter::AssInterToInde}.

The proof of this first case is very similar to the one in~\cite{DawGarLDP} in Chapter~5.4 for the mean-field setting.
However differences arise due to the space and random environment dependency.
Moreover we show a large deviation principle on the space $\Cem$ and not like in \cite{DawGarLDP} on $\Cem_{\varphi,\infty}$ equipped even with another topology than the subspace topology.

\emptyline
\step
If we assumed in Assumption~\ref{ass::LDPempP}~\ref{ass::LDPempP::bCont} that the continuity of $b$ holds on $\M_{\varphi,R}$ and not only on $\M_{\varphi,R} \cap \MOneL{\TWR}$, then we  could handle the case of $\ol\mu_{\Ti} \in \Cem_{\varphi,\infty}$ with $ \ol{\mu}_{t} \not \in \MOneL{\TWR}$ for some $t \in \Ti$, as in the previous step.
However to keep the assumption more general, we have to use a new approach.
Indeed we show that no empirical process is within an $\epsilon$-ball around $\ol\mu_{\Ti}$ for $N$ large enough.
From this we infer the claims of Theorem~\ref{thm::LDPempP::Local}.

\emptyline
\step
When $\ol\mu_{\Ti}$  is not in $\Cem_{\varphi,\infty}$, then the first statement of Theorem~\ref{thm::LDPempP::Local} is obviously satisfied and the second statement follows from Lemma~\ref{lem::LDPempP::Inter::OutCemR+Initial}.

\end{steps}

\begin{proof}
Fix an arbitrary $\ol\mu_{\Ti}  \in \Cem$.

\emptyline
\begin{steps}[label=\textbf{Case \arabic{stepsi}: },ref=Case \arabic{stepsi}]
\step[{$\ol{\mu}_{\Ti} \in \Cem_{\varphi,\infty} \cap  \Cem^{L}$}]
\label{pf::thm::LDPempP::Local::step::InML}

\begin{steps}
\step[{Definition of a system of diffusions with a fixed effective field}]
\label{pf::thm::LDPempP::Local::step::DefIndeSys}

We set $b^{I} \left( x, \w, \theta,t \right)  \defeq b \left( x, \w, \theta,\ol{\mu}_{t}  \right) $
and use this function as drift coefficient to define the time dependent diffusion generator $\Gen^{I}_{t,x,\w}$ (defined as in \eqref{def::GeneratorIndependent}).
Then $\Gen^{I}_{t,x,\w} = \Gen_{\ol{\mu}_{t},x,\w} $.
Moreover we define the measures $P^{I,N} \in \MOne{ \Csp{\Ti}^{\Nd} } $ as in Notation~\ref{nota::LDPempP::Inde::General}.

As shown in Chapter~\ref{sec::LDPempP::Inter::AssInterToInde}, the Assumption~\ref{ass::LDPempP} implies the Assumptions~\ref{ass::LDPempP::Inde} for the generator $\Gen^{I}_{t,x,\w}$.
Therefore the Theorem~\ref{thm::LDPempP::Inde}  is applicable for $P^{I,N}$.

\emptyline
\step[{Comparison of the two processes with help of the Girsanov theorem}]
\label{pf::thm::LDPempP::Local::step::Girsan}

We claim that for each $\ul{\w}^{N} \in \Wsp^{\Nd}$,
$P^{N}_{\ul{\w}^{N}}$ is absolutely continuous with respect to $P^{I,N}_{\ul{\w}^{N}}$, with Radon-Nikodym derivative
\begin{align}
	\label{eq::thm::LDPempP::Local::RadNikoInterToNonInter}
	\frac{ \dd P^{N}_{\ul{\w}^{N}} }{\dd P^{I,N}_{\ul{\w}^{N}} }
	=
	e^{M^{N}_{\ul{\w}^{N},T} \:\: - \frac{1}{2} \langle\langle M^{N}_{\ul{\w}^{N}} \rangle\rangle_{T}}
\end{align}
for all $\ul\theta^{N} \in \RN$.
Here $M^{N}_{\ul{\w}^{N},t}$ is a continuous local $P^{I,N}_{\ul{\w}^{N}}$ martingale with quadratic variation
\begin{align}
	\label{eq::thm::LDPempP::Local::QuadraticVariation}
	\langle\langle M^{N}_{\ul{\w}^{N}} \rangle\rangle_{t}  \left( \ul{\theta}^{N}_{\Ti}  \right) 
	=
	\Nd
	\int_{0}^{t}
		\int_{\TWR}
				\sigma^{2} \abs{b \left( x, \w, \theta,\empPt{u}  \right) -b \left( x, \w, \theta,\ol{\mu}_{u}  \right) }^{2}
				\empPt{u} \left(\dd x, \dd \w, \dd \theta \right) 
	\dd u
	,
\end{align}
where $\empPt{u}$ is the empirical measure defined by $\ul{\theta}^{N}_{u}$ and $\ul{\w}^{N}$.
This can be shown by a spatial localisation argument. 
The generators $\Gen^{N}_{.}$ and $\Gen^{I,N}_{.}$ only differ in their drift coefficients. 
The martingale problems corresponding to both generators are well defined.
Moreover $b^{N}$ (defined in Assumption~\ref{ass::LDPempP}~\ref{ass::LDPempP::blockbd}) and $b^{I}$ (as continuous function) are both locally bounded.
By spatial localisation (see \cite{StrVarMultidi} Theorem~10.1.1) it is hence enough to consider bounded drift coefficients.
For bounded drift coefficients, we know by \cite{StrVarMultidi} Theorem~6.4.2 the claimed representation of the Radon-Nikodym formula.

\emptyline
\step[{The proof of~\ref{thm::LDPempP::Local::LowerBound}}]
\label{pf::thm::LDPempP::Local::step::LowerBound}

For $ S_{\nu,\zeta} \left( \ol\mu_{\Ti}  \right) = \infty$, \ref{thm::LDPempP::Local::LowerBound} is obviously satisfied.
Therefore assume that $ S_{\nu,\zeta} \left( \ol\mu_{\Ti}  \right) < \infty$.
Fix an open neighbourhood $V \subset \Cem$ of $\ol\mu_{\Ti} $ and an arbitrary $\gamma>0$.

The Lemma~\ref{lem::LDPempP::Inter::OutCemR+Initial} can also be applied to $P^{I,N}$ instead of $P^{N}$  by Assumption~\ref{ass::LDPempP}~\ref{ass::LDPempP::Lyapu}.
This lemma then states (with $s = S_{\nu,\zeta} \left( \ol\mu_{\Ti} \right) + \gamma$), that there is a $R>0$ such that
\begin{align}
	\label{eq::thm::LDPempP::Local::CemPhiWithoutCemR::ForOverlineP}
	P^{I,N} \left[ \empP \in \Cem \backslash \Cem_{\varphi,R} \right]
			\leq
			e^{- \Nd S_{\nu,\zeta} \left( \ol\mu_{\Ti}  \right) }
			e^{- \Nd \gamma}
	.
\end{align}
Assume that this $R$ is so large that $\ol\mu_{\Ti} \in \Cem_{\varphi,R}$.
We choose now two constants $p,q >1$ with $\frac{1}{p}+\frac{1}{q} =1$ and a $\delta>0$ such that
\begin{align}
	\label{eq::thm::LDPempP::Local::defpqdelta}
		\frac{1}{2} \left( 1+ \frac{p}{q} \right) \delta 
		+
		p S_{\nu,\zeta} \left( \ol\mu_{\Ti}  \right)
	\leq
		S_{\nu,\zeta} \left( \ol\mu_{\Ti}  \right)  
		+
		\gamma
	.
\end{align}

\emptyline

By Assumption~\ref{ass::LDPempP}~\ref{ass::LDPempP::MuIntCont} and \eqref{eq::thm::LDPempP::Local::QuadraticVariation}, there is a open neighbourhood $W \subset \Cem$ of $\ol \mu_{\Ti}$ such that 
$W \cap \Cem_{\varphi,R} \subset  V$ and $\langle\langle M^{N}_{\ul{\w}^{N}} \rangle\rangle_{T} \left( \ul{\theta}^{N}_{\Ti} \right)  \leq \Nd \delta$ for $\ul{\w}^{N} \in \Wsp^{\Nd}$ and $\ul{\theta}^{N}_{\Ti} \in \Csp{\Ti}^{\Nd}$ when the corresponding empirical processes $\empP \in W \cap \Cem_{\varphi,R}$.
With the same arguments as Dawson and G\"artner we can show by using the Radon-Nikodym derivative \eqref{eq::thm::LDPempP::Local::RadNikoInterToNonInter}  that for each $\ul{\w}^{N} \in \Wsp$ 
\begin{align}
\label{eq::thm::LDPempP::Local::PNVomega}
	P^{N}_{\ul{\w}^{N}}  \left[ \empP \in V \right]
	\geq
		 P^{N}_{\ul{\w}^{N}} \left[ \empP \in W \cap \Cem_{\varphi,R} \right]
	\geq
		e^{-\frac{1}{2} \left( 1+ \frac{p}{q} \right) \delta \Nd}
		\left( P^{I,N}_{\ul{\w}^{N}} \left[ \empP \in W \cap \Cem_{\varphi,R}  \right] \right)^{p}
	.
\end{align}
We integrate \eqref{eq::thm::LDPempP::Local::PNVomega} with respect to $\zeta^{N}$ and apply the Jensen inequality,
\begin{align}
\label{eq::thm::LDPempP::Local::PNV}
P^{N}  \left[ \empP \in V \right]
\geq
e^{-\frac{1}{2} \left( 1+ \frac{p}{q} \right) \delta \Nd}
\left( P^{I,N} \left[ \empP \in W \cap \Cem_{\varphi,R}  \right] \right)^{p}
.
\end{align}
Moreover
\begin{align}
\label{eq::thm::LDPempP::Local::PNW}
	P^{I,N} \left[ \empP \in W \cap \Cem_{\varphi,R} \right]
	\geq
		P^{I,N} \left[ \empP \in W \right]
			\left( 
			1
			-
			e^{-\Nd \frac{\gamma}{2}}
			\right)
	,
\end{align}
for $N$ large enough.
Indeed \eqref{eq::thm::LDPempP::Local::PNW} holds,  by the triangle inequality and 
\begin{align}
\begin{split}
P^{I,N} \left[ \empP \not \in \Cem_{\varphi,R} \right]
\leq
e^{- \Nd S_{\nu,\zeta} \left( \ol\mu_{\Ti}  \right) } e^{- \Nd \gamma}
\leq
e^{- \Nd \frac{\gamma}{2}}	 P^{I,N} \left[ \empP \in W \right]
,
\end{split}
\end{align}
by \eqref{eq::thm::LDPempP::Local::CemPhiWithoutCemR::ForOverlineP} and because $W$ is an open set and $\left\{\empP, P^{I,N}\right\}$  satisfies a large deviation principle (Theorem~\ref{thm::LDPempP::Inde}).

Combine \eqref{eq::thm::LDPempP::Local::PNV} and \eqref{eq::thm::LDPempP::Local::PNW}, we get
\begin{align}
\label{eq::thm::LDPempP::Local::PNValmost}
\begin{split}
	&\liminf_{N \rightarrow \infty} N^{-d} \log P^{N} \left[ \empP \in V \right]
	\\
	&\geq
		- \frac{1}{2} \left( 1+ \frac{p}{q} \right) \delta
		+
		p \liminf_{N \rightarrow \infty} N^{-d} \log P^{I,N} \left[ \empP \in W \right]
	.
\end{split}
\end{align}
Finally we conclude by the large deviation principle of $\left\{\empP, P^{I,N}\right\}$ (Theorem~\ref{thm::LDPempP::Inde})
and \eqref{eq::thm::LDPempP::Local::defpqdelta} 
\begin{align}
\begin{split}
	\textnormal{\eqref{eq::thm::LDPempP::Local::PNValmost}}
	\geq
		-\frac{1}{2} \left( 1+ \frac{p}{q} \right) \delta
		-
		p S_{\nu,\zeta} \left( \ol\mu_{\Ti}  \right) 
	\geq
		- S_{\nu,\zeta} \left( \ol\mu_{\Ti}  \right)  - \gamma
	.
\end{split}
\end{align}
This inequality holds for all $\gamma>0$. Hence we have proven \ref{thm::LDPempP::Local::LowerBound} for this case.

\emptyline
\step[{The proof of~\ref{thm::LDPempP::Local::UpperBound}}]
\label{pf::thm::LDPempP::Local::step::UpperBound}

We assume $S_{\nu,\zeta} \left( \overline{\mu} \right) < \infty$. The case when it is not finite can be treated analogue.
Fix a $\gamma>0$.
Due to Lemma~\ref{lem::LDPempP::Inter::OutCemR+Initial} it is sufficient to find for $R>0$ large enough with $\ol{\mu}_{\Ti} \in \Cem_{\varphi,R}$, an open neighbourhood $V \subset \Cem$ of $\ol\mu_{\Ti}$ such that 
\begin{align}
	\label{eq::thm::LDPempP::Local::UB::RestrictToCR}
		\limsup_{N \rightarrow \infty} N^{-d} \log P^{N} \left[ \empP \in V \cap \Cem_{\varphi,R} \right]
			\leq
			- S_{\nu,\zeta} \left( \ol\mu_{\Ti}  \right)  + \gamma
		.
\end{align}
Fix again $p,q >1$ with $\frac{1}{p}+\frac{1}{q} =1$ and a $\delta>0$, such that
\begin{align}
	\label{eq::thm::LDPempP::Local::UB::defPQDelta}
	\frac{p-1}{2} \delta + \frac{1}{q} \left( -S_{\nu,\zeta} \left( \ol\mu_{\Ti}  \right)  + \frac{\gamma}{2} \right)
	\leq
		- S_{\nu,\zeta} \left( \ol\mu_{\Ti}  \right)  + \gamma
	.
\end{align}

By Assumption~\ref{ass::LDPempP}~\ref{ass::LDPempP::MuIntCont} and \eqref{eq::thm::LDPempP::Local::QuadraticVariation}
and by Theorem~\ref{thm::LDPempP::Inde},
there is a small open neighbourhood $V \subset \Cem$ of $\ol{\mu}_{\Ti}$, such that
$\langle\langle M^{N}_{\ul{\w}^{N}} \rangle\rangle_{T} \left( \ul{\theta}^{N}_{\Ti} \right)  \leq \Nd \delta$ 
for $\ul{\w}^{N} \in \Wsp^{\Nd}$ and $\ul{\theta}^{N}_{\Ti} \in \Csp{\Ti}^{\Nd}$ when the corresponding empirical processes $\empP \in V \cap \Cem_{\varphi,R}$,
and such that
\begin{align}
	\label{eq::thm::LDPempP::Local::UB::IndepLDP}
	\limsup_{N \rightarrow \infty} N^{-d} \log P^{I,N} \left[ \empP \in V \right]
			\leq
			- S_{\nu,\zeta} \left( \ol\mu_{\Ti}  \right)  + \frac{\gamma}{2}
	.
\end{align}
In the last inequality we use that $S^{I}_{\nu,\zeta}$ is lower semi-continuous and $S_{\nu,\zeta} \left( \ol\mu_{\Ti}  \right) = S^{I}_{\nu,\zeta} \left( \ol\mu_{\Ti}  \right)$
As in \cite{DawGarLDP}, we can show, by using the Radon-Nikodym derivative \eqref{eq::thm::LDPempP::Local::RadNikoInterToNonInter}, that for all $\ul{\w}^{N} \in \Wsp$
\begin{align}
\begin{split}
	P^{N}_{\ul{\w}^{N}} \left[ \empP \in V \cap \Cem_{\varphi,R} \right]
	\leq
		e^{ \frac{p-1}{2} \delta N}
		\left(  P^{I,N}_{\ul{\w}^{N}} \left[ \empP \in V \right]  \right)^{\frac{1}{q}}
	.
\end{split}
\end{align}
To conclude \eqref{eq::thm::LDPempP::Local::UB::RestrictToCR}, integrate both sides with respect to $\zeta^{N}$, apply the Jensen inequality 
and finally use \eqref{eq::thm::LDPempP::Local::UB::defPQDelta} and \eqref{eq::thm::LDPempP::Local::UB::IndepLDP}.
Hence we showed \ref{thm::LDPempP::Local::UpperBound} for this case.
\end{steps}

\emptyline
\step[{$\ol{\mu}_{\Ti} \in \Cem_{\varphi,\infty}$ and $\ol{\mu}_{\Ti} \not \in  \Cem^{L}$}]
\label{pf::thm::LDPempP::Local::step::OutML}

Fix an arbitrary $\ol{\mu}_{\Ti} \in \Cem_{\varphi,\infty}$ with $\ol{\mu}_{\Ti} \not \in  \Cem^{L}$.
Then $S_{\nu,\zeta} \left( \ol{\mu}_{\Ti} \right) = \infty$, by the definition of the rate function.
This implies that \ref{thm::LDPempP::Local::LowerBound} of Theorem~\ref{thm::LDPempP::Local} is obviously satisfied.

Now we prove that \ref{thm::LDPempP::Local::UpperBound} of Theorem~\ref{thm::LDPempP::Local}  holds.
At first we fix an open ball around $\ol{\mu}_{\Ti}$, that does not intersect $\Cem^{L}$ (\ref{pf::thm::LDPempP::Local::step::OutML::Ball}).
Then we show that in such an open ball there is no empirical process with $N$ large enough (\ref{pf::thm::LDPempP::Local::step::OutML::NoEmpP}).
From this we conclude \ref{thm::LDPempP::Local::UpperBound} (in \ref{pf::thm::LDPempP::Local::step::OutML::Conclude}).
\emptyline
\begin{steps}
\step[{A open ball around $\ol{\mu}_{\Ti}$}]
\label{pf::thm::LDPempP::Local::step::OutML::Ball}
The set $\MOneL{\TWR}$ is closed in $\MOne{\TWR}$ (see e.g. \cite{AttButMichVar} Proposition~4.3.1).
This implies that also the set $ \Cem^{L}$ is closed in $ \Cem$.
Hence there is a $\epsilon>0$ such that 
\begin{align}
		\dist{\ol{\mu}_{\Ti}}{ \Cem^{L} } 
		= 
		\inf_{\pi \in \Cem^{L} } \left\{ \sup_{t \in \Ti} \rho^{Lip} \left( \ol{\mu}_{t} , \pi_{t} \right) \right\} 
		>
		2\epsilon
		,
\end{align}
where $\rho^{Lip} $ is the bounded Lipschitz norm on $\MOne{\TWR}$.

Define the open $\epsilon$ ball $B_{\epsilon} \left( \ol{\mu}_{\Ti} \right)$ around $\ol{\mu}_{\Ti}$ in this norm. 

\emptyline
\step[{No empirical process in the open ball for $N$ large enough}]
\label{pf::thm::LDPempP::Local::step::OutML::NoEmpP}
Assume that we could find a sequence $N_{\ell} \nearrow \infty$ in $\N$, such that 
for each $N_{\ell}$ there is an empirical process $\mu^{N_{\ell}}_{\Ti} \in B_{\epsilon} \left( \ol{\mu} \right)$, with a
 $\ul{\theta}^{N_{\ell}}_{\Ti} \subset \Csp{\Ti}^{\Nd_{\ell}}$ and a $\ul{\w}^{N} \in \Wsp^{\Nd}$.
We claim that this leads to a contradiction.
For each $N_{\ell}$ in the sequence, define $\mu^{(\ell)}_{t,x} = \delta_{\w^{k,N_{\ell}}}\delta_{\theta^{k,N_{\ell}}_{t}}$ when $\abs{x-\frac{k}{N_{\ell}}}<\frac{1}{2 N_{\ell}}$.
Then $\left\{ t \mapsto \mu^{(\ell)}_{t} \defeq \dd x \otimes \mu^{(\ell)}_{t,x} \right\} \in \Cem^{L}$.
For each $f \in \Csp{\TWR}$ that is Lipschitz continuous with $\iNorm{f}+\abs{f}_{Lip} \leq 1$,
\begin{align}
\begin{split}
		&\abs{
		\int_{\TWR} f \left( x, \w, \theta \right) \mu^{N_{\ell}}_{t} \left( \dd x, \dd \w, \dd \theta \right) 
		-
		\int_{\TWR} f \left( x, \w, \theta \right) \mu^{(\ell)}_{t} \left( \dd x, \dd \w, \dd \theta \right) 
		}
		\\
		&=
		\sum_{k \in \Td_{N_{\ell}}}
			\abs{ 	\frac{1}{N_{\ell}} f \left( \frac{k}{N_{\ell}} , \w^{k,N_{\ell}}, \theta^{k,N_{\ell}}_{t} \right) 
							- \int_{\Delta_{k,N_{\ell}}} f \left( x, \w^{k,N_{\ell}}, \theta^{k,N_{\ell}}_{t} \right) \dd x }
		\leq
		\sum_{k \in \TN}
			\abs{f}_{Lip} \left( \frac{1}{N_{\ell}} \right)^{2}
		\leq
		\frac{1}{N_{\ell}}
		,
\end{split}
\end{align}
with $\Delta_{k,N_{\ell}}$ defined as in Assumption~\ref{ass::LMF::J}.
Hence the distance between $\mu^{N_{\ell}}_{\Ti} $ and $ \Cem^{L}$ vanishes, a contraction.
Therefore we can fix an $\ol{N}\in \N$, such that there is no empirical process $\empP$ in $B_{\epsilon} \left( \ol{\mu}_{\Ti} \right)$ when $N>\ol{N}$.

\emptyline
\step[{Conclusion of \ref{thm::LDPempP::Local::UpperBound}}]
\label{pf::thm::LDPempP::Local::step::OutML::Conclude}
From the previous step we infer that for $N > \ol{N}$,
\begin{align}
	P^{N} \left[ \empP \in  B_{\epsilon} \left( \ol{\mu}_{\Ti} \right) \right]	
	=
	0
	.
\end{align}
This implies \ref{thm::LDPempP::Local::UpperBound} of Theorem~\ref{thm::LDPempP::Local} for this case.
\end{steps}

\emptyline
\step[{$\ol{\mu}_{\Ti} \not \in \Cem_{\varphi,\infty}$}]
\label{pf::thm::LDPempP::Local::step::NotCphi}

Because $\ol{\mu}_{\Ti} \not \in \Cem_{\varphi,\infty}$, $S_{\nu,\zeta} \left( \ol\mu_{\Ti} \right) = \infty$.
Therefore the condition \ref{thm::LDPempP::Local::LowerBound} of Theorem~\ref{thm::LDPempP::Local} is obviously satisfied.
To prove \ref{thm::LDPempP::Local::UpperBound} of Theorem~\ref{thm::LDPempP::Local}, note that for each $R>0$, the open set $\Cem \backslash \Cem_{\varphi,R}$ is a neighbourhood of $\ol\mu_{\Ti}$.
By Lemma~\ref{lem::LDPempP::Inter::OutCemR+Initial}, there is for each $\gamma$  and $R$ such that
\begin{align}
		P^{N} \left[  \empP \in  \Cem \backslash \Cem_{\varphi,R} \right]
		\leq 
		e^{-\Nd \gamma}
		.
	\end{align}
This implies the claimed condition \ref{thm::LDPempP::Local::UpperBound} of Theorem~\ref{thm::LDPempP::Local} in this case.

\end{steps}\vspace{-\baselineskip}
\end{proof}

\subsection{The concrete example \texorpdfstring{\protect\eqref{SDE::LocalMF}}{(\ref{SDE::LocalMF})} of a local mean field model}
\label{sec::LocalMF}

In this chapter we show that the concrete example \eqref{SDE::LocalMF} of a local mean field model, defined by $\sigma =1 $ and $b$ given by \eqref{eq::DriftCoef::LocalMF},
 with Assumption~\ref{ass::LMF::Init}, Assumption~\ref{ass::LMF::Init::Integral}, Assumption~\ref{ass::LMF::J} and Assumption~\ref{ass::LMF::Psi}, satisfies the Assumptions~\ref{ass::LDPempP}.

\begin{proof}
Fix $\varphi \left(  \theta \right) \defeq 1+ \theta^{2}$.
We show now separately that each item of Assumption~\ref{ass::LDPempP} is satisfied.

\emptyline
\begin{steps}
\step[{Assumption~\ref{ass::LDPempP}~\ref{ass::LDPempP::bCont}}]
\label{step::LocalMF::bCont}

The function $\partial_{\theta}\Psi$ is continuous by Assumption~\ref{ass::LMF::Psi}.
Hence the drift coefficient is continuous on $\TWR \times \left( \M_{\varphi,R} \cap \MOneL{\TWR} \right)$ if the map 
\begin{align}
\label{eq::step::LocalMF::bCont::beta}
		\left( x, \w , \mu \right) 
		\mapsto 
		\beta \left( x, \w, \mu \right) \defeq \int_{\TWR} J \left( x-x', \w , \w'  \right) \theta' \mu \left( \dd x', \dd w', \dd \theta' \right)
\end{align} 
is continuous on this space.
This holds if for $R>0$ and each sequence $\left( x^{(n)}, \w^{(n)}, \mu^{(n)} \right) \rightarrow  \left( x, \w, \mu \right)$ in $\TW \times \left( \M_{\varphi,R} \cap \MOneL{\TWR} \right)$, the following absolute value vanishes
\begin{align}
\label{eq::step::LocalMF::bCont::betaCont}
\begin{split}
		\abs{ \beta \left( x^{(n)}, \w^{(n)}, \mu^{(n)} \right) - \beta \left( x, \w, \mu \right) }
		\leq&
				\abs{ \beta \left( x, \w, \mu^{(n)} \right)  - \beta \left( x^{(n)}, \w^{(n)}, \mu^{(n)} \right) }
				\\
				&+
				\abs{ \beta \left( x, \w, \mu^{(n)} \right)  - \beta \left( x, \w, \mu\right)}
		\eqdef 
			\encircle{1} + \encircle{2}
		.
\end{split}
\end{align}
We show now that \encircle{1} and \encircle{2} vanish when $n$ tends to infinity.

\begin{steps}
\step[{\protect\encircle{1} }]
\label{step::LocalMF::bCont::Contx}
There is a sequence of continuous functions $J_{\ell} \in \Csp{\TW \times \Wsp}$, such that $J_{\ell} \rightarrow J$ in $\Ltwo \left( \Td , \Csp{ \Wsp\times \Wsp} \right)$, because $J \in \Ltwo \left( \Td , \Csp{ \Wsp\times \Wsp} \right)$.
This implies that for all $\ol{x} \in \Td$, $\ol{\w} \in \Wsp$ and $n \in \N$
\begin{align}
\label{eq::step::LocalMF::bCont::JLtwoJk}
\begin{split}
	&\abs{\int_{\TW} \left( J - J_{\ell} \right)\left( \ol{x} -x', \ol{\w} , \w' \right) 
				\int_{\R} \theta' \mu^{(n)}_{x',\w'} \left(  \dd \theta' \right) \mu^{(n)}_{x',\Wsp} \left( \dd \w' \right)  \dd x'
	}
	\\
	&\leq
		 \left( \int_{\Td} \left( \sup_{\w',\w'' \in \Wsp} \abs{ \left( J - J_{\ell} \right)\left( x , \w'', \w' \right) } \right)^{2} \dd x \right)^{\oh}
		R
	,
\end{split}
\end{align}
 because $\mu^{(n)} \in \M_{\varphi,R}$.
Therefore
\begin{align}
		\encircle{1}
		\leq
			\sup_{x' \in \Td, \w' \in \Wsp} \abs{ J_{\ell} \left( x^{(n)}  - x', \w^{(n)}, \w'  \right) - J_{\ell} \left( x - x' , \w , \w'  \right)}
			 \left( 1 +R \right)
			 + 
			 2 \absabs{J-J_{\ell}}  R
		\leq
			\epsilon
		,
\end{align}
for $k \in \N$ and $n \in \N$ large enough, because $J_{\ell}$ is uniformly continuous on the compact set $\TW \times \Wsp$.

\emptyline
\step[{\protect\encircle{2}}]
\label{step::LocalMF::bCont::ContMu}
To bound $\encircle{2}$, define the function $\chi_{M} \left( \theta \right)  \defeq \left( \theta \wedge M \right) \vee -M $ and approximate $J$ by $J_{\ell}$ as in the previous step.
Then
\begin{align}
\label{eq::LocalMF::bCont::ContMu::5sum}
\begin{split}
	\encircle{2}
	&\leq
	\abs{ \int_{\TW} \left( J-J_{\ell} \right) \left( x-x', \w, \w' \right) \int_{\R} \theta' \mu^{(n)}_{x'} \left(  \dd \theta' \right) \mu^{(n)}_{x',\Wsp} \left( \dd \w' \right) \dd x' }
	+
	\textnormal{ (this integral with $\mu$) }
	\\
	&+
	\abs{ \int  J_{\ell} \left( x-x', \w, \w'  \right)  \left( \theta' - \chi_{M} \left( \theta' \right) \right) \mu^{(n)} \left( \dd x', \dd \w', \dd \theta' \right)  }
	+
	\textnormal{ (this integral with $\mu$) }
	\\
		&+
	\abs{ \int  J_{\ell} \left( x-x', \w, \w'  \right)  \chi_{M} \left( \theta' \right) \left( \mu^{(n)} - \mu \right) \left( \dd x', \dd \w', \dd \theta' \right)  }
	\defeq
	\encircle{A} + \encircle{B} + \encircle{C} + \encircle{D} + \encircle{E}
	.
\end{split}
\end{align}
The \encircle{A} and \encircle{B} are bounded by $\epsilon$, when $k$ and $n$ are large enough as shown in \eqref{eq::step::LocalMF::bCont::JLtwoJk}.
We bound \encircle{C} by
\begin{align}
\begin{split}
	\encircle{C}
	&\leq
		\iNorm{J_{\ell}}  \int \abs{\theta'} \1_{\abs{\theta'}>M} \mu^{(n)}\left( \dd x', \dd \w, \dd \theta' \right) 
	\\
	&
	\leq
		\iNorm{J_{\ell}}
			\mu^{(n)} \left[ \left( x, \w, \theta' \right) : \abs{\theta'}>M \right]^{\oh}
			 \left( \int \left( \theta' \right)^{2} \mu^{(n)} \left( \dd x', \dd \w', \dd \theta' \right)  \right)^{\oh}
	.
\end{split}
\end{align}
For an arbitrary fixed $k \in \N$ and for all $n \in \N$, the right hand side is bounded by $\epsilon$ for $M$ large enough,
because  $\mu^{(n)} \in \M_{\varphi,R}$
and by the tightness of $\left\{ \mu^{(n)} \right\}_{n}$ (as a converging sequence).
The same arguments show \encircle{D} is bounded by $\epsilon$.
The \encircle{E} converges to zero when $n \rightarrow \infty$, for arbitrary fixed $k$ and $M$, because the integrand  is bounded and continuous.

Therefore, we fix at first a $k \in \N$, then an $M >0$. Then for $n \in \N$ large enough, \encircle{2} is bounded by $\epsilon$.

\end{steps}

We have hence shown that \eqref{eq::step::LocalMF::bCont::betaCont} vanishes when $n$ tends to infinity, i.e. that $\beta$ is continuous.

\emptyline
\step[{Assumption~\ref{ass::LDPempP}~\ref{ass::LDPempP::blockbd}}]

Fix an arbitrary $N \in \N$ and an arbitrary $\ul{\w}^{N} \in \Wsp$.
The function $b^{N} : \RN \rightarrow \RN$ is continuous, by 
Assumption~\ref{ass::LMF::Psi}
and because $\frac{1}{\Nd} \sum_{j \in \TN} J \left( \frac{i-j}{N}, \w^{i,N}, \w^{j,N} \right) \theta^{j,N}$ is continuous (because $J\left( \frac{i-j}{N}, \w, \w' \right)$ is finite for all $i,j \in \TN$ and all $\w, \w' \in \Wsp$ by Assumption~\ref{ass::LMF::J}).
Hence $b^{N}$ is locally bounded.

\emptyline
\step[{Assumption~\ref{ass::LDPempP}~\ref{ass::LDPempP::IntBound}}]

Let $\Gen^{\textnormal{LMF}}_{\empM,.,.}$ be the generator of the local mean field model, defined as \eqref{def::InteractingGenerator}, with $\empM$ the empirical measure corresponding to $\ul{\theta}^{N} \in \RN$ and $\ul{\w}^{N} \in \Wsp$.
Then for $N$ large enough
\begin{align}
\label{eq::LocalMF::ass::IntGenVarphi}
\begin{split}
			&\int_{\TWR} \Gen^{\textnormal{LMF}}_{\empM,x, \w} \varphi \left( \theta \right) + \frac{1}{2} \abs{\partial_{\theta} \varphi\left( \theta \right) }^{2}
						\empM \left( \dd x, \dd \w, \dd \theta \right)  
			\\
			&=
					2+
					2\int_{\TWR} \left(-  \ol\Psi' \left( \theta\right)  \theta - \w \theta^{2} +  \theta^{2} \right) \empM \left( \dd x, \dd \w, \dd \theta \right)  
					+
					2 B^{N}_{\ul{\w}^{N}} \left( \ul{\theta}^{N} \right)
\end{split}
\end{align}
where
\begin{align}
\label{eq::LocalMF::ass::BoundOnInteraction}
\begin{split}
		&B^{N}_{\ul{\w}^{N}} \left( \ul{\theta}^{N} \right)
		\defeq
				\frac{1}{N^{2d}} \sum_{i,j \in \TN} J \left( \frac{i-j}{N}, \w^{i,N}, \w^{j,N} \right) \theta^{i,N} \theta^{j,N}
		\\
		&\leq 
			\left(
				\frac{1}{\Nd}
				\sum_{i \in \TN } 
				\sup_{\w,\w' \in \Wsp}
				\abs{ J  \left( \frac{i}{N}, \w , \w' \right)  -  \Nd \int_{\Delta_{i,N}} J \left( x, \w , \w' \right) \dd x }
				+
				\LpN{1}{\ol{J}}
			\right)
			\frac{1}{\Nd} \sum_{j \in \TN} \left( \theta^{j,N} \right)^{2}
		\\
		&\leq
			\left( \delta +\LpN{1}{\ol{J}}
			\right)
			\frac{1}{\Nd} \sum_{j \in \TN} \left( \theta^{j,N} \right)^{2}
		,
\end{split}
\end{align}
with $\delta>0$ if $N>\ol{N}_{\delta}$ by Assumption~\ref{ass::LMF::J}.
With this upper bound on $B^{N}$, $\Psi$ being a polynomial of even degree with positive coefficient of this degree (Assumption~\ref{ass::LMF::Psi}) and $\Wsp$ being compact, we conclude that \eqref{eq::LocalMF::ass::IntGenVarphi} is lower or equal to
\begin{align}
\begin{split}
					C+
					2 \int_{\TWR}  \left(  \abs{\w}+ 1+ \LpN{1}{\ol{J}} + \delta \right) \theta^{2}  
						\empM \left( \dd x, \dd \w, \dd \theta \right)  
			\leq
					\lambda \int_{\TWR} \varphi \left( \theta \right) \empM \left( \dd x, \dd \w, \dd \theta \right)  
			.
\end{split}
\end{align}
Here the constant $\lambda$ only depends on $\Psi$ and $J$ for $N$ large enough but not on $\empM$. Hence the Assumption~\ref{ass::LDPempP}~\ref{ass::LDPempP::IntBound} is satisfied.

\emptyline
\step[{Assumption~\ref{ass::LDPempP}~\ref{ass::LDPempP::Lyapu}}]

Fix an arbitrary $\mu_{\Ti} \in \Cem_{\varphi,\infty} \cap  \Cem^{L}$.
We know by \ref{step::LocalMF::bCont}, that $\left( x, \w, t \right) \mapsto \beta \left( x, \w,  \mu_{t} \right) $ is continuous.
Moreover the set $\TW \times \left\{ \mu_{t} \right\}_{t \in \Ti}$ is compact in $\TW \times \MOne{\TWR}$, by Prokhorov's theorem.
Hence $\beta$ is bounded on this set by a constant $C_{\beta}$.
Then for all $\left( t,x, \w, \theta \right) \in \TTWR$
\begin{align}
\begin{split}
	\Gen^{\textnormal{LMF}}_{\mu_{t},x, \w} \varphi \left( \theta \right) + \frac{1}{2} \abs{\partial_{\theta} \varphi\left( \theta \right) }^{2}
				&=
						- 2 \partial_{\theta} \ol\Psi \left( \theta\right)  \theta
						-
							2 \w \theta^{2}
			 			+
						2  \theta 
						\beta \left( x, \mu_{t} \right) 
						+
						2
						+
						2 \theta^{2}
			\\
			&\leq
					- 2  \partial_{\theta} \ol\Psi \left( \theta\right)  \theta
					+
					2 \abs{\w} \theta^{2}
					+
					2 \abs{\theta} C_{\beta}
					+
					2
					+
					2 \theta^{2}
			\leq
					\lambda \left( \mu_{\Ti} \right) \varphi \left( \theta \right)
			,
\end{split}
\end{align}
because $\ol\Psi$ is a polynomial of even degree (Assumption~\ref{ass::LMF::Psi}) and $\Wsp$ is compact.

\emptyline
\step[{Assumption~\ref{ass::LDPempP}~\ref{ass::LDPempP::MuIntCont}}]
\label{step::LocalMF::MuIntCont}

Fix an  $R>0$ and a $\ol\mu_{\Ti}  \in \Cem_{\varphi,R} \cap \Cem^{L}$.
Take an arbitrary sequence $ \left\{ \mu^{\left( n \right)}_{\Ti}\right\} $ from one of the sets given in Assumptions~\ref{ass::LDPempP}~\ref{ass::LDPempP::MuIntCont}, such that $\mu^{\left( n \right)}_{\Ti}\rightarrow \ol\mu_{\Ti}$.
We show in the subsequent steps that
		\begin{align}
		\label{eq::LocalMF::MuIntCont::Conv}
			\int_{0}^{T} 
				\int_{\TWR}
							\abs{ \beta \left( x, \w, \mu^{\left( n \right)}_{t} \right) -\beta \left( x, \w, \ol{\mu}_{t}  \right) }^{2}
				 \mu^{\left( n \right)}_{t} \left(\dd x, \dd \w, \dd \theta \right) 
			\dd t 
			\rightarrow
			0
			.
		\end{align}

\begin{steps}
\step[{Case: All $\mu^{(n)}_{\Ti} \in \Cem^{L}$}]
\label{step::LocalMF::MuIntCont::CemL}

Assume at first that $\mu^{(n)}_{\Ti} \in \Cem_{\varphi,R} \cap \Cem^{L}$ for all $n \in \N$.
For each $t \in \Ti$, $\mu^{\left( n \right)}_{t} \rightarrow \ol{\mu}_{t}$ in $\M_{\varphi,R}$ by the uniform topology on $\Cem$.
Therefore the set $U_{t} \defeq \left\{ \mu^{\left( n \right)}_{t}  \right\}_{n} \cup \left\{ \ol\mu_{t} \right\}$ is compact.
$
\TW \times U_{t} 
\ni 
\left( x, \w,  \mu \right) 
\mapsto 
\abs{ \beta \left( x, \w, \mu \right) -\beta \left( x, \w, \ol{\mu}_{t}  \right) }
$ is uniformly continuous (we show the continuity in \ref{step::LocalMF::bCont}).
Hence for each $t \in \Ti$, the absolute value in \eqref{eq::LocalMF::MuIntCont::Conv} converges uniformly in $\left( x,\w \right) \in \TW$ to zero, when $n$ tends to infinity.
Moreover this absolute value is uniformly bounded, because for all $\left( x,\w \right) \in \TW$
\begin{align}
\label{eq::step::LocalMF::MuIntCont::CemL::Bound}
	\abs{ \int_{\TWR} J \left( x- x', \w, \w'  \right) \theta' \mu^{(n)}_{t} \left( \dd x', \dd \w', \dd \theta' \right) }^{2}
	\leq
		\LtwoN{\ol{J}}^{2}
		 \int_{\TWR} \abs{\theta'}^{2} \mu^{(n)}_{t} \left( \dd x', \dd \w', \dd \theta' \right) 
	.
\end{align}
The right hand side is bounded by $\LtwoN{\ol{J}}^{2} R$,
because all $\mu^{(n)}_{\Ti}$ are in $\Cem_{\varphi,R}$.
This implies the convergence \eqref{eq::LocalMF::MuIntCont::Conv} for sequences in $\Cem^{L}$.

\emptyline
\step[{Case: All $\mu^{(n)}_{\Ti}$ are empirical processes}]
\label{step::LocalMF::MuIntCont::EmpP}

Fix a sequence of empirical processes $\left\{ \mu^{(n)}_{\Ti} \right\}_{n} \subset \Cem_{\varphi,R}$,
such that $\mu^{(n)}_{\Ti} \rightarrow \ol{\mu}_{\Ti}$.
Fix  $N_{n} \in \N$, $\theta^{i,N_{n}}_{\Ti} \in \Csp{\Ti}$, $\w^{i,N_{n}} \in \Wsp$ such that
  $\mu^{(n)}_{\Ti} = \frac{1}{\Nd_{n}}  \sum_{i \in \T^{d}_{N_{n}}} \delta_{\left( \frac{i}{N_{n}} , \w^{i,N_{n}},  \theta^{i,N_{n}}_{t} \right) }$.
Note that we do not get for this sequence the continuity of $\beta$ at $t \in \Ti$ from \ref{step::LocalMF::bCont}.

For each $t \in \TN$ and $n \in \N$, the inner integral in \eqref{eq::LocalMF::MuIntCont::Conv} is given by
		\begin{align}
		\label{eq::step::LocalMF::MuIntCont::EmpP::Pointw}
			\frac{1}{\Nd_{n}} \sum_{j \in \Td_{N_n} }
							\abs{ \beta \left( \frac{j}{N_n}, \w^{j,N_{n}}, \mu^{\left( n \right)}_{t} \right) -\beta \left( \frac{j}{N_n},\w^{j,N_{n}}, \ol{\mu}_{t}  \right) }^{2}
			.
		\end{align}
We show in the following that this sum converges for each $t \in \Ti$ pointwise to zero (\ref{step::LocalMF::MuIntCont::EmpP::Pointw}).
Moreover we show that this sum is uniformly bounded (\ref{step::LocalMF::MuIntCont::EmpP::Bound}).
From these two results we conclude \eqref{eq::LocalMF::MuIntCont::Conv} by the dominated convergence theorem.

\begin{steps}
\step[{\protect\eqref{eq::step::LocalMF::MuIntCont::EmpP::Pointw} vanishes pointwise}]
\label{step::LocalMF::MuIntCont::EmpP::Pointw}
To show that  \eqref{eq::step::LocalMF::MuIntCont::EmpP::Pointw} vanishes, we divide the absolute value as in \eqref{eq::LocalMF::bCont::ContMu::5sum} into five summands.
Fix an arbitrary small $\epsilon>0$.
By fixing $k \in \N$ and $M>0$ large enough, the \encircle{B}, \encircle{C} and \encircle{D} of these summands are smaller than $\epsilon$ for all $\left( x, \w \right) \in \TW$ for fixed $k$ and all $n \in \N$ large enough, by the same arguments that we use in \ref{step::LocalMF::bCont::ContMu}.
Hence to  bound \eqref{eq::step::LocalMF::MuIntCont::EmpP::Pointw} we only need to bound the following two summands
\begin{align}
\begin{split}
			\encircle{A}
			&\defeq
			\frac{1}{\Nd_{n}} \sum_{j \in \Td_{N_n} }
							\abs{ \frac{1}{\Nd_{n}} \sum_{i \in \Td_{N_n} }   
														\theta^{i,N_n}_{t}   
														 \left( J-J_{\ell} \middle) \middle( \frac{j - i}{N_n} , \w^{j,N_{n}},\w^{i,N_{n}}  \right)  }
			\\
			\encircle{E}
			&\defeq
			\frac{1}{\Nd_{n}} \sum_{j \in \Td_{N_n} }
										\abs{ \int  J_{\ell} \left( \frac{j}{N_{n}} - x' , w^{j}, w^{i} \right)  \chi_{M} \left( \theta' \right) 
													\left( \mu^{(n)}_{t} - \ol\mu_{t} \right) \left( \dd x', \dd \theta' \right)  }
			.
\end{split}
\end{align}
We prove now that \encircle{A} and \encircle{E} are smaller than $\epsilon$ when $n$ is large enough (\ref{step::LocalMF::MuIntCont::EmpP::Pointw::A} and \ref{step::LocalMF::MuIntCont::EmpP::Pointw::B}).
Both proofs require that $N_{n}$ converges to infinity. 
We show in \ref{step::LocalMF::MuIntCont::EmpP::Pointw::Nn}, that this is a consequence of the convergence of $\mu^{(n)}_{t}$ to a measures in $\MOneL{\TWR}$.

\emptyline
\begin{steps}
\step[{The sequence $N_{n} \rightarrow \infty$}]
 \label{step::LocalMF::MuIntCont::EmpP::Pointw::Nn}
 Assume that this were not the case, i.e. that there is a subsequence $\left\{ N_{n_\ell}\right\}_{\ell=1}^{\infty}$ such that $N_{n_\ell} \leq \ol{N} < \infty$. 
 This is a contradiction to the convergence of $\mu^{(n)}_{t}$ to $\ol{\mu}_{t}$.
 Indeed, choose $f \in \CspL{b}{\TWR}$ such that $f \left( x, \w,  \theta \right) = f \left( x \right) \geq 0$ for all $\left( x, \w,  \theta \right) \in \TWR$, $\int_{\Td} f \left( x \right) \dd x >0$ and $f\left( \frac{k}{N} \right) =0$ for all $N \leq \ol{N}$, $k \in \TN$.
 Then $\int f \left( x \right) \mu^{(n_{\ell})}_{t} =0$ for all $\ell \in \N$, but $\int f \left( x \right) \ol{\mu}_{t} >0$.
 A contradiction.
 
    \emptyline
    \step[{\protect\encircle{E}}]
    \label{step::LocalMF::MuIntCont::EmpP::Pointw::B}
    The function $J_{\ell}$ is uniformly continuous on $\TW$.
    By the compactness of $\TW$, there are finitely many $\left\{ x_{a} \right\}_{a \in A} \subset \Td$ and finitely many $\left\{ \w_{a'} \right\}_{a' \in A'} \subset \Wsp$, such that
    \begin{align}
		    \encircle{E}
		    \leq
			    2 \epsilon M
			    +
			    \max_{ a \in A, a' \in A' }
										\abs{ \int  J_{\ell} \left( x_{a}- x', \w_{a'}, \w' \right)  \chi_{M} \left( \theta' \right) 
												\left( \mu^{(n)}_{t} - \ol\mu_{t} \right) \left( \dd x', \dd \w', \dd \theta' \right)  }
			    .
    \end{align}
	The maximum is only over a finite number of values, hence the convergence of $\mu^{(n)}_{t}$ to $\ol\mu_{t} $ implies that for $n$ large enough, the maximum is bounded by $\epsilon$.

    \emptyline
    \step[{\protect\encircle{A}}]
    \label{step::LocalMF::MuIntCont::EmpP::Pointw::A}
	We bound \encircle{A} by $\epsilon$ through a similar estimate as in \ref{step::LocalMF::bCont::Contx}.
	In particular we use the following estimate instead of \eqref{eq::step::LocalMF::bCont::JLtwoJk}. For all $j \in \Td_{N_{n}}$,
\begin{align}
\label{eq::LocalMF::MuIntCont::ConvFixtx}
\begin{split}
	\encircle{A}
		&\leq
		\abs{ \sum_{i \in \Td_{N_n} } \theta^{i,N_n}_{t}  \int_{\Delta_{i,N_n}} \left( J-J_{\ell}  \middle) \middle( \frac{j}{N_{n}}-x'  , \w^{j,N_{n}},\w^{i,N_{n}} \right) \dd x'  }
		\\
		&+
		\abs{ \frac{1}{\Nd_{n}} \sum_{i \in \Td_{N_n} } \theta^{i,N_n}_{t}
										\left( J_{\ell}  \left( \frac{j - i}{N_n} , \w^{j,N_{n}},\w^{i,N_{n}}  \right)  
											- \Nd_{n} \int_{\Delta_{i,N_n}} J_{\ell} \left( \frac{j}{N_{n}}-x'  , \w^{j,N_{n}},\w^{i,N_{n}}  \right) \dd x'   \right)
				}
		\\
		&+
				\abs{ \frac{1}{\Nd_{n}} \sum_{i \in \Td_{N_n} } \theta^{i,N_n}_{t}
												\left( J  \left( \frac{j - i}{N_n} , \w^{j,N_{n}},\w^{i,N_{n}}  \right)  -  \Nd_{n} \int_{\Delta_{i,N_n}} J \left( \frac{j}{N_{n}}-x' , \w^{j,N_{n}},\w^{i,N_{n}}  \right) \dd x'   \right)
						}
		.		 
\end{split}
\end{align}
We denote the three summands by \encircle{A1}, \encircle{A2} and \encircle{A3} and we bound them  separately.
By applying twice the \Holder inequality
\begin{align}
	\encircle{A1}
	\leq 
		\left( \frac{1}{\Nd_{n}}  \sum_{i \in \Td_{N_n} } \abs{ \theta^{i,N_n}_{t}}^{2} \right)^{\oh}  
		\left( \int_{\Td} \left( \sup_{\w,\w' \in \Wsp} \abs{ \left( J-J_{\ell} \middle) \middle( x', \w, \w' \right) }  \right)^{2} \dd x' \right)^{\oh}
	\leq
		R \epsilon
	,
\end{align}
 for $k$ large enough.
 \begin{align}
 	\encircle{A2}
 	\leq 
 		 \frac{1}{\Nd_{n}}  \sum_{i \in \Td_{N_n} } \abs{ \theta^{i,N_n}_{t}} 
	 	\sup_{ \abs{y-y'}\leq \frac{1}{N_n} }
	 	\sup_{\w,\w' \in \Wsp} \abs{  J_{\ell} \left( y', \w, \w' \right)  -J_{\ell} \left( y, \w, \w' \right) }
 	\leq
 		R \epsilon
 	,
 \end{align}
  for each $k$, when $n$ (and hence $N_n$) is large enough. Last but not least, by a change of variables 
 \begin{align}
 	\encircle{A3}
 	\leq 
 		\left( \frac{1}{\Nd_{n}}  \sum_{i \in \Td_{N_n} } \abs{ \theta^{i,N_n}_{t}}^{2} \right)^{\oh}  
 		\left(   \sum_{i \in \Td_{N_n} } 
						 		\sup_{\w,\w' \in \Wsp}
						 		\abs{  \int_{\Delta_{i,N_n}} J  \left( \frac{i}{N_n} , \w, \w' \right) - J \left( x' , \w, \w' \right) \dd x' }^{2} 
			\right)^{\oh}
 	,
 \end{align} 
 which is also bounded by $R \epsilon$, when $n$ is large enough by Assumption~\ref{ass::LMF::J}.
 
\end{steps}
  \emptyline
   \step[{\protect\eqref{eq::step::LocalMF::MuIntCont::EmpP::Pointw} is uniformly (in $t \in \Ti$) bounded}]
   \label{step::LocalMF::MuIntCont::EmpP::Bound}
	 
	We show that each summand of \eqref{eq::step::LocalMF::MuIntCont::EmpP::Pointw} is bounded uniformly in $t \in \Ti$, $j \in \Td_{N_n}$, $n \in \N$.
	By applying the \Holder inequality we get
	\begin{align}
	\begin{split}
			&\abs{ \beta \left( \frac{j}{N_{n}}, \w^{j,N_{n}}, \mu^{(n)}_{t} \right) }^{2}
			\\
			&\leq
				 	\left(
				 	\frac{1}{\Nd}
				 	\sum_{i \in \TN } 
				 	\sup_{\w,\w' \in \Wsp}
				 	\abs{ J  \left( \frac{i}{N}, \w , \w' \right)  -  \Nd \int_{\Delta_{i,N}} J \left( x, \w , \w' \right) \dd x }^{2}
				 	+
				 	\LpN{2}{\ol{J}}
				 	\right)
				 	\frac{1}{\Nd_{n}}  \sum_{i \in \Td_{N_n} } \abs{ \theta^{i,N_n}_{t}}^{2} 
			.
		\end{split}
	\end{align}
	This is bounded by $R \left( \LtwoN{\ol{J}} + \delta \right)$, 
	 for a $\delta>0$, when $N_{n}$ is large enough, by Assumption~\ref{ass::LMF::J}.
	 Moreover we get a uniform upper bound on $\abs{\beta \left( \frac{j}{N_{n}}, \w^{j,N_{n}}, \ul\mu_{t} \right)}$ as in \eqref{eq::step::LocalMF::MuIntCont::CemL::Bound}.

\end{steps}
\end{steps}

We have hence proven Assumption~\ref{ass::LDPempP}~\ref{ass::LDPempP::MuIntCont}.
\end{steps}

\emptyline
Summarized, the specific model considered in Chapter~\ref{sec::Intro::ResultLMF}, satisfies the Assumption~\ref{ass::LDPempP}.
\end{proof}

\begin{remark}
	When considering only continuous $J$, the proofs are much simpler.
	However also interaction weights that are not continuous are of particular interest (for some examples see Example~\ref{exa::LMF::J}).
\end{remark}

%%%%%%%%%%%%%%
%%%%%%%%%%%%%%
%%%%%%%%%%%%%%

\section{Representations of the rate function for the LDP of the empirical process}
\label{sec::OtherRepRFS}

In this chapter, we state three other representations of the rate function $S_{\nu,\zeta}$, besides the two given in Theorem~\ref{thm::LDPempP}.
To state these representations we need the following notation.
\begin{notation}
\label{nota::ForOtherRepS}
For $\mu_{\Ti} \in \Cem_{\varphi,\infty}$ and $\left( t,x,\w, \theta \right) \in \TTWR$, set
\begin{align}
	b^{I,\mu_{\Ti}} \left( t, x, \w, \theta \right) \defeq b \left( x, \w, \theta, \mu_{t} \right) 
	.
\end{align} 
With $b^{I,\mu_{\Ti}}$ as drift coefficient, define the generator $\Gen^{I,\mu_{\Ti}}_{t,x,\w}$ as in \eqref{def::GeneratorIndependent}.
For this system, the Assumption~\ref{ass::LDPempP::Inde} are satisfied if the assumptions of Theorem~\ref{thm::LDPempP}  hold (as shown in Chapter~\ref{sec::LDPempP::Inter::AssInterToInde}).
In particular the corresponding martingale problem has for each $\left( x, \w, \theta \right) \in \TWR$ a unique solution, which we denote by $P^{I,\mu_{\Ti}}_{x,\w,\theta}$.
Then we define $P^{I,\mu_{\Ti}}_{x,\w}  \in \MOne{ \Csp{\Ti} }$, $P^{I,N,\mu_{\Ti}}_{\ul{\w}^{N}} \in \MOne{ \Csp{\Ti}^{\Nd} }$ and  $P^{I,N,\mu_{\Ti}} \in \MOne{ \Wsp^{\Nd} \times \Csp{\Ti}^{\Nd} }$  
as in Notation~\ref{nota::LDPempP::Inde::General}.

Moreover we denote by $U_{s,t}^{\mu_{\Ti}}  $ the operator $U_{s,t}$ defined in \eqref{def::LDPempP::OperatorU} with $P^{I}$ replaced by $P^{I,\mu_{\Ti}}$.

\end{notation}

\begin{theorem}
\label{thm::RepS}
Let the assumptions of Theorem~\ref{thm::LDPempP} hold.
		Take a  $\mu_{\Ti} \in \Cem$, with $S_{\nu,\zeta} \left( \mu_{\Ti}  \right)  < \infty$.
		Then $S_{\nu,\zeta}$ has the following representations.
		\begin{enuRom}
		\item
		\label{thm::RepS::First}
		\begin{align}
				S_{\nu,\zeta} \left( \mu_{\Ti}  \right)
				=
				\inf_{ \substack{Q \in \MOne{ \TWC }  \\
						 \Pi \left( Q \right)_{\Ti} =\mu_{\Ti} }
					}
						\relE{Q}{\dd x \otimes \zeta_{x} \left( \dd \w \right) \otimes P^{I,\mu_{\Ti}}_{x,\w}   }
		\end{align}
		
		\item	
		\label{thm::RepS::Sec}
		$S_{\nu,\zeta} \left( \mu_{\Ti}  \right)$ is equal to			
				\begin{align}
					\begin{split}
						\hspace{-0.5cm} 
						\sup_{\substack{r \in \N, \\0 \leq t_{1} < ... < t_{r} \leq T }} 
						&\left[
						\sup_{f } 
								\left\{  
									\int_{\TWR} \mkern-24mu f \left( x,\w,\theta\right) \mu_{t_{1}}
								- 
									\int_{\Td}    \log \left( 
									\int_{\Wsp\times \R} \mkern-18mu
											U^{\mu_{\Ti}}_{0,t_{1}} e^{f} \left( x, \w, \theta \right)  
									\nu_{x} \left(\dd \theta\right) 
									\zeta_{x} \left( \dd \w \right)
									\right) \dd x
								\right\} 
						\right.
						\\
						&+
						\left.
						\sum_{i=2}^{r}
						\sup_{f  } 
								\left\{  
									\int_{\TWR} \mkern-24mu   f \left( x,\w, \theta\right) \mu_{t_{i}}
									-
									\int_{\TWR}  \mkern-24mu  \log U^{\mu_{\Ti}}_{t_{i-1},t_{i}} e^{f}  \left( x,\w,\theta\right) \mu_{t_{i-1}}
								\right\} 
					\right]
						,
					\end{split}
				\end{align}
				where the $\mu_{t_{i}}$ integrate with respect to the variables $\dd x, \dd \w, \dd \theta$ and 
				the functions $f$ in the suprema are in the set $\CspL[\infty]{c}{\TWR}$.
			
		\item
		\label{thm::RepS::h}
		There is a function $h^{\mu_{\Ti} }  \in \what L^{2}_{\mu_{\Ti}} \left( 0,T \right) $ (this space is defined in the \ref{pf::lem::LDPempP::indep::3Rep::Step::L2} of the proof of Lemma~\ref{lem::LDPempP::indep::3Rep}),
		such that
		\begin{align}
						\hspace{-0.5cm} 
					S_{\nu,\zeta} \left( \mu_{\Ti}  \right)
					=
							\frac{1}{2}
							\int_{0}^{T} 
							\int_{\TWR} 
							\frac{\sigma^{2}}{2} \left( h^{\mu_{\Ti}} \left( t,x, \w, \theta \right) \right) ^{2} 
							\mu_{t} \left( \dd x, \dd \w, \dd \theta\right)
							\dd t
						+
							\relE{\mu_{0}}{\dd x \otimes  \zeta_{x} \otimes \nu_{x}}
						.
		\end{align}
		Moreover $\mu_{\Ti}$ satisfies in a weak sense (i.e. when integrated against an arbitrary function in $\CspL[1,0,2]{c}{ \TTWR  }$) the PDE
		\begin{align}
		\label{thm::RepS::h::PDE}
				\partial_{t}\mu_{t}
				=
				\left(\Gen_{\mu_{t},.,.} \right)^{*}  \mu_{t}
				+
				\sigma^{2} \partial_{\theta} \left( \mu_{t} h^{\mu_{\Ti} } \left( t \right) \right)
				.
		\end{align}
		
		\end{enuRom}
\end{theorem}

\begin{proof}
When $S_{\nu,\zeta} \left( \mu_{\Ti}  \right)  < \infty$, then $\mu_{\Ti} \in \Cem_{\varphi,\infty} \cap \Cem^{L}$.
Therefore we know by Chapter~\ref{sec::LDPempP::Inter::AssInterToInde}, that the measure $P^{I,\mu_{\Ti}}_{x,\w}$ is well defined.
Moreover all the results of Chapter~\ref{sec::LDPempP::Inde} hold for the independent spin system with the drift coefficient $b^{I}$ of Notation~\ref{nota::ForOtherRepS}.

\emptyline
The representations \ref{thm::RepS::First} and \ref{thm::RepS::Sec} follow directly from Lemma~\ref{lem::LDPempP::indep::1Rep} and Lemma~\ref{lem::LDPempP::indep::2Rep}.

\emptyline
The representation \ref{thm::RepS::h}, follows from Lemma~\ref{lem::LDPempP::indep::3Rep} and the proof of this lemma, in particular \eqref{eq::lem::LDPempP::indep::3Rep::supIeqH} in \ref{pf::lem::LDPempP::indep::3Rep::Step::L2} of this proof.
That $\mu_{\Ti}$ is a weak solution of the PDE \eqref{thm::RepS::h::PDE}, follows from \eqref{eq::lem::LDPempP::indep::3Rep::weakPDE} and \eqref{pf::lem::LDPempP::indep::3Rep::defEll}.
\end{proof}

%%%%%%%%%%%%%%%%
%%%%%%%%%%%%%%%%
%%%%%%%%%%%%%%%%

     \section{The LDP of the empirical measure}
\label{sec::LDPLN}

In this chapter we show the large deviation principle for the empirical measures $L^{N}$ under the assumptions of Chapter~\ref{sec::LDPempP} and the following  exponential tightness assumption.
\begin{assumption}
\label{ass::LNPNexptight}
The family $\left\{ L^{N} , P^{N} \right\}$ is exponential tight, i.e.
	for all $s>0$,
	there is a compact set $\calK_{s} \subset \MOne{\TWC}$, such that
	\begin{align}
		\label{thm::LDPLN::ExpBound::eq}
		\limsup_{N \rightarrow \infty} N^{-d} \log 
		P^{N} \left[ L^{N} \not \in \calK_{s} \right]
		\leq 
		-s
		.
	\end{align}
	
\end{assumption}

To state the large deviation principle result, we need the following definitions and notations.
\begin{definition}
We say $Q \in \calM_{\varphi, R}$ if and only if $\Pi \left(  Q \right)_{\Ti} \in \M_{\varphi,R}$, for $R \in \left( 0, \infty \right]$.
\end{definition}

For fixed $x \in \Ti$ and  $Q \in \calM_{\varphi, R} \cap \MOneL{\TWC} $, 
we define $b^{I,\Pi \left( Q \right)}$, $\Gen^{I,\Pi \left( Q \right)}_{t,x,\w}$ and the measures $P^{I,\Pi \left( Q \right)}_{x,\w} \in \MOne{ \Csp{\Ti} }$ and $P^{I,N,\Pi \left( Q \right)} \in \MOne{ \Wsp^{\Nd} \times \Csp{\Ti}^{\Nd} }$ as in Notation~\ref{nota::ForOtherRepS}.

\begin{theorem}
\label{thm::LDPLN}
	If the assumptions of Theorem~\ref{thm::LDPempP} and the Assumption~\ref{ass::LNPNexptight} hold,
	then the family of empirical measures $\left\{ L^{N} , P^{N} \right\}$  satisfies on $\MOne{\TWC}$ a large deviation principle with good rate function 
	\begin{align}
	\label{eq::thm::LDPLN::RF}
		I \left( Q \right)
		\defeq
		\begin{cases}
			\relE{ Q } { P^{I, \Pi \left( Q \right) }  }
			\qquad 
			&\textnormal{if }
			Q \in \MOneL{\TWC}  \cap \calM_{\varphi,\infty}
			,
			\\
			\infty
			&\textnormal{otherwise.}
		\end{cases}
	\end{align}
where
$P^{I, \Pi \left( Q \right) }  \defeq \dd x \otimes \zeta_{x} \left( \dd \w \right) \otimes P^{I, \Pi \left( Q \right) }_{x,\w} \in \MOne{\TWC}$.
\end{theorem}

To prove this theorem, we use the same approach as for the proof of the large deviation principle for the empirical process $\empP$, that we give in Chapter~\ref{sec::LDPempP}.
To summarise it, we show at first a large deviation principle for spins that evolve according to an independent system of SDEs.
Form this we derive  a local LDP for the interacting system.
This requires exponential bounds on the probability that the empirical measures leave the set $\calM_{\varphi,R}$.
Finally we infer from the local LDP, the desired LDP of the interacting system.
In this last step, we have to show that $I$ is a good rate function.
This prove is different from the corresponding one in Chapter~\ref{sec::LDPempP}.
Moreover we require the exponential tightness in the last step.
We need to assume it in this chapter, because we are not able to prove it in general as in Chapter~\ref{sec::LDPempP}.

\emptyline
In Chapter~\ref{sec::LDPLN::LMF} we show that the  concrete example \eqref{SDE::LocalMF} of a local mean field model satisfies the exponential tightness.
Moreover we show a second representation of the rate function for this model.

\subsection{Proof of the LDP (Theorem~\ref{thm::LDPLN})}

To prove the Theorem~\ref{thm::LDPLN}, we show at first that the measure in the relative entropy in \eqref{eq::thm::LDPLN::RF} is actually a probability measure.

\begin{lemma}
\label{lem::LDPLN::PQwellDef}
	For each $Q \in \calM_{\varphi,\infty} \cap \MOneL{\TWC}$, the measure 
	$P^{I, \Pi \left( Q \right) }$ is well defined.
\end{lemma}

\begin{proof}[of Lemma~\ref{lem::LDPLN::PQwellDef}]
Fix a $Q \in \calM_{\varphi,\infty} \cap \MOneL{\TWC}$.
The function $b^{I, \Pi \left( Q \right) }$ is continuous. 
Indeed, $t \mapsto \Pi \left( Q \right)_{t}$ is continuous (Lemma~\ref{lem::MapPi::WellDef}), $\Pi \left( Q \right)_{t} \in \M_{\varphi,R} \cap \MOneL{\TWR}$ for all $t \in \Ti$
and  $b$ is continuous on $\TWR \times \left( \M_{\varphi,R} \cap \MOneL{\TWR} \right)$ (Assumption~\ref{ass::LDPempP}~\ref{ass::LDPempP::bCont}).
Therefore we can apply \cite{StrVarMultidi}~Theorem~11.1.4 to get the continuity of $\left( x, \w, \theta \right) \mapsto P^{I,\Pi \left( Q \right)}_{x,\w,\theta }$ (see also Lemma~\ref{lem::LDPempP::indep::Feller}).
By this continuity, the Assumption~\ref{ass::LMF::Init} and the Assumption~\ref{ass::LMF::Medium}, we conclude (as in Lemma~\ref{lem::SanovT::ass::dxQxnu}) that the measure $P^{I, \Pi \left( Q \right) }$ is well defined. 
\end{proof}

\subsubsection{The independent system}

Fix a $Q \in \calM_{\varphi, \infty} \cap \MOneL{\TWC}$.
We get, as in the proof of Lemma~\ref{lem::LDPempP::indep::1Rep},  the following large deviation principle for the independent system (by Lemma~\ref{lem::SanovT} and Lemma~\ref{lem::SanovT::RelEntropy} with $r=1$, $Y=\Csp{\Ti}$).
\begin{lemma}
\label{thm::LDPLN::Inde}
		The family $\left\{ L^{N},  P^{I,N,\Pi \left( Q \right)} \right\}$ satisfies on $ \MOne{\TWC}$ a large deviation principle with rate function
		\begin{align}
			I^{Q} \left( \Gamma \right)
			=
			\relE{ \Gamma } { P^{I, \Pi \left( Q \right) } }
		\end{align}
		for $\Gamma \in \MOneL{\TWC}$ and infinity otherwise.
\end{lemma}

\subsubsection{The interacting system}

As in Chapter~\ref{sec::LDPempP::Inter}, we show at first the following local version of the LDP.
\begin{lemma}
\label{thm::LDPLN::Local}
	Under the same assumptions as in Theorem~\ref{thm::LDPempP} 
	and for each  $\overline{Q} \in \MOne{\TWC}$,
	the following statements are true.
	\begin{enuRom}
	\item
	\label{thm::LDPLN::Local::LowerBound}
		For all open neighbourhoods $V \subset \MOne{\TWC}$ of $\overline{Q} $
		\begin{align}
		\label{thm::LDPLN::Local::eq::LowerBound}
			\liminf_{N \rightarrow \infty} N^{-d} \log P^{N} \left[ L^{N} \in V \right]
			\geq
			- I \left(  \ol{Q} \right)
			.
		\end{align}
	\item
	\label{thm::LDPLN::Local::UpperBound}
		For each $\gamma>0$, there is an open neighbourhood $V \subset \MOne{\TWC}$ of $\overline{Q}  $
		such that
		\begin{align}
		\label{thm::LDPLN::Local::eq::UpperBound}
			\limsup_{N \rightarrow \infty} N^{-d} \log P^{N} \left[ L^{N} \in V \right]
			\leq
					\begin{cases}
							- I\left(  \ol{Q} \right) + \gamma
							\quad&\textnormal{if }   I \left(  \ol{Q} \right)  < \infty
							\\
							- \gamma
							&\textnormal{otherwise.}
					\end{cases}
		\end{align}
	\end{enuRom}
\end{lemma}

The Lemma~\ref{thm::LDPLN::Local} can be proven as the Theorem~\ref{thm::LDPempP::Local}.
This proof requires Lemma~\ref{thm::LDPLN::Inde} and the following exponential bound instead of Lemma~\ref{lem::LDPempP::Inter::OutCemR+Initial}..
\begin{lemma}
\label{lem::LDPLN::OutCemR+Initial}
	For all $s >0$, there is a $R=R_{s}>0$, such that for all $N \in \N$
	\begin{align}
		\sup_{\ul{\w}^{N} \in \Wsp^{\Nd}}
		P^{N}_{\ul{\w}^{N}}  \left[  L^{N} \not \in \calM_{\varphi,R} \right]
		\leq 
		e^{-\Nd s}
		.
	\end{align}
\end{lemma}
The Lemma~\ref{lem::LDPLN::OutCemR+Initial} follows directly from Lemma~\ref{lem::LDPempP::Inter::OutCemR+Initial}, because $L^{N} \in \calM_{\varphi,R}$ if and only if $\Pi \left( L^{N} \right)_{\Ti} \in \Cem_{\varphi,R}$, i.e.
	\begin{align}
		P^{N}_{\ul{\w}^{N} } \left[  L^{N} \in \calM_{\varphi,R} \right]
		= 
		P^{N}_{\ul{\w}^{N} } \left[  \empP \in \Cem_{\varphi,R} \right]
		.
	\end{align}

Then Lemma~\ref{thm::LDPLN::Local}
and Assumption~\ref{ass::LNPNexptight} imply the lower and upper large deviation bound with the good rate function $I$.
Indeed we show in the next lemma that $I$ is a good rate function.
This finishes the proof of the Theorem~\ref{thm::LDPLN}.

\begin{lemma}
\label{lem::LDPLN::GoodRF}
	The function $Q \mapsto I \left( Q \right) $ is a good rate function.
\end{lemma}

\begin{proof}
We show at first that the level set $\calL^{\leq s} \left( I \right)  $ is relatively compact and then that it is closed.
\begin{steps}
	\step[{$\calL^{\leq s} \left( I \right)  $ is relatively compact}]
					
	By Assumption~\ref{ass::LNPNexptight} and Lemma~\ref{lem::LDPLN::OutCemR+Initial} we know that there is a compact set $\calK_{s+\epsilon} \subset \calM_{\varphi,R}$, for $R>0$ large enough, such that \eqref{thm::LDPLN::ExpBound::eq} holds.
	We claim that $\calL^{\leq s} \left( I \right)  \subset \calK_{s+\epsilon}$.
	Assume that there is a $Q \in \calL^{\leq s} \left( I \right) $ that is not in $\calK_{s+\epsilon}$.
	Then we know by \eqref{thm::LDPLN::ExpBound::eq}  and
	Theorem~\ref{thm::LDPLN::Local}~\ref{thm::LDPLN::Local::LowerBound} (because $ \MOne{\TWC}  \backslash \calK_{s+\epsilon}$ is an open neighbourhood of $Q$), 
	that $s+\epsilon \leq I \left( Q \right)$, a contradiction.
	
	\emptyline		
	
	\step[{$\calL^{\leq s} \left( I \right)  $ is closed}]
	
	By the definition of $I$ and the previous step,
	 $\calL^{\leq s} \left( I \right)  \subset \calK_{s+\epsilon} \cap \MOneL{\TWC}$.
	Fix an arbitrary converging sequence $\left\{ Q^{(n)} \right\}_{n} \subset  \calL^{\leq s} \left( I \right)$. The limit point $Q^{*}$ of this sequence is in $\calK_{s+\epsilon} \cap \MOneL{\TWC}$.
	We prove now that $Q \in  \calL^{\leq s} \left( I \right)$.
	
	This follows if we knew that for all $\left( x, \w \right) \in \TW$, $P^{I,\Pi \left( Q^{(n)} \right) }_{x,\w} \rightarrow P^{I,\Pi \left( Q^{*} \right) }_{x,\w}$.
	Indeed, this implies that also $\dd x \otimes \zeta_{x} \left( \dd \w \right) \otimes P^{I,\Pi \left( Q^{(n)} \right) }_{x,\w} \rightarrow \dd x \otimes \zeta_{x} \left( \dd \w \right)  \otimes P^{I,\Pi \left( Q^{*} \right) }_{x,\w}$.
	Then we conclude the lower semi-continuity of $I$, from the lower semi-continuity of the relative entropy in both variables.
	
	The convergence of $P^{I,\Pi \left( Q^{(n)} \right) }_{x,\w}$ follows from \cite{StrVarMultidi} Theorem~11.1.4.
	This theorem is applicable if for each $M \in \R$
	\begin{align}
	\label{eq::pf::lem::LDPLN::GoodRF::ConvSV}
		\lim_{n \rightarrow \infty} \int_{0}^{T} \sup_{\abs{\theta} \leq M} 
										\abs{ b \left( x, \w, \theta, \Pi \left( Q^{(n)} \right)_{t} \right) - b \left( x, \w, \theta, \Pi \left(  Q^{*} \right)_{t} \right) }
									\dd t
						= 
						0
			.
	\end{align}
	This convergence follows if
			\begin{align}
			\label{eq::pf::lem::LDPLN::GoodRF::ConvSVsup}
			\sup_{t \in \Ti} 
			\sup_{\abs{\theta} \leq M} 
			\abs{ b \left( x, \w, \theta, \Pi \left( Q^{(n)} \right)_{t}  \right) - b \left( x, \w, \theta, \Pi \left( Q^{*} \right)_{t}  \right) }
			\rightarrow
			0
			.
			\end{align}
		The function		
		$\TTWR \times \left( \calM_{\varphi,R} \cap \MOneL{\TWC} \right) \ni \left( t, x, \w, \theta, Q \right)  \mapsto  b \left( x, \w, \theta, \Pi \left(  Q \right)_{t} \right)$ is continuous as composition of continuous function (Assumption~\ref{ass::LDPempP}~\ref{ass::LDPempP::bCont}).
		Moreover only the compact set $\Ti$, $\abs{\theta}\leq M$ and $Q^{(n)},Q^{*} \in \calK_{s+\epsilon}$ is considered in  \eqref{eq::pf::lem::LDPLN::GoodRF::ConvSVsup}.
		Therefore we conclude \eqref{eq::pf::lem::LDPLN::GoodRF::ConvSVsup} from the uniform convergence of $b$ on this set.
	
\end{steps}\vspace{-\baselineskip}
\end{proof}

\subsection{The concrete example \texorpdfstring{\protect\eqref{SDE::LocalMF}}{(\ref{SDE::LocalMF})} of a local mean field model}
\label{sec::LDPLN::LMF}

In this chapter we consider the large deviation principle for the family $\left\{L^{N}, P^{N} \right\}$ of the local mean field model, defined by $\sigma =1 $ and the drift coefficient \eqref{SDE::LocalMF} with Assumption~\ref{ass::LMF::Init}, Assumption~\ref{ass::LMF::Init::Integral}, Assumption~\ref{ass::LMF::Medium}, Assumption~\ref{ass::LMF::J} and Assumption~\ref{ass::LMF::Psi}.

Choose $\varphi=\theta^{2}+1$. 
We know by Chapter~\ref{sec::LocalMF}, that the Assumption~\ref{ass::LDPempP} is satisfied. 
Hence by  Theorem~\ref{thm::LDPLN}, the empirical measure $L^{N}$ of the local mean field model, satisfies a large deviation principle, provided that the 
the exponential tightness Assumptions~\ref{ass::LNPNexptight} holds.
We claim the exponential tightness in the next lemma, which we prove in Chapter~\ref{sec::LDPLN::LMF::PfExpTight}.
		\begin{lemma}
			\label{lem::LDPLN::LMF::expTight}
			For the concrete example \eqref{SDE::LocalMF} of a local mean field model, the family $\left\{ L^{N},P^{N} \right\}$ is exponentially tight, i.e. the Assumptions~\ref{ass::LNPNexptight} is satisfied.
		\end{lemma}

The measure $P^{I, \Pi \left( Q \right) }_{x,\w}$ is for each $\left( x, \w \right) \in \TW$ the law of the following one dimensional SDE
\begin{align}
\label{eq::LDPLN::LMF::SDE::FixedInterMeasure}
\begin{split}
	\dd \what\theta^{x}_{t} &= \left(-\Psi \left( \what\theta^{x}_{t} , \w \right)  + \beta \left( x, \w, \Pi \left( Q \right)_{t}  \right)   \right)  \dd t + \dd B_{t}
	\\
		\what\theta^{x}_{0} &\sim \nu_{x} 
	,
\end{split}
\end{align}
where the function $\beta : \TW \times \left\{ \mu \in \MOne{ \TWR } : \int_{\TWR} \theta^{2} \mu \left( \dd x, \dd \w, \dd \theta \right) < \infty \right\}   \rightarrow \R$ is defined in \eqref{eq::step::LocalMF::bCont::beta}.
We interpret $\beta\left( x, \w, \mu \right) $ as the effective field corresponding to the measure $\mu \in \MOne{\TWR}$ at the spatial position $x \in \Td$ with fixed environment $\w \in \Wsp$.
If we inserted $\empPt{t}$ in $\beta$,  the drift coefficient of the SDE \eqref{eq::LDPLN::LMF::SDE::FixedInterMeasure} would equal the drift coefficient of the SDE \eqref{SDE::LocalMF}.
Hence the rate function $I$ (defined in \eqref{eq::thm::LDPLN::RF}) measures the deviation of $Q$ from the measure of the solution to the SDE with effective field $Q$.

\emptyline

Now we show that this rate function has another representation, in which the influence of the entropy and of the interaction becomes obvious.
We need the following notation.
\begin{notation}
\label{nota::LDPLN::LMF::MeasuresW}
\begin{itemNoLeftIntend}
\item
	We denote by $W^{0}_{x}$ the law of a Brownian motion with initial distribution $\nu_{x}$.
\item
	We use the symbol $W^{-\Psi}_{x,\w}$ for the law of the solution of the SDE with drift coefficient $-\partial_{\theta} \Psi \left( ., w \right)$ and with initial distribution $\nu_{x}$, for $\w \in \Wsp$.
\item
	With these measures we define the products of these measures $W^{N,0}_{\ul{\w}^{N}}, W^{N,\Psi}_{\ul{\w}^{N}} \in \MOne{ \Csp{\Ti}^{\Nd} }$ and 
	the product with $\zeta^{N}$ by $W^{N,0}, W^{N,\Psi}  \in \MOne{ \Wsp^{\Nd} \times \Csp{\Ti}^{\Nd} }$ similar as in Notation~\ref{nota::LMF::General}.
	
\end{itemNoLeftIntend}
\end{notation}
\begin{remark}
All these measures exist, because the corresponding Martingale problems are well posed (by the Assumption~\ref{ass::LMF::Psi}). 
Note that the $\Nd$ diffusion processes described by $W^{N,0}_{\ul{\w}^{N}}$  and $W^{N,-\Psi}_{\ul{\w}^{N}}$ do not interact.
\end{remark}

\begin{theorem}
	\label{thm::LDPLN::LMF::LDP}
	If Assumption~\ref{ass::LMF::Init}, Assumption~\ref{ass::LMF::Init::Integral}, Assumption~\ref{ass::LMF::Medium}, Assumption~\ref{ass::LMF::J} and Assumption~\ref{ass::LMF::Psi} hold,
	then the family $\left\{ L^{N}, P^{N} \right\}$ satisfies on $\MOne{\TWC}$ a large deviation principle with good rate function 
	\begin{align}
	\label{eq::thm::LDPLN::LMF::LDP::RF}
		\ol{ I } \left( Q \right)  
		=
		\begin{cases}
			 \relE{ Q }{ \dd x \otimes \zeta_{x} \left( \dd \w \right) \otimes W^{-\Psi}_{x,\w} } - F \left( Q \right) 
			\quad &\textnormal{if } Q \in \calM_{\varphi,\infty} \cap \MOneL{\TWC}
			\\
			\infty 
			&\textnormal{otherwise,}
		\end{cases}
	\end{align}
with
	\begin{align}
	\label{eq::thm::LDPLN::LMF::LDP::defF}
	\begin{split}
	F \left( Q \right) 
	\defeq&
		-
		 \frac{1}{2} \int \int
		 \int J \left( x''-x , \w'' , \w \right)  J \left( x''-x', \w'' , \w' \right)  Q \left( \dd x'', \dd \w'', \dd \eta_{\Ti}  \right)  \int_{0}^{T} \theta'_{t} \theta_{t} \dd t
		\\
		&\mkern350mu
	 	Q \left( \dd x', \dd \w', \dd \theta'_{\Ti}  \right) 
		Q \left( \dd x, \dd \w, \dd \theta_{\Ti}  \right) 
		\\
		&
		+
		\frac{1}{2}
		\int \int
		J \left( x-x' ,\w, \w' \right) 
		\left[ 
		\theta_{T} \theta'_{T} 
		-
		\theta_{0} \theta'_{0} 
		\right] 
		Q \left( \dd x, \dd \w, \dd \theta_{\Ti}  \right)  Q \left( \dd x', \dd \w', \dd \theta'_{\Ti}  \right)
		,
	\end{split}
	\end{align}
where the integrals $\int$ are over the space $\TWC$.	

By the uniqueness of the rate function of a large deviation system, $\ol{I} = I$ (for $I$ defined in Theorem~\ref{thm::LDPLN}).
\end{theorem}

\subsubsection{Proof of Theorem~\ref{thm::LDPLN::LMF::LDP}}
\label{sec::LDPLN::LMF::LDP}

\begin{proof}
We know already by Theorem~\ref{thm::LDPLN} and Lemma~\ref{lem::LDPLN::LMF::expTight}, that $L^{N}$ satisfies a large deviation principle with rate function $I$ defined in \eqref{eq::thm::LDPLN::RF}.
Hence we only have to show that $I$ equals $\ol{I}$.

Note that $I \left( Q \right)= \infty$ and $\ol{I} \left( Q \right) = \infty$, if $Q \not \in \calM_{\varphi,\infty} \cap \MOneL{\TWC}$ or if not $Q << \dd x \otimes \zeta_{x} \left( \dd w \right) \otimes P^{I, \Pi \left( Q \right) }_{x,\w}$.
Indeed if $Q \in  \calM_{\varphi,\infty} \cap \MOneL{\TWC}$, then $Q << \dd x \otimes \zeta_{x} \left( \dd w \right) \otimes W^{-\Psi}_{x,\w}$ if and only if $Q << \dd x \otimes \zeta_{x} \left( \dd w \right) \otimes P^{I, \Pi \left( Q \right) }_{x,\w}$.
This is the case because $\left(t, x,\w \right) \mapsto \beta \left(x, \w, \Pi \left( Q \right)_{t}  \right)$ is uniformly bounded (see \eqref{eq::step::LocalMF::MuIntCont::CemL::Bound}).
Moreover $F \left( Q \right)$ is bounded  for such a $Q$ (because $\ol{J} \in \Ltwo \left( \Td \right)$ by Assumption~\ref{ass::LMF::J}).

Hence Theorem~\ref{thm::LDPLN::LMF::LDP} follows if for all $Q \in \calM_{\varphi,\infty} \cap \MOneL{\TWC}$ with $Q<<\dd x \otimes \zeta_{x} \left( \dd \w \right) \otimes W^{-\Psi}_{x,\w}$,  $I \left( Q \right)$ equals $\ol{I} \left( Q \right)$.
For such a $Q$, $F$ has the following different representation.
\begin{lemma}
\label{lem::LDPLN::LMF::F2ndRep}
	For $Q \in \calM_{\varphi,\infty} \cap \MOneL{\TWC}$ with $Q<<\dd x \otimes \zeta_{x} \left( \dd \w \right) \otimes W^{-\Psi}_{x,\w}$,
	\begin{align}
		F\left( Q \right)  
		=
			\int_{\TWC} \log \frac{ \dd P^{I, \Pi \left( Q \right) }_{x,\w } }{ \dd W^{-\Psi}_{x,\w}} \left( \theta_{\Ti}  \right) Q \left( \dd x, \dd \w, \dd \theta_{\Ti}  \right)
		.
	\end{align}
\end{lemma}

From this lemma, we immediately infer the equality of $I$ and $\ol{I}$ and hence Theorem~\ref{thm::LDPLN::LMF::LDP}.
\end{proof}

\begin{proof}[of Lemma~\ref{lem::LDPLN::LMF::F2ndRep}]
	Fix a $Q \in \calM_{\varphi,\infty} \cap \MOneL{\TWC}$ with $Q << \dd x \otimes \zeta_{x} \left( \dd \w \right) \otimes W^{-\Psi}_{x,\w}$.
	By the well-posedness of the Martingale problems, the measures $P^{I, \Pi \left( Q \right) }_{x,\w }$ and $W^{-\Psi}_{x,\w }$ are equivalent.
	By the Girsanov theorem, their Radon-Nikodym derivative can be written as
	\begin{align}
	\label{eq::pf::lem::LDPLN::LMF::F2ndRep::Gir}
		\log \frac{ \dd P^{I, \Pi \left( Q \right) }_{x,\w} }{ \dd W^{-\Psi}_{x,\w}  } \left( \theta_{\Ti}  \right) 
		=
			-
			\oh \int_{0}^{T} \left( \beta \left( x, \w, \Pi \left( Q \right)_{t}  \right) \right)^{2} \dd t
			+
			\int_{0}^{T} \beta \left( x, \w, \Pi \left( Q \right)_{t}  \right) \dd \theta_{t} 
			\eqdef
			\encircle{1}+\encircle{2}
		.
	\end{align}
	Integrating $\encircle{1}$ w.r.t $Q$ we get the first term in $F$.
	We show now that $\encircle{2}$ leads to the second term of $F$.
	Therefore remark at first that
	\begin{align}
	\label{eq::pf::lem::LDPLN::LMF::F2ndRep::1}
	\begin{split}
		&\int \encircle{2} \; Q\left( \dd x, \dd \w, \dd \theta \right) 
		=
		\int \int  J\left( x-x' , \w, \w' \right)  
						\int_{0}^{T} \theta'_{t}  \dd \theta_{t}  \quad
					Q\left( \dd x, \dd w, \dd \theta_{\Ti}  \right)  Q\left( \dd x', \dd w', \dd \theta'_{\Ti} \right) 
		\\
		&=
		\oh \int \int  J\left( x-x' , \w, \w' \right)  
				\left( 	\int_{0}^{T} \theta'_{t}  \dd \theta_{t} 
				+ \int_{0}^{T} \theta_{t}  \dd \theta'_{t} 
			\right)
		Q\left( \dd x, \dd w, \dd \theta_{\Ti}  \right)  Q\left( \dd x', \dd w', \dd \theta'_{\Ti} \right) 
		,
	\end{split}
	\end{align}
	where we use that $J$ is an even function. Integrals without integration bounds integrate over the space $\TWC$.
		The stochastic integrals are well defined because $\theta_{\Ti} $ is a $Q_{x,\w}$-semimartingale, because $Q_{x,\w} << W^{-\Psi}_{x,\w}<< W^{0}_{x}$ for almost all $\left( x,\w \right) \in \TW$.
	By integration by parts formula for the \Ito~integral, 
	\eqref{eq::pf::lem::LDPLN::LMF::F2ndRep::1} equals to the second summand on the right hand side of \eqref{eq::thm::LDPLN::LMF::LDP::defF}.

\end{proof}

\subsubsection{Preliminaries}

In this chapter we state and prove some results, that we need in the proof of Lemma~\ref{lem::LDPLN::LMF::expTight}.

\begin{lemma}
\label{lem::LDPLN::LMF::WPsiFeller}
	The map $\TW \ni \left( x,\w \right) \mapsto W^{-\Psi}_{x,\w} \in \MOne{\Csp{\Ti}}$ is Feller continuous.
\end{lemma}
\begin{proof}
By similar estimates as in Lemma~\ref{lem::LDPempP::indep::Feller}, Lemma~\ref{lem::SanovT::ass::FellerWithAddx} and \eqref{eq::pf::lem::SanovT::ass::Init::Convergence::Hf}, we get the the Feller continuity.
This requires the Assumption~\ref{ass::LMF::Psi} and Assumption~\ref{ass::LMF::Init}.
\end{proof}

\begin{lemma}
\label{lem::LDPLN::LMF::eTheta2Bound}
For all $\kappa<c_{\Psi}$, there is a constant $C_{\kappa}>0$ such that 
	\begin{enuRom}
	\item
	\label{lem::LDPLN::LMF::eTheta2Bound::T}
		\hfill $\begin{aligned}
				\sup_{x\in \Td} \sup_{\w \in \Wsp}
				E_{W^{-\Psi}_{x,\w} } \left[ e^{ \kappa \left[ \left( \theta_{T}  \right)^{2}  + \left( \theta_{0}  \right)^{2}  \right]} \right]
				<
				C_{\kappa}
				\quad\textnormal{and}
		\end{aligned}$
		\hfill \mbox{}
	\item
	\label{lem::LDPLN::LMF::eTheta2Bound::Int0T}
	\hfill
		$\begin{aligned}
			\sup_{x\in \Td} \sup_{\w \in \Wsp}
			E_{W^{-\Psi}_{x,\w} } \left[ e^{ \kappa \int_{0}^{T} \left( \theta_{t} \right)^{2} \dd t} \right]
			< 
			C_{\kappa}
			.
		\end{aligned}$
		\hfill \mbox{}
	\end{enuRom}
\end{lemma}

\begin{proof}
\begin{enuRomNoIntendBf}
\item
Fix arbitrary $\left( x, \w \right) \in \TW$.

By the Girsanov theorem and  \Ito's lemma we have
\begin{align}
\label{eq::GirsaPsiTo0}
	\frac{ \dd W^{- \Psi}_{x,\w} }{ \dd W^{0}_{x} }
	=
		e^{ \Psi \left( \theta_{0} , \w  \right)  
			-
			\Psi \left( \theta_{T} , \w  \right)  
			+ 
			\oh \int_{0}^{T} \partial^{2}_{\theta^{2}} \Psi \left( \theta_{t}  , \w \right)  \dd t 
			-
			\oh \int_{0}^{T} \left( \partial_{\theta} \Psi \left( \theta_{t}  , \w \right)  \right)^{2} \dd t
		 }
	\quad.
\end{align}
By $\ol{\Psi}$ being a polynomial of even degree (Assumption~\ref{ass::LMF::Psi})
and by $\Wsp$ being compact, 
the following upper bound on the  Radon-Nikodym derivative holds
\begin{align}
		\frac{ \dd W^{- \Psi}_{x,\w} }{ \dd W^{0}_{x} }
	\leq
		e^{ \Psi \left( \theta_{0}, \w  \right)  
					-
					\Psi \left( \theta_{T} , \w  \right)  
					+ 
					TC
				 }
	.
\end{align}
Therefore 
\begin{align}
\label{eq::EPsiBoundeT}
\begin{split}
	E_{W^{-\Psi}_{x,\w}} \left[ e^{ \kappa \left[ \left( \theta_{T}  \right)^{2}  + \left( \theta_{0}  \right)^{2}  \right]} \right]
	&\leq
	e^{TC} E_{W^{0}_{x}} \left[ e^{ \kappa \left( \theta_{T}  \right)^{2} - \Psi \left( \theta_{T} , \w \right) } 
													e^{\kappa \left(  \theta_{0}  \right)^{2}  + \Psi \left( \theta_{0} ,\w  \right) } 
											\right]
	\\
	&\leq	
	e^{TC} e^{C}  \int_{\R}  e^{\kappa \left(  \theta  \right)^{2}  + \ol \Psi \left( \theta  \right) + \w_{1} \theta }   \nu_{x} \left( \dd \theta \right)
	,
\end{split}
\end{align}
where we use Assumption~\ref{ass::LMF::Psi}, $\Wsp$ being compact and $\kappa<c_{\Psi}$ in the second inequality.
The right hand side of \eqref{eq::EPsiBoundeT} is bounded by a constant uniformly in $\left( x, \w \right) \in \TW$, by Assumption~\ref{ass::LMF::Init::Integral}, Assumption~\ref{ass::LMF::Psi} and by $\Wsp$ being compact.

\item

By Assumption~\ref{ass::LMF::Psi}, the Radon-Nikodym derivative in \eqref{eq::GirsaPsiTo0} is also be bounded by
\begin{align}
	\frac{ \dd W^{- \Psi}_{x,\w} }{ \dd W^{0}_{x} }
	\leq
		e^{ \Psi \left( \theta_{0},\w  \right)  
				+
					C
					-
					\int_{0}^{T} 
							 c \left( \theta_{t} \right)^{2} 
					\dd t 
				 }
	,
\end{align}
for constants $c \in (0,c_{\Psi})$ and $C=C  \left( c \right) >0$.
Using this bound, we get 
\begin{align}
\label{eq::LDPLN::LMF::EPsiBoundeInt0T}
\begin{split}
	E_{W^{-\Psi}_{x,\w}} \left[ e^{ \kappa \int_{0}^{T} \left( \theta_{t} \right)^{2} \dd t} \right]
	&\leq
			e^{C} 
			E_{W^{0}_{x}} 
				\left[
				e^{   \int_{0}^{T} 
							\left( \kappa - c \right)  \left( \theta_{t} \right)^{2} 
						\dd t } 
				e^{\Psi \left( \theta_{0} , \w  \right) }
				 \right]
	\leq
	e^{C} 
	\int_{\R} e^{ \ol \Psi \left( \theta  \right) + \w_{1} \theta  } \nu_{x} \left( \dd \theta \right)
	.
\end{split}
\end{align}
The right hand side is bounded by a constant uniformly in $\left( x, \w \right) \in \TW$ by Assumption~\ref{ass::LMF::Init::Integral}, Assumption~\ref{ass::LMF::Psi} and by $\Wsp$ being compact.

\end{enuRomNoIntendBf}\vspace{-\baselineskip}
\end{proof}

Now we derive the Radon-Nikodym derivative between $P^{N}_{\ul{\w}^{N}}$ and $W^{N,-\Psi}_{\ul{\w}^{N}}$ by using the Girsanov theorem.
	\begin{lemma}
	\label{lem::LDPLN::LMF::PNtoWN}
	For $L^{N} \in \MOne{ \TWC}$, defined by $\ul{\w}^{N} \in \Wsp^{\Nd}$ and $\ul{\theta}^{N}_{\Ti} \in \Csp{\Ti}^{\Nd}$,
		\begin{align}
			\frac{\dd P^{N}_{\ul{\w}^{N}}}
				{\dd W^{N,- \Psi}_{\ul{\w}^{N}}}
			 \left( \ul{\theta}^{N}_{\Ti}  \right) 
			=
				e^{
					\Nd F \left( L^{N}   \right) 
					-
					\oh \frac{T}{N} \sum_{i \in \TN} J \left( 0, \w^{i,N}, \w^{i,N} \right)
				}
			.
		\end{align}
	\end{lemma}

\begin{proof}[of Lemma~\ref{lem::LDPLN::LMF::PNtoWN}]
Fix $\ul{\w}^{N} \in \Wsp^{\Nd}$.
To shorten the notation, we define for $\ul{\theta}^{N} \in \RN$
\begin{align}
\label{def::LDPLN::Interaction}
	 B^{N}_{\ul{\w}^{N}} \left( \ul\theta^{N} \right)  
	 \defeq
	    \frac{1}{2N^d} 
	    	\sum_{i,j\in \TN} J\left(\frac{i - j}{N}, \w^{i,N}, \w^{j,N} \right)  \theta^{i,N} \theta^{j,N}
	.
\end{align}
By the Girsanov theorem and $J$ being even,
\begin{align}
\label{eq::lem::LDPLN::LMF::PNtoWN::Girs}
		\log \left( 
			\frac{\dd P^{N}_{\ul{\w}^{N}}}
			{\dd W^{N,-\Psi}_{\ul{\w}^{N}}}
		 \left( \ul{\theta}^{N}_{\Ti}  \right) 
			\right)
		=
		 -\frac{1}{2}  \int_{0}^{T} 
			 \underbrace{ \sum_{i \in \TN} 
							 \left( \partial_{\theta^{i,N}_{t}} B^{N}_{\ul{\w}^{N}} \left( \ul\theta^{N}_{t}  \right)  \right) ^{2} 
							 }_{\encircle{1}}
			 	\dd t
		+
		\underbrace{ \sum_{i \in \TN} \int_{0}^{T} \partial_{\theta^{i,N}_{t}} B^{N}_{\ul{\w}^{N}} \left( \ul\theta^{N}_{t}  \right)   \dd \theta^{i,N}_{t} 
							}_{\encircle{2}}
	.
\end{align}
The first summand of \eqref{eq::lem::LDPLN::LMF::PNtoWN::Girs} equals the first summand of $\Nd F \left( L^{N} \right) $, because for each $t \in \Ti$
\begin{align}
\label{eq::lem::LDPLN::LMF::PNtoWN::FirstSum}
\begin{split}
	\encircle{1}
	&=
	\frac{1}{\Nd} \sum_{j,k \in \TN}  
		\left( \frac{1}{\Nd} \sum_{i \in \TN} J \left( \frac{i-j}{N} , \w^{i,N}, \w^{j,N} \right)  J \left( \frac{i-k}{N} , \w^{i,N}, \w^{k,N} \right)  \right) 
		\theta^{j,N}_{t}  \theta^{k,N}_{t} 
	\\
	&= 
	\Nd \int\int
				\quad
				\theta'_{t} \theta_{t} 
				\quad
				\int
				 J \left( x''-x, \w'', \w \right)  J \left( x''-x', \w'' , \w' \right)  L^{N} \left( \dd x'', \dd \w'', \dd \eta_{\Ti} \right) 
	\\
	&\mkern350mu	
			L^{N} \left( \dd x, \dd \w, \dd \theta_{\Ti} \right)  L^{N} \left( \dd x', \dd \w', \dd \theta'_{\Ti} \right) 
	,
\end{split}
\end{align}
where the integrals in the last line are over the sets $\TWC$.

For \encircle{2} we apply \Ito's lemma.
Under $W^{N,-\Psi}_{\ul{\w}^{N}}$, the $\theta^{k,N}_{\Ti}$ is a  \Ito~process with drift coefficient $-\partial_{\theta} \Psi \left( ., \w^{k,N} \right)$, for  each $k \in \TN$.
Hence
\begin{align}
\begin{split}
\encircle{2}
	&=
		B^{N}_{\ul{\w}^{N}} \left( \ul\theta^{N}_{T}  \right) -B^{N}_{\ul{\w}^{N}} \left( \ul\theta^{N}_{0}  \right) 
		-
		\frac{1}{2} 
		\sum_{i \in \TN} \int_{0}^{T}  \partial^{2}_{\left( \theta^{i,N}\right)^{2}} B^{N}_{\ul{\w}^{N}} \left( \ul\theta^{N}_{t}  \right)  \dd t
	\\
	&=
		B^{N}_{\ul{\w}^{N}} \left( \ul\theta^{N}_{T}  \right) -B^{N}_{\ul{\w}^{N}} \left( \ul\theta^{N}_{0}  \right) 
		-
		\frac{1}{2}\sum_{i \in \TN} T  \frac{J \left( 0, \w^{i,N}, \w^{i,N} \right) }{\Nd}
	.
\end{split}
\end{align}
Using \eqref{def::LDPLN::Interaction}, we conclude that \encircle{2} is equal to the second summand  of $F \left( L^{N} \right)$.

\end{proof}

\subsubsection{Proof of the exponential tightness (Lemma~\ref{lem::LDPLN::LMF::expTight})}
\label{sec::LDPLN::LMF::PfExpTight}

\begin{proof}[of Lemma~\ref{lem::LDPLN::LMF::expTight}]
To show that the family $\left\{ L^{N}, P^{N} \right\} $ is exponential tight, we first construct compact sets $K^{\ell} \subset \MOne{ \TWC } $ for which the family $\left\{ L^{N}, W^{N,-\Psi} \right\} $ is exponential tight. 
Then we show that this leads to the exponential tightness of $\left\{ L^{N}, P^{N} \right\} $. 

\emptyline
\begin{steps}
\step[{Exponential tightness of  $\left\{ L^{N}, W^{N,-\Psi} \right\} $}]

To show the exponential tightness of  $\left\{ L^{N}, W^{N,-\Psi} \right\} $ we generalise Lemma~6.2.6 in \cite{DemZeiLarge}. 
In contrast to \cite{DemZeiLarge}, the measures 
$\calA \defeq \left\{ \zeta_{x} \left( \dd \w \right) \otimes W^{-\Psi}_{x,\w} \right\}_{x \in \Td} \subset \MOne{\Wsp \times \Csp{\Ti}}$ are not identically distributed, 
due to the dependency of the initial distribution and of the random environment on $x \in \Td$.

Therefore we show at first that $\calA$ is a tight set of measures.
Take an arbitrary sequence in $\calA$.
Then there is a sequence $\left\{ x_{n} \right\}_{n} \subset \Td$ such that the sequence is given by
 $\left\{ \zeta_{x_{n}} \left( \dd \w \right) \otimes W^{-\Psi}_{x_{n},\w} \right\}_{n}$.
This implies that there is a converging subsequence $x_{n_{k}} \rightarrow x^{*} \in \Td$ (due to the compactness of $\Td$).
By the continuity of $x \mapsto \zeta_{x} \left( \dd \w \right) \otimes W^{-\Psi}_{x,\w}$ (this could be shown as the continuity of \eqref{eq::pf::lem::SanovT::ass::Init::Convergence::Hf} by Assumption~\ref{ass::LMF::Medium} and Lemma~\ref{lem::LDPLN::LMF::WPsiFeller}), we get a converging subsequence.
Therefore $\calA$ is sequentially compact.
Moreover $\Wsp \times \Csp{\Ti}$ is a  separable metric space.
Then the Prokhorov's theorem implies that $\calA$ is tight.
\emptyline
The tightness of the set $\calA$ implies that  there is a compact set $\Gamma_{a} \subset \Wsp \times \Csp{\Ti}$ such that for all $x \in \Td$
\begin{align}
	\int_{\Wsp}
	\int_{\Csp{\Ti}}
				\1_{\left\{  \left( \w,\theta_{\Ti} \right) \not \in \Gamma_{a} \right\} } 
		W^{-\Psi}_{x_{n},\w} \left( \dd \theta_{\Ti} \right)
		\zeta_{x_{n}} \left( \dd \w \right)
	\leq
		e^{-2a^{2}}\left( e^{a}-1 \right)
		.
\end{align}
Now define $K^{a} \defeq \left\{ Q \in \MOne{ \TWC }  : Q \left( \Td \times \Gamma_{a} \right)  \geq 1- \frac{1}{a} \right\}$.
The sets $K^{a}$ are closed by the Portmanteau Lemma.
Moreover for each $A \in \N$ the sets 
\begin{align}
\label{eq::lem::LDPLN::LMF::expTight::KL}
	K_{A}
	\defeq
	 \bigcap_{a=A}^{\infty} K^{a}
\end{align}
 are compact by Prokhorov's theorem and the definition of $K^{a}$.
Then we get  
\begin{align}
\begin{split}
	&W^{N,-\Psi} \left[ L^{N} \not \in K^{a} \right]
	=
		W^{N,-\Psi} \left[ L^{N} \left( \Td \times \Gamma_{a}  \right)  > \frac{1}{a} \right]
	\\
		&\leq
		e^{-2 \Nd a } E_{W^{N,-\Psi}} \left[ e^{2 a^{2} \sum_{i \in \TN}
		 \1_{\left( \w, \theta^{i,N}_{\Ti} \right)  \not \in \Gamma_{a}}} \right]
	\\
	&\leq
		e^{-2 \Nd a }  \prod_{i \in \TN} 
				\left( 1
					+
					e^{2 a^{2}}
					\int_{\Wsp}
					\int_{\Csp{\Ti}}
						\1_{\left\{  \left( \w,\theta_{\Ti} \right) \not \in \Gamma_{a} \right\} } 
					W^{-\Psi}_{x_{n},\w} \left( \dd \theta_{\Ti} \right)
					\zeta_{x_{n}} \left( \dd \w \right) 
				\right)
	\\
	&\leq 
		e^{-2 \Nd a } 
		\left( 1+  e^{a} - 1  \right)^{\Nd}
	\leq
		e^{- \Nd a}
	.
\end{split}
\end{align}
This implies that
\begin{align}
\label{eq::lem::LDPLN::LMF::expTight::WNExpTight}
	W^{N,-\Psi} \left[ L^{N} \not \in K_{A} \right]
	\leq 
	2 e^{-\Nd A}
	,
\end{align}
hence the claimed exponential tightness.

\emptyline
\step[Transferring the exponential tightness]

We show now that $\left\{ L^{N}, P^{N} \right\} $ are also exponential tight with respect to the same sets $\Gamma_{a}$.
We get by the Radon-Nikodym derivative derived in Lemma~\ref{lem::LDPLN::LMF::PNtoWN} and by the \Holder inequality ($ \frac{1}{p}+\frac{1}{p'}=1$)
\begin{align}
\label{eq::pf::lem::LDPLN::LMF::eFoutOfNR::EpFirst}
\begin{split}
	E_{P^{N}}
			\left[ 
			\1_{L^{N} \not \in K_{A} } 
			\right]  
	&=
	E_{W^{N,-\Psi}}
		\left[ 
		e^{N^d F  \left( L^{N}  \right)  } \1_{L^{N} \not \in K_{A} } 
		\right]  
	e^{-\oh T \ol{J} \left( 0 \right) }
	\\
	&\leq
		e^{-\oh T \ol{J} \left( 0 \right) }
		E_{W^{N,-\Psi}}
		\left[ 
		e^{N^d p F  \left( L^{N}  \right)  } 
		\right]^{\frac{1}{p}}
		W^{N,-\Psi} \left[ L^{N} \not \in K_{A} \right]^{\frac{1}{p'}}
	.
\end{split}
\end{align}
Note that  $F \left( L^{N} \right)  $ is bounded from above by
\begin{align}
\label{eq::pf::lem::LDPLN::LMF::eFoutOfNR::BoundF}
	F\left( L^{N} \right)  
	\leq
		\oh \frac{1}{\Nd} 
		\left( \LpN{1}{\ol{J}}  +  \delta \right) 
		\sum_{i \in \TN} \left( \theta^{i,N}_{T} \right)^{2} 
		+
		\left( \theta^{i,N}_{0}  \right)^{2}
		,
\end{align}
for each $\delta>0$ when $N>N_{\delta}$ (see \eqref{eq::LocalMF::ass::BoundOnInteraction}).
When $p>1$ is not too large and $N$ large enough, $p \left( \LpN{1}{\ol{J}} + \delta \right) < c_{\Psi}$ (by Assumption~\ref{ass::LMF::Psi}).
Therefore we get by Lemma~\ref{lem::LDPLN::LMF::eTheta2Bound}~\ref{lem::LDPLN::LMF::eTheta2Bound::T},
\begin{align}
\label{eq::pf::lem::LDPLN::LMF::eFoutOfNR::EpFLNbound}
\begin{split}
	E_{W^{N,-\Psi}}
		\left[ 
		e^{N^d p F  \left( L^{N}  \right)  } 
		\right]
	&\leq
		\prod_{i \in \TN}
		\sup_{\w \in \Wsp}
		E_{W^{-\Psi}_{\frac{i}{N},\w }}
		\left[ 
		e^{p \left( \LpN{1}{\ol{J}} + \delta \right) \left(  \left( \theta_{T}  \right)^{2} + \left( \theta_{0}  \right)^{2} \right)}
		\right]
	\leq 
	C^{\Nd}
	.
\end{split}
\end{align}
Combining \eqref{eq::pf::lem::LDPLN::LMF::eFoutOfNR::EpFirst}, \eqref{eq::lem::LDPLN::LMF::expTight::WNExpTight} and \eqref{eq::pf::lem::LDPLN::LMF::eFoutOfNR::EpFLNbound}, we conclude
\begin{align}
	E_{P^{N}}
	\left[ 
	\1_{L^{N} \not \in K_{A} } 
	\right]  
	\leq
		C^{\Nd}
		e^{-\Nd  \frac{1}{p'} A}
	,
\end{align}
when $N$ and $A$ are large enough.

\end{steps}\vspace{-\baselineskip}
\end{proof}

\section{Comparison of the LDPs of the empirical measure and of the empirical process}
\label{sec::Comp}

In this chapter we state at first (Chapter~\ref{sec::Comp::Mini}) a one to one relation between the minimizer of the rate functions $I$ (of $\left\{L^{N},P^{N}\right\}$ derived in Theorem~\ref{thm::LDPLN}) and $S_{\nu,\zeta}$ (of $\left\{\empP, P^{N} \right\}$ derived in Theorem~\ref{thm::LDPempP}).
Then we explain how one can easily infer from the large deviation principle of the empirical measure $\left\{L^{N}\right\}$, the large deviation principle for the empirical process $\left\{\empP\right\}$ in  $\Cem$.
This follows by a simple application of the contraction principle (see Theorem~\ref{thm::LdpLNtoEmp}).
However the derived rate function does not have the expression $S_{\nu,\zeta}$ defined in \eqref{eq::thm::LDPempP::RF::ST}.
We show in Chapter~\ref{sec::Comp::UpperBoundSnu} that the derived rate function is at least an upper bound on $S_{\nu,\zeta}$.

\subsection{Relation between the minimiser of the rate function}
\label{sec::Comp::Mini}

We know by  Theorem~\ref{thm::RepS}~\ref{thm::RepS::First} and \eqref{eq::thm::LDPLN::RF} the following relation between $S_{\nu,\zeta}$ and $I$
	\begin{align}
	\label{eq::Comp::Mini::Relation}
	S_{\nu,\zeta} \left( \mu_{\Ti}  \right)
	=
	\inf_{ \substack{Q \in \MOne{ \TWC }  \\
			\Pi \left( Q \right)_{\Ti} =\mu_{\Ti} }
	}
	\relE{Q}{\dd x \otimes \zeta_{x} \left( \dd \w \right) \otimes P^{I,\mu_{\Ti}}_{x,\w}   }
	=
	\inf_{ \substack{Q \in \MOne{ \TWC }  \\
			\Pi \left( Q \right)_{\Ti} =\mu_{\Ti} }
	}
	I \left( Q \right)
	.
	\end{align}
We show in the next theorem a one to one relation between the minimizer of $I$ and $S_{\nu,\zeta}$.
Note that in general there can be two $Q,Q' \in \MOne{\TWC}$ with the same projection $\Pi \left( Q \right) = \Pi \left( Q' \right)$ and with $I \left( Q \right)= I \left( Q' \right)$.
However when $S_{\nu,\zeta} \left( \Pi \left( Q \right)  \right) =0$, then this is not the case.

\begin{theorem}
\label{thm::LDPLN::RelMinimaRF}

\begin{enuRom}
\item
\label{thm::LDPLN::RelMinimaRF::ItoS}
If $ I \left( Q \right) =0$, then $S_{\nu,\zeta} \left( \Pi \left(  Q \right)_{\Ti} \right)=0$.

\item
\label{thm::LDPLN::RelMinimaRF::StoI}
If  $S_{\nu,\zeta} \left( \mu_{\Ti} \right) =0$, then there is exactly one $Q \in \MOne{\TWC}$ with $\Pi \left( Q \right)_{\Ti} = \mu_{\Ti}  $ and $I \left( Q \right) =0$.
This $Q$ equals $\dd x \otimes \zeta_{x} \left( \dd w \right) \otimes P^{I,\mu_{\Ti}}_{x,\w}$.

\end{enuRom}
\end{theorem}

\begin{proof}

By \eqref{eq::Comp::Mini::Relation}, \ref{thm::LDPLN::RelMinimaRF::ItoS} is obviously satisfied.	

Now we show the opposite direction \ref{thm::LDPLN::RelMinimaRF::StoI}.
Fix a $\mu_{\Ti} \in \Cem$ with $S_{\nu,\zeta} \left( \mu_{\Ti} \right) =0$.
Then $Q \in \MOne{\TWC}$, with $\Pi \left( Q \right)_{\Ti} = \mu_{\Ti}  $ and $I \left( Q \right)$ implies that
$Q = Q^{*} \defeq \dd x \otimes \zeta_{x} \left( \dd w \right) \otimes P^{I,\mu_{\Ti}}_{x,\w}$.
This implies that there is at most one minimizer with $\Pi \left( Q \right)_{\Ti} = \mu_{\Ti}  $ and $I \left( Q \right)$.

Now we show that there exist an arbitrary $\ol{Q} \in \MOne{\TWC}$, with $\Pi \left( \ol{Q} \right)_{\Ti} = \mu_{\Ti}  $ such that $I \left( \ol{Q} \right) =0$.
This implies in particular that $\Pi \left( Q^{*} \right)_{\Ti} = \mu_{\Ti}  $.
By Chapter~\ref{sec::LDPempP::Inter::AssInterToInde} the results of Chapter~\ref{sec::LDPempP::Inde} hold for the SDE with fixed interaction $\mu_{\Ti}$.
		Then we get by the beginning of \ref{pf::lem::LDPempP::indep::UpperB::step::Minimiz} of the proof of Lemma~\ref{lem::LDPempP::indep::UpperB}, that there is a $Q^{*} \in \MOne{\TWC}$, with $\Pi \left( Q^{*} \right)_{\Ti} = \mu_{\Ti}  $ and with $I \left( Q^{*} \right) = 0$.

\end{proof}

\subsection{From the LDP of the empirical measure to the LDP of the empirical process}
\label{sec::Comp::LdpLNtoEmp}

In the following theorem, we derive the large deviation principle of the empirical process $\left\{\empP , P^{N} \right\} $ from the large deviation principle of the empirical measure $\left\{L^{N} , P^{N} \right\}$.
This is a simple application of the contraction principle.
This theorem requires only the large deviation principle of $\left\{ L^{N} \right\}$ (in contrast to the relation \eqref{eq::Comp::Mini::Relation} between the rate function).
However the rate function for the empirical processes is only described via a minimizing problem (see Chapter~\ref{sec::Comp::UpperBoundSnu} for a further discussion).

\begin{theorem}
\label{thm::LdpLNtoEmp}
If the assumptions of Theorem~\ref{thm::LDPLN::LMF::LDP} hold, then
the family of the empirical processes $\left\{ \empP, P^{N} \right\}$ satisfies on  $\Cem$  a large deviation principle with rate function
\begin{align}
\label{eq::thm::LdpLNtoEmp::DefjInf}
	j \left( \mu_{\Ti} \right)  
	\defeq
	\inf_{Q \in \MOne{ \TWC } : \Pi \left( Q \right)_{\Ti}  = \mu_{\Ti} } I\left( Q \right) 
	.
\end{align}

\end{theorem}

\begin{proof}
The family $\left\{ L^{N} ,P^{N} \right\}   $ satisfies by Theorem~\ref{thm::LDPLN::LMF::LDP} a LDP on $\MOne{ \TWC }$ with rate function $I$.
Moreover the map $\Pi  :  \MOne{ \TWC } \rightarrow \Cem $ is continuous (Lemma~\ref{lem::MapPi::Cont}).
Then the contraction principle implies the LDP of $\left\{\empP, P^{N} \right\}$ with the rate function $j$.
\end{proof}

\subsection{An upper bound on the rate function \texorpdfstring{$S_{\nu,\zeta}$}{Snu}}
\label{sec::Comp::UpperBoundSnu}

By Theorem~\ref{thm::LdpLNtoEmp}, $j$ is the rate function of the large deviation principle of $\left\{\empP, P^{N} \right\}$.
Moreover, by Theorem~\ref{thm::LDPempP} and the uniqueness of rate functions, $j$ has to be equal to $S_{\nu,\zeta}$ and $S^{\TW}_{\nu,\zeta}$.
We show now that $j$ is equal to $S^{\TW}_{\nu,\zeta}$ at least when $j$ is finite, without using the Theorem~\ref{thm::LDPempP}
(we need only Lemma~\ref{lem::LDPempP::indep::1Rep} and Lemma~\ref{lem::LDPempP::indep::UpperB}).
However $j$ is not everywhere finite (see also Remark~\ref{rem::LDPempP::indep::admissMu} for the concept of admissible flows).
Therefore this is only an upper bound on $S^{\TW}_{\nu,\zeta}$.
Nevertheless the upper bound on $S^{\TW}_{\nu,\zeta} $, implies at least a large deviation upper bound with $S^{\TW}_{\nu,\zeta} $ as rate function.
For the large deviation lower bound (and another proof of the upper bound) we refer to Chapter~\ref{sec::LDPempP}.

\begin{lemma}
\label{lem::LdpLNtoEmp::SuppBound}
Let the assumptions of Theorem~\ref{thm::LDPLN::LMF::LDP} hold.

	If $j \left( \mu_{\Ti} \right) < \infty$ for a $\mu_{\Ti} \in \Cem$, then $j \left( \mu_{\Ti} \right) = S^{\TW}_{\nu,\zeta} \left( \mu_{\Ti}  \right)$.
	
	In particular this implies $j \left( \mu_{\Ti}  \right) \geq S^{\TW}_{\nu,\zeta} \left( \mu_{\Ti}  \right) \geq S_{\nu,\zeta} \left( \mu_{\Ti}  \right)$.
\end{lemma}
 
 \begin{remark}
 In \cite{PraHolMKV} a proof of the equality between the counterparts of $j$ and $S^{\TW}_{\nu,\zeta}$ is  given.
 However in that proof the authors accidentally use a circular reasoning (in the equality (2.24) in \cite{PraHolMKV}). 
 We are also not able to prove the missing lower bound on $S^{\TW}_{\nu,\zeta}$,  without using Theorem~\ref{thm::LDPempP},.
 \end{remark}
 
 \begin{proof}[of Lemma~\ref{lem::LdpLNtoEmp::SuppBound}]
 Fix a $\mu_{\Ti} \in \Cem$ with $j \left( \mu_{\Ti} \right) < \infty$.
 Then there is a $R>0$, such that $\mu_{\Ti} \in \Cem_{\varphi,R}$, because there has to be a $Q \in \calM_{\varphi,\infty}$ with $I \left(Q\right) < \infty$ and $\Pi \left(Q\right)_{\Ti} = \mu_{\Ti}$.
By the same argument $\mu_{\Ti} \in \Cem^{L}$.
 
 Define $b^{I,\mu_{\Ti}} \left( t, x, \w, \theta \right) \defeq b \left( x, \w, \theta, \mu_{t}\right)$ as in Notation~\ref{nota::ForOtherRepS}.
 With this $b^{I,\mu_{\Ti}}$, we can define a system of independent SDEs as in \eqref{eq::SDEIndep}.
 This system satisfies the Assumption~\ref{ass::LDPempP::Inde} as shown in Chapter~\ref{sec::LDPempP::Inter::AssInterToInde}.
 Then the Lemma~\ref{lem::LDPempP::indep::1Rep} is applicable and we denote the rate function \eqref{eq::lem::LDPempP::indep::1Rep::S1Inf} by $S^{I,1,\mu_{\Ti}}_{\nu,\zeta}$, i.e.
 \begin{align}
 j  \left( \mu_{\Ti} \right)  
 = 
 \inf_{Q \in \MOne{ \TWC } : \Pi \left( Q \right)_{\Ti}  = \mu_{\Ti} } I\left( Q \right) 
 =
 S^{I,1,\mu_{\Ti}}_{\nu,\zeta} \left( \mu_{\Ti} \right) 
 .
 \end{align}
From this equality and Lemma~\ref{lem::LDPempP::indep::UpperB} (which is applicable for the same reasons), we conclude the Lemma~\ref{lem::LdpLNtoEmp::SuppBound}.

 \end{proof}

%%%%%%%%%%%%%%%%
%%%%%%%%%%%%%%%%
%%%%%%%%%%%%%%%%

\section[The LDP of the empirical measure for \texorpdfstring{\protect\eqref{SDE::LocalMF}}{(\ref{SDE::LocalMF})}, via a generalisation of Varadhan's lemma]{The LDP of the empirical measure for the concrete example \texorpdfstring{\protect\eqref{SDE::LocalMF}}{(\ref{SDE::LocalMF})} of a local mean field model via a generalisation of Varadhan's lemma}
\label{sec::LMF::LDPLN}

We show in Chapter~\ref{sec::LDPLN::LMF} that the family of empirical measure $\left\{L^{N}\right\}$ of the local mean field model \eqref{SDE::LocalMF}, satisfies a large deviation principle and we derive  two representations of the rate function (Theorem~\ref{thm::LDPLN} and Theorem~\ref{thm::LDPLN::LMF::LDP}).
In the proof of Theorem~\ref{thm::LDPLN} we use the same approach as in the proof of the large deviation principle of the empirical process $\left\{\empP\right\}$ in Chapter~\ref{sec::LDPempP}.
In particular we investigate the SDE with a fixed effective field and derive a LDP for this system. 
Then we infer from this LDP a LDP of the the SDE with interaction.
From Theorem~\ref{thm::LDPLN} we infer Theorem~\ref{thm::LDPLN::LMF::LDP}, by showing an equality of the two formulas of the rate function.

\emptyline

In this chapter we prove  Theorem~\ref{thm::LDPLN::LMF::LDP} by another approach.
We look at first at the SDEs with drift coefficient $-\Psi'$, i.e. at spin values that are distributed according to $W^{N,-\Psi}$ (defined in Notation~\ref{nota::LDPLN::LMF::MeasuresW}).
Then we apply the generalised Varadhan's Lemma (Theorem~\ref{vara::thm::GenVara}).
By the Laplace principle we infer finally the claimed large deviation principle with the rate function \eqref{eq::thm::LDPLN::LMF::LDP::RF}.

The Theorem~\ref{thm::LDPLN} can then be derived from Theorem~\ref{thm::LDPLN::LMF::LDP}. 
Indeed one only has to show the equality of the two representations of the rate function, which follows by the proof of Theorem~\ref{thm::LDPLN::LMF::LDP} given in in Chapter~\ref{sec::LDPLN::LMF::LDP}.

\begin{notation}
We fix $\varphi \left( \theta \right) =1+\theta^{2}$.
To simplify the notation we use $\calM_{R} $ and $\calM_{\infty}$ instead of $\calM_{\varphi,R} $, $\calM_{\varphi,\infty}$ in this chapter, i.e.
\begin{align}
\label{def::calMR}
	\calM_{R} 
	\defeq
		\left\{ Q \in \MOne{\TWC} : \sup_{ t \in \Ti } \int \left( \theta_{t}  \right)^{2} Q( \dd \theta_{\Ti} ) \leq R-1 \right\}
	,
\end{align}
and
\begin{align}
	\calM_{\infty} \defeq \bigcup_{R} \calM_{R} \subset \MOne{ \TWC } 
	,
\end{align}
equipped with the subspace topology induced by $\MOne{ \TWC } $.
\end{notation}

\begin{proof}[Second proof of Theorem~\ref{thm::LDPLN::LMF::LDP}]~

	We know from Lemma~\ref{lem::SanovT::RelEntropy} that  $\left\{ L^{N} , W^{N,-\Psi} \right\} $  satisfies a large deviation principle with rate function 
	$\relE{ Q }{ \dd x \otimes \zeta_{x} \left( \dd \w \right) \otimes W^{-\Psi}_{x,\w} }$ if $Q \in \MOneL{\TWC}$ and infinity otherwise.
	To infer the LDP of $\left\{ L^{N} , P^{N} \right\}$  from the LDP of $\left\{ L^{N} , W^{N,-\Psi} \right\} $,
	we need at first the following result, which states the validity of the Laplace principle.
	\begin{lemma}
		\label{lem::LMF::LDPLN::Varadhan}
		For any $G \in \CspL{b}{\MOne{ \TWC } }$ bounded continuous functional,  
		\begin{align}
		\label{lem::LMF::LDPLN::Varadhan::eq}
		\begin{split}
		\lim_{N \to \infty}
			 \frac{1}{N^d} \log E_{P^{N}} \left[ e^{ N^d  G  \left(  L^{N}  \right) }  \right]
			&=
		 	\sup_{Q \in \calM_{\infty}} \left\{  G  \left( Q  \right)  +  F  \left( Q  \right)  -   \relE{ Q }{ \dd x \otimes \zeta_{x} \left( \dd \w \right) \otimes W^{-\Psi}_{x,\w} }  \right\} 
		 	\\
		 	&=
		 	\sup_{Q \in \MOne{\TWC}} \left\{  G  \left( Q  \right)  -   \ol{I} \left( Q \right)  \right\} 
		 	<
		 	\infty
		 	 ,
		\end{split}
		\end{align}
	with $F$ and $\ol{I}$ defined in Theorem~\ref{thm::LDPLN::LMF::LDP}.
	\end{lemma}
	In the proof of this lemma, we apply at first the Girsanov theorem to replace the integral with respect to $P^{N}$, by an integral with respect to $W^{N,-\Psi}$.
	Thus we get in the exponent $G-F$, by Lemma~\ref{lem::LDPLN::LMF::PNtoWN}.
	However the function $F$ is neither bounded nor continuous, due to the unbounded terms in the integrals in $F$.
	Therefore we can not apply the original Varadhan's lemma, but we have to use a generalised version of it (see Appendix~\ref{sec::Vara}).
	\emptyline

Moreover we need that $\ol{I}$ is a good rate function.
		\begin{lemma}
		\label{lem::LMF::LDPLN::RFGood}
			The rate function $\ol{I} $ is good, i.e. the level sets $\calL^{\leq c} \left( \ol{I} \right) \defeq \left\{ Q : \ol{I} \left( Q \right)  \leq c \right\}$ are compact for each $c \geq 0$.
		\end{lemma}
	
			By \cite{DupEllAWeakCon}  Theorem~1.2.3 the validity of the Laplace principle for all $G \in \CspL{b}{\MOne{ \TWC } }$ (shown in Lemma~\ref{lem::LMF::LDPLN::Varadhan}) and the fact that $\ol{I}$ is a good rate function  (Lemma~\ref{lem::LMF::LDPLN::RFGood}),  implies the claimed large deviation principle of $\left\{L^{N}\right\}$ under $\left\{P^{N}\right\}$.
\end{proof}

\subsection{Proof of Lemma~\ref{lem::LMF::LDPLN::Varadhan}}
\label{sec::LMF::LDPLN::pfLemVara}

In this chapter we prove Lemma~\ref{lem::LMF::LDPLN::Varadhan}.
We explain at first the strategy of this proof.
To prove the first equality in Lemma~\ref{lem::LMF::LDPLN::Varadhan}, we need to show that for any $G \in \CspL{b}{\MOne{ \TWC } }$,  
		\begin{align}
		\label{eq::pf::lem::LMF::LDPLN::Varadhan::WithF}
		\lim_{N \to \infty}
			 \frac{1}{N^d} \log E_{W^{N,-\Psi}} \left[ e^{ N^d \left( G  + F  \right) \left(  L^{N}  \right) } \right]
			=
		 	\sup_{Q \in \calM_{\infty}} \left\{ \left(  G+  F \right)  \left( Q  \right)    -   \relE{ Q }{ \dd x \otimes \zeta_{x} \left( \dd \w \right) \otimes W^{-\Psi}_{x,\w} }  \right\} 
		 	,
		\end{align}
by Lemma~\ref{lem::LDPLN::LMF::PNtoWN}.
The second equality in Lemma~\ref{lem::LMF::LDPLN::Varadhan} follows from the definition of $\ol{I}$.

\emptyline
The equation \eqref{eq::pf::lem::LMF::LDPLN::Varadhan::WithF} would follow directly from Varadhan's Lemma (see Theorem~4.3.1 in \cite{DemZeiLarge}), if $F$ were continuous.
But this is not the case, because the functions in the integrals in $F$ are not bounded.
Therefore we can not use the usual Varadhan's lemma, but we need a generalisation (Theorem~\ref{vara::thm::GenVara}).
We prove at the end of this chapter that the conditions of this generalisation are satisfied.
This requires some results, that we state now.
Also larger sets than $\calM_{R}$ are required in that proof
(we refer to Chapter~\ref{sec::LMF::LDPLN::LDP::Pf::DiffSets} for a discussion why we need larger sets).
For each $R \in \R_{+}$ define
\begin{align}
\label{def::calNR}
	\begin{split}
		\calN_{R} \defeq 
		\bigg\{ &Q \in \MOne{\TWC }  : 
		\int_{\TWR} \int_{0}^{T} \left( \theta_{t} \right)^{2}  \dd t Q \left( \dd x, \dd \w, \dd \theta_{\Ti}   \right)  \leq R
		\textnormal{ and }
		\\
		&\left.
		\int_{\TWR} \left( \theta_{T}  \right)^{2} Q \left( \dd x, \dd \w, \dd \theta_{\Ti}   \right)  \leq R
		\textnormal{ and } 
		\int_{\TWR} \left( \theta_{0}  \right)^{2} Q \left( \dd x, \dd \w, \dd \theta_{\Ti}   \right)  \leq R
		\right\}
		,
	\end{split}
\end{align}
and denote the subspace of $\MOne{\TWC}$, of the union of these sets, by
\begin{align}
	\calN_{\infty} \defeq \bigcup_{R} \calN_{R} \subset \MOne{ \TWC } 
\end{align}
equipped with the topology induced by $\MOne{ \TWC } $.

In the first lemma, we show that the probability of being outside of $\calN_{R}$ under $W^{N,- \Psi}$, decays exponentially fast.
\begin{lemma}
\label{lem::LMF::LDPLN::ExpBoundNR}
For all $\kappa< c_{\Psi}$ (defined in Assumption~\ref{ass::LMF::Psi}), there is a constant $C>0$, such that for all $N$ and $R$ large enough 
\begin{align}
	W^{N,- \Psi} \left[  L^{N} \not \in \calN_{R} \right]
	\leq
	e^{-\Nd \kappa R} C^{\Nd}
	.
\end{align}
\end{lemma}

Then we show that the probability of being outside of $\calN_{R}$ also decays (at least asymptotically) exponential fast under $P^{N}$.
\begin{lemma}
	\label{lem::LMF::LDPLN::eFoutOfNR}
~\vspace{-1.8em}
	\begin{align}
 	\begin{split} 
 		&\limsup_{R\to\infty}\limsup_{N\to\infty} \frac{1}{N^d} 
 	  		\log P^{N}
 	 		\left[ 
 	 		L^{N} \not \in \calN_{R} 
 			\right]  
 		\\
 		&=
 		\limsup_{R\to\infty}\limsup_{N\to\infty} \frac{1}{N^d} 
  		\log E_{W^{N,-\Psi}}
 		\left[ 
 		e^{N^d F  \left( L^{N}  \right)  } \1_{L^{N} \not \in \calN_{R} } 
		\right]  
 	= 
		-\infty
	.
	\end{split}
	\end{align}
\end{lemma}

Moreover we show that the sets $\calN_{R}$ are closed and that the restriction of $F$ to particular sequences in these sets is continuous.
	\begin{lemma}
	\label{lem::LMF::LDPLN::NRclosed}
		The sets $\calN_{R}$ are closed.
	\end{lemma}
	\begin{lemma}
	\label{lem::LMF::LDPLN::FcontNR}
		For each $R>0$, for each $Q \in \calN_{R}$ with  $\relE{ Q }{ \dd x \otimes \zeta_{x} \left( \dd \w \right) \otimes W^{-\Psi}_{x,\w} } < \infty$ and
		for each sequence of empirical measures $\left\{ L_{N_{n}} \right\} \subset \calN_{R}$ with $L_{N_{n}} \rightarrow Q$,
		the sequence $F \left( L_{N_{n}} \right) \rightarrow F \left( Q \right)$.
	\end{lemma}

For the last equality in \eqref{lem::LMF::LDPLN::Varadhan::eq}, we need that
$\relE{ Q }{ \dd x \otimes \zeta_{x} \left( \dd \w \right) \otimes W^{-\Psi}_{x,\w} } - F \left( Q \right) = \infty$ if $Q \in \calN_{\infty} \backslash \calM_{\infty}$. 
Because $F$ is bounded on $\calN_{\infty}$, it is enough to prove the following lemma.
\begin{lemma}
\label{lem::LMF::LDPLN::RelEoutM}
$ \relE{ Q }{ \dd x \otimes \zeta_{x} \left( \dd \w \right) \otimes W^{-\Psi}_{x,\w} } = \infty$ if $Q \not \in \calM_{\infty}$.
Therefore this also holds for $Q \not \in \calN_{\infty}$.
\end{lemma}

We state the proofs of these lemmas in Chapter~\ref{sec::LMF::LDPLN::pfLemChapVara}.
In the rest of this chapter, we infer Lemma~\ref{lem::LMF::LDPLN::Varadhan} from these lemmas.
	
\begin{proof}[of Lemma~\ref{lem::LMF::LDPLN::Varadhan}]
To prove Lemma~\ref{lem::LMF::LDPLN::Varadhan}, we show that \eqref{eq::pf::lem::LMF::LDPLN::Varadhan::WithF} holds by applying the generalised Varadhan's Lemma (Theorem~\ref{vara::thm::GenVara}).
In \ref{pf::lem::LMF::LDPLN::Varadhan::Step::ExtVara} we show that the conditions of Theorem~\ref{vara::thm::GenVara} hold. 
Then we derive in  \ref{pf::lem::LMF::LDPLN::Varadhan::Step::SupFinite}, that the supremum on the right hand side of \eqref{eq::pf::lem::LMF::LDPLN::Varadhan::WithF} is finite.

\emptyline
\begin{steps}

\step[{Applying the Theorem~\ref{vara::thm::GenVara}}]
	\label{pf::lem::LMF::LDPLN::Varadhan::Step::ExtVara}
	To apply Theorem~\ref{vara::thm::GenVara} we show that the model we consider here is within the class defined in Chapter~\ref{vara::sec::Example::Class}.
	We take as increasing sets the $\calN_{R}$.

	\begin{steps}
		\step[{\ref{vara::ex:MR}}]
			See Lemma~\ref{lem::LMF::LDPLN::NRclosed}.
		
		\step[{\ref{vara::ex:Iinfty}}]
			See Lemma~\ref{lem::LMF::LDPLN::RelEoutM}, because $\calM_{\infty} \subset \calN_{\infty}$.

		\step[{\ref{vara::ex:PhiMRcont}}]
			See Lemma~\ref{lem::LMF::LDPLN::FcontNR}.
				
		\step[{\ref{vara::ex:PhiBound}}]
		For a empirical process $L^{N} \in \calN_{R}$, we get by \eqref{eq::pf::lem::LDPLN::LMF::eFoutOfNR::BoundF}, that 
\begin{align}
\label{eq::pf::lem::LMF::LDPLN::Varadhan::BoundF}
\begin{split}
	F \left( L^{N} \right)
	\leq 
		\left( \LpN{1}{\ol{J}}  +  \delta \right) 
		\oh \frac{1}{\Nd}
		\sum_{i \in \TN} \left( \theta^{i,N}_{T} \right)^{2} 
		+
		\left( \theta^{i,N}_{0}  \right)^{2}
	\leq
		R \left( \LpN{1}{\ol{J}}+ \delta \right)
	.
\end{split}
\end{align}
for all $\delta>0$, when $N> N_{\delta}>0$ (by Assumption~\ref{ass::LMF::J}).
Set $\alpha\left( R \right)  = R \left( \LpN{1}{\ol{J}}+ \delta \right) $, for $\delta$ small enough.

		\step[{\ref{vara::ex:ProbOutMRBound}}]
			This follows from Lemma~\ref{lem::LMF::LDPLN::ExpBoundNR} with 
			$\beta\left( R \right)  = c_{\Psi}R - C$ for a constant $C>0$.
		
		\step[{\ref{vara::ex:AlphaMBeta}}]
			$\alpha\left( R \right) -\beta\left( R \right)   \rightarrow - \infty$
			by Assumption~\ref{ass::LMF::Psi}.
			
		\step[{\ref{vara::ex:DiffCondAlmost}}]
			See Lemma~\ref{lem::LMF::LDPLN::eFoutOfNR}.
		
		\step[{\ref{vara::ex:Moment}}]		
			The sufficient moment condition is satisfied, because $G$ is bounded and because there is a $C>0$ and a $\gamma>1$ not too large, such that
\begin{align}
\label{eq::pf::lem::LMF::LDPLN::Varadhan::Moment} 
	E_{W^{N,-\Psi}} \left[ e^{\gamma \Nd  F \left( L^{N} \right) } \right]
	\leq
		C^{\Nd }
\end{align}
for all $N \in \N$, by \eqref{eq::pf::lem::LMF::LDPLN::eFoutOfNR::EpFLNbound}  (in the proof of Lemma~\ref{lem::LMF::LDPLN::eFoutOfNR}).

	\end{steps}
\emptyline
Hence the model we consider here is within the class defined in Chapter~\ref{vara::sec::Example::Class}.
Therefore all conditions of  Theorem~\ref{vara::thm::GenVara} are satisfied and we get
	\begin{align}
	\label{eq::pf::lem::LMF::LDPLN::Varadhan::Res} 
	\begin{split}
		&\lim_{N\to\infty} \frac{1}{\Nd}\log E_{W^{N,-\Psi}} \left[  e^{ \Nd \left( G  + F \right) \left( L^{N}  \right) } \right] 
		\\
		&= 
		\sup_{Q \in \MOne{\TWC} : \relE{Q}{\dd x \otimes \zeta_{x} \left( \dd \w \right) \otimes W^{-\Psi}_{x,\w}} < \infty} \left\{  \left( G + F \right) \left( Q \right)  - \relE{Q}{\dd x \otimes \zeta_{x} \left( \dd \w \right) \otimes W^{-\Psi}_{x,\w}} \right\} 
		\\
		&=
		\sup_{Q \in \calM_{\infty}} \left\{  \left( G + F \right) \left( Q  \right)   - \relE{Q}{\dd x \otimes \zeta_{x} \left( \dd \w \right) \otimes W^{-\Psi}_{x,\w}} \right\} 
		.
	\end{split}
	\end{align}
In the last equality we use Lemma~\ref{lem::LMF::LDPLN::RelEoutM} and that $F$ is finite on $\calM_{\infty} \subset \calN_{\infty}$.

\emptyline
\step[{The suprema in \protect\eqref{eq::pf::lem::LMF::LDPLN::Varadhan::Res}  are finite}]
\label{pf::lem::LMF::LDPLN::Varadhan::Step::SupFinite}
For a lower bound on the right hand side of \eqref{eq::pf::lem::LMF::LDPLN::Varadhan::Res}, take $Q = \dd x \otimes \zeta_{x} \left( \dd \w \right) \otimes W^{-\Psi}_{x,\w} \in \calN_{R}$.
Moreover the left hand side of \eqref{eq::pf::lem::LMF::LDPLN::Varadhan::Res} is bounded from above, because
\begin{align}
	E_{W^{N,-\Psi}} \left[ e^{\Nd \left( G + F \right) \left( L^{N} \right) } \right] 		
	\leq
		e^{\Nd \iNorm{G}} 
		\left( E_{W^{N,-\Psi}} \left[ e^{\Nd \gamma F  \left( L^{N} \right) } \right] \right)^{\frac{1}{\gamma}}
	\leq
		e^{\Nd \iNorm{G}} 
		C^{\frac{\Nd}{\gamma}}
		,
\end{align}
for each $N \in \N$,
by \eqref{eq::pf::lem::LMF::LDPLN::Varadhan::Moment}. 
Hence also the suprema in \eqref{eq::pf::lem::LMF::LDPLN::Varadhan::Res} are finite.

\end{steps}\vspace{-\baselineskip}
\end{proof}

\subsection{Proofs of the lemmas of Chapter~\ref{sec::LMF::LDPLN::pfLemVara}}
\label{sec::LMF::LDPLN::pfLemChapVara}

In this chapter we prove the lemmas that we state in Chapter~\ref{sec::LMF::LDPLN::pfLemVara}.

\begin{proof}[of Lemma~\ref{lem::LMF::LDPLN::ExpBoundNR}]
At first we split the set $\calN_{R}$ in its three conditions. Then we show separately for each of the three terms an exponential small upper bound.
\begin{align}
\begin{split}
	W^{N,-\Psi} \left[  L^{N} \not \in \calN_{R} \right]
	&\leq
		W^{N,-\Psi} \left[  \frac{1}{\Nd} \sum_{i \in \TN} \left( \theta^{i,N}_{T}  \right)^{2} >R \right]
	+
		W^{N,-\Psi} \left[  \frac{1}{\Nd} \sum_{i \in \TN} \left( \theta^{i,N}_{0}  \right)^{2} >R \right]
	\\
	&\quad
	+
		W^{N,-\Psi} \left[  \frac{1}{\Nd} \sum_{i \in \TN} \int_{0}^{T} \left( \theta^{i,N}_{t}  \right)^{2} \dd t>R \right]
	\eqdef 
	\kreis{1} + \kreis{2} + \kreis{3}
	.
\end{split}
\end{align}
Fix a $\kappa<c_{\Psi}$.
With the exponential Chebychev inequality, we get
\begin{align}
	\kreis{1}
	\leq
		e^{- \Nd R \kappa}
 		\prod_{i \in \TN}
 		\sup_{\w \in \Wsp}
		E_{W^{-\Psi}_{\frac{i}{N},\w} } \left[ e^{ \kappa \left( \theta_{T}  \right)^{2}} \right]
	\leq
		e^{- \Nd R \kappa}
		C^{\Nd}
	,
\end{align}
where we use Lemma~\ref{lem::LDPLN::LMF::eTheta2Bound}~\ref{lem::LDPLN::LMF::eTheta2Bound::T} to get the last inequality.
By the Chebychev inequality we get also for $\kreis{2}$
\begin{align}
	\kreis{2} 
	=
		 \nu^{N} \left[ \frac{1}{\Nd} \sum_{i \in \TN} \left( \theta^{i,N}_{0}  \right)^{2} >R \right]
	\leq
		e^{- \Nd R \kappa}
			\prod_{i \in \TN}
			\int_{\R} e^{ \kappa  \theta^{2}} \nu_{\frac{i}{N}} \left( \dd \theta \right)
	\leq
			e^{- \Nd R \kappa} C^{\Nd}
		,
\end{align}
where we use Assumption~\ref{ass::LMF::Init::Integral} and Assumption~\ref{ass::LMF::Psi} in the last inequality.
For $\kreis{3}$ we get from the exponential Chebychev inequality and Lemma~\ref{lem::LDPLN::LMF::eTheta2Bound}~\ref{lem::LDPLN::LMF::eTheta2Bound::Int0T}
\begin{align}
	\kreis{3}
	\leq
		e^{- \Nd R \kappa}
		\prod_{i \in \TN}
		\sup_{\w \in \Wsp} E_{W^{-\Psi}_{\frac{i}{N},\w} } \left[ e^{ \kappa \int_{0}^{T} \left( \theta_{t} \right)^{2} \dd t} \right]
	\leq
		e^{- \Nd R \kappa}
		C^{\Nd}
	.
\end{align}
\end{proof}

\begin{proof}[of Lemma~\ref{lem::LMF::LDPLN::eFoutOfNR}]
With $ \frac{1}{p}+\frac{1}{p'}=1$, we get by the \Holder inequality
\begin{align}
	E_{W^{N,-\Psi}}
		\left[ 
		e^{N^d F  \left( L^{N}  \right)  } \1_{L^{N} \not \in \calN_{R} } 
		\right]  
	\leq
		E_{W^{N,-\Psi}}
		\left[ 
		e^{N^d p F  \left( L^{N}  \right)  } 
		\right]^{\frac{1}{p}}
		W^{N,-\Psi} \left[ L^{N} \not \in \calN_{R} \right]^{\frac{1}{p'}}
	.
\end{align}
For $\delta>0$ and $p>1$ small enough, such that $p \left( \LpN{1}{\ol{J}} + \delta \right) < c_{\Psi}$, we can bound $F$ form above as in \eqref{eq::pf::lem::LMF::LDPLN::Varadhan::BoundF} when  $N>N_{\delta}$. Then
\begin{align}
\label{eq::pf::lem::LMF::LDPLN::eFoutOfNR::EpFLNbound}
\begin{split}
	E_{W^{N,-\Psi}}
		\left[ 
		e^{N^d p F  \left( L^{N}  \right)  } 
		\right]
	&\leq
		\prod_{i \in \TN}
		\sup_{\w \in \W}
		E_{W^{-\Psi}_{\frac{i}{N},\w} }
		\left[ 
		e^{p \left( \LpN{1}{\ol{J}} + \delta \right) \left(  \left( \theta_{T}  \right)^{2} + \left( \theta_{0}  \right)^{2} \right)}
		\right]
	\leq 
	C^{\Nd}
	,
\end{split}
\end{align}
where we use Lemma~\ref{lem::LDPLN::LMF::eTheta2Bound}~\ref{lem::LDPLN::LMF::eTheta2Bound::T} in the last inequality.
With Lemma~\ref{lem::LMF::LDPLN::ExpBoundNR}, we conclude
\begin{align}
	E_{W^{N,-\Psi}}
		\left[ 
		e^{N^d F  \left( L^{N}  \right)  } \1_{L^{N} \not \in \calN_{R} } 
		\right]  
	\leq
		C^{\Nd}
		e^{-\Nd  \frac{1}{p'} \kappa R}
	,
\end{align}
for $\kappa<c_{\Psi}$, when $N$ and $R$ are large enough.
This proves Lemma~\ref{lem::LMF::LDPLN::eFoutOfNR}.
\end{proof}

\begin{proof}[of Lemma~\ref{lem::LMF::LDPLN::NRclosed}]
Take an arbitrary sequence $Q^{(n)} \in \calN_{R}$ that converges in $\MOne{\TWC}$ to a $Q \in \MOne{\TWC}$. 
Now we want to show that $Q \in \calN_{R}$.
To this end, we define the cutoff functions for $M \in \R_{+}$
\begin{align}
\label{eq::pf::lem::LMF::LDPLN::NRclosed::Cutoff}
	 \chi_M  \left( \theta  \right)  
	 \defeq
	  \left( \theta \wedge M \right) \vee -M 
	  .
\end{align}
The function $\TWC \ni \left( x,\w, \theta_{\Ti} \right)   \mapsto \chi_{M} \left( \theta_{T} ^{2} \right) \in \R$ is a continuous, bounded function.
By the weak convergence of $Q^{(n)}$, $\int \chi_{M} \left( \theta_{T} ^{2} \right) Q^{(n)} \rightarrow \int \chi_{M} \left( \theta_{T} ^{2} \right) Q$ for each $M \in \R_{+}$. Hence
\begin{align}
\label{eq::pf::lem::LMF::LDPLN::NRclosed::Bounded}
	\int \chi_{M} \left( \theta_{T} ^{2} \right) Q 
	\leq
	R
	.
\end{align}
By the monotone convergence theorem this implies
\begin{align}
	\int  \theta_{T} ^{2} Q 
	\leq
	R
	.
\end{align}
Similar calculations  with $\chi_{M} \left( \theta_{0} ^{2} \right)$ and  $\chi_{M} \left( \int_{0}^{T} \theta_{t} ^{2}  \dd t \right)$ imply also the boundedness by $R$ of the other two conditions in $\calN_{R}$.
Hence $\calN_{R}$ is closed.
\end{proof}

\begin{proof}[of Lemma~\ref{lem::LMF::LDPLN::FcontNR}]
	Fix an $R>0$, a measure $Q \in \calN_{R}$, with $\relE{ Q }{ \dd x \otimes \zeta_{x} \left( \dd \w \right) \otimes W^{-\Psi}_{x,\w} } < \infty$.
	Hence in particular that $Q \in \MOneL{\TWC}$.
	Moreover fix a weakly converging sequence of empirical measures $L^{N_{n}} \rightarrow Q$ in $\calN_{R}$. 
	We show now that $F \left( L^{N_{n}} \right)  \rightarrow F \left( Q \right) $.
	
	Because the integrands in $F$ are neither continuous nor bounded, we approximate $J$ by continuous $J_{\ell}$ and the spin values by cutoff functions $\chi_{M}  \left( \theta, \theta' \right)  = \chi_{M}  \left( \theta \right) \chi_{M} \left( \theta' \right)$  as in \eqref{eq::LocalMF::bCont::ContMu::5sum}.
	The five arising summands (similar to \eqref{eq::LocalMF::bCont::ContMu::5sum}), can all be bounded by $\epsilon$, by the same approach that we use in \ref{step::LocalMF::MuIntCont::EmpP::Pointw} in Chapter~\ref{sec::LocalMF}.
	This implies the convergence of  $F \left( L^{N_{n}} \right)  \rightarrow F \left( Q \right) $.
	
	This approach requires that $L^{N_{n}}, Q \in \calN_{R}$, the Assumption~\ref{ass::LMF::J}, that $L^{N_{n}} \otimes L^{N_{n}}$ is tight and that on compact subsets of $\left( \TWC \right)^{2} $ the spin value paths are equibounded.
	
	We need the last two properties to bound \encircle{C} and \encircle{D} of \eqref{eq::LocalMF::bCont::ContMu::5sum}.
	Indeed when considering the second summand of $F$ we get for we get for \encircle{C}
	\begin{align}
	\begin{split}
			&\abs{ \int J_{\ell} \left( x-x', \w, \w' \right)   
						\left( \theta'_{T}  \theta_{T}  - \chi_{M} \left( \theta'_{T}  \right) \chi_{M} \left( \theta_{T}  \right)  \right) 
						\left( L^{N_{n}} \otimes L^{N_{n}} \right) 
					}
			\\
			&\leq  
			L^{N_{n}} \otimes L^{N_{n}} \Big[  \left( \theta_{\Ti} ,\theta'_{\Ti}  \right)  : \iNorm{\theta_{\Ti} }>M  
														\textnormal{ or } \iNorm{\theta'_{\Ti} } > M
											\Big]
				\iNorm{J_{\ell}}  
				\int_{\TWC}  \left( \theta_{T} \right) ^{2} L^{N_{n}}
			\\
			&\leq
				L^{N_{n}} \otimes L^{N_{n}} \left[ \left( \TWC  \right)^{2} \backslash K_{\epsilon}  \right]
				\iNorm{J_{\ell}} R
			\leq 
			C \epsilon^{\oh}
			,
	\end{split}
	\end{align}
	for a suitably chosen compact set $K_{\epsilon}$ and $M>0$, such that $\theta_{\Ti}  \in K_{\epsilon}$ implies that $\iNorm{\theta_{\Ti} } \leq M$.
	The  first summand of $F$ can be bounded analogue.
	
	Note that we get the tightness of $L^{N_{n}} \otimes L^{N_{n}}$ by Prokhorov's theorem and because the convergence of $L^{N_{n}} \rightarrow Q$ implies the convergence of $L^{N_{n}} \otimes L^{N_{n}} \rightarrow Q \otimes Q$.
	Moreover on compact subsets of $\left( \TWC \right)^{2} $, the spin value paths are equibounded by Lemma~\ref{lem::AA}~\ref{lem::AA::AA}.
\end{proof}

\begin{proof}[of Lemma~\ref{lem::LMF::LDPLN::RelEoutM}]
	To prove this lemma, we use that the probability of $L^{N}$ being outside of $\calM_{R}$ under $P^{N}$ decays exponentially fast, i.e. there is a $\lambda \in \R_{+}$ and a $C>0$ such that for all $N \in \N$ and $R$ large enough 
	\begin{align}
		\label{eq::pf::lem::LMF::LDPLN::RelEoutM::ExpBoundMR}
		W^{N,-\Psi} \left[  L^{N} \not \in \calM_{R} \right]
		\leq
		e^{-\Nd \lambda R} C^{\Nd}
		.
	\end{align}
	We get this exponential bound from Lemma~\ref{lem::LDPempP::Inter::OutCemR+Initial}, because $L^{N} \in \calM_{R}$ if and only if $\empP \in \Cem_{\varphi,R}$ with $\varphi= \theta^{2}$.
	The necessary assumptions for this  lemma (i.e. Assumption~\ref{ass::LDPempP}~\ref{ass::LDPempP::bCont}-\ref{ass::LDPempP::Lyapu}) are satisfied for the drift coefficient $-\Psi'$ (by the Assumption~\ref{ass::LMF::Psi} and the same arguments as in Chapter~\ref{sec::LocalMF}).
	
	\emptyline

If $Q \in \MOne{\TWC} \backslash \calM_{\infty}$, then $Q \not\in \calM_{R}$ for all $R>0$.
Then for $R$ large enough
\begin{align}
	\label{eq::pf::lem::LMF::LDPLN::RelEoutM::RelEBoundMR}
	\begin{split}
		- \relE{ Q }{ \dd x \otimes \zeta_{x} \left( \dd \w \right) \otimes W^{-\Psi}_{x,\w} } 
		&\leq
		- \inf_{ Q' \not \in \calM_{R}} \relE{ Q' }{ \dd x \otimes \zeta_{x} \left( \dd \w \right) \otimes W^{-\Psi}_{x,\w} } 
		\\
		&\leq
		\liminf_{N \rightarrow \infty} \frac{1}{\Nd} \log W^{N,-\Psi}  \left[ L^{N} \not \in \calM_{R} \right]
		\leq
			-\lambda R  + \log C
		,	
	\end{split}
\end{align}
by \eqref{eq::pf::lem::LMF::LDPLN::RelEoutM::ExpBoundMR}, by the large deviation principle of $\left\{ L^{N},W^{N,-\Psi} \right\}$ and by $\calM_{R}$ being closed (by a similar proof as for Lemma~\ref{lem::LMF::LDPLN::NRclosed}).
\end{proof}

\begin{remark}
	\label{rem::LMF::LDPLN::RelEoutM::OtherProof}
	We could also prove \eqref{eq::pf::lem::LMF::LDPLN::RelEoutM::ExpBoundMR} without using the Lemma~\ref{lem::LDPempP::Inter::OutCemR+Initial}.
	One would have to show at first the exponential decay of the probability of being outside of $\calM_{R}$ under $W^{N,-\Psi}$.
	This can be done by the direct approach to transfer the problem to the measure $W^{N,0}$ (by the Girsanov theorem) and then to use the Doob submartingale inequality.
\end{remark}

\subsection{\texorpdfstring{$\ol{I}$}{overline I} is a good rate function (Lemma~\ref{lem::LMF::LDPLN::RFGood})}
\label{sec::LMF::LDPLN::ExpGood}

	\begin{proof}[of Lemma~\ref{lem::LMF::LDPLN::RFGood}]
		We have to prove that $\calL^{\leq c} \left( \ol{I} \right)$ is compact.
		Therefore we show at first that $\calL^{\leq c} \left( \ol{I} \right)$ is a subset of $\calN_{R}$ for $R$ large enough, then that it is closed and finally that it is compact.
		
		\emptyline
		\begin{steps}
		\step[{$\calL^{\leq c} \left( \ol{I} \right)$  is a subset of $\calN_{R}$ for $R$ large enough}]

		Fix a $Q \in \calL^{\leq c} \left( \ol{I} \right)$.
		Then $\ol{I} \left( Q \right) \leq c$ implies that $Q \in \calN_{\infty}$ and $Q \in \MOneL{\TWC}$.
		Choose $R>0$ such that $Q \in \calN_{R+\frac{1}{R}} \backslash \calN_{R}$.
		Then
		\begin{align}
		\label{eq::pf::lem::LMF::LDPLN::RFGood::IboundMR}
			\ol{I} \left( Q \right)
			=
				\relE{Q}{\dd x \otimes \zeta_{x} \left( \dd \w \right) \otimes W^{-\Psi}_{x,\w}}  
				-
				F \left( Q \right)
			\geq
				\kappa R + C_{R}
				-
				\LpN{1}{\ol{J}} \left(  R+\frac{1}{R}  \right)
			,
		\end{align}
		with $\kappa \in \R$ such that $\LpN{1}{\ol{J}} < \kappa < c_{\Psi}$ (possible due to Assumption~\ref{ass::LMF::Psi}).
		In this equality we use a similar calculation as in \eqref{eq::pf::lem::LMF::LDPLN::RelEoutM::RelEBoundMR} to bound the relative entropy (with $\calN_{R}$ instead of $\calM_{R}$ and with Lemma~\ref{lem::LMF::LDPLN::ExpBoundNR}).
		The upper bound on $F$ holds by Assumption~\ref{ass::LMF::J} and by $Q$ having the Lebesgue measure as projection to $\Td$.
		
		The right hand side of \eqref{eq::pf::lem::LMF::LDPLN::RFGood::IboundMR} tends to infinity when $R$ increases (by Assumption~\ref{ass::LMF::Psi}).
		Hence  there is a $R$ large enough such that $Q \in \calN_{R}$ if $Q \in \calL^{\leq c} \left( \ol{I} \right)$.
		
		\emptyline
		\step[{$\calL^{\leq c} \left( \ol{I} \right)$  is a closed}]
				
				Take a sequence $\left\{ Q^{(n)} \right\}_{n} \subset \calL^{\leq c} \left( \ol{I} \right) \subset \calN_{R}$, such that $Q^{(n)} \rightarrow Q$ in $\MOne{\TWC}$.
				Then $Q \in \calN_{R}$ because this set is closed (Lemma~\ref{lem::LMF::LDPLN::NRclosed}).
				By $F$ being continuous on $\calN_{R}$ and the relative entropy being lower semi-continuous,
				\begin{align}
							\ol{I} \left( Q \right)  
							\leq 
							\liminf_{n \rightarrow \infty} \ol{I} \left( Q^{(n)} \right)  
							\leq 
							c
							.
				\end{align}

		\step[{$\calL^{\leq c} \left( \ol{I} \right)$  is compact}]
				
				We use now the exponential tightness of $\left\{ L^{N}, W^{N,-\Psi} \right\}$ derived in Chapter~\ref{sec::LDPLN::LMF::PfExpTight}.
				The corresponding compact sets $K_{A}$ are defined in \eqref{eq::lem::LDPLN::LMF::expTight::KL}.
				We claim that there is a $A>0$ such that $\calL^{\leq c} \left( \ol{I} \right) \subset K_{A}$.
				Take a $Q \in \calL^{\leq c} \left( \ol{I} \right) \subset \calN_{R}$ with $Q \not \in K^{A}$ for a $A>0$.
				Then 
					\begin{align}
							\ol{I} \left( Q \right)
							=
								\relE{Q}{\dd x \otimes \zeta_{x} \left( \dd \w \right) \otimes W^{-\Psi}_{x,\w}}  
								-
								F \left( Q \right)
							\geq
								A
								-
								\LpN{1}{\ol{J}}   R 
							,
						\end{align}
						where we bound $F$ as in \eqref{eq::pf::lem::LMF::LDPLN::RFGood::IboundMR}
						and the lower bound on the relative entropy follows by the same calculation as in \eqref{eq::pf::lem::LMF::LDPLN::RelEoutM::RelEBoundMR}  with $K_{A}$ instead of $\calM_{R}$ and with \eqref{eq::lem::LDPLN::LMF::expTight::WNExpTight}.
				This implies that there has to be a $A>0$, such that $\calL^{\leq c} \left( \ol{I} \right) \subset K^{A}$.
		\end{steps}
		
		Therefore we conclude that the sets $\calL^{\leq c} \left( \ol{I} \right) $ are compact as closed subsets of a compact set.
	\end{proof}

\subsection{Discussion of the different subsets that we use in the proofs}
\label{sec::LMF::LDPLN::LDP::Pf::DiffSets}

In the previous chapters we use the three different subsets of $\MOne{\TWC}$, $\calM_{R}$ (defined in \eqref{def::calMR}), $\calN_{R}$ (defined in \eqref{def::calNR}) and $K^{L}$ (defined in \eqref{eq::lem::LDPLN::LMF::expTight::KL}).
Let us briefly discuss why we do not restrict the attention to one of these sets.
\emptyline
	On the one hand, we need a compact set for the proof that $I$ is a good rate function (Lemma~\ref{lem::LMF::LDPLN::RFGood}).
	On the other hand, we need a closed set on which the function $F$ is continuous and bounded in Lemma~\ref{lem::LMF::LDPLN::Varadhan}. 
	
	Each of the sets $K_{L}$ and $\calN_{R}$ only satisfies one of these properties.
	The set $\calN_{R}$ is not compact, because the definition does not include equicontinuity. 
	Whereas the set $K_{L}$ is only abstractly defined, what seems to make it impossible to show that $F$ is continuous on this set.

\emptyline
	
	Moreover we need for the proof that one representation of the rate function implies the other (see the proof of Theorem~\ref{thm::LDPLN::LMF::LDP} in Chapter~\ref{sec::LDPLN::LMF::LDP}), that the function $\beta$ (defined in \eqref{eq::step::LocalMF::bCont::beta}) is uniformly bounded in time for $t \in \Ti$, such that the martingale problem for the SDE \eqref{eq::LDPLN::LMF::SDE::FixedInterMeasure}  well posed.
	This is the case for all measure in $\calM_{\infty}$, but not for all measures in $\calN_{\infty}$.

	Then the natural question arises, why not using $\calM_{R}$ in the proof in Chapter~\ref{sec::LMF::LDPLN}  of the Theorem~\ref{thm::LDPLN::LMF::LDP}.
	In \ref{pf::lem::LMF::LDPLN::Varadhan::Step::ExtVara} of the proof of Lemma~\ref{lem::LMF::LDPLN::Varadhan}, we apply the extended Varadhan's Lemma (Theorem~\ref{vara::thm::GenVara}).
	To apply this lemma, we need that the probability $	W^{N,-\Psi} \left[  L^{N} \not \in \calM_{R} \right]$ decays like $e^{\kappa R}$ with $\kappa>\LpN{1}{\ol{J}}$.
	However in general we can show an exponential decay but not with a $\kappa > \LpN{1}{\ol{J}}$  (e.g. by the approach in \eqref{eq::pf::lem::LMF::LDPLN::RelEoutM::ExpBoundMR} or the one sketched in Remark~\ref{rem::LMF::LDPLN::RelEoutM::OtherProof}), 
	
	This is different for the sets $\calN_{R}$, because in the proof of the exponential decay for these sets (Lemma~\ref{lem::LMF::LDPLN::ExpBoundNR}), we can benefit from the integrals with respect to time arising by the Girsanov theorem.
	These integrals can however not approximate the supremum in $\calM_{R}$.

 \begin{remark}
 	The rate function is infinite outside of the set $\calM_{\infty}$. 
 	Moreover $P^{N} \left[  L^{N}  \in \calM_{\infty} \right] =1$ (this follows e.g. by similar arguments that show \eqref{eq::pf::lem::LMF::LDPLN::RelEoutM::ExpBoundMR}).
 	Therefore $\left\{ L^{N}, P^{N} \right\}$ satisfies a large deviation principle on $\calM_{\infty}$ (by Lemma~4.1.5~(b)  in \cite{DemZeiLarge}).
 	
 	We could consider the subspace $\calM_{\infty}$ from the very first, i.e. to consider the LDP of $\left\{ L^{N} , W^{N,-\Psi} \right\}$ on the subspace  $\calM_{\infty}$ (this LDP exists due to  Lemma~4.1.5~(b)  in \cite{DemZeiLarge}, by \eqref{eq::pf::lem::LMF::LDPLN::RelEoutM::ExpBoundMR} and Lemma~\ref{lem::LMF::LDPLN::RelEoutM}).
	Then by the same arguments and steps as in the proof in Chapter~\ref{sec::LMF::LDPLN}  of Theorem~\ref{thm::LDPLN::LMF::LDP}, we would get a LDP of $\left\{ L^{N} , P^{N} \right\}$ on $\calM_{\infty}$, with the rate function $I$.
 	
 	However it simplifies the proof of Chapter~\ref{sec::LMF::LDPLN}  only marginally and the result for the whole space $\MOne{\TWR}$ is stronger.
 \end{remark}

\appendix
\section*{Appendices}
%\addcontentsline{toc}{section}{Appendices}

\renewcommand{\thesection}{\Alph{section}}

\section{A generalisation of Varadhan's lemma to nowhere continuous functions}
\label{sec::Vara}
 
 \renewcommand{\thesection}{\Alph{section}}
  \renewcommand{\thesubsection}{\thesection.\arabic{subsection}}
 \renewcommand{\thesubsubsection}{\thesection.\arabic{subsubsection}}
 \renewcommand{\theparagraph}{\thesection.\arabic{subsection}.\arabic{paragraph}}

\subsection{Introduction}

Varadhan proved in \cite{VarAsymp} in Chapter~3 a generalisation of the Laplace method, that is referred to as Varadhan's lemma.
The lemma is a consequence of the large deviation principle. It gives a precise description of the logarithmic asymptotic (for $N \rightarrow \infty$) of expectations like
\begin{align}
	\E \left[ e^{N \phi\left( \xi_{N} \right) } \right]
	.
\end{align}
For example Varadhan's lemma can be used to transfer the LDP from $\P$ to $e^{\phi}\P$, by the relation between the Laplace principle and the large deviation principle (see \cite{DupEllAWeakCon} Chapter 1).
It requires usually that the function $\phi$ is continuous and satisfies a tail condition or that it is even bounded (\cite{VarAsymp}~Chapter~3, \cite{DemZeiLarge}~Theorem~4.3.1, \cite{DupEllAWeakCon}~Theorem~1.2.1, \cite{denHolLargD}~Theorem~III.13). 

In the following we generalise this to functions $\phi$ that are not continuous.
We only require that $\phi$ can be approximated (in an appropriate sense) by two sequences of measurable functions.
Moreover we need beside the tail condition, further condition concerning the difference of $\phi$ and the approximating functions.
For continuous $\phi$, the sequences can be chosen to equal $\phi$ everywhere and our conditions shrink to the usual tail condition.

\emptyline

In \cite{JansOnTheUppVar}, the upper bound of Varadhan's lemma is extended to hold for functions $\phi$ that are not upper semi-continuous.  However the authors require that either the rate function is continuous or that the level sets $\phi^{-1} \left( \left[ a, \infty \right] \right)$ are closed for all $a$ large enough.
In the models that we have in mind none of these two conditions are satisfied. 

\emptyline

We state the main result in Chapter~\ref{vara::sec::MainResult}.
Then we discuss in Chapter~\ref{vara::sec::SimplerCond} some less general conditions, that might be simpler to prove, under which the main results still holds.
We state our proof of the extended Varadhan's lemma in Chapter~\ref{vara::sec::Pf}.
Finally we give some examples, that indicate why the extension is useful.

\subsection{The main result}
\label{vara::sec::MainResult}

Let $X$ be a regular topological space and $\left\{ \xi_{N} \right\}$ be a family of random variables with values in $X$.
We denote by $\left\{ \P^{N} \right\}$ the probability measures associated with $\left\{ \xi_{N} \right\}$.

\emptyline

In the following theorem we state a Varadhan type lemma for a non-continuous and unbounded function $\phi$, that satisfies the usual tail condition (see condition~\ref{vara::ass::thm::GenVara::TailCondition}).
Moreover we require the existence of two sequences of functions ($\ul{\phi}_{R}$ and $\ol{\phi}_{R}$) that approximate (in a appropriate lower/upper semicontinous way) $\phi$.
These sequences have to satisfy two conditions (\ref{vara::ass::thm::GenVara::BiggerEpsIrrelevant} and \ref{vara::ass::thm::GenVara::DiffPhiPhiR}).
We show in Chapter~\ref{vara::sec::SimplerCond} how the conditions of this Theorem can be simplified.

\emptyline

Set
\begin{align}
\label{vara::eq::def::SStar}
	S^{*} \defeq \sup_{x \in X : I\left( x \right) < \infty} \left\{ \left( \phi  - I \right) \left( x \right) \right\}
	\in
	\R \cup \left\{ -\infty , \infty \right\}
	.
\end{align}

\begin{theorem}
\label{vara::thm::GenVara}
	Assume that $\left( \xi_{N},\P^{N} \right)$ satisfies a LDP with speed $a_{N}$ on $X$ with good rate function $I : X \rightarrow \left[ 0, \infty \right]$.
	Let $\phi$ be a measurable function $X \rightarrow \R \cup \left\{ -\infty , \infty \right\}$.
	Assume that the following conditions are satisfied.
	\begin{enuAlph}
	\item
	\label{vara::ass::thm::GenVara::ConUlPhi}
			There is a family of measurable functions $\left\{ \ul{\phi}_{R} \right\}_{R \in \R_{+}}$ from $X \rightarrow \R\cup \left\{ -\infty,\infty \right\}$ , 
			such that 
			\begin{align*}
				\forall  \left( x \in X  \textnormal{ with } I \left( x \right) < \infty  \right)
				\quad
				\forall \delta>0
				\quad
				\exists R^{*}_{x,\delta}>0
				\quad
				\textnormal{ such that }
				\quad
				\forall R>R^{*}_{x,\delta}
			\end{align*}
			 exists an open neighbourhood $U_{x,\delta,R} \subset X$ of $x$, such that
			\begin{align}
			\label{vara::eq::thm::GenVara::ConUlPhi}
				\inf_{y \in  U_{x,\delta,R} \cap \supp \left\{ \xi_{N} \right\}}  \ul{\phi}_{R}\left( y \right)  \geq   \phi\left( x \right)  - \delta
				.
			\end{align}
		\item
		\label{vara::ass::thm::GenVara::ConOlPhi}
			There is a family of measurable  functions $\left\{ \ol{\phi}_{R} \right\}_{R \in \R_{+}}$ from  $X \rightarrow \R\cup \left\{ -\infty,\infty \right\}$,
			such that 
						\begin{align*}
							\exists R^{*}>0
							\quad
							\textnormal{ such that }
							\quad
							\forall R>R^{*}
							\quad
							\forall  \left( x \in X  \textnormal{ with } I \left( x \right) < \infty  \right)
							\quad
							\forall \delta>0
						\end{align*}
 exists an open neighbourhood $U_{x,\delta,R} \subset X$ of $x$, such that 
			\begin{align}
			\label{vara::eq::thm::GenVara::ConOlPhi}
				\sup_{y \in  U_{x,\delta,R} \cap \supp \left\{ \xi_{N} \right\}}  \ol{\phi}_{R}\left( y \right)  \leq  \max \left\{ S^{*}, \phi\left( x \right)  \right\} + \delta
				.
			\end{align}
	
		\item
		\label{vara::ass::thm::GenVara::BiggerEpsIrrelevant}
			For all $\epsilon>0$ 
			\begin{align}
			\label{vara::eq::VaradLower::BiggerEpsIrrelevant}
			\begin{split}
				&\lim_{R \rightarrow \infty} \liminf_{N \rightarrow \infty}
				a_{N}^{-1} \log
				\E_{\P^{N}} 
					\left[ 
						e^{a_{N}  \ul\phi_{R}\left( \xi_{N} \right)  }
					 \1_{\left\{  \phi \left( \xi_N  \right)   > \ul{\phi}_{R} \left( \xi_N  \right)  -\epsilon \right\}} 
					\right]
				\\
				&=
				\lim_{R \rightarrow \infty} \liminf_{N \rightarrow \infty}
				a_{N}^{-1} \log 
					\E_{\P^{N}} 
						\left[ 
							e^{a_{N} \ul\phi_{R}\left( \xi_{N} \right) }
						\right]
				.
			\end{split}
			\end{align}
		\item
		\label{vara::ass::thm::GenVara::DiffPhiPhiR}
			For all $\epsilon>0$
			\begin{align}
			\label{vara::eq::VaradUpper::DiffPhiPhiR}
				 \limsup_{R\to\infty}  \limsup_{N\to\infty} a_{N}^{-1} 
					 \log \E_{\P_{N}} 
						 \left[ 
						 	e^{a_{N} \phi \left( \xi_N  \right)  } 
							\1_{\left\{ \phi \left( \xi_N  \right)   >  \ol{\phi}_{R} \left( \xi_N  \right) + \epsilon  \right\}} 
						 \right] 
				 \leq 
				 	S^{*} 
				 .
			 \end{align}
		\item
		\label{vara::ass::thm::GenVara::TailCondition}
			The following tail condition holds
			\begin{align}
			\label{vara::eq::VaradUpper::TailCondition}
				\lim_{M \rightarrow \infty} 
				\limsup_{N \rightarrow \infty} 
					a_{N}^{-1} \log 
			 				\E_{\P_{N}} 
			 					\left[ 
									e^{ a_{N} \phi \left( \xi_{N} \right) } 
									\1_{ \phi \left( \xi_{N} \right) \geq M} 
								\right]
				 \leq
				 	S^{*}
				 .
			\end{align}
	\end{enuAlph}

	Then
	\begin{align}
	\label{vara::eq::thm::GenVara}
		\lim_{N \rightarrow \infty} a_{N}^{-1} \log \E_{\P^{N}} \left[ e^{ a_{N}  \phi  \left( \xi_{N} \right) } \right]
		=
			S^{*}
		.
	\end{align}

\end{theorem}

\begin{remark}
	The condition $I \left( x \right) < \infty$ in the supremum in the definition \eqref{vara::eq::def::SStar} of $S^{*}$ can be dropped, if $\phi\left( x \right) < \infty$ for all $x \in X$.
\end{remark}

\begin{remark}
\label{vara::rem::usualVaradhan}
	Theorem~\ref{vara::thm::GenVara} implies the usual Varadhan's lemma.
	If $\phi : X \rightarrow \R \cup \left\{ -\infty,  \infty \right\}$ is continuous, then set $\ul{\phi}_{R}=\ol{\phi}_{R}=\phi$ and all the conditions except the tail condition~\ref{vara::ass::thm::GenVara::TailCondition} are immediately satisfied.
	
\end{remark}

\begin{remark}
We need the  conditions \ref{vara::ass::thm::GenVara::BiggerEpsIrrelevant}, \ref{vara::ass::thm::GenVara::DiffPhiPhiR}, to reduce the proofs of the Varadhan lower and upper bound for $\phi$ to the analysis of Varadhan lower and upper bounds on $\ul\phi_{R}$ and $\ol{\phi}_{R}$ respectively.
The Varadhan lower and upper bounds for these sequences can finally be shown by a similar proof as for the original Varadhan's lemma (e.g. \cite{DemZeiLarge}~Theorem~4.3.1), due to the conditions~\ref{vara::ass::thm::GenVara::ConUlPhi} and~\ref{vara::ass::thm::GenVara::ConOlPhi}.
\end{remark}

\begin{remark}
\label{vara::rem::CondULOL}
The conditions~\ref{vara::ass::thm::GenVara::ConUlPhi} and~\ref{vara::ass::thm::GenVara::ConOlPhi} on $\ul{\phi}_{R}$ and $\ol{\phi}_{R}$ differ. 
For the latter the lower bound $R^{*}$ on $R$ is uniformly in $\delta$ and $x$.
Whereas for $\ul{\phi}_{R}$ it does not have to be uniformly in these variables.
\end{remark}

\begin{remark}
	As in \cite{VarAsymp} Chapter 3, we could also treat a different function $\phi^{N}$ for each $N \in \N$.
	Then we would get
	\begin{align}
		\lim_{N \rightarrow \infty}
			a_{N}^{-1} \log \E_{\P^{N}} \left[ e^{a_{N} \phi^{N} \left( \xi_{N} \right)} \right]
			=
			S^{*}
	\end{align}
	if the conditions~\ref{vara::ass::thm::GenVara::BiggerEpsIrrelevant},~\ref{vara::ass::thm::GenVara::DiffPhiPhiR} and~\ref{vara::ass::thm::GenVara::TailCondition} hold with $\phi$ replaced by $\phi^{N}$.

\end{remark}

\begin{remark}
	To apply Bryc's Lemma (inverse Varadhan Lemma), one can use $\phi = F + G$ with $G \in \CspL{b}{X}$ and $F:X \rightarrow \R$ such that the $F-I$ is the new rate function.
	By the continuity and boundedness of $G$, it is enough to find pointwise (to $F$) converging sequences $\left\{ \ul{F}_{R} \right\}$ and $\left\{ \ol{F}_{R} \right\}$ and to look at $\ul{\phi}_{R} = \ul{F}_{R} +G$ and $\ol{\phi}_{R} = \ol{F}_{R} +G$.

\end{remark}

\emptyline

We split the proof in showing that the right hand side of \eqref{vara::eq::thm::GenVara} is a lower bound  (Lemma~\ref{vara::lem::VaradLower}) and an upper bound (Lemma~\ref{vara::lem::VaradUpper}) for the left hand side.

\begin{lemma}
\label{vara::lem::VaradLower}
	Let $\phi$ and  $\left\{ \ul{\phi}_{R} \right\}$ be defined as in Theorem~\ref{vara::thm::GenVara}.
	and the large deviation lower bound for $\left( \xi_{N},\P^{N} \right)$  with speed $a_{N}$ holds with rate function $I$.
	When \ref{vara::ass::thm::GenVara::ConUlPhi} and \ref{vara::ass::thm::GenVara::BiggerEpsIrrelevant} hold,
	then 
	\begin{align}
	\label{eq::vara::lem::VaradLower}
		\liminf_{N \rightarrow \infty} a_{N}^{-1} \log \E_{\P^{N}} \left[ e^{a_{N} \phi \left( \xi_{N} \right)  } \right]
		\geq
			S^{*}
		.
	\end{align}
		
\end{lemma}

\begin{lemma}
\label{vara::lem::VaradUpper}
	Let $\phi$ and  $\left\{ \ol{\phi}_{R} \right\}$ be defined as in Theorem~\ref{vara::thm::GenVara}.
	and the large deviation upper bound  for $\left( \xi_{N},\P^{N} \right)$  with speed $a_{N}$ holds with rate function $I$.
	When \ref{vara::ass::thm::GenVara::ConOlPhi}, \ref{vara::ass::thm::GenVara::DiffPhiPhiR} and \ref{vara::ass::thm::GenVara::TailCondition} hold,
	then
	\begin{align}
	\label{eq::vara::lem::VaradUpper}
		\limsup_{N \rightarrow \infty} a_{N}^{-1} \log  \E_{\P^{N}} \left[ e^{a_{N} \phi \left( \xi_{N} \right)  } \right]
		\leq
			S^{*}
		.
	\end{align}
\end{lemma}

\subsection{Proof of the generalised Varadhan's lemma}
\label{vara::sec::Pf}

\paragraph{Proof of the lower bound (Lemma~\protect\ref{vara::lem::VaradLower})}
\label{vara::sec::PfLowerBound}

\begin{proof}
This proof is organised as follows.
  At first we show that the function $\phi$ in the exponent on the left hand side of \eqref{eq::vara::lem::VaradLower}  can be replaced by $\ul{\phi}_{R}$, with an arbitrary small error $\epsilon$ for $R$ large enough (see \eqref{vara::pf::lem::VaradLower::eq::ReplacePhi} and \eqref{vara::pf::lem::VaradLower::eq::ReplaceLowerEps}).
  This requires the condition~\ref{vara::ass::thm::GenVara::BiggerEpsIrrelevant}.
  Then we use a similar idea as in the proof of the lower bound of the usual Varadhan Lemma in \cite{DemZeiLarge} (proof of Lemma~4.3.4).
  Here we use condition \ref{vara::ass::thm::GenVara::ConUlPhi} (compare this to the application of the lower semi continuity condition in the proof in \cite{DemZeiLarge}).
This leads to the claimed lower bound for each $x \in X$ (see \eqref{vara::pf::lem::VaradLower::eq::VaradPhiR}).

\emptyline

Fix a $\epsilon>0$ and a $R$, then
\begin{align}
\label{vara::pf::lem::VaradLower::eq::ReplacePhi}
\begin{split}
	&\liminf_{N \rightarrow \infty} a_{N}^{-1} \log \E_{\P_{N}} \left[ e^{ a_{N}  \phi\left( \xi_{N} \right)  } \right]
	\\
	&\geq
	\liminf_{N \rightarrow \infty} a_{N}^{-1} \log  \E_{\P_{N}} \left[ e^{a_{N}  \phi\left( \xi_{N} \right)  }
			\1_{\left\{  \phi \left( \xi_N  \right)   > \ul{\phi}_{R} \left( \xi_N  \right)  -\epsilon \right\}} 
			  \right]
	\\
	&\geq
		-\epsilon 
		+
		\liminf_{N \rightarrow \infty} a_{N}^{-1} \log
				 \E_{\P_{N}} \left[ 
				 		e^{a_{N}  \ul{\phi}_{R}\left( \xi_{N} \right)  }
					\1_{\left\{  \phi \left( \xi_N  \right)   > \ul{\phi}_{R} \left( \xi_N  \right)  -\epsilon \right\}}  
				\right]
	.
\end{split}
\end{align}
By condition~\ref{vara::ass::thm::GenVara::BiggerEpsIrrelevant}, we get for $R$ large enough
that the right hand side of \eqref{vara::pf::lem::VaradLower::eq::ReplacePhi} is greater or equal to
\begin{align}
\label{vara::pf::lem::VaradLower::eq::ReplaceLowerEps}
	\geq
		- 2 \epsilon 
		+
		\liminf_{N \rightarrow \infty} a_{N}^{-1} \log 
					 \E_{\P_{N}} \left[ 
					 			e^{a_{N}  \ul{\phi}_{R}\left( \xi_{N} \right)   }  
								\right]
	.
\end{align}

We fix an arbitrary $x \in X$ with $I \left( x \right) < \infty$ and an arbitrary $\delta>0$.
By the condition \ref{vara::ass::thm::GenVara::ConUlPhi} there is a open neighbourhood $U_{x,\delta,R}$ of $x$ such that 
\begin{align}
	\inf_{y \in U_{x,\delta,R} \cap \supp \left\{ \xi_{N} \right\}}  \ul{\phi}_{R}\left( y \right) 
	\geq 
	\phi\left( x \right) -\delta
	.
\end{align}
Using this the large deviation lower bound of $\xi_{N}$, the right hand side of \eqref{vara::pf::lem::VaradLower::eq::ReplacePhi} is greater or equal to
\begin{align}
\label{vara::pf::lem::VaradLower::eq::VaradPhiR}
\begin{split}
	&\geq
		-2\epsilon + \liminf_{N \rightarrow \infty} a_{N}^{-1} \log  \E_{\P_{N}} \left[ e^{a_{N}  \ul{\phi}_{R}\left( \xi_{N} \right)   }  
									\1_{\left\{ \xi_{N} \in U_{x,\delta,R} \right\}} 
								\right]
	\\
	&\geq
		- 2 \epsilon 
		+
		\inf_{y \in U_{x,\delta,R} \cap \supp \left\{ \xi_{N} \right\}}  \ul{\phi}_{R}\left( y \right)  
		+\liminf_{N \rightarrow \infty} a_{N}^{-1} \log \P_{N}  \left[ \xi_{N} \in U_{x,\delta,R} \right]
	\\
	&\geq
		- 2 \epsilon 
		+
		\phi \left( x \right) - \delta - I\left( x \right) 
	.
\end{split}
\end{align}
Now let $\delta$ and $\epsilon$ tend to zero.
Hence we get for all  $x \in X$ (with $I \left( x \right) < \infty$),
\begin{align}
	\liminf_{N \rightarrow \infty} a_{N}^{-1} \log \E_{\P_{N}} \left[ e^{ a_{N}  \phi\left( \xi_{N} \right)  } \right]
	\geq 
	\phi \left( x \right)  - I\left( x \right) 
.
\end{align}
This implies the Varadhan lower bound  \eqref{eq::vara::lem::VaradLower}.
\end{proof}

\paragraph{Proof of the upper bound (Lemma~\protect\ref{vara::lem::VaradUpper})}
\label{vara::sec::PfUpperBound}

\begin{proof}
To prove the Varadhan upper bound we replace the function $\phi$ in the exponent of the left hand side of \eqref{eq::vara::lem::VaradUpper} by the function $\ol{\phi}_{R}$.
Therefore we split the expectation on the left hand side of \eqref{eq::vara::lem::VaradUpper} into the events when $\phi$ is greater or lower than a $M \in \R$.  Moreover we split it again in the events that $\phi$ exceeds $\ol{\phi}_{R}$ more or less than $\epsilon$.
Then we interchange the $\max$ and $\limsup$ to investigate the three situations separately (see \eqref{vara::pf::lem::VaradUpper::eq::ReplacePhi}).
Only the situation when $\phi$ exceeds $\ol{\phi}_{R}$ less than $\epsilon$ needs further investigation due to the conditions~\ref{vara::ass::thm::GenVara::DiffPhiPhiR} and~\ref{vara::ass::thm::GenVara::TailCondition} (see also \eqref{vara::pf::lem::VaradUpper::eq::InsertAll}).
In that situation we replace $\phi$ in the exponent by $\ol{\phi}_{R}$ with error $\epsilon$.
Then for $\ol{\phi}_{R}$ in the exponent we use parts of the usual proof of Varadhan's upper bound of \cite{DemZeiLarge} Lemma~4.3.6.
This leads to \eqref{vara::pf::lem::VaradUpper::eq::VardhanPhiR}.
Here we use condition \ref{vara::ass::thm::GenVara::ConOlPhi} (compare this to the application of the upper semi continuity condition in the proof in \cite{DemZeiLarge}).
Finally in \eqref{vara::pf::lem::VaradUpper::eq::InsertAll} we combine these calculations and conclude the claimed upper bound.

\emptyline

Fix an $M \in \R$, an $\epsilon>0$ and an $R>0$.
\begin{align}
\label{vara::pf::lem::VaradUpper::eq::ReplacePhi}
\begin{split}
	&\limsup_{N \rightarrow \infty} a_{N}^{-1} \log \E_{\P_{N}} \left[ e^{ a_{N}  \phi\left( \xi_{N} \right)  } \right]
	\\
	&\leq
	\left\{ \limsup_{N \rightarrow \infty} a_{N}^{-1} \log \left( \E_{\P_{N}} \left[ 
						e^{a_{N} \left(  \ol{\phi}_{R} \left( \xi_{N} \right)  \wedge M \right)  } \right] e^{a_{N} \epsilon} \right) \right\}
	\\
	&\textnormal{\hspace{1cm}} \vee
		\left\{ \limsup_{N \rightarrow \infty} a_{N}^{-1} \log  \E_{\P_{N}} \left[ 
						e^{a_{N}  \phi \left( \xi_{N} \right)    } 
						\1_{ \phi \left( \xi_{N} \right)  \geq M}
						\right]  \right\}
	\\
	&\textnormal{\hspace{1cm}} \vee
	\left\{ \limsup_{N \rightarrow \infty} a_{N}^{-1} \log \E_{\P_{N}} \left[ e^{a_{N}  \phi\left( \xi_{N} \right)  }
			\1_{\left\{ \phi \left( \xi_N  \right)   >  \ol{\phi}_{R} \left( \xi_N  \right) + \epsilon \right\}} \right] \right\}
	.
\end{split}
\end{align}
We show that the first $\limsup$ of the right hand side leads to the claimed upper bound.
Therefore we use parts of the proof of \cite{DemZeiLarge} Lemma~4.3.6 for the upper bound of Varadhan's Lemma.
Fix  $\alpha,\delta \in \R_{+}$.
The function $I$ is lower semi continuous and $\ol{\phi}_{R}$ satisfies \ref{vara::ass::thm::GenVara::ConOlPhi}. 
Hence (for $R$ large enough) for each $x \in X$ there is a open neighbourhood $A_{x}=A_{x,\delta,R} \subset X$ of $x$ such that 
\begin{align}
	\sup_{y \in A_{x} \cap \supp \left\{ \xi_{N} \right\}} \ol{\phi}_{R}  \left( y \right) 
	\leq
	 	\max \left\{ S^{*}, \phi\left( x \right)  \right\} + \delta
	\qquad \qquad \textnormal{ and } \qquad \qquad
		\inf_{y \in \ol{A_{x}}} I \left( y \right) 
	\geq
	 	I \left( x \right) - \delta
	.
\end{align}
Using these $A_{x}$ we find a finite cover $U_{i=1}^{N\left( \alpha \right) } A_{x_{i}}$ of the level sets $I^{-1} \left( \left[ 0,\alpha \right] \right) $ with $x_{i} \in I^{-1} \left( \left[ 0,\alpha \right] \right) $  by the compactness of $I^{-1} \left( \left[ 0,\alpha \right] \right) $.
Then
\begin{align}
	\E_{\P_{N}} \left[ e^{  a_{N} \left(  \ul{\phi}_{R} \left( \xi_{N} \right)   \wedge M \right) } \right]
	\leq
		\sum_{i =1}^{N\left( \alpha \right) } 
						e^{a_{N}  \left( \max \left\{ S^{*}, \phi\left( x_{i} \right)  \right\}  +\delta \right) } 
						\P^{N} \left[ \xi_{N} \in  \ol{A_{x_{i}}} \, \right]
		+ e^{a_{N} M}  \P^{N} \left[ \left(  \bigcup_{i_=1}^{ N\left( \alpha \right)  } A_{x_{i}}  \right)^{\!\!c}  \; \right] 
	.
\end{align}
By the large deviation upper bound of $\xi_{N}$ with rate function $I$ we get 
\begin{align}
\label{vara::pf::lem::VaradUpper::eq::VardhanPhiR}
\begin{split}
	&\limsup_{N \rightarrow \infty} \frac{1}{a_{N} } \log \E_{\P_{N}} \left[ e^{  a_{N} \left(  \ol{\phi}_{R} \left( \xi_{N} \right)   \wedge M \right) } \right]
	\\
	&\leq
	\max \left\{ \left\{\max_{i=1}^{N\left( \alpha \right) } \left\{  \max \left\{ S^{*}, \phi\left( x_{i} \right)  \right\}   - \inf_{x \in \ol{A_{x_{i}}} } I\left( x \right)  \right\}
			 + \delta \right\}  
		, \left\{ M- \alpha \right\} \right\}
	\\
	&\leq
	\max \left\{ 
		S^{*} + \delta
		,
		\left\{\max_{i=1}^{N\left( \alpha \right) } \left\{  \left( \phi  - I \right)\left( x_{i} \right)  \right\} + 2 \delta \right\}  
		,
		 \left\{ M- \alpha \right\} \right\}
	\\
	&\leq
	\max \left\{
		\left\{\sup_{x \in X : I\left( x \right)  < \infty} \left\{  \left( \phi - I \right)\left( x \right)  \right\}+ 2 \delta \right\} 
		,
		\left\{ M- \alpha \right\} \right\}
	.
\end{split}
\end{align}
Combining  \eqref{vara::pf::lem::VaradUpper::eq::VardhanPhiR} into \eqref{vara::pf::lem::VaradUpper::eq::ReplacePhi}, we get for $R> R\left( \alpha,\gamma,\delta \right) $
\begin{align}
\label{vara::pf::lem::VaradUpper::eq::InsertAll}
\begin{split}
	&\limsup_{N \rightarrow \infty} a_{N}^{-1} \log \E_{\P_{N}} \left[ e^{ a_{N}  \phi\left( \xi_{N} \right)   } \right]
	\\
	&\leq
	\left\{
	\epsilon 
	+
	\sup_{x \in X, I \left( x \right) < \infty}  \left\{  \phi\left( x \right)  - I\left( x \right)  \right\} + 2 \delta 
	\right\}
	\vee 
		\left\{
			\epsilon + M- \alpha
		\right\}
	\\
	&\qquad
		\vee
		\left\{
			 \limsup_{N \rightarrow \infty} a_{N}^{-1} \log 
			 \E_{\P_{N}} \left[ e^{ a_{N}  \phi \left( \xi_{N} \right)   } \1_{ \phi \left( \xi_{N} \right)  \geq M} \right]
		 \right\}
	\\
	&\qquad
	 \vee
			\left\{
			\limsup_{N \rightarrow \infty} a_{N}^{-1} \log 
					\E_{\P_{N}} \left[ e^{a_{N} \phi\left( \xi_{N} \right) } 
					\1_{\left\{ \phi \left( \xi_N  \right)   >  \ol{\phi}_{R} \left( \xi_N  \right) + \epsilon  \right\}} \right]
			\right\}
	.
\end{split}
\end{align}
Take now at first the limit $R \rightarrow \infty$. Then the last $\limsup$ vanishes due to condition~\ref{vara::ass::thm::GenVara::DiffPhiPhiR}.
Because $\alpha, \delta$ are arbitrary take then the limits  $\alpha \rightarrow \infty$, $\delta \rightarrow 0$.
Afterwards  let $M$ tend to infinity and apply condition~\ref{vara::ass::thm::GenVara::TailCondition}.
Finally $\epsilon$ is also arbitrary small.
Hence we have proven Lemma~\ref{vara::lem::VaradUpper}.
\end{proof}

\subsection{Simplifications of the conditions in Theorem~\ref{vara::thm::GenVara}}
\label{vara::sec::SimplerCond}

\paragraph{The condition  \ref{vara::ass::thm::GenVara::ConUlPhi} on \texorpdfstring{$\ul{\phi}_{R}$}{underlined phi R}}

\begin{lemma}
\label{vara::lem::SuffiCond::ulPhiR}
	The condition~\ref{vara::ass::thm::GenVara::ConUlPhi} on $\ul{\phi}_{R}$ is satisfied if 
	\begin{itemNoItemsep}
		\item
			each $\ul{\phi}_{R}: X \rightarrow \R$ is lower semi continuous and  
		\item
			$\left\{ \ul{\phi}_{R} \right\}$ converge pointwise to $\phi$ on  $\left\{ x \in X : I \left( x \right) < \infty \right\} $ for $R \rightarrow \infty$.
	\end{itemNoItemsep}
\end{lemma}

\begin{proof}
	By the lower semi continuity there is for each $x \in X$ with $I\left( x \right) <\infty$, each $\delta>0$ and each $R>0$, a neighbourhood $ U_{x,\delta,R} \subset X$ of $x$ such that
	\begin{align}
		\inf_{y \in  U_{x,\delta,R}} \ul{\phi}_{R}\left( y \right)  \geq \ul{\phi}_{R}\left( x \right)  - \frac{\delta}{2}
		.
	\end{align}
	Moreover by the pointwise convergence there is a $R^{*}_{x,\delta}$ such that for all $R>R^{*}_{x,\delta}$
	\begin{align}
		\ul{\phi}_{R}\left( x \right) 
		>
		\phi\left( x \right)  - \frac{\delta}{2}
		.
	\end{align}
	Therefore the claimed condition~\ref{vara::ass::thm::GenVara::ConUlPhi} is proven.
\end{proof}

\paragraph{The condition \ref{vara::ass::thm::GenVara::ConOlPhi}  on \texorpdfstring{$\ol{\phi}_{R}$}{overlined phi R} }

\begin{lemma}
\label{vara::lem::SuffiCond::olPhiR}
	The condition \ref{vara::ass::thm::GenVara::ConOlPhi} on $\ol{\phi}_{R}$ is satisfied if 
	\begin{itemNoItemsep}
		\item
			each $\ol{\phi}_{R}: X \rightarrow \R$ is upper semi continuous and  
		\item
			for each $\delta>0$ there is a $R^{*}_{\delta} \in \R_{+}$ such that for all $R \geq R^{*}_{\delta}$, 
			$\ol{\phi}_{R}\left( x \right)  \leq \max \left\{ S^{*}, \phi\left( x \right)  \right\}+ \delta$
			for all  $x \in \left\{ x \in X : I \left( x \right) < \infty \right\} $. 
	\end{itemNoItemsep}
\end{lemma}

\begin{proof}
	By the upper semi continuity we have that for each $x \in  \left\{ x \in X : I \left( x \right) < \infty \right\} $, each $\delta$ and each $R \in \R_{+}$ there is a open set $A_{x,\delta,R} \subset X$ such that 
	\begin{align}
		\sup_{y \in A_{x,\delta,R} } \ol{\phi}_{R}\left( y \right)  
		\leq
			\ol{\phi}_{R}\left( x \right)  + \delta
		.
	\end{align}

	If we have $R>R^{*}_{\delta}$, then we know by the second property that the right hand side is lower or equal to
	\begin{align}
		\leq 
			\max \left\{ S^{*}, \phi\left( x \right)  \right\}+ 2 \delta
		.
	\end{align}
\end{proof}
	
	 The class of functions that satisfies  \ref{vara::ass::thm::GenVara::ConOlPhi} is in general larger than the class defined in Lemma~\ref{vara::lem::SuffiCond::olPhiR}. 
	 For an abstract example, see Chapter~\ref{vara::sec::Example::Class}.

\paragraph{The condition \protect\ref{vara::ass::thm::GenVara::BiggerEpsIrrelevant}}

\begin{lemma}
\label{vara::lem::VaradLower::CondSatis1}
The condition~\ref{vara::ass::thm::GenVara::BiggerEpsIrrelevant} holds if 
\begin{align*}
	\forall \epsilon>0
	\quad
	\exists R_{\epsilon}>0
	\quad \textnormal{ such that } \quad
	\forall R>R_{\epsilon}
	\quad
	\exists N_{\epsilon,R} \in \N
	\quad \textnormal{ such that } \quad
	\forall N>N_{\epsilon,R}
\end{align*}
\begin{align}
\begin{split}
	\E_{\P^{N}} 
		\left[ 
			e^{a_{N}  \ul\phi_{R}\left( \xi_{N} \right)  }
			\1_{\left\{  \phi \left( \xi_N  \right)   > \ul{\phi}_{R} \left( \xi_N  \right)  -\epsilon \right\}} 
		\right]
	\geq
	\E_{\P^{N}} 
		\left[ 
			e^{a_{N}  \ul\phi_{R}\left( \xi_{N} \right)  }
		\1_{\left\{  \phi \left( \xi_N  \right)   < \ul{\phi}_{R} \left( \xi_N  \right)  -\epsilon \right\}} 
		\right]
		.
\end{split}
\end{align}
\end{lemma}

\begin{proof}
	For $a,b \geq 0$,
	\begin{align}
		\log \left( a+b \right) 
		\leq
			\max \left\{ \log \left( 2 a \right) , \log \left( 2 b \right) \right\}
		=
			\max \left\{ \log \left( a \right) , \log \left( b \right) \right\} + \log \left( 2 \right)
		\leq
			\log\left( a + b \right) + \log \left( 2 \right)
		.
\end{align}
This implies by the assumption of this lemma that
	\begin{align}
	\label{vara::eq::lem::VaradLower::CondSatisf1::max}
	\begin{split}
		&\lim_{R \rightarrow \infty} \liminf_{N \rightarrow \infty}
				a_{N}^{-1} \log 
					\E_{\P^{N}} 
						\left[ 
							e^{a_{N} \ul\phi_{R}\left( \xi_{N} \right) }
						\right]
		\\
		&=
				\lim_{R \rightarrow \infty} \liminf_{N \rightarrow \infty}
				a_{N}^{-1} 
				\log
				\max \left\{ 
					\E_{\P^{N}} 
						\left[ 
							e^{a_{N} \ul\phi_{R}\left( \xi_{N} \right) }
						\1_{\left\{  \phi \left( \xi_N  \right)   > \ul{\phi}_{R} \left( \xi_N  \right)  -\epsilon \right\}} 
						\right]
				\right.
				\\
				&\hspace{4cm}
				\left.
				, 
					\E_{\P^{N}} 
						\left[ 
							e^{a_{N} \ul\phi_{R}\left( \xi_{N} \right) }
						\1_{\left\{  \phi \left( \xi_N  \right)   < \ul{\phi}_{R} \left( \xi_N  \right)  -\epsilon \right\}} 
						\right]
					\right\}
		\\
		&=
		\lim_{R \rightarrow \infty} \liminf_{N \rightarrow \infty}
				a_{N}^{-1} 
					\log 
					\E_{\P^{N}} 
						\left[ 
							e^{a_{N} \ul\phi_{R}\left( \xi_{N} \right) }
						\1_{\left\{  \phi \left( \xi_N  \right)   > \ul{\phi}_{R} \left( \xi_N  \right)  -\epsilon \right\}}  
						\right]
		.
	\end{split}
	\end{align}
\end{proof}

\begin{lemma}
\label{vara::lem::VaradLower::CondSatis}
The condition~\ref{vara::ass::thm::GenVara::BiggerEpsIrrelevant} holds if 
\begin{enuRom}
	\item
	\label{vara::ass::VaradLower::BiggerEpsIrrelevant::GrEps}
		$
		\forall C>0
		\quad
		\forall \epsilon>0
		\quad
		\exists R_{C,\epsilon}>0
		\quad \textnormal{ such that } \quad
		\forall R>R_{C,\epsilon}
		\quad
		\exists N_{C,\epsilon,R} \in \N
		\quad \textnormal{ such that } \quad
		\forall N>N_{C,\epsilon,R}
	$
		\begin{align}
		\label{vara::eq::VaradLower::BiggerEpsIrrelevant::GrEps}
				\E_{\P^{N}} 
					\left[ 
						e^{a_{N}  \ul\phi_{R}\left( \xi_{N} \right)  }
					\1_{\left\{  \phi \left( \xi_N  \right)   < \ul{\phi}_{R} \left( \xi_N  \right)  -\epsilon \right\}} 
					\right]
			\leq
				e^{- a_{N} C }
			\quad \textnormal{ or}
		\end{align}
	\item
		$
			\forall \epsilon>0
			\quad
			\exists \beta_{\epsilon} \in \left( 0,1 \right]
			\exists R_{\epsilon} >0
			\quad \textnormal{ such that } \quad
			 \forall R>R_{\epsilon}
			\quad
			\forall N \in \N
		$
		\begin{align}
		\label{vara::eq::VaradLower::BiggerEpsIrrelevant::beta}
				\E_{\P^{N}} 
					\left[ 
						e^{a_{N}  \ul\phi_{R}\left( \xi_{N} \right)  }
					\1_{\left\{  \phi \left( \xi_N  \right)   > \ul{\phi}_{R} \left( \xi_N  \right)  -\epsilon \right\}} 
					\right]
				\geq
				\beta 
					\E_{\P^{N}} 
						\left[ 
							e^{a_{N} \ul\phi_{R}\left( \xi_{N} \right) }
						\right]
				.
			\end{align}

\end{enuRom}

\end{lemma}

\begin{proof}
	We only have to show that the left hand side of \eqref{vara::eq::VaradLower::BiggerEpsIrrelevant} is greater or equal to the right hand side.
	\begin{enuRomNoIntendBf}
	\item
	
	By the condition of this lemma we get as in \eqref{vara::eq::lem::VaradLower::CondSatisf1::max} that  for each $C>0$ 
	\begin{align}
	\begin{split}
		&\lim_{R \rightarrow \infty} \liminf_{N \rightarrow \infty}
				a_{N}^{-1} \log 
					\E_{\P^{N}} 
						\left[ 
							e^{a_{N} \ul\phi_{R}\left( \xi_{N} \right) }
						\right]
		\\
		&\leq
				\lim_{R \rightarrow \infty} \liminf_{N \rightarrow \infty}
				\max \left\{ 
					a_{N}^{-1} 
					\log
					\E_{\P^{N}} 
						\left[ 
							e^{a_{N} \ul\phi_{R}\left( \xi_{N} \right) }
						\1_{\left\{  \phi \left( \xi_N  \right)   > \ul{\phi}_{R} \left( \xi_N  \right)  -\epsilon \right\}} 
						\right]
				,
				-C
				\right\}
		\\
		&=
				\max \left\{ 
					\lim_{R \rightarrow \infty} \liminf_{N \rightarrow \infty}
					a_{N}^{-1} 
					\log
					\E_{\P^{N}} 
						\left[ 
							e^{a_{N} \ul\phi_{R}\left( \xi_{N} \right) }
						\1_{\left\{  \phi \left( \xi_N  \right)   > \ul{\phi}_{R} \left( \xi_N  \right)  -\epsilon \right\}} 
						\right]
				,
				-C
				\right\}
			.
	\end{split}
	\end{align}
	Now we let $C$ tend to infinity, what proves the claim.

	\item
		In this case the equality \eqref{vara::eq::VaradLower::BiggerEpsIrrelevant} follows by inserting \eqref{vara::eq::VaradLower::BiggerEpsIrrelevant::beta} and using that $\frac{\log \left( \beta \right) }{a_{N}} \rightarrow 0$.

	\end{enuRomNoIntendBf}\vspace{-\baselineskip}
\end{proof}

The following lemma is a corollary of Lemma~\ref{vara::lem::VaradLower::CondSatis}~\ref{vara::ass::VaradLower::BiggerEpsIrrelevant::GrEps}.

\begin{lemma}
\label{vara::lem::VaradLower::CondSatis::CondSatis}
	If 	$\forall C>0
			\quad
			\forall \epsilon>0
			\quad
			\exists R_{C,\epsilon}>0
			\quad \textnormal{such that} \quad
			\forall R>R_{C,\epsilon}
			\quad
			\exists N_{C,\epsilon,R} \in \N
			\quad \textnormal{such that} \quad
		$
	\begin{align}
	\label{vara::lem::VaradLower::CondSatis::CondSatis::eq}
		e^{a_{N} \sup_{x \in \supp \left\{ \xi_{N} \right\}} \ul\phi_{R} \left( x \right) } \P^{N} \left[   \phi \left( \xi_N  \right)  < \ul{\phi}_{R} \left( \xi_N  \right)  -\epsilon  \right]
		\leq
		e^{-a_{N} C }
		,
	\end{align}
	for all $N>N_{C,\epsilon,R}$,
	then condition Lemma~\ref{vara::lem::VaradLower::CondSatis}~\ref{vara::ass::VaradLower::BiggerEpsIrrelevant::GrEps}
	and hence Condition~\ref{vara::ass::thm::GenVara::BiggerEpsIrrelevant} hold. 
\end{lemma}

\begin{remark}
\label{vara::rem::VaradLower::CondSatis::CondSatis}
The supremum in \eqref{vara::lem::VaradLower::CondSatis::CondSatis::eq} could  be restricted to  $\left\{ x :   \phi \left( x  \right)  < \ul{\phi}_{R} \left( x  \right)  -\epsilon \right\}$.
\end{remark}

In the following lemma we show that a simpler condition (than the condition Lemma~\ref{vara::lem::VaradLower::CondSatis}~\ref{vara::ass::VaradLower::BiggerEpsIrrelevant::GrEps})  implies the condition~\ref{vara::ass::thm::GenVara::BiggerEpsIrrelevant}. 
The proof of this lemma requires parts of the proof of the Varadhan lower bound (of Chapter~\ref{vara::sec::PfLowerBound}).
Of course not those parts that require the condition~\ref{vara::ass::thm::GenVara::BiggerEpsIrrelevant}.

\begin{lemma}
\label{vara::lem::VaradLower::SimplerCondTo::BiggerEpsIrrel}
	Let the condition~\ref{vara::ass::thm::GenVara::ConUlPhi} hold.
	If
	\begin{align*}
	\forall \epsilon>0
		\quad
		\exists \gamma_{\epsilon}>0
		\quad
		\exists R^{*}_{\epsilon}>0
		\quad \textnormal{such that} \quad
		\forall R>R_{\epsilon}
		\quad
		\exists N_{R} \in \N
		\quad \textnormal{such that} \quad
		\forall N>N^{*}_{R}
	\end{align*}\vspace{-1.5em}
	\begin{align}
	\label{vara::eq::lem::VaradLower::SimplerCondTo::BiggerEpsIrrel}
		\E_{\P^{N}} \left[ e^{a_{N} \ul{\phi}_{R}\left( \xi_{N} \right) }  
			\1_{\left\{  \phi \left( \xi_N  \right)  < \ul{\phi}_{R} \left( \xi_N  \right)  -\epsilon \right\}} 
				\right] 
		\leq
			e^{a_{n} \left( S^{*} - \gamma_{\epsilon} \right)}
		,
	\end{align}
	then the condition~\ref{vara::ass::thm::GenVara::BiggerEpsIrrelevant} holds.

\end{lemma}

\begin{proof}
	Let us fix a $\epsilon>0$ small enough.
	Then 
	\begin{align}
	\begin{split}
			\liminf_{N \rightarrow \infty} a_{N}^{-1} \log 
					 \E_{\P_{N}} \left[ 
					 			e^{a_{N}  \ul{\phi}_{R}\left( \xi_{N} \right)   }  
								\right]
			&=
			\liminf_{N \rightarrow \infty} 
				\max \left\{
					a_{N}^{-1} \log 
					 \E_{\P_{N}} \left[ 
					 			e^{a_{N}  \ul{\phi}_{R}\left( \xi_{N} \right)   }  
							\1_{\left\{  \phi \left( \xi_N  \right)  > \ul{\phi}_{R} \left( \xi_N  \right)  -\epsilon \right\}} 
								\right]
					\right.
			\\
			&\hspace{2cm}
				\left.
					,
					a_{N}^{-1} \log 
					 \E_{\P_{N}} \left[ 
					 			e^{a_{N}  \ul{\phi}_{R}\left( \xi_{N} \right)   }  
							\1_{\left\{  \phi \left( \xi_N  \right)  < \ul{\phi}_{R} \left( \xi_N  \right)  -\epsilon \right\}} 
								\right]
				\right\}
			\\
			&\hspace{-1cm}
			\leq
			\max \left\{
				\liminf_{N \rightarrow \infty} 
					a_{N}^{-1} \log 
					 \E_{\P_{N}} \left[ 
					 			e^{a_{N}  \ul{\phi}_{R}\left( \xi_{N} \right)   }  
							\1_{\left\{  \phi \left( \xi_N  \right)  > \ul{\phi}_{R} \left( \xi_N  \right)  -\epsilon \right\}} 
								\right]
					,
				S^{*}-\gamma_{\epsilon}
			\right\}
		.
	\end{split}
	\end{align}
	
	The left hand side is larger or equal to $S^{*}$ as shown in the proof of Lemma~\ref{vara::lem::VaradLower} (\eqref{vara::pf::lem::VaradLower::eq::ReplaceLowerEps} to \eqref{vara::pf::lem::VaradLower::eq::VaradPhiR} requires only condition~\ref{vara::ass::thm::GenVara::ConUlPhi}).
	This implies that 
	\begin{align}
		\liminf_{N \rightarrow \infty} a_{N}^{-1} \log 
						 \E_{\P_{N}} \left[ 
						 			e^{a_{N}  \ul{\phi}_{R}\left( \xi_{N} \right)   }  
									\right]
	=
	\liminf_{N \rightarrow \infty} 
						a_{N}^{-1} \log 
						 \E_{\P_{N}} \left[ 
						 			e^{a_{N}  \ul{\phi}_{R}\left( \xi_{N} \right)   }  
								\1_{\left\{  \phi \left( \xi_N  \right)  > \ul{\phi}_{R} \left( \xi_N  \right)  -\epsilon \right\}} 
									\right]
					,
	\end{align}	
	for all $R>R^{*}_{\epsilon}$, i.e. the condition~\ref{vara::ass::thm::GenVara::BiggerEpsIrrelevant} holds. 
\end{proof}

\paragraph{The condition \protect\ref{vara::ass::thm::GenVara::TailCondition}}

\begin{lemma}
\label{vara::lem::VaradUpper::MomentThenTail}
	The tail condition~\ref{vara::ass::thm::GenVara::TailCondition}  is satisfied if for some $\gamma>1$
	\begin{align}
	\label{vara::lem::VaradUpper::MomentThenTail::cond}
		\limsup_{N \rightarrow \infty} a_{N}^{-1} \log
				 \E_{\P_{N}} \left[ e^{\gamma a_{N}  \phi\left( \xi_{N} \right)  } \right]
	 	< 
			\infty
	 	.
	\end{align}
\end{lemma}

\begin{proof}
	As shown in Lemma 4.3.8 in \cite{DemZeiLarge} the moment condition implies the tail condition. The continuity of $\phi$ is not required.
\end{proof}

\begin{lemma}
\label{vara::lem::VaradUpper::AsympTail}
The condition~\ref{vara::ass::thm::GenVara::DiffPhiPhiR} and the following asymptotic tail condition
	\begin{align}
	\label{vara::eq::lem::VaradUpper::AsympTail}
		\lim_{M \rightarrow \infty} 
		\lim_{R \rightarrow \infty}
		\limsup_{N \rightarrow \infty} 
			a_{N}^{-1} \log 
	 				\E_{\P_{N}} 
	 					\left[ 
							e^{ a_{N} \ol\phi_{R} \left( \xi_{N} \right) } 
							\1_{ \ol\phi_{R} \left( \xi_{N} \right) \geq M} 
						\right]
		\leq
			S^{*}
	\end{align}
imply the tail condition~\ref{vara::ass::thm::GenVara::TailCondition}.
\end{lemma}

\begin{proof}
For all $R>0$ we have
	\begin{align}
	\begin{split}
		\E_{\P_{N}} 
	 		\left[ 
				e^{ a_{N} \phi \left( \xi_{N} \right) } 
				\1_{ \phi \left( \xi_{N} \right) \geq M} 
			\right]
		&\leq
			e^{a_{N} \epsilon}
			\E_{\P_{N}} 
 				\left[ 
					e^{ a_{N} \ol\phi_{R} \left( \xi_{N} \right) } 
					\1_{ \ol\phi_{R} \left( \xi_{N} \right) \geq M-\epsilon} 
				\right]
			\\
			&\quad+
			\E_{\P_{N}} 
 				\left[ 
					e^{ a_{N} \phi \left( \xi_{N} \right) } 
					\1_{\left\{ \phi \left( \xi_N  \right)   >  \ol{\phi}_{R} \left( \xi_N  \right) + \epsilon  \right\}} 
				\right]
		.
	\end{split}
	\end{align}
By taking first the maximum of the summands of the right hand side and then the limit $R \rightarrow \infty$, the contribution of the second summand vanishes due to condition~\ref{vara::ass::thm::GenVara::DiffPhiPhiR}. 
Finally we take the limit $M \rightarrow \infty$ and use  \eqref{vara::eq::lem::VaradUpper::AsympTail} to conclude the claimed condition~\ref{vara::ass::thm::GenVara::TailCondition}.
\end{proof}

 \subsection{Example}
 \label{vara::sec::Example}

  If $\phi$ is continuous, then all conditions of Theorem~\ref{vara::thm::GenVara} simplify to the usual conditions of  Varadhan's lemma (see Remark~\ref{vara::rem::usualVaradhan}).
  
  We state now at first a simple example of sums Bernoulli random variables and show that the generalised Varadhan's lemma (Theorem~\ref{vara::thm::GenVara}) hold for functions with one jump point.
  Finally we show that an abstract setting implies the conditions of Theorem~\ref{vara::thm::GenVara}. 
  In this setting the function $\phi$ might be nowhere continuous.

 \paragraph{A simple example}
 
 Let $\theta^{i}$ be i.i.d. random variables with distribution  $Bern\left( p \right) $ for $p>\oh$, i.e. $\theta^{i} =1$ with probability $p$ and $\theta^{i} = -1$ with probability $1-p$.
 Define the random variables $\xi_{N}= \frac{1}{N} \sum_{i=1}^{N} \theta^{i}$.
 These random variables satisfy a large deviation principle with rate function
 \begin{align}
 	\Lambda^{*} \left( x \right)  
 	=
 	\oh \left[ \left( x+1 \right) \log \left( \frac{x+1}{p} \right) 
 	+
 	\left( 1-x \right) \log \left( \frac{1-x}{1-p} \right) \right] 
 	-
 	\log\left( 2 \right) 
 	,
 \end{align}
 for  $x \in [-1,1]$ and infinity otherwise.
 
 For a $\alpha \in \left( 0, \Lambda^{*} \left( 0 \right) \right)$, set
 \begin{align}
 	\phi \left( x \right)  = 
	\begin{cases}
		0 \qquad &\textnormal{if } x > 0
		\\
		\alpha &\textnormal{else}
		.
	\end{cases}
 \end{align}

We show now that the conditions of Theorem~\ref{vara::thm::GenVara} hold.
Set  $\ol{\phi}_{R}=\phi$.
Then condition~\ref{vara::ass::thm::GenVara::DiffPhiPhiR} holds. 
Moreover condition~\ref{vara::ass::thm::GenVara::ConOlPhi} is satisfied, because $\phi$ is upper semi continuous.
By the boundedness of $\phi$ also the tail condition~\ref{vara::ass::thm::GenVara::TailCondition} holds.
Define
\begin{align}
	\ul{\phi}_{R} \left( x \right)
	=
	\begin{cases}
		0 \qquad &\textnormal{if } x \geq \frac{1}{R}
		\\
		\alpha &\textnormal{else}
		.
	\end{cases}
\end{align}
This function is lower semi-continuous and converges pointwise to $\phi$.
This implies by Lemma~\ref{vara::lem::SuffiCond::ulPhiR}  the condition~\ref{vara::ass::thm::GenVara::ConUlPhi}.

We only have to show that condition~\ref{vara::ass::thm::GenVara::BiggerEpsIrrelevant} holds.
We do this now with help of the Lemma~\ref{vara::lem::VaradLower::CondSatis1}.

Then
\begin{align}
\begin{split}
	&\E_{\P^{N}} 
					\left[ 
						e^{N  \ul\phi_{R}\left( \xi_{N} \right)  }
					\1_{\left\{ \phi \left( \xi_N  \right)   <  \ol{\phi}_{R} \left( \xi_N  \right) - \epsilon \right\}} 
					\right]
	=
		e^{N \alpha}
		\P^{N} \left[ \xi_{N} \in \left( 0, \frac{1}{R} \right) \right]
	\\
	&\leq
		e^{N \alpha}
		\P^{N} \left[ \xi_{N} < \frac{1}{R}  \right]
	\leq
		e^{N \alpha}
		e^{-N \inf_{x \leq \frac{1}{R}}\left( \Lambda^{*} \left( x \right) +o_{N}\left( 1 \right)  \right) }
	\leq
		e^{N \alpha}
		e^{-N \left( \Lambda^{*} \left( \frac{1}{R} \right) +o_{N}\left( 1 \right)  \right) }
	,
\end{split}
\end{align}
for $R$ large enough, because $p>\oh$.
Moreover
\begin{align}
	&\E_{\P^{N}} 
			\left[ 
				e^{a_{N}  \ul\phi_{R}\left( \xi_{N} \right)  }
			\1_{\left\{ \phi \left( \xi_N  \right)   >  \ul{\phi}_{R} \left( \xi_N  \right) - \epsilon \right\}} 
			\right]
	\geq
				\P^{N} \left[ \xi_{N} \geq \frac{1}{R} \right]
	\geq
		e^{- N 
			\left(
			\inf_{x \geq \frac{1}{R}-\epsilon } \Lambda^{*} \left( x \right)  
			+ o_{N}\left( 1 \right) 
			\right)
		}
	=
		e^{-N o_{N}\left( 1 \right) }
	,
\end{align}
for $R$ large enough, because $p>\oh$.
 Hence the condition of Lemma~\ref{vara::lem::VaradLower::CondSatis1} is satisfied because $\alpha < \Lambda^{*}\left( 0 \right)$.

\paragraph{Class of examples}
\label{vara::sec::Example::Class}
 
 A class of examples is given by the following abstract setting. We show for example in Chapter~\ref{sec::LMF::LDPLN}, that the concrete example \eqref{SDE::LocalMF} of the local mean field model satisfies these conditions.
 
 \begin{enuRom}[label=\emph{(\ref*{vara::sec::Example::Class}.\roman*)},
 	ref=(\ref{vara::sec::Example::Class}.\roman*) ,
 	leftmargin=5em,
 	labelwidth=5em]
  \item
  \label{vara::ex:Irate}
		  $I$ be the good rate function as stated in Theorem~\ref{vara::thm::GenVara} and 
		  $\phi$ be a measurable function $X \rightarrow \R \cup \left\{ -\infty , \infty \right\}$.
 \item
 \label{vara::ex:MR}
 	For each $R \in \N$, let $M_{R} \subset X$ be a closed subset, such that $M_{R} \subset M_{R+1}$.
	
	We set $\Xi \defeq \bigcup M_{R}$.
\item
\label{vara::ex:Iinfty}
	$I\left( x \right)  = \infty$ if $x \not \in \Xi$.
\item
\label{vara::ex:PhiMRcont}
	For each $R \in \R$, $x \in M_{R}$ with $I \left( x \right) < \infty$ and each sequence $\left\{ x^{(n)} \right\} \subset M_{R} \cap \supp \left\{ \xi_{N} \right\}$ with $x^{(n)} \rightarrow x$, the following convergence holds: $\phi \left( x^{(n)} \right) \rightarrow \phi \left( x \right)$.
	
	This holds in particular when the restriction of $\phi$ to $M_{R}$ is continuous for each $R$.
\item
\label{vara::ex:PhiBound}
	There is an $\alpha : \R_{+} \rightarrow \R_{+}$ and a $N^{*} \in \N$ such that, for all $N>N^{*}$,
	$\phi\left( \xi_N \right)  \leq \alpha\left( R \right) $, when $\xi_N \in M_{R} $.
	
\item
\label{vara::ex:ProbOutMRBound}
	There is a function  $\beta : \R_{+} \rightarrow \R$, such that
	$\P^{N} \left[ \xi_{N} \not \in M_{R} \right] \leq e^{- a_{N} \beta\left( R \right) }$.
\item
\label{vara::ex:AlphaMBeta}
	$\lim_{R \rightarrow \infty} \alpha\left( R \right)  - \beta\left( R \right)  = -\infty$.
\item
\label{vara::ex:DiffCondAlmost}
~\vspace{-1em}
		\begin{align}
		\begin{split}
			\limsup_{R\to\infty}\limsup_{N\to\infty}
			\frac{1}{a_{N}}\log 
			\E_{\P^{N}} \left[  e^{ a_{N} \phi\left( \xi_N \right)  }
			\1{\left\{ \xi_N\not\in M_{R}\right\}} 
			\right]
			=
			-\infty .
		\end{split}
		\end{align}
\item	
\label{vara::ex:Moment}
	$\phi$ satisfies the tail condition~\ref{vara::ass::thm::GenVara::TailCondition} or it satisfies \eqref{vara::lem::VaradUpper::MomentThenTail::cond}.
 \end{enuRom}

\begin{remark}
		Note that the most important condition is \ref{vara::ex:PhiMRcont}.
		In this abstract setting, the function $\phi$ might be nowhere continuous, as long as this condition holds.
		The other conditions are only necessary to reduce the analysis to sequences the fit in the setting of \ref{vara::ex:PhiMRcont}.
\end{remark}

We define
	\begin{align}
		\label{vara::ex::DefULOLPhi}
		\ul{\phi_{R}}\left( x \right)  
		=
			\begin{cases}
				\phi\left( x \right)  \quad &\textnormal{if } x \in M_{R}
				 \\ 
				\alpha\left( R \right) 
				&\textnormal{otherwise}
			\end{cases}
		\hspace{2cm}
		\ol{\phi_{R}}\left( x \right)  
		=
			\begin{cases}
				\phi\left( x \right)  \quad &\textnormal{if } x \in M_{R}
				\\ 
				S^{*}
				&\textnormal{otherwise.}
			 \end{cases} 
	\end{align}

\begin{lemma}
These conditions allow the application of Theorem~\ref{vara::thm::GenVara}.
\end{lemma}

\begin{proof}
\begin{steps}

\step[{Condition \ref{vara::ass::thm::GenVara::ConUlPhi}}]

The $\ul{\phi_{R}}$ is measurable, because $\ul{\phi_{R}}$ restricted to $M_{R}$ is measurable (by \ref{vara::ex:Irate}),  the $M_{R}$ are closed and $\ul{\phi_{R}}$ is constant outside of $M_{R}$.

To show the other part of this condition, fix an $R>0$, a $\delta>0$ and an $x \in M_{R}$, with $I \left( x \right) < \infty$.
On $M_{R}$,  $\ul{\phi_{R}} = \phi$.
Hence by~\ref{vara::ex:PhiMRcont}, there is a open set $\what U_{x,\delta,R} \subset M_{R}$ such that 
\begin{align}
	\inf_{y \in \what U_{x,\delta,R} \cap \supp \left\{ \xi_{N} \right\}} \ul{\phi_{R}}\left( y \right)  
	\geq  
	\phi \left( x \right)  - \delta
	.
\end{align}
Denote by $U_{x,\delta,R}$ the open subset of $X$, such that $\what U_{x,\delta,R} = U_{x,\delta,R} \cap M_{R}$. 
For each $y \in  U_{x,\delta,R} \backslash \what U_{x,\delta,R}$, with $y \in \supp \left\{ \xi_{N} \right\}$,  $\ul{\phi_{R}}\left( y \right) = \alpha\left( R \right)  \geq \ul{\phi_{R}}\left( x \right) $ by~\ref{vara::ex:PhiBound}. 
Hence $\ul{\phi_{R}}$ satisfies condition~\ref{vara::ass::thm::GenVara::ConUlPhi} for $x \in M_{R}$.

For $x \not \in M_{R}$, there is a open neighbourhood $U_{x,\delta,R}$ of $x$ with $U_{x,\delta,R} \cap  M_{R} = \emptyset$ (because $M_{R}$ is closed).
On $U_{x,\delta,R}$, $\ul{\phi_{R}}$  is constant and equals $S^*$.
This implies condition~\ref{vara::ass::thm::GenVara::ConUlPhi}.

\step[{Condition~\ref{vara::ass::thm::GenVara::ConOlPhi}}]

The function $\ol{\phi_{R}}$ is measurable by the same arguments as $\ul{\phi_{R}}$.

Fix an arbitrary $R>0$, a $\delta>0$ and a  $x \in M_{R}$ with $I \left( x \right) < \infty$.
By \ref{vara::ex:PhiMRcont}, there is an open neighbourhood  $\what{U}_{x,\delta,R} \subset M_{R}$ such that
			\begin{align}
				\sup_{y \in  \what{U}_{x,\delta,R} \cap \supp \left\{ \xi_{N} \right\}} 
								 \ol{\phi}_{R}\left( y \right)  
				\leq
				\phi \left( x \right)+ \delta
				.
			\end{align}
Denote by $U_{x,\delta,R}$ the open subset of $X$, such that $\what{U}_{x,\delta,R} = U_{x,\delta,R}  \cap M_{R}$.
Then 
		\begin{align}
				\sup_{y \in  U_{x,\delta,R} \cap \supp \left\{ \xi_{N} \right\}}  \ol{\phi}_{R}\left( y \right)  
				\leq 
					\max \left\{
							\sup_{y \in  \what{U}_{x,\delta,R} \cap \supp \left\{ \xi_{N} \right\}}  \ol{\phi}_{R}\left( y \right)  
							,
							S^{*}
					\right\}
				\leq
					\max \left\{
							\phi \left( x \right)+ \delta
							,
							S^{*}
					\right\}
				.
			\end{align}
The case when $x \not\in  M_{R}$, 
can be treat as in condition~\ref{vara::ass::thm::GenVara::ConUlPhi}.
This implies condition~\ref{vara::ass::thm::GenVara::ConOlPhi}.

	\step[{Condition~\ref{vara::ass::thm::GenVara::BiggerEpsIrrelevant} holds}]
	
	To show this condition we show the sufficient condition of Lemma~\ref{vara::lem::VaradLower::CondSatis::CondSatis}.
	By the definition of $\ul{\phi_{R}}$ and \eqref{vara::ex:PhiBound}, we know that $ \phi \left( \xi_N  \right)  < \ul{\phi}_{R} \left( \xi_N  \right)  -\epsilon$ implies that $\xi_{N} \not \in M_{R}$.
	Hence
\begin{align}
\begin{split}	
	&e^{a_{N}  \sup_{x \in \supp \left\{ \xi_{N} \right\}} \ul{\phi_{R}} \left( x \right) }
	\P^{N} \left[  \phi \left( \xi_N  \right)  < \ul{\phi}_{R} \left( \xi_N  \right)  -\epsilon \right]
	\\
	&\leq
		e^{a_{N} \alpha\left( R \right)  }
		\P^{N} \left[ \xi_{N} \not \in M_{R} \right]
	\leq
		e^{a_{N} \left( \alpha\left( R,N \right)  - \beta\left( R \right)   \right) }
\end{split}
\end{align}
by \ref{vara::ex:PhiBound}, \ref{vara::ex:ProbOutMRBound}.
Finally \ref{vara::ex:AlphaMBeta} implies the condition of Lemma~\ref{vara::lem::VaradLower::CondSatis::CondSatis}.

	\step[{Condition~\ref{vara::ass::thm::GenVara::DiffPhiPhiR} holds}]
	
	This condition is satisfied by \ref{vara::ex:DiffCondAlmost} and because  $\phi \left( \xi_N  \right)   >  \ol{\phi}_{R} \left( \xi_N  \right) + \epsilon $ implies that $\xi_{N} \not\in M_{R}$.
	
	\step[{Condition~\ref{vara::ass::thm::GenVara::TailCondition} holds}]
	
	We assume this in \ref{vara::ex:Moment}.
\end{steps}\vspace{-\baselineskip}
\end{proof}

%\bibliographystyle{gerplain} 
%\bibliographystyle{alpha} 

%%%%%
%%%%%

\AdresseAtEnd

\end{document}